\documentclass[11pt]{article}
\RequirePackage[OT1]{fontenc}
\RequirePackage{amsthm,amsmath}
\RequirePackage{hyperref}
\usepackage{amsfonts}
\usepackage{amssymb}
\usepackage{amstext}
\usepackage{parskip}
\usepackage{fullpage}
\usepackage{graphicx}
\usepackage{subcaption}

\RequirePackage[numbers]{natbib}
\usepackage{amsfonts,amssymb,amsthm,amsmath,enumerate}

\DeclareMathOperator{\E}{\mathbvb{E}}
\DeclareMathOperator*{\Span}{span}

\def \P {\mathbb{P}}
\def \R {\mathbb{R}}

\def \G {\mathcal{G}}

\def \Z {\mathbb{Z}}

\def \DD {\mathbb{D}}

\def \UU {\mathbb{U}}

\def \EE {\mathcal{E}}

\def \NN {\mathcal{N}}

\def \d {\delta}

\def \< {\langle}
\def \> {\rangle}

\def \one {{\bf 1}}

{\end{eqnarray}\end{subequations}\hskip-4.0pt}

\def\fatnorm#1{|\kern-.2ex|\kern-.2ex| #1 |\kern-.2ex|\kern-.2ex|}
\newcommand{\twonorm}[1]{\left\lVert#1\right\rVert_2}

\newcommand{\shnorm}[1]{\lVert#1\rVert}
\newcommand{\fnorm}[1]{\left\lVert#1\right\rVert_F}
\newcommand{\norm}[1]{\left\lVert#1\right\rVert}
\newcommand{\abs}[1]{\left\lvert#1\right\rvert}

\newcommand{\Sp}{\mathbb{S}}

\newcommand{\M}{\mathcal{M}}

\newcommand{\sparse}{\texttt{sparse}}

\newcommand{\N}{{\mathcal N}}

\newcommand{\cov}{\textsf{Cov}}

\newcommand{\half}{\ensuremath{\frac{1}{2}}}
\newcommand{\inv}[1]{\frac{1}{#1}}
\def\conv{\mathop{\text{\rm conv}\kern.2ex}}

\newcommand{\RE}{\textnormal{\textsf{RE}}}
\newcommand{\ip}[1]{\;\langle{\,#1\,}\rangle\;}
\newcommand{\size}[1]{\ensuremath{\left|#1\right|}}
\newcommand{\onenorm}[1]{\ensuremath{\left|#1\right|_1}}

\newcommand{\maxnorm}[1]{\ensuremath{\left\|#1\right\|_{\max}}}
\newcommand{\expct}[1]{\ensuremath{\mathbb E}#1}
\newcommand{\silent}[1]{}
\newcommand{\mvec}[1]{\rm{vec}\left\{\,#1\,\right\}}

\newcommand{\ve}{\varepsilon}

\newcommand{\ul}{\underline}
\newcommand{\ol}{\overline}

\def\qed{\hskip1pt $\;\;\scriptstyle\Box$}
\def\Ber{\mathop{\text{Bernoulli}\kern.2ex}}
\def\supp{\mathop{\text{supp}\kern.2ex}}
\def\corr{\mathop{\text{corr}\kern.2ex}}
\def\prec{\mathop{\text{precision}\kern.2ex}}
\def\recall{\mathop{\text{recall}\kern.2ex}}
\def\cov{\mathop{\text{Cov}\kern.2ex}}
\def\mnorm{\mathcal{N}_{f,m}\kern.2ex}
\def\var{\mathop{\text{Var}\kern.2ex}}
\def\ess{\mathop{\text{ess}\kern.2ex}}
\def\dom{\mathop{\text{dom}\kern.2ex}}
\def\lin{\mathop{\text{lin}\kern.2ex}}
\newcommand{\event}{\mathcal{E}}
\newcommand{\func}[1]{\ensuremath{\mathrm{#1}}}
\newcommand{\diag}{\func{diag}}
\newcommand{\thresh}{\func{thresh}}

\newcommand{\offd}{\func{offd}}

\newcommand{\mask}{\func{mask}}
\newcommand{\overall}{\func{overall}}
\newcommand{\spin}{\textsf{span}}

\newcommand{\up}{{\upsilon}}

\newcommand{\bs}{{\bf s}}

\newcommand{\bv}{{\bf v}}
\newcommand{\dm}{{\diamond}}
\newcommand{\vecp}{\ensuremath{{\bf p}}}

\let\hat\widehat
\let\tilde\widetilde

\newcommand{\vecone}{{\bf 1}}

\newcommand{\vecb}{{\ensuremath{\bf{b}}}}

\newcommand{\tr}{{\rm tr}}

\def\E{{\mathbb E}}

\def\supp{\mathop{\text{\rm supp}\kern.2ex}}
\def\argmin{\mathop{\text{arg\,min}\kern.2ex}}

\newcommand{\prob}[1]{\ensuremath{\mathbb P}\left(#1\right)}

\newcommand{\beq}{\begin{equation}}
\newcommand{\eeq}{\end{equation}}
\newcommand{\ben}{\begin{eqnarray}}
\newcommand{\een}{\end{eqnarray}}
\newcommand{\bnum}{\begin{enumerate}}
\newcommand{\enum}{\end{enumerate}}
\newcommand{\bit}{\begin{itemize}}
\newcommand{\eit}{\end{itemize}}
\newcommand{\bens}{\begin{eqnarray*}}
\newcommand{\eens}{\end{eqnarray*}}
\newcommand{\X}{{\mathcal X}}
\newcommand{\bE}{{\bf E}}
\newcommand{\be}{{\bf e}}

\newcommand{\F}{\ensuremath{\mathcal F}}

\newcommand{\TM}{\ensuremath{\mathbb M}}

\newcommand{\calX}{\ensuremath{\mathcal X}}

\newcommand{\HH}{\ensuremath{\mathbb H}}
\newcommand{\Sc}{\ensuremath{S^c}}
\newcommand{\distinct}{\ensuremath{\text{distinct}}}
\newcommand{\cone}{\ensuremath{\text{Cone}}}

\newcommand{\QA}{\ensuremath{p}}

\newcommand{\minus}{\ensuremath{-}}

\newcommand{\ga}{\gamma}
\newcommand{\z}{p}

\newcommand{\onen}{\textstyle \frac{1}{n}}

\newcommand{\D}{{\mathcal D}}

\newcommand{\Ball}{{B}}
\newcommand{\B}{\mathcal{B}}

\newtheorem{theorem}{Theorem}[section]

\newtheorem{lemma}[theorem]{Lemma}
\newtheorem{proposition}[theorem]{Proposition}

\newtheorem{definition}[theorem]{Definition}

\newtheorem{remark}[theorem]{Remark}
\newtheorem{corollary}[theorem]{Corollary}

\def\qed{\hskip1pt $\;\;\scriptstyle\Box$}

\newenvironment{proofof}[1]{\hspace*{20pt}{\it Proof}{ of #1}.\hskip10pt}{\qed\vskip5pt}
\newenvironment{proofof2}{\hskip10pt}{\qed\vskip5pt}

\begin{document}

\title{Concentration of measure bounds for matrix-variate data
    with missing values}

  \author{Shuheng Zhou\\
    University of California, Riverside, CA 92521}

\date{}

\maketitle

\begin{abstract}
We consider the following data perturbation model, where the 
covariates incur multiplicative errors.
For two $n \times m$ random matrices $U, X$, we denote by $U \circ X$
the Hadamard or Schur product, which is defined as $(U \circ X)_{ij} = (U_{ij}) \cdot (X_{ij})$.
In this paper, we study the subgaussian matrix variate model,
where we observe the matrix variate data $X$ through a random mask $U$:
\bens
\label{eq::Xgenabs}
\X = U \circ X \; \; \; \text{ where} \; \; \;X = B^{1/2} \Z A^{1/2},
\eens
where $\Z$ is a random matrix with independent subgaussian entries, and
$U$ is a mask matrix with either zero or positive entries, where $\E
U_{ij} \in [0, 1]$ and all entries are mutually independent.
Under the assumption of independence between $U$ and $X$,
we introduce componentwise unbiased estimators for estimating covariance $A$ and 
$B$, and prove the concentration of measure bounds in the sense of guaranteeing the restricted eigenvalue($\RE$) 
conditions to hold on the unbiased estimator for $B$, when columns of data
matrix $X$ are sampled with different rates.
We further develop multiple regression methods for estimating the
inverse of $B$ and show statistical rate of convergence.
Our results provide insight for sparse recovery for relationships
among entities (samples, locations, items) when features (variables,
time points, user ratings) are present in the observed data matrix
$\X$ with heterogeneous rates. Our proof techniques can certainly be extended
to other scenarios. We provide simulation evidence illuminating the theoretical predictions.
\end{abstract}

\section{Introduction}
\label{sec::intro}
In this paper, we study the multiplicative measurement errors on
matrix-variate data in the presence of missing values,
and sometimes entirely missed rows or columns.
Missing value problems appear in many application areas such as
energy, genetics, social science and demography, 
and spatial statistics;
see~\cite{DLR77,HWang86,Full:1987,Pig01,LR02,HK10,CW11,CCZ15} and references therein.
For complex data arising from these application domains, 
missing values is a norm rather than an exception.
For example, in spatio-temporal models in geoscience, it is common
some locations will fail to observe certain entries, or at different
time points, the number of active observation stations varies~\citep{Smith03,SKS15}. 
In social science and demography, the United States Census Bureau was
involved in a debate with the U.S. Congress  and the U.S. Supreme
Court over the handling of the undercount in the 2000
U.S. Census~\citep{Pig01}.
In addition to missing values, data are often contaminated with an
additive source of noise on top of the multiplicative noise such as
missing values~\citep{HWang86,CGL93,carr:rupp:2006}.

For two $n \times m$ random matrices $U, X$, denote by $U \circ X$
the Hadamard or Schur product, which is defined as $(U \circ X)_{ij} =
(U_{ij}) \cdot (X_{ij})$.
Let $U$ be a random mask with
either zero or positive entries. Consider the following data
perturbation model, where the covariates incur multiplicative errors;
that is, instead of $X  =  [x^1, \ldots, x^{m}]$,  we observe
\ben
\label{eq::perturb}
\X = U \circ X \quad \text{ where} \; \;  U =[u^1, \ldots, u^m]
\; \text{ and } \; \E U_{ij} \in [0, 1];
\een
Subsampling in rows, or columns, or random sampling
of entries of $X$ are special cases of this
model~\cite{RT10,LW12,RT13,BRT14}.

Consider a space-time model $X(\bs,t)$ where $\bs$ denotes spatial location and
$t$ denotes time. In the space-time model literature,
a common assumption on the covariance of $X$ is separability, namely
\ben
\label{eq::separable}
\cov(X(\bs,t),X(\bs',t')) = A_0(\bs,\bs')B_0(t,t')
\een
where $A_0$ and $B_0$ are each covariance functions~\cite{CW11}.
In this model, a mean zero column vector $x^j$
corresponds to values observed across $n$ spatial locations at a
single time point $t_j, j=1, 2, \ldots$.
When $X$ is observed in full and free of noise, the theory is already in 
place on estimating matrix variate Gaussian graphical models:
under sparsity conditions,~\cite{Zhou14a} is the first
in literature to show with theoretical guarantee that one can estimate
the graphs, covariance and inverse covariance matrices well using only
one instance from the matrix-variate normal distribution.
See~\cite{Dawid81,GV92,AT10,LT12,THZ13,KLLZ13,Horns19,GZH19} and 
references therein for more applications of the matrix variate models. 

In this paper, we consider this separable covariance model defined 
through the tensor product of $A_0, B_0$, however,
now under the much more general subgaussian distribution, 
where we also model the sparsity in data with a random mask in
\eqref{eq::perturb}.
While the general aim is to recover the full rank
covariance matrices $A_0, B_0 \succ 0$ in the tensor-product,
we specifically focus on estimating $B_0$ and its inverse $\Theta_0$ for the bulk
of this paper. The reason is because in space-time applications, we
are often interested in discovering the relationships between spatial coordinates,
where for each node (row), there is a continuous data stream over
time. In such data sets, the time dimension $m$ can dominate the spatial
dimension $n$ as they may  differ by many orders of magnitude, e.g., ten thousand time points versus one hundred locations, while at each time point, we may only have a subset of the $n$ observations. 

\subsection{Our approach and contributions}
Our main task is on deriving concentration of measure bounds 
for a componentwise oracle estimator $\tilde{B}_0$ in the space-time 
context~\eqref{eq::separable} under the observation 
model~\eqref{eq::perturb}.
An oracle estimator that provides componentwise unbiased estimate
for covariance $B_0$ was introduced in~\cite{Zhou19},
\bens 
\label{eq::BestIntro}
\tilde{B}_0 = \X \X^T \oslash \M \; \text{ for } \; \M \; \text{ as
  defined in }~\eqref{eq::Bmask},
\eens 
where $\oslash$ denotes componentwise division and we use the
convention  of $0/0=0$. The matrix $\M$ is a linear combination of rank-one matrices
$M_1, \ldots, M_m$
\ben 
\label{eq::Bmask}
\M & := & \sum a_{jj} M_j \; \; \text{ where} \; \; 
M_j = \expct{(u^j \otimes u^j)} \in \R^{n \times n}
\een
and $a_{11}, a_{22}, \ldots$ are diagonal entries of matrix
$A_0$. Specifically, suppose we observe $\X = U \circ X$, where data is randomly 
subsampled.  Here and in the sequel, we assume that each column $u^j =(u^j_1,
\ldots, u^j_n)^T \in \{0,1\}^{n}$
of the mask matrix $U$ is composed of independent Bernoulli random variables such 
that
\bens
\E u^j_k = p_j,  \forall j, k; \; \text{and instead of $x^j$, we
  observe  } \; x^j \circ u^j, j=1,  \ldots, m; 
\eens
Moreover, $X$ and $U$ are   independent of each other.
In Theorem~\ref{thm::main-intro},
we prove new concentration of measure inequalities for
quadratic forms involving sparse and nearly sparse vectors,
in particular,
\ben
\label{eq::quadorig}
&& q^T(\X \X^T \oslash \M- \E  (\X \X^T \oslash \M))q \quad
\text{ over a class of vectors } \quad q \in \R^{n}
\een
satisfying the following cone constraint,
for $0< s_0 \le n$ and $\M$ as defined in~\eqref{eq::Bmask},
\ben
\label{eq::cone}
\cone(s_0) := \{v: \onenorm{v} \le \sqrt{s_0} \twonorm{v}\}, \; \text{where } 
\; \onenorm{v} =: \sum \abs{v_i}, \text {and } \; \twonorm{v}^2 = \sum v_i^2.
\een
This enables us to obtain operator norm-type of bounds on estimating 
submatrices of $B_0$ with the whole $B_0$ as a special case, using
novel matrix concentration of measure analyses.
Theorem~\ref{thm::RE} shows that certain restricted eigenvalue ($\RE$) conditions
hold on $\tilde{B}_0$ with high probability for the perturbation model
as in Definition~\ref{def::RMS}.
Moreover, we introduce an estimator $\hat{\TM}$
for matrix $\M$ and show that entries of $\hat{\TM}$ (cf.~\eqref{eq::MHat}) 
are
tightly concentrated around their mean values in $\M$ under 
Definition~\ref{def::RMS}.

We make the following theoretical contributions:
{\bf [a]} concentration of measure bounds on quadratic forms and certain functionals of large random matrices
$\X \X^T$ and $\X^T \X$;
{\bf [b]} consistency and the rate of convergence in the operator norm of
the inverse covariance $\Theta_0 =B_0^{-1}$: Under the $\RE$ conditions which we show to hold for $\tilde{B}_0$ and 
$\hat{B}^{\star} :=  \X \X^T \oslash \hat{\TM}$
(cf.~Theorems~\ref{thm::RE} and~\ref{thm::main-coro}),
we further propose using the multiple regression method~\cite{MB06} in 
combination with techniques subsequently developed in~\cite{Yuan10,LW12} to construct an estimator $\hat{\Theta}$ for the 
inverse covariance $\Theta_0$ under the more general sparse
subgaussian matrix-variate model in Definition~\ref{def::RMS};
{\bf [c]} combining the results we obtain on \eqref{eq::quadorig} and on estimating $\M$ with estimator
$\hat{\TM}$ (cf.~\eqref{eq::MHat}), we obtain an error bound on  the quadratic form
\ben
\label{eq::quadpreview}
\abs{q^T(\X \X^T \oslash \hat{\TM}- \E  (\X \X^T \oslash \M))q} \; \; \text{for 
  all } q \in \Sp^{n-1} \cap \cone(s_0),
\een
where $\Sp^{n-1}$ is the unit sphere;
cf. Lemma~\ref{lemma::unbiasedmask} and Theorem~\ref{thm::main-coro}.
For completeness, in Theorem~\ref{thm::gramsparse} and Corollary~\ref{coro::Bhatnorm}, 
we control the quadratic forms~\eqref{eq::quadorig}
and~\eqref{eq::quadpreview} over the entire sphere. 

{ \bf Roadmap.}
In this paper, we develop a set of new tools and ideas in order to 
prove tight concentration of measure bounds for matrix variate data
with missing values.
To establish the RE conditions for
random matrices as in Definition~\ref{def::RMS}, 
we had to introduce a new approach based on the conditioning arguments 
to deal with randomness in matrix variate data $X$ and mask $U$
simultaneously.
Conditioned on the {\it good events} related to $U$, we can gain 
uniform control over a family of random matrices, 
cf.~\eqref{eq::tensor} and \eqref{eq::defineADM}, in the sense that we obtain a uniform bound on 
the operator and the Frobenius norm for this family of random
matrices; cf. Theorems~\ref{thm::mainop2} and~\ref{thm::uninorm2}.
This in turn allows a uniform control over the quadratic form in 
\eqref{eq::quadorig} over the cone or the entire sphere, thanks to the 
Hanson-Wright inequality~\cite{RV13}, and the sparse Hanson-Wright 
inequalities as developed in~\cite{Zhou19} and the present work. 
We introduce
Theorem~\ref{thm::Bernmgf} and its Corollary~\ref{coro::offdn}, and give a proof sketch of
Theorems~\ref{thm::main-intro} and~\ref{thm::gramsparse} in
Sections~\ref{sec::prelim} and~\ref{sec::reduction} respectively, while highlighting where such
inequalities are being applied.
From these initial estimators, we establish consistency and
obtain the rate of convergence in the operator norm 
for the penalized estimator $\hat\Theta$, because the RE conditions are also established 
for $\hat{B}^{\star}$
in  Theorem~\ref{thm::main-coro}.

{ \bf Organization.}
The rest of the paper is organized as follows.
In Section~\ref{sec::related}, we review the related work to place our work in context.
In Section~\ref{sec::method}, we define our model and the method.
Section~\ref{sec::theory} presents  in Theorems~\ref{thm::main-intro}
and~\ref{thm::RE} our main technical results on analyzing the random quadratic
form~\eqref{eq::quadorig}.
We present in
Theorems~\ref{thm::AD},~\ref{thm::XX^T} and  
Corollary~\ref{coro::tartan} properties regarding the gram matrix $X X^T$ 
for the fully observed matrix variate data.
Section~\ref{sec::Bhatmain} presents error bounds on the quadratic 
form \eqref{eq::quadpreview}. 
In Section~\ref{sec::strategy},
we elaborate upon the proof strategies for 
Theorems~\ref{thm::main-intro} and~\ref{thm::gramsparse},
highlighting the structural complexity emerging from the sparse
matrix-variate model due to complex dependencies when analyzing the
quadratic forms. Section~\ref{sec::inverse} presents theoretical results on inverse
covariance estimation, where we develop convergence bounds
in Theorem~\ref{coro::thetaDet}.
Section~\ref{sec::examples} shows numerical results that validate
our theoretical predictions.
Section~\ref{sec::mainresult} presents
a complete exposition on the main proof ideas for 
Theorems~\ref{thm::main-intro} and~\ref{thm::main}.
Section~\ref{sec::mainproofs} contains proofs of
Proposition~\ref{prop::projection} and Theorem~\ref{thm::main-intro}.
We conclude in  Section~\ref{sec::conclude}.
We place all technical proofs in the supplementary material.

\subsection{Related work}
\label{sec::related}
Restricted eigenvalue conditions ($\RE$)
have been widely explored in the literature for various families
of random design matrices, see for
example~\cite{BRT09,GB09,RWY10,RZ13,LW12,RZ17} and references
therein. See~\cite{HW71,HW73,RV13,BVZ19} and references therein,
for classical results and recent expositions on quadratic forms over dense random vectors.
In this work, we address the key challenges arising from the intricate interactions 
between $U$ and $X$ by developing a new proof architecture, as elaborated in 
Sections~\ref{sec::strategy} and~\ref{sec::mainresult}.
Because of the sparsity and complex dependencies, 
our analyses on~\eqref{eq::quadorig} and~\eqref{eq::quadpreview} draw 
upon and yet significantly extend the concentration of measure
inequalities on quadratic forms involving (sparse) subgaussian random
vectors as studied in~\cite{RV13,Zhou19}, which we refer to as the
(sparse) Hanson-Wright (HW) inequalities.
Under sparsity and neighborhood stability conditions,  it was first shown
in~\cite{MB06} that the graph corresponding to $\Theta_0 = B_0^{-1}$
can be estimated efficiently using the penalized multiple regression
approach with a sample size that dominates the maximum node degree (sparsity).
Subsequent line of work illustrates that certain restricted
eigenvalue conditions imposed on the design matrices not only
guarantee sparse recovery in high dimensional linear
models~\cite{BRT09,GB09} but also enable inverse covariance estimation under additional spectral
conditions~\cite{ZRXB11,LW12}.
See also~\cite{RBLZ08,RWRY08,ZLW08,Yuan10,RSZZ15} and 
references therein for theoretical developments and results on 
Gaussian graphical modeling using either graphical Lasso or nodewise regression type of 
estimators~\cite{FHT07,YL07,BGA08}.

We emphasize that our inverse covariance estimation works under much more general
subgaussian distribution as we will elaborate in Section~\ref{sec::method}.
It turns out that such concentration of measure properties are also essential to ensure
algorithmic convergence, and hence to bound both optimization and
statistical errors, for example, when approximately 
solving optimization problems such as the corrected Lasso
using the gradient-descent type of algorithms
\citep[cf.][]{ANW12,LW12}. These estimators were introduced to tackle high 
dimensional errors-in-variables regression problems including the 
missing values. See~\cite{RT10,RT13,BRT14}, where Dantzig 
selector-type of estimators have been designed and studied.
Such quadratic forms in errors-in-variables
models were also analyzed in~\cite{RZ17}, where data $X$ is contaminated
with a perturbation matrix $\Delta \in \R^{n \times m}$ such that $\E
\Delta_{ij} = 0, \forall i, j$
that consists of spatially correlated subgaussian noise as column vectors;
there we introduced the additive errors in the covariates,
resulting in a non-separable class of space-time covariance models for
the observation matrix.

The problem we study here is different from matrix completion,
which focuses on recovering low-rank structures;
See~\cite{CR09,CT10,PVY16,Vers18,Wain19} and references therein.
We focus here on recovering the full rank covariance matrices
$A_0, B_0 \succ 0$ in the matrix variate model,
but with incomplete data, or intentionally subsampled data.
One can use our analysis to significantly subsample the data matrix, for example, in an
i.i.d. manner by setting column sampling rate $p_j = p, \forall j$, while still ensuring
that the gram matrices $XX^T$ and $X^TX$  (upon adjustment) can be
used for statistical analysis, for instance, regression or inverse
covariance estimation. This idea has been explored before in the
context of high dimensional linear regression;
See for instance~\cite{ZLW09} and references therein for 
other examples of data perturbation model, where $X$ is subject to a random affine transformation characterized by multiplicative noise and additive noise: $\X := \Phi X + \Delta$, 
where $\Phi$ is a random matrix with i.i.d. Gaussian or symmetric
Bernoulli entries. More broadly, our work is also related to the {\it data masking } literature in the context of privacy~\citep[cf.][]{duncan:91}, which allows the possibility 
of deleting records, sampling, and suppressing subsets of variables.
See also~\cite{AT10,SB12,SSB14}, where EM based method for sparse
inverse covariance estimation and missing value imputation algorithms
in the matrix-variate  normal model were considered. In contrast to
the methods studied here, there were no theoretical guarantee on
statistical convergence.

\subsection{Definitions and notations}
\label{sec::notation}
Let $e_1, \ldots, e_n$ be the canonical basis of $\R^n$.
For a set $J \subset \{1, \ldots, n\}$, denote
$E_J = \Span\{e_j: j \in J\}$.
We denote by $[n]$ the set $\{1, \ldots, n\}$.
We refer to a vector $v \in \R^n$ with at most 
$d, d \in [n]$ nonzero entries as an $d$-sparse vector.
Let $\Ball_2^n$ and $\Sp^{n-1}$ be the unit Euclidean ball and the
unit sphere, respectively.
 Let $B$ be an $n \times n$ matrix.
A $k \times k$ submatrix of $B$
formed by deleting $n-k$ rows of $B$, and the same $n-k$ columns of 
$B$, is called a principal submatrix of $B$.
For two subsets $I, J \in [n]$, denote by $B_{I, J}$ the submatrix of 
$B$ with rows and columns indexed by set $I$ and $J$ respectively.
For a symmetric matrix $A$, let $\lambda_{\max}(A)$ and
$\lambda_{\min}(A)$ be the largest and the smallest
eigenvalue of $A$ respectively.
For a matrix $A$, we use $\twonorm{A}$ to denote its operator norm
and $\fnorm{A}$ the Frobenius norm, given by $\fnorm{A} = (\sum_{i, j} a_{ij}^2)^{1/2}$.  For a vector
$x = (x_1, \ldots, x_n) \in \R^n$, denote by $\twonorm{x} = \sqrt{
  \sum_{i=1}^n x_i^2}$ and $\onenorm{x} := \sum_{j} \abs{x_j}$.
For a matrix $A = (a_{ij})$ of size $m \times n$,
let  $\norm{A}_{\infty} = \max_{i} \sum_{j=1}^n |a_{ij}|$ and  $\norm{A}_1 = \max_{j} \sum_{i=1}^m |a_{ij}|$ denote  the maximum absolute row and column sum of the matrix $A$ respectively.
Let $\norm{A}_{\max} = \max_{i,j} |a_{ij}|$ denote the componentwise matrix maximum norm.
Let $\diag(A)$ be the diagonal of $A$. Let $\offd(A) = A - \diag(A)$.
Let $r(A) = {\tr(A)}/{\twonorm{A}}$ denote the effective
rank of $A$.
For a given vector $x \in \R^m$, $\diag(x)$ denotes the diagonal matrix whose  
main diagonal entries are the entries of $x$.  
For a finite set $V$, the cardinality is denoted by $\abs{V}$. 
Let $\kappa(A) = \lambda_{\max}(A)/\lambda_{\min}(A)$ denote the condition number for matrix $A$.
For two numbers $a, b$, $a \wedge b := \min(a, b)$, and 
$a \vee b := \max(a, b)$.
We write $a \asymp b$ if $ca \le b \le Ca$ for some positive absolute
constants $c,C$ that are independent of $n, m$, sparsity, and
sampling parameters.
We write $f = O(g)$ or $f \ll g$ if $\abs{f} \le C g$ for some absolute constant
$C< \infty$ and $f=\Omega(g)$ or $f \gg g$ if $g=O(f)$.
We write $f = o(g)$ if $f/g \to 0$ as $n \to \infty$, where the
parameter $n$ will be the size of the matrix under consideration.
In this paper, $C, c, c', C_1, C_2, \ldots$, etc, denote various absolute positive constants which may change line by line.

\section{The model and the method}
\label{sec::method}
We first introduce our data generative model and estimators.
Let $B_0 = (b_{ij}) \in \R^{n \times n}$ and $A_0 = (a_{ij}) \in \R^{m
  \times m}$ be positive definite matrices.
Denote by $B_0^{1/2}$ and $A_0^{1/2}$ the unique square root 
of $B_0, A_0 \succ 0$ respectively.
Denote by $X  =  [x^1, \ldots, x^{m}]$
the full (but not fully observed) $n \times m$ data matrix with column 
vectors $x^1, \ldots, x^m \in \R^{n}$ and row vectors $y^1, \ldots, 
y^n$. For a random variable $Z$, the subgaussian (or $\psi_2$) 
 norm of $Z$ denoted by $\norm{Z}_{\psi_2}$ is defined as 
$\norm{Z}_{\psi_2} = \inf\{t > 0\; : \; \E \exp(Z^2/t^2) \le 2 \}.$
 \begin{definition}{(Random mask sparse model)}
   \label{def::RMS}
Consider the data matrix $X_{n \times m}$ generated from a subgaussian
random matrix $\Z_{n \times m} = (Z_{ij})$ with independent mean-zero unit variance components whose $\psi_2$ norms are uniformly bounded:
\ben
\label{eq::missingdata}
X & = & B_0^{1/2} \Z A_0^{1/2}, \; \; \text{ where } \; \expct{Z_{ij}} = 0, \; \E Z_{ij}^2 =1,
\; \text{ and} \; \norm{Z_{ij}}_{\psi_2} \leq K, \forall i,j.
\een
Without loss of generality, we assume $K=1$. Suppose we observe 
\ben
\label{eq::Xgen}
\X & = & U\circ X, \; \text{ where $U \in \{0,1\}^{n \times m}$ is a
  mask matrix and } \\
\label{eq::maskUV}
U & = & [u^1| u^2| \ldots |u^m] =  [v^1| v^2| \ldots |v^n]^T
\text{ is independent of $X$}, 
\een
with independent row vectors $v^1, \ldots, v^n \sim {\bf v}  \in
\{0,1\}^{m}$, where ${\bf v}$ is composed of independent Bernoulli random variables with $\E v_k = p_k, 
k=1,\ldots, m$. 
\end{definition}
Since the tensor product $A_0 \otimes B_0 = A_0 \eta \otimes 
\inv{\eta} B_0$ for any $\eta>0$, we can only estimate $A_0$ and $B_0$ up to 
a scaled factor.
Hence, without loss of generality,  we propose to estimate 
\ben 
\label{eq::starscale}
A_{\star} := A_0 \tr(B_0)/n \quad \text{ and } \quad B_{\star} := n 
B_0/\tr(B_0) 
\een
with the following set of oracle estimators $\tilde{B}_{0}$ and
$\tilde{A}_{0}$ as well as the sample based plug-in
estimators $\hat{A}_{\star}$ and  $\hat{B}^{\star}$ as defined in 
\eqref{eq::Astar} and~\eqref{eq::Bstar}, which are completely
data-driven.
\begin{definition}
\label{def::RMSet}
For the rest of the paper, we assume $B_0 \succ 0$ is scaled such 
that $\tr(B_0) = n$ in view of \eqref{eq::starscale}; hence 
$\twonorm{B_0} \ge 1$.  Let $\X = U \circ X$.
In order to estimate $B_0$, we define the following oracle estimator:
\ben
\label{eq::entrymean}
\quad \tilde{B}_0 = \X \X^T \oslash
\M \; \text{ where } \;
\M_{k \ell} = \left\{\begin{array}{rl} \sum_{j=1}^m a_{jj} \z_j  
& \text{ if }  \; \ell = k, \\
\sum_{j=1}^m a_{jj} \z_j^2 & \text{ if }  \; \ell \not= k.
\end{array}\right.
\een
Clearly, $\E \tilde{B}_0 = B_0$, 
where expectation denotes the componentwise expectation.
\end{definition}
For completeness, we also present the corresponding oracle estimator $\tilde{A}_{0}$ as 
studied in~\cite{Zhou19}:
\ben
\label{eq::baseN}
\tilde{A}_{0}  & = & \calX^T \calX \oslash \N
 \; \; \text{ where } \; \; \N := \tr(B_0) \expct{v^i \otimes v^i}, \\
\label{eq::baseNBer}
&&
 \;\text{ and }  \; \N_{ij} = \tr(B_0)\left\{
\begin{array}{rl} {\z}_i & \text{ if }  \; i  = j, \\
{\z}_i {\z}_j  & \text{ if }  \; i  \not= j,
\end{array}\right.
\een
Similar to~\eqref{eq::Bmask},~\eqref{eq::baseN} works for general 
distributions of $U$, while~\eqref{eq::baseNBer} works for the 
model~\eqref{eq::maskUV} under consideration.
Denote by  $\vecp := (\z_1, \ldots, \z_m)$ the vector of column sampling 
probabilities. 
Let $\hat\vecp = (\hat{\z}_1, \hat{\z}_2,\ldots, \hat{\z}_m)$ 
denote the estimate of sampling probabilities $\vecp$,
where $\hat{\z}_j = \onen \sum_{k=1}^n u^j_k$ is the average number
of non-zero (observed) entries for column $j$. 
Let $\hat{M}$ be as defined in~\eqref{eq::Astar}.
Construct
\ben
\label{eq::Astar}
\hat{A}_{\star} & = & 
 \onen \X^T \X \oslash \hat M \; \;
 \text{ where }
 \; \; \hat{M}_{ij} = \left\{
\begin{array}{rl} \hat{\z}_i & \text{ if }  \; i  = j, \\
\hat{\z}_i \hat{\z}_j  & \text{ if }  \; i  \not= j; \\
\end{array}\right. \\
\label{eq::Bstar}
\; \text{ and } \; \; 
\hat{B}^{\star} & = &  \X \X^T \oslash \hat{\TM}  \; \; \text{ where } \\
\label{eq::MHat}
\hat{\TM}_{k \ell} & = &    \left\{
\begin{array}{rl} \onen \tr(\X^T \X)  & \text{ if }  \; k = \ell, \\
  \frac{1}{n-1} \tr(\X^T \X \circ \hat{M}) - \inv{n(n-1)} \tr(\X^T \X)
& \text{ if }  \; k \not= \ell.
\end{array}\right. 
\een
By using the pair of plug-in estimators~\eqref{eq::Astar}
and~\eqref{eq::Bstar},
based on the observed data as well as their sparsity patterns as elaborated above, 
we are able to estimate $B_{\star}$ and $A_{\star}$ as specified 
in~\eqref{eq::starscale}: (a) Clearly $\E \hat{M} = D_{\vecp} +
\offd(\vecp \otimes \vecp)$, where $D_{\vecp} = \diag(\vecp) =
\diag(\z_1, \ldots, \z_m)$
denotes the  diagonal matrix with entries of  $\vecp$ along its 
main diagonal;
(b) It is also straightforward to check that $\hat{\TM}$  is a componentwise unbiased 
estimator for $\M$ when $\tr(B_0) = n$; (c) The design of $\hat{B}^{\star}$ makes it scale-free 
as we divide $\X \X^T$ by the mask matrix $\hat{\TM}$:
$\tr(\X^T \X) =\tr(\X \X^T)\;\; \text{and hence by
  construction} \; \; \tr(\hat{B}^{\star}) = n.$
Further justifications for the oracle estimators \eqref{eq::entrymean}
and~\eqref{eq::baseN}, 
as well as~\eqref{eq::Astar} and~\eqref{eq::Bstar},
appear in Lemma~\ref{lemma::unbiasedmask}, Section~\ref{sec::strategy}, and
the supplementary Sections~\ref{sec::maskMN} and~\ref{sec::maskest}.

\subsection{The multiple regression functions}
\label{sec::inverselasso}
The nodewise regression method was introduced in~\cite{MB06} using the 
relations in~\eqref{eq::ggm1}, which we now review in the context of
matrix variate normal model \eqref{eq::matrix-normal-rep-intro}.
When $\Z$ in~\eqref{eq::missingdata}
is a Gaussian random ensemble,
with i.i.d. $N(0, 1)$ entries,  we say the random matrix $X$ as defined in~\eqref{eq::missingdata}
follows the matrix-variate normal distribution
\ben 
\label{eq::matrix-normal-rep-intro}
X_{n \times m} \sim \N_{n,m}(0, A_{0, m \times m} \otimes B_{0, n \times n}). 
\een
This is equivalent to say $\mvec{X}$ follows a multivariate normal distribution with mean
${{\bf 0}}$ and a separable covariance $\Sigma = A_0 \otimes B_0$,
where $\mvec{X}$ is formed by stacking the columns of $X$ into a vector in $\R^{mn}$.
To ease the exposition, we first
consider~\eqref{eq::matrix-normal-rep-intro} and assume that $A_0$ is
a correlation matrix and $B_0 \succ 0$.
Denote by $(X_1, X_2, \ldots, X_n) \in \R^n$ any column vector
from the data matrix $X$~\eqref{eq::matrix-normal-rep-intro},
then $(X_1, X_2, \ldots, X_n) \sim \N_n(0, B_0)$ and we have by the
Gaussian regression model,
for each $j =1, \ldots, n$, 
\ben
\label{eq::ggm1}
 X_{j} &=&
 \sum_{k \not=j} X_{k} \beta_k^{j*} + V_j,
 \text{ where}  \; \;
\forall j, V_{j} \sim \N(0, \sigma^2_{V_j}) \; \text{ is independent 
  of } \{X_k; k \not=j\}.
\een
There exist explicit relations between the regression coefficients,
error variances and the concentration matrix $\Theta_0 =B_0^{-1} :=
(\theta_{ij}) \succ 0$: for all $j$, we have $\sigma^2_{V_j} =
1/{\theta_{jj}} >0$, $\beta_k^{j*} = -{\theta_{jk}}/{\theta_{jj}}$,
and $\theta_{jk}=  \beta^{j*}_k=  \beta^{k*}_j= 0$  if $X_j$ is
independent of $X_k$ given the other variables~\citep{laur96}.

For the subgaussian model as considered in the present work, 
instead of independence, we have by the Projection Theorem, $\cov(V_j, 
X^k) = 0, \;\forall k \not=j$; see for example Theorem 2.3.1.~\cite{BD06}.
Proposition~\ref{prop::projection} illuminates this zero correlation 
condition for the general model~\eqref{eq::missingdata} as well
as the explicit relations between the regression coefficients $ \beta_k^{j*}$, error
variances for $\{V_j, j \in [n]\}$, and the inverse covariance
$\Theta_0= (\theta_{ij}) \succ 0$. This result may be of independent
interests and holds for general covariance $A_0, B_0 \succ 0$. We
prove Proposition~\ref{prop::projection} in Section~\ref{sec::proofofprojection}.

\begin{proposition}{\textnormal{\bf (Matrix subgaussian regression model)}}
\label{prop::projection} 
Let $\{X^j, j =1, \ldots, n\}$ be row vectors of the data matrix $X =
B_0^{1/2} \Z A_0^{1/2}$ as in Definition~\ref{def::RMS},
where we assume that $A_0, B_0\succ 0$. Let $\Theta_0 = B_0^{-1} = (\theta_{ij}) \succ
0$, where $0< 1/{\theta_{jj}} < b_{jj}, \forall j$.
Consider many regressions, where we regress one row vector $X^j$
against all other row vectors.
Then for each $j \in [n]$,
\ben
\label{eq::regress}
&& X^j = \sum_{k \not=j} X^{k} \beta_k^{j*} + V_j (j = 1, \ldots,  n),
\text{ where }   \beta_k^{j*} = - {\theta_{jk}}/{\theta_{jj}},  \\
\label{eq::residual}
& & \cov(V_j, V_j) = A_0/{\theta_{jj}} 
\; \text{ and} \; \cov(V_j, X^k) = 0 \; \; \forall k \not=j, \forall j \in [n].
\een
\end{proposition}

\subsection{Inverse covariance estimator}
Our method on  inverse covariance estimation corresponds to the 
proposal in \cite{Yuan10,LW12},  only now dropping the i.i.d. or 
Gaussian assumptions, as we use the relations~\eqref{eq::regress}
and~\eqref{eq::residual} which are valid for the much more general 
subgaussian matrix-variate model~\eqref{eq::missingdata}. 
We need to introduce some notation. Fix $j \in [n]$.
We use $B_{\minus j, j}$ as a shorthand to denote $B_{I, J}$
when $I = n \setminus \{j\}$ and $J = \{j\}$, that is, the $j^{th}$
column of $B$ without the diagonal entry $B_{jj}$.
We use ${B}_{\minus j, \minus j}$ to denote the submatrix 
of $B$ formed by throwing away row $j$ and column $j$. 
Let $\X_j$ denote the $j^{th}$ row vector of $\X$ and $\X_{\minus j}$ 
the submatrix of $\X$ with row $j$ removed. 
Let $\hat{B}^{\star}$ be as defined in~\eqref{eq::Bstar},
\ben
\label{eq::GammaMain}
\; \hat\Gamma^{(j)} &:= & \hat{B}^{\star}_{\minus j, \minus j} =
\X_{\minus j}  \X_{\minus j}^T \oslash \hat{\TM}_{\minus j, {\minus 
    j}}  \in \R^{(n-1) \times (n-1)},\\
\label{eq::gammaMain}
\; \text{ and } \; \hat\gamma^{(j)} &:= & \hat{B}^{\star}_{\minus j,
  j} =\X_{\minus j} \X_j^T /\hat{\TM}_{\minus j, j}  \in \R^{(n-1) \times 1}.
\een
Let $\hat \beta^{(j)}  := \{\hat \beta_k^{(j)}, k \in [n], k 
\not=j\}$. For a chosen penalization parameter $\lambda > 0$ and a fixed
$\ell_1$-radius $b_1 > 0$, we consider the following  variant of the
Lasso estimator~\cite{Tib96,Chen:Dono:Saun:1998,LW12} for
the nodewise regression on the design $\X$, for each $j =1, \ldots, n$,
\begin{eqnarray}
\label{eq::origin} \; \; 
\hat \beta^{(j)} & = & \arg\min_{\beta \in \R^{n-1} \onenorm{\beta} \le b_1} \left\{ \frac{1}{2} \beta^T \hat\Gamma^{(j)} \beta - \ip{\hat\gamma^{(j)}, \beta} + \lambda \onenorm{\beta}\right\},
\end{eqnarray}
and the constraint parameter $b_1$ is to be chosen as an upper bound
on the $\ell_1$-norm of vector $\beta^{j*}$, namely, $b_1 \ge
\onenorm{\beta^{j*}} \; \forall j$, for $\beta^{j*}$ as in
\eqref{eq::regress}. See~\cite{LW12} for justifications of
\eqref{eq::origin}.
\begin{definition}{\textnormal{(\bf Estimating inverse  covariance)}}
\label{def::TopHat}
To construct an estimator for $\Theta = B_{\star}^{-1}$ with $\X$ as defined in~\eqref{eq::Xgen}:\\
\noindent{\bf Step 1.}
Obtain $n$ vectors of  $\hat{\beta}^{(j)}, j \in [n]$ by
solving \eqref{eq::origin} with $\hat\Gamma^{(j)}$ and
$\hat\gamma^{(j)}$ as in~\eqref{eq::GammaMain}
and~\eqref{eq::gammaMain}; \\
\noindent{\bf Step 2.}
Obtain an estimate for each row of $\Theta_0$ as follows:
$\forall j \in [n]$,
\ben
\label{eq::digest}
\quad  \tilde{\Theta}_{jj} =  (\hat{B}^{\star}_{jj} - \hat{B}^{\star}_{j, 
    \minus j} \hat{\beta}^{(j)}  )^{-1}
\;\text{ and }\; \tilde{\Theta}_{j, \minus j} = 
  -\tilde{\Theta}_{jj} \hat\beta^{(j)},
\een
where $\tilde{\Theta}_{j, \minus j}$ denotes the $j^{th}$ row of
$\tilde{\Theta}$ with diagonal entry $\tilde{\Theta}_{jj}$ removed,
and $\hat\beta^{(j)} = \{\hat{\beta}^{(j)}_k; \ k \in [n], k \not =j \}$.
Thus for $\hat{\NN}_{\beta}$ as defined in~\eqref{eq::negbeta}, we
have in the matrix form,
\ben
\label{eq::negbeta}
\Tilde\Theta
& = &
\diag(\tilde{\Theta}_{11}, \ldots, \tilde{\Theta}_{nn}) \cdot
\hat\NN_{\beta}; \; \; \text{ where} \; \; 
\hat{\NN}_{\beta}  =
\left[
\begin{array}{cccc} 
1 & -\hat{\beta}^{(1)}_{2} & \ldots & -\hat{\beta}^{(1)}_{n}\\
-\hat{\beta}^{(2)}_{1}  & 1 & \ldots & -\hat{\beta}^{(2)}_{n}\\
  \vdots & \vdots & \ddots & \vdots \\
  -\hat{\beta}^{(n)}_{1}  & -\hat{\beta}^{(n)}_{2}  & \ldots & 1 
\end{array} \right]_{n \times n };
\een
\noindent{\bf Step 3.}
Set $\hat\Theta =\arg\min_{\Theta \in S^{n}} \shnorm{\Theta -
  \tilde\Theta}_{\infty}$, where $S^n$ is the set of $n \times n$ symmetric matrices.
\end{definition}

\subsection{Preliminary results}
\label{sec::prelim}
Theorem~\ref{thm::Bernmgf} shows a concentration of measure bound on a
quadratic form with non-centered Bernoulli random variables where an explicit
dependency on $p_i, i =1, \ldots, m$ is shown.
Theorem~\ref{thm::Bernmgf} is crucial in proving 
Theorems~\ref{thm::gramsparse} and~\ref{thm::main}, as we need it to derive a concentration bound 
on $S_{\star}(q)$ as  defined in \eqref{eq::starship} and
$S_{\star}(q, h)$ as in~\eqref{eq::stardust}.
The setting here is different from Theorem~\ref{thm::HW} as we deal 
with a quadratic form that involves non-centered Bernoulli random 
variables, and the bound is especially useful in the present work as 
we allow $p_j \to 0$ for some or all coordinates of $j \in \{1,
\ldots, m\}$.
Naturally, the proof for Theorem~\ref{thm::Bernmgf} builds upon that
of the Hanson-Wright inequality in Theorem~\ref{thm::HW}~\cite{RV13}.
Moreover, although Theorem~\ref{thm::Bernmgf} follows the same line of arguments 
as Theorem 2.10 in~\cite{Zhou19}, we state the bound differently
here. Hence we prove Theorem~\ref{thm::Bernmgf} in the supplementary 
Section~\ref{sec::proofofbasemgf} for self-containment.
The constants presented in Theorem~\ref{thm::Bernmgf} statement are 
arbitrarily chosen. 
\begin{theorem}{\textnormal{\bf (Moment generating function for Bernoulli
      quadratic form)}}
\label{thm::Bernmgf}
Let $\xi = (\xi_1, \ldots, \xi_m) \in \{0, 1\}^m$ be a random vector 
with independent Bernoulli random variables $\xi_i$ such that 
$\xi_i = 1$ with probability $p_i$ and $0$ otherwise. 
Let $A = (a_{ij})$ be an $m \times m$ matrix.
Let $\sigma_i^2 = p_i(1-p_i)$.
Let $S_{\star} :=  \sum_{i,j} a_{ij} \xi_i \xi_j - \E 
\sum_{i,j} a_{ij} \xi_i \xi_j.$
Denote by  $D_{\max} := \norm{A}_{\infty} \vee \norm{A}_1$.
Then, for every $\abs{\lambda} \le 1/{(16(\norm{A}_1 \vee \norm{A}_{\infty}))}$,
\bens
\lefteqn{\E \exp(\lambda S_{\star})
\le   \exp\big(32.5 \lambda^2 D_{\max} e^{8 \abs{\lambda} D_{\max}}
  \sum_{i\not=j} \abs{a_{ij}} \sigma^2_j \sigma^2_i \big)  \cdot}\\
&& \exp\big(2 \lambda^2 D_{\max} e^{4\abs{\lambda} D_{\max}} 
\big(\sum_{i=1}^m \abs{a_{ii}} \sigma^2_i + 2 \sum_{i\not=j} \abs{a_{ij}} p_j p_i \big)  \big) 
\eens
\end{theorem}

\begin{theorem}{\textnormal{(Hanson-Wright inequality)~\textnormal{\cite{RV13}}}}
\label{thm::HW}
Let $Z \sim \mvec{\Z^T}$ for $\Z$ as defined in
\eqref{eq::missingdata}.
Let $A$ be an $mn \times mn$ matrix. Then, for every $t > 0$, 
\bens 
\prob{\abs{Z^T A Z - \expct{Z^T A Z} } > t} 
\leq  2 \exp \big(- c\min\big({t^2}/{(K^4\fnorm{A}^2)}, {t}/{(K^2\twonorm{A})} \big)\big). 
\eens 
\end{theorem}

Next, we need the following definitions on the maximum sparse eigenvalue 
for $B_0$ and $\abs{B_0}$ respectively. 
We then state in Lemma~\ref{lemma::converse} an upper bound on 
$\psi_B(s_0)$ that depends on $s_0$ rather than $n$. 
\begin{definition}{\textnormal{(Sparse eigenvalue for $\abs{B_0}$)}}
\label{def::C0}
Let $1\le s_0 \le n$.
 Let $\abs{B_0} := (\abs{b_{ij}})$, the entrywise absolute value of
 a matrix $B_0$.
 Denote by $\psi_B(n) = {\twonorm{\abs{B_0}}}/{\twonorm{B_0}}$.
 Let 
 \bens
 \psi_B(s_0) := {\rho_{\max}(s_0, \abs{B_0})}/{\twonorm{B_0}} , \; \text{ where } \;
\rho_{\max}(s_0, \abs{B_0}) := \max_{q \in \Sp^{n-1}, s_0-\sparse} \sum_{i, j} \abs{b_{ij}}
\abs{q_i}\abs{q_j}.
\eens
\end{definition}

\begin{definition}
  \label{def::sparse-eigen}
For $1\le s_0 \leq  n$, we define the largest $s_0$-sparse eigenvalue
of an $n \times n$ matrix $B_0 \succ 0$ to be:  $\rho_{\max}(s_0, B_0) := 
\max_{v \in \Sp^{n-1}; s_0-\text{sparse}} \; \; v^T B_0 v$.
As a consequence of the Rayleigh-Ritz theorem,
\ben 
\label{eq::eigen-Sigma}
\max_{j} b_{jj} =: b_{\infty} & \le &\rho_{\max}(s_0, B_0) \le 
\twonorm{B_0} \le \twonorm{ (\abs{b_{ij}})}
\een
\end{definition}

\begin{lemma}
\label{lemma::converse}
Let $\abs{B_0}_{S,S}$ denote the principal submatrix of 
  $\abs{B_0}$ with rows and columns indexed by $S \subseteq [n] = \{1, 
  \ldots, n\}$. Let $1 \le s_0 \le n$.  Denote by $b_{\infty} = \max_{j} b_{jj}$.
   Then for $\psi_B(s_0)$ and $\rho_{\max}(s_0, \abs{B_0})$ as in 
Definition~\ref{def::C0}, we have $\rho_{\max}(s_0, \abs{B_0}) \ge  b_{\infty} \ge 1$, 
\bens 
\rho_{\max}(s_0, \abs{B_0})
& = &
 \max_{S \subset [n]: \abs{S} = s_0}\lambda_{\max}(\abs{B_0}_{S,S}) 
 \le \sqrt{s_0} \twonorm{B_0},\;  \text{ and }  \; \psi_B(s_0) 
\le \sqrt{s_0}.
\eens 
\end{lemma}
Here and in the sequel, denote by $a_{\infty} = \max_{j}
a_{jj}$, $a_{\min} = \min_{j} a_{jj}$, and $\eta_A =
\sqrt{{a_{\infty}}/{a_{\min}}}$.
Let $W :=\diag(b^{1/2}_{11}, \ldots, b^{1/2}_{nn})$ have bounded positive
entries. We state Assumption (A1).\\
\noindent{\bf[A1.]}
Denote by $\rho(B) = W^{-1} B_0 W^{-1}$
the correlation matrix for $B_0$ and $\Omega$  its inverse.
Suppose $1\le  \twonorm{\rho(B)} \le  M_{\rho} \;  \text{ and } \; 1\le
\twonorm{\Omega}  \le M_{\Omega}$. Let $\kappa_{\rho} :=
\twonorm{\rho(B)} \twonorm{\Omega}$ and $\tilde{\kappa}_{\rho} =
M_{\rho} M_{\Omega}$ be an upper  estimate.
\section{The main theorems}
\label{sec::theory}
We now present our first main result in Theorem~\ref{thm::main-intro}.
We are not optimizing over the logarithmic factors in this paper.
Similar results for the plug-in estimator $\hat{B}^{\star}$ are
presented in Theorem~\ref{thm::main-coro}.
\begin{theorem}{\textnormal{\bf (Overall bounds for oracle $\tilde{B}_0$)}} 
\label{thm::main-intro}
Consider the data generating random matrices as in \eqref{eq::Xgen}
and \eqref{eq::maskUV}.
Let $B_0 \succ 0$ and $\tilde{B}_0$ be given in
Definition~\ref{def::RMSet}.
Set $1 \le  s_0 \le n$. 
Let $\psi_B(s_0)$ be as in Definition~\ref{def::C0}.
Let $0< \ve < 1/2$.
Let $c, c', C_1, C_2, C_4>1$ be absolute constants.
Suppose
\ben 
\label{eq::baseline}
{\sum_{j=1}^m a_{jj} p_j^2}/{ \twonorm{A_0}} \ge C_4 \eta_A^2 ( \psi_B(d) \vee 1)
 s_0 \log (n \vee m), \; \text{where} \; d = 2s_0 \wedge n.
 \een
 Then with probability at least $1-{c'}/{(n \vee m)^4} -4 \exp(-c s_0 \log  (3e n/(s_0 \ve)))$, we  have 
\ben
\label{eq::ratemain}
\sup_{q \in \sqrt{s_0} B_1^n \cap B_2^n} 
\inv{\twonorm{B_0} }\abs{q^T (\tilde{B}_0 -B_0) q}
  \le C_1 \eta_A  r_{\offd}(s_0) \ell^{1/2}_{s_0, n} + C_2
  r_{\offd}^2(s_0) \psi_B(d),
  \een   
where $B_1^n, B_2^n$ denote the unit $\ell_1$ and $\ell_2$ balls
respectively, 
\ben
\label{eq::offdrate}
\quad
r_{\offd}(s_0) =
\sqrt{s_0 \log \big(\frac{3e n}{s_0\ve}\big) 
  \frac{\twonorm{A_0}}{\sum_{j} a_{jj}  p_j^2}}, \; \text{ and }
 \ell_{s_0, n} = \frac{\log (n \vee m)}{\log (3en/(s_0 \ve) )}.
\een
\end{theorem}
\noindent{\bf Remarks.}
We now unpack Theorem~\ref{thm::main-intro}
and show its connection to the $\RE$ conditions in Definition~\ref{def::lowRE}.
\begin{definition}{\textnormal{(Lower-and-Upper-$\RE$ conditions)~\cite{LW12}}}
\label{def::lowRE}
The matrix $\Gamma$ satisfies a Lower-$\RE$ condition with curvature
$\alpha >0$ and tolerance $\tau > 0$ if 
$\theta^T \Gamma \theta \ge \alpha \twonorm{\theta}^2 - \tau \onenorm{\theta}^2 \; \;  \forall
\theta \in \R^m.$
The matrix $\Gamma$ satisfies an upper-RE condition with smoothness $\tilde\alpha >0$ and tolerance $\tau > 0$ if 
$\theta^T \Gamma \theta \le \tilde\alpha \twonorm{\theta}^2 + \tau 
\onenorm{\theta}^2, \forall \theta \in \R^m$.
\end{definition}
As $\alpha$ becomes smaller, or as $\tau$ becomes larger, 
the Lower-RE condition is easier to be satisfied. 
Similarly, as $\tilde\alpha$ or $\tau$ becomes larger, 
the Upper-RE condition is easier to be satisfied. 

Clearly, the set of vectors as defined in \eqref{eq::cone} satisfy 
\bens 
\cone(s_0) \cap \Sp^{n-1} \subset \sqrt{s_0} B_1^n \cap B_2^n.
\eens
Therefore, we have for $\d \asymp
\eta_{A} r_{\offd}(s_0)  \ell^{1/2}_{s_0, n} + r_{\offd}^2(s_0) \psi_B(2s_0 \wedge n) $, 
\ben 
\label{eq::RElower}
q^T \tilde{B}_0 q \ge q^T B_0 q - \delta \twonorm{B_0} \ge 
\lambda_{\min}(B_0) - \delta \twonorm{B_0}> 0,
\een
for all $q \in \cone(s_0) \cap \Sp^{n-1}$ so long as the condition number $\twonorm{B_0}
/\lambda_{\min}(B_0)<\infty$ is bounded and the lower bound on
the sampling rate is sufficiently strong in the sense
of~\eqref{eq::baseline}.
The lower bound in~\eqref{eq::RElower} ensures that the
RE condition as defined in~\cite{BRT09} holds.
Here and in the sequel, denote by
\ben
\label{eq::diagr}
r_{\diag} & := &\eta_A \big({\twonorm{A_0} \log (m \vee n) 
}/{\norm{\M}_{\diag}}  \big)^{1/2}, \; \text{ and}   \\
 \label{eq::paritydual}
 \ul{r_{\offd}}
 & := & \eta_A \big(\twonorm{A_0} \log (m \vee 
   n)/\norm{\M}_{\offd} \big)^{1/2}, \; \text{ where} \\
\label{eq::definmaxM}
\norm{\M}_{\diag} & := & \sum_{j=1}^m  a_{jj} \z_j \; \text{ and } \; \norm{\M}_{\offd} := \sum_{j=1}^m   a_{jj}
\z^2_j
\een
denote the componentwise matrix max norm for $\diag(\M)$ and $\offd(\M)$  respectively.
As we show in Theorems~\ref{thm::diagmain} and~\ref{thm::main-coro}
and Lemma~\ref{lemma::pairwise},
$r_{\diag}$ and $\ul{r_{\offd}}$ dominate the maximum of entrywise errors for estimating the
   diagonal and the off-diagonal elements of $B_0$ respectively.
The diagonal component has a tighter concentration than the
off-diagonal component since the probability that we {\it observe }
two entries $y^i_k, y^j_k, \forall i \not=j, k \in [m]$ simultaneously is much smaller than that of
any single entry, where recall $y^j, j \in [n]$ are row vectors of
$X$. Hence a rate slower than $r_{\diag}$ is expected for
estimating $b_{ij}$ with $\tilde{B}_{0, ij}, \forall i \not=j$;
cf. Lemma~\ref{lemma::pairwise} and Theorem~\ref{thm::main-coro}.
Consequently, the RE conditions are predominantly controlled by the
concentration bounds with regard to the off-diagonal component of
$\tilde{B}_0$, and the first term in~\eqref{eq::ratemain} is closely related to $\ul{r_{\offd}}$, namely,    
   \bens
   (I): \quad
   r_{\offd}(s_0) \ell^{1/2}_{s_0, n} =\sqrt{s_0 \log (n \vee m)}
\sqrt{\frac{\twonorm{A_0}}{\sum_{j} a_{jj}  p_j^2} }\asymp \sqrt{s_0} \ul{r_{\offd}}
   \eens
for $r_{\offd}(s_0)$ and $\ell_{s_0, n}$ as defined in
\eqref{eq::offdrate}. The effective sample size~\eqref{eq::baseline} ensures that in the
missing value setting, the second term in \eqref{eq::ratemain} also
converges to 0:
\bens
(II): \quad r^2_{\offd}(s_0) \psi_B(2s_0 \wedge n)
\asymp s_0 \psi_B(2s_0 \wedge n) \log \big(\frac{3e n}{s_0\ve}\big)
\frac{\twonorm{A_0}}{\sum_{j} a_{jj}  p_j^2} \to 0;
\eens
As we will illustrate in our numerical examples in  
Section~\ref{sec::examples}, when the sample size is small, the second
term (II) dominates; otherwise, the first term
(I) dominates. Moreover, in order for the first term, namely,
$r_{\offd}(s_0)$ to completely dominate the second term of $O(r_{\offd}^2(s_0)  \psi_B(2s_0 \wedge n))$ in
\eqref{eq::ratemain}, the sample size needs to satisfy: ${\sum_{j=1}^m a_{jj} p_j^2
}/{\twonorm{A_0}} = \Omega\big(\psi_B^2(2s_0 \wedge n) s_0
\log\big(\frac{3en}{s_0\ve}\big)\big)$,
which is clearly more restrictive than~\eqref{eq::baseline}. Hence 
both terms are needed in characterizing the final
rate~\eqref{eq::ratemain}.
We will discuss the error rate~\eqref{eq::ratemain} and sample 
size lower bound~\eqref{eq::baseline} further in Section
\ref{sec::RE}.

\subsection{The restricted eigenvalue conditions}
\label{sec::RE}
We now state in Theorem~\ref{thm::RE} that, under the slightly stronger condition \eqref{eq::forte}
on the effective sample size,  the Lower and  Upper-$\RE$ conditions hold for $\tilde{B}_0$
with suitably chosen curvature $\alpha$, smoothness $\tilde{\alpha}$, 
and tolerance $\tau$ parameters.
Let $\kappa_B :=  \twonorm{B_0}/\lambda_{\min}(B_0)$.
Let $C_{\diag}, C_{\sparse}$ be absolute constants which we use 
throughout this paper; cf. Theorems~\ref{thm::diagmain}
and~\ref{thm::main}.
\begin{theorem}
  \label{thm::RE}
  Set $0< \ve < 1/2$.  Set $1\le s_0 \le n$.
  Suppose $n \ge 3e/(2\ve)$.  Suppose (A1) holds. 
Let $c, C$, $C_{\RE} = 2 (C_{\diag} \vee C_{\sparse} \vee 1)$ be absolute constants.
Suppose all conditions in Theorem~\ref{thm::main-intro} hold and
\ben 
\label{eq::forte}
\quad \quad {\sum_{j} a_{jj} p_j^2}/{\twonorm{A_0}} \ge 12^2 C^2_{\RE} \kappa_B
s_0 \log (m \vee n) \eta^2_A
\big(\psi_B(2s_0 \wedge n) \vee  \kappa_B \big).
\een
Then with probability at least $1-{C}/{(m \vee n)^4} - 4\exp(-c
s_0 \log (3e n/(s_0 \ve)))$, for all $q \in \R^{n}$,
\ben
\label{eq::BDlow}
\quad q^T \tilde{B}_0 q & \ge & \frac{5}{8} \lambda_{\min}(B_0) 
 \twonorm{q}^2 -\frac{3 \lambda_{\min}(B_0)}{8 s_0} \onenorm{q}^2, \;
 \text{ and } \\
 \label{eq::BDup}
\quad q^T \tilde{B}_0 q & \le & ( \lambda_{\max}(B_0) + 
 \frac{3}{8} \lambda_{\min}(B_0)) 
 \twonorm{q}^2  +\frac{3\lambda_{\min}(B_0)}{8 s_0} \onenorm{q}^2.
 \een
\end{theorem}

\noindent{\bf Discussions.}
It was shown previously in~\cite{Zhou19} that 
$\tilde{A}_{0}$ and $\tilde{B}_{0}$ provide accurate 
componentwise estimates for $A_0$ and $B_0$, once the sample size is
sufficiently large, namely,  {\bf (a)} $\forall i \not=j$,
$p_i p_j =\Omega\big({\log m  \twonorm{B_0}}/{\tr(B_0)} \big)$ and
 {\bf (b)} $\sum_{j} a_{jj} p_j^2  =\Omega\big({\twonorm{A_0} \eta_A^2 \log
  m}\big)$.
However, it is not at all clear that one can obtain convergence in the operator norm
from such results for estimating the covariance $A_{0}$
and $B_{0}$ as a whole as they may not be close to the
positive-semidefinite cones of appropriate dimensions.
This issue is both inherent in the matrix variate model \eqref{eq::separable} due to the scarcity in the
samples available to estimate the larger covariance matrix as
mentioned earlier, and also clearly exacerbated by the overwhelming
presence of missing values~\cite{Zhou14a,Zhou19}.

To prove uniform concentration of measure bounds for
the quadratic form we pursue in Theorems~\ref{thm::main-intro} and~\ref{thm::RE}, we
drop condition {\bf (a)} above while strengthening condition {\bf (b)}, because
condition  {\bf (a)} is only needed in order for $\tilde{A}_0$ to have componentwise convergence.
More precisely, ignoring the logarithmic factors, an additional factor 
of $s_0 \psi_B(2s_0 \wedge n)$ is now needed in both~\eqref{eq::baseline}
and~\eqref{eq::forte} in order to control the quadratic form~\eqref{eq::quadorig} over the cone
$\cone(s_0) \cap \Sp^{n-1}$. When $p_j = 1, \forall j$, such concentration bounds
were shown before for random design matrix with independent subgaussian rows or columns; see for
example~\cite{RZ13,LW12,RZ17} and references therein.
Specifically, we will state in Theorem~\ref{thm::AD} a result for
 the fully observed matrix variate subgaussian data
 $X$~\eqref{eq::missingdata}.
 It is evident that due to the mathematical complexities arising from missing 
 values, as considered in~\eqref{eq::missingdata}
 and~\eqref{eq::Xgen}, one would not expect to derive 
 results identical to the complete data scenarios, even when $p_j$s 
 are close to 1; however, one can hope these are close.
\begin{theorem}{\textnormal{\bf (RE conditions for full data matrix $X$)}}
\label{thm::AD}
Set $1\le s_0 \le n$. 
Let $B_0 \succ 0$ be given in Definition~\ref{def::RMSet}, and
$\rho_{\max}(s_0, B_0)$ be as in Definition~\ref{def::sparse-eigen}.
Let $A_0 \in \R^{m \times m}$ be symmetric positive definite
and $X = B_0^{1/2} Z A_0^{1/2}$ as in~\eqref{eq::missingdata}.
Let $c_2, c' > 1$, and $C$ be some absolute constants.
Suppose 
\ben
\label{eq::trB}
\frac{\tr(A_0)}{\twonorm{A_0}} & \ge & c' \frac{s_0}{\ve^2}
\log\left(\frac{3e n}{s_0 \ve}\right) \;
\text{ where} \; \; \ve \le \frac{3 \lambda_{\min}(B_0)}{128C \rho_{\max}(s_0, B_0)}.
\een
Then with probability at least $1- 4 \exp(-c_2 \ve^2 r(A_0))$,
the Lower and Upper $\RE$ conditions \eqref{eq::BDlow} and
\eqref{eq::BDup} hold with $\hat{B} = X X^T/{\tr(A_0)}$ replacing
$ \tilde{B}_0$. \silent{Then clearly,
the lower and upper $\RE$ conditions hold for the gram matrix
 $XX^T/\tr(A_0)$ with probability at least $1- 4 \exp(-c_2 \ve^2 r(A_0))$, 
\bens
\text{curvature} && 
\alpha = \frac{5}{8}\lambda_{\min}(B_0), \;\text{smoothness} \; \; \tilde\alpha
= \frac{11}{8}\lambda_{\max}(B_0) \\
\text{ and tolerance } && 
 \tau :=  \frac{\lambda_{\min}(B_0) - \alpha}{s_0}  = \frac{3
   \lambda_{\min}(B_0)}{8 s_0}, \; \text{ where} \; \alpha,
 \tilde\alpha, \tau \; \text{ are as in Definition~\ref{def::lowRE}}.
 \eens}
 \end{theorem}

\silent{for all $q \in \R^n$, 
\ben
\label{eq::BDloworig}
q^T \hat{B} q & \ge & \frac{5}{8} \lambda_{\min}(B_0) 
 \twonorm{q}^2 -\frac{3 \lambda_{\min}(B_0)}{8 s_0} \onenorm{q}^2 \\
 \label{eq::BDuporig}
 q^T \hat{B} q & \le & ( \lambda_{\max}(B_0) +\frac{3}{8} \lambda_{\min}(B_0)) 
 \twonorm{q}^2  +\frac{3\lambda_{\min}(B_0)}{8 s_0} \onenorm{q}^2.
 \een}

\silent{
for $M_B \asymp
\rho_{\max}(s_0, B_0)/\lambda_{\min}(B_0)$ for $\lambda_{\min}(B_0) >0$ and $\tr(B_0) \asymp n$,
\ben
\label{eq::fulldata}
\frac{\tr(A_0)}{\twonorm{A_0}} \ge \frac{c' s_0}{\ve^2} \log
\big(\frac{3e n}{s_0\ve}\big) \; \; \text{ where  } \;
\; \ve = \inv{2M_B} \asymp \frac{\lambda_{\min}(B_0)}{2\rho_{\max}(s_0, B_0)}
\een}

\noindent{\bf Implications.}
The key difference between \eqref{eq::trB} and \eqref{eq::baseline} is
the extra $\psi_B(2s_0 \wedge n)$ factor appearing in
\eqref{eq::baseline}.
When $\psi_B(2s_0 \wedge n)$ grows only mildly with $s_0$, 
the sample lower bounds~\eqref{eq::baseline} and~\eqref{eq::trB}
are almost identical when we set $p_j  = 1, \forall j$; in particular,
this holds when $B_0 \succ 0$ is a diagonal matrix or a matrix with
all positive entries.
However, when $\psi_B(2s_0 \wedge n)$ grows with $s_0$, then potentially we have 
superlinear (but sub-quadratic) dependency on $s_0$, since $\psi_B(2s_0 \wedge n) \le \sqrt{2s_0 
  \wedge n}$ by Lemma~\ref{lemma::converse}.
We mention in passing that the potential superlinear dependency on $s_0$
as in \eqref{eq::forte} is the cost associated with our rather complex
probabilistic arguments when dealing with missing values in the matrix
variate model as in Definition~\ref{def::RMS}.
We will discuss key proof strategies for 
Theorem~\ref{thm::main-intro} in Section~\ref{sec::strategy} and defer 
its proof to Section~\ref{sec::appendintromain}. 
We elaborate upon the sample size requirements further in
Remark~\ref{rem::threefactor}, and
Sections~\ref{sec::reduction} and~\ref{sec::mainresult}.
We prove Theorem~\ref{thm::RE} and the related Lemma~\ref{lemma::converse} in the supplementary 
Section~\ref{sec::appendproofofRE}. 
We prove Theorem~\ref{thm::AD} in Section~\ref{sec::proofofthmAD} for
 the sake of self-containment.

\subsection{Error bounds on $\hat{\ensuremath{\mathbb M}}$ and $\hat{B}^{\star}$}
\label{sec::Bhatmain}
Our main goal in this section is to  show in
Theorem~\ref{thm::main-coro}
the error bounds on the quadratic form $\abs{q^T(\X \X^T \oslash
  \hat{\TM}- \E  (\X \X^T  \oslash \M))q}$ as defined
in~\eqref{eq::quadpreview}.
For completeness,  we also state the large deviation bounds on 
$\hat{\TM}$ in Lemma~\ref{lemma::unbiasedmask}, which is 
crucial for proving Theorem~\ref{thm::main-coro} and Corollary~\ref{coro::Bhatnorm}. 
As we now elaborate in Theorem~\ref{thm::diagmain},
for diagonal matrices, the elementwise matrix maximum norm and the 
operator norm coincide.
For general matrices, however, they do not 
match and hence new tools must be developed to obtain sharp 
operator norm bounds.
Theorem~\ref{thm::diagmain} follows from~\cite{Zhou19}, cf. Theorems 
5.3, 5.4 and eq. (33) therein, upon adjusting the constants.
The proof is omitted.
On the other hand,
the componentwise rates of convergence for $\offd(B_0)$ are
restated in Lemma~\ref{lemma::pairwise}.
Let $c, C, C', C_{\diag},  C_{\offd}, \ldots$ be absolute constants which may change line by 
line.  Let $r_{\diag}$ and   $\underline{r_{\offd}}$ be as defined in 
\eqref{eq::diagr} and~\eqref{eq::paritydual} respectively.

\begin{theorem}{\textnormal{\cite{Zhou19}}}
  \label{thm::diagmain}
Suppose ${\sum_{j=1}^m a_{jj} p_j}/{\twonorm{A_0}} =
\Omega\big(\eta_A^2 \log (n \vee m)\big)$.
Then
\bens
\prob{\maxnorm{\diag(\X \X^T) -  \E\diag(\X 
     \X^T)}/{\norm{\M}_{\diag}} >  C_{\diag} b_{\infty} r_{\diag}} =: \prob{\F_{\diag}} 
\le {C}/{(n \vee m)^d}
\eens
for $d \ge 4$. Moreover, we have on $\F_{\diag}^c$,
\ben
\label{eq::convex}
\twonorm{\diag(\tilde{B}_0 - B_0)}
:=   \sup_{q \in \Sp^{n-1}} \abs{q^T (\diag(\X \X^T)-
  \E \diag(\X  \X^T))    q }/{\norm{\M}_{\diag}} \le C_{\diag} b_{\infty} r_{\diag}
  \een
\end{theorem}

\begin{lemma}{\textnormal{\cite{Zhou19}}}
  \label{lemma::pairwise}
 Suppose ${\sum_{i=1}^m a_{jj} p_i^2}/{\twonorm{A_0}}=
 \Omega(\log (m \vee n))$. Suppose (A1) holds. Suppose $n$ is sufficiently large.
Then
\bens
\label{eq::pairwise}
\prob{\F_{6}}  & := &
\prob{\max_{i, j, i \not=j} \; \; 
\abs{b_{ij} - {\ip{X^i \circ v^i, X^{j} \circ v^j}}/{\shnorm{\M}_{\offd}}}
>  C_{\offd}\sqrt{b_{ii} b_{jj}}   \underline{r_{\offd}}} \le  {c_6}/{(n \vee   m)^4}.
\eens
\silent{Suppose all conditions in Lemma~\ref{lemma::unbiasedmask} hold.
Then, we have on event $\F_5^c \cap \F_{6}^c$, where
$\prob{\F_5^c \cap \F_{6}^c} \ge 1-{c_8}/{(n \vee m)^4}$,
\bens
\label{eq::pairwisemask}
\forall i \not=j, \; \; \abs{ b_{ij} -    \ip{X^i \circ v^i, X^{j}
        \circ v^j}/{\shnorm{\hat{\TM}}_{\offd}}  }\le  C_{\offd}(1+o(1)) \sqrt{b_{ii} b_{jj}}  \ul{r_{\offd}}.
\eens}
\end{lemma}
We now state in Lemma~\ref{lemma::unbiasedmask}
that $\hat{\TM}$ as defined in~\eqref{eq::Bstar} is a componentwise unbiased 
estimator for $\M$ as defined in \eqref{eq::entrymean};
Moreover, the entries $\hat{\TM}_{k \ell}, \forall k,
\ell$ are tightly concentrated around their mean values under
Definition~\ref{def::RMS}.
Here $\hat{\TM}$  has only two unique entries.
We prove Lemmas ~\ref{lemma::pairwise},~\ref{lemma::unbiasedmask}, 
Theorem~\ref{thm::main-coro}, and Corollary~\ref{coro::Bhatnorm}
in the supplementary Sections~\ref{sec::proofofpairwise}, ~\ref{sec::maskest}, and~\ref{sec::Bstarthm}
respectively. 
We will elaborate on the plug-in estimator $\hat{\TM}$ 
and prove concentration bounds for $\hat{\ensuremath{\mathbb M}}$ in
Sections~\ref{sec::maskMN} and~\ref{sec::TMunbiased}
respectively.
\begin{lemma}
  \label{lemma::unbiasedmask}
  Suppose that $\sum_{s=1}^m a^2_{ss} p_s^2 > C a^2_{\infty} \log (m\vee 
  n)$. 
  Denote by
\bens
\hat{\TM}_{k \ell} & =: & \shnorm{\hat{\TM}}_{\offd} \text{ in case } k
\not= \ell, \text{ and }  \hat{\TM}_{\ell \ell} =:
\shnorm{\hat{\TM}}_{\diag} =\onen \tr(\X^T \X) \; \; \forall \ell,
\eens
the off-diagonal and diagonal components in $\hat{\TM}$ as defined in~\eqref{eq::Bstar},
respectively. 
Moreover, $\E \hat{\TM} = \M$, for $\M$ as defined in \eqref{eq::entrymean}.
Then we have with probability at least $1 - {C'}/{(n \vee m)^4}$,
\ben
\label{eq::delmask}
&& \quad \quad  \delta_{\mask}
:= {\abs{\shnorm{\M}_{\offd} -\shnorm{\hat{\TM}}_{\offd}}}/{\shnorm{\M}_{\offd}}
=  O\big({\underline{r_{\offd}} }/{\sqrt{r(B_0)}}\big), \\
\label{eq::maskorder}
&&
\delta_{m,\diag} :=
{\abs{\shnorm{\M}_{\diag}    -\shnorm{\hat{\TM}}_{\diag}}}/{\shnorm{\M}_{\diag}}
=  O\big( {r_{\diag}}/{ \sqrt{r(B_0)}}\big),
\een
where $r(B_0)={\tr(B_0)}/{\twonorm{B_0}} = \Omega(n)$ for $\tr(B_0) =n$.
\end{lemma}
Lemma~\ref{lemma::unbiasedmask} guarantees that 
the relative errors in the operator norm for the oracle $\tilde{B}_0$ and the plug-in
estimator $\hat{B}^{\star}$ as in~\eqref{eq::Bstar} are controlled 
essentially at the same order so long as $r(B_0) =\Omega(n)$;
cf. Theorem~\ref{thm::gramsparse}, and Corollary~\ref{coro::Bhatnorm};
Moreover, $\RE$ conditions hold for both estimators with the same
parameters as we now show in Theorem~\ref{thm::main-coro}.

\begin{theorem}{\textnormal{\bf (Overall bounds with $\hat{B}^{\star}$)}} 
  \label{thm::main-coro}
  Let $\hat{B}^{\star}$ be as defined in~\eqref{eq::Bstar}.
  Suppose (A1) holds and $n$ is sufficiently large. Let $s_0 \in [n]$.
Suppose all conditions in Theorem~\ref{thm::RE} hold and the lower
bound~\eqref{eq::forte} holds with $C_{\RE} \ge 2  (C_{\diag} \vee
C_{\sparse}\vee 1)(1+o(1))$. Then, the following statements hold with
probability at least $1- {C}/{(n \vee m)^4} -4 \exp(-c s_0  \log  (3e
n/(s_0 \ve)))$, \\
(a) $\shnorm{\diag(\hat{B}^{\star} - B_0)}_{\max} \le  C_{\diag} (1+o(1))
b_{\infty} r_{\diag}$; \\
(b)
$\shnorm{\offd(\hat{B}^{\star} - B_0)}_{\max} \le  C_{\offd} (1+o(1)) b_{\infty} \ul{r_{\offd}}$; \\
(c) The Lower-and-Upper-$\RE$ conditions as  
in~\eqref{eq::BDlow} and \eqref{eq::BDup} hold with $\hat{B}^{\star}$ 
replacing $\tilde{B}_0$, and
\bens
\label{eq::eventBstar}
&& \sup_{q \in \sqrt{s_0} B_1^n \cap B_2^n} 
\abs{q^T (\hat{B}^{\star} - B_0)  q}/{\twonorm{B_{0}} } =O\big(\ul{r_{\offd}}
\sqrt{s_0} + \psi_B(2s_0 \wedge n) r_{\offd}^2(s_0)\big).
\eens
\end{theorem}
\subsection{The operator norm error bounds for $\tilde{B}_0$ and
   $\hat{B}^{\star}$}
 \label{sec::opnorm}
In Theorem~\ref{thm::gramsparse} and Corollary~\ref{coro::Bhatnorm},
we control the quadratic forms~\eqref{eq::quadorig}
and~\eqref{eq::quadpreview} over the entire sphere.
In the proof of Theorem~\ref{thm::gramsparse}, we 
  show a more refined bound: $\shnorm{\tilde{B}_0 -B_0}_2 
  =O_P(\delta_q(B))$; cf~\eqref{eq::Bhatop}. 
The same error bound holds for $\hat{B}^{\star}$ as 
we state in Corollary~\ref{coro::Bhatnorm} and illustrate numerically
in Section~\ref{sec::examples}.
\begin{theorem}{\textnormal{\bf (Operator norm bound)}} 
  \label{thm::gramsparse}
  Suppose all conditions in Theorem~\ref{thm::main-intro} hold with
  $s_0 = n$. Suppose that $n = \Omega(\log (m \vee n))$.
Let  $r_{\offd}(n) := \sqrt{{n \twonorm{A_0}}/{(\sum_{j}
    a_{jj}  p_j^2)}}$. Let $\psi_B(n) =
{\twonorm{\abs{B_0}}}/{\twonorm{B_0}}$.
Then we have for some absolute constants $C, C_1, C_2$,
 with probability at least $1-{C}/{(n \vee m)^4}$,  
 \ben
\label{eq::opnorm} 
&& \twonorm{\tilde{B}_0 -B_0}/{\twonorm{B_0} }  \le 
C_1 \eta_A r_{\offd}(n)  + C_2 r^2_{\offd}(n) \psi_B(n)
\log^{1/2} (n 
\vee m)\big).
\een
\end{theorem}
\begin{corollary}
  \label{coro::Bhatnorm}
Let  $p_{\max} := \max_{i} p_i$.
Under the conditions in Theorems~\ref{thm::main-coro} and  \ref{thm::gramsparse}, we have
with probability at least $1-C'/(n \vee m)^4$,
$\quad \shnorm{\hat{B}^{\star}  -B_0}_2/{\twonorm{B_0} }  =O(\delta_q(B))$, where
 \ben
\nonumber
\lefteqn{\delta_q(B) \asymp \eta_A r_{\offd}(n) +   \big(1 +
  {a_{\infty} \psi_B(n)}/{\twonorm{A_0}} \big) r^2_{\offd}(n) + \tau_p,}\\
\label{eq::Bhatoplocal}
& &\text{ for } \quad \tau_p \asymp \eta^{1/2}_A \sqrt{p_{\max} } r^{3/2}_{\offd}(n) \psi_B(n)
\big({\log (n \vee  m) }/{n}\big)^{1/4}.
\een
\end{corollary}

\begin{remark}
  \label{rem::threefactor}
When $\psi_B(n) = O(\sqrt{n/\log (n \vee m)})$, $\tau_p$ in \eqref{eq::Bhatoplocal} can be absorbed into
  the other two and
\ben
\nonumber
\label{eq::classic}
&&
\twonorm{\hat{B}^{\star}  -B_0}/{\twonorm{B_0} }
=  O_P\Big(r_{\offd}(n)  + \big(p_{\max} + 1/\twonorm{A_0}\big) 
r^2_{\offd}(n) \psi_B(n) \Big).
\een
\end{remark}

Clearly, in the setting of missing values under consideration, 
\bens
x = 1/{r_{\offd}(n)^2}
\; \text{is used to replace} \; {r(A_0)}/{n} := {\tr(A_0)}/{(\twonorm{A_0}  n)} 
\; \text{ when} \; p_j = 1, \forall j.
\eens
In the classical setting of covariance estimation with independent
columns, the rate of convergence is 
\bens
\shnorm{\tilde{B}  -B_0}_2/{\twonorm{B_0} }
=O_P\big(1/\sqrt{x} + 1/{x} \big), \; \; \text{ where } \; x =
r(A_0)/n = m/n
\eens
since the effective rank $r(A_0) = m$ in case $A_0 = I_m$; See
Exercise 4.7.3~\cite{Vers18}.
Hence, we can immediately recover the classical result in case $B_0 = I_n$, or 
more generally,  when $\psi_B(n) \asymp 1$, using the error bounds in
Theorem~\ref{thm::gramsparse} in  view of \eqref{eq::classic}.

Next, we compare with the case when the matrix-variate data $X$
is observed in full and free of noise, that is,  when $p_j =1,
\forall j$.
We will present Theorem~\ref{thm::XX^T} by Rudelson~\cite{Rudelson13} and its
Corollary~\ref{coro::tartan}, which state that $XX^T /\tr(A_0)
\approx B_0$ and the relative error for covariance estimation is
guaranteed to be bounded by $O(\ve)$,
for any $1> \ve>0$, so long as the effective rank $r(A_0)$ is proportional to $n$,
or more precisely, so long as $r(A_0) = \Omega({n \log(3/\ve)}/{\ve^2})$.
We prove Corollary~\ref{coro::tartan} in
Section~\ref{sec::proofoftartan}, following the same line of arguments in~\cite{Rudelson13}.
Let $U,V$ be $n \times n$ symmetric matrices. We will write $U \le V$
if the matrix $V-U$ is positive semidefinite.
\begin{theorem}{{\textnormal{\bf (Operator norm bound for full matrix $X$)}}~\cite{Rudelson13} }
  \label{thm::XX^T}
  Let $n \le m$, and let $X$ be an $n \times m$ random matrix as
  defined in~\eqref{eq::missingdata}.  Then for any $1> \ve>0$,
 \bens
   \P \left( (1-\ve)B_0 \le XX^T /\tr(A_0)\le (1+\ve) B_0 \right)
   \ge 1 - \exp \left( -c \ve^2 {\tr(A_0)}/\twonorm{A_0} \right).
   \eens
   provided that $\tr(A_0)/\twonorm{A_0} \ge c'  {n
     \log(3/\ve)}/{\ve^2}$, where $c, c'$ are absolute constants.

\end{theorem}

\begin{corollary}{{\textnormal{\bf (Relative error for covariance
        estimation)}}}
\label{coro::tartan}
Under the conditions in Theorem~\ref{thm::XX^T},
we have for some absolute constants $c, c', C$, with probability at least $1-\exp(c \ve^2 r(A_0))$,  
\bens
\twonorm{XX^T/\tr(A_0) - B_0} /\twonorm{B_0}
\le  4 C \ve\; \; \text{ where } \; \ve^2 r(A_0) \ge c' n \log(3/\ve).
\eens
\end{corollary}
\silent{Notice that since $1> \ve >0$ and $\ve^2 =\Omega( n/r(A_0))$, we have
\bens
1/{\ve} =O\left(\sqrt{{r(A_0)}/{n}}\right), \; \text{  and hence }\;
\; \log (3/\ve) = O\left(\log\left({r(A_0)}/{n}\right) \right).
\eens}
Therefore, it is sufficient to set for $r(A_0) := \tr(A_0)/\twonorm{A_0}$, 
\ben 
\label{eq::rateB}
\ve \asymp \sqrt{{n}/{r(A_0)} \log\left({r(A_0)}/{n}\right) }
\een 
which
is clearly comparable to the error rate $r_{\offd}(n)$ as in~\eqref{eq::opnorm}, where we
replace the effective rank $r(A_0)$ with the effective sample size parameter ${\sum_{j=1}^m a_{jj} p_j^2
}/{\twonorm{A_0}}$.

\noindent{\bf Discussions.}
As is well-understood in the literature, results such as those in
Theorems~\ref{thm::main-intro},~\ref{thm::RE},~\ref{thm::main-coro},
~\ref{thm::gramsparse}, and~\ref{thm::XX^T} can not be derived using the entrywise
deviation directly.
In the present setting, using entrywise bounds to control quadratic
forms~\eqref{eq::quadorig} or~\eqref{eq::quadpreview} will result in
a stronger dependency between the effective sample size and the sparsity parameter $s_0$ through
$\psi_B(s_0)$.
Elementwise bounds such as $\shnorm{\tilde{B}_0 -
  B_0}_{\max}$ (and $\shnorm{\tilde{A}_0 -A_0}_{\max}$)
were proved and thoroughly discussed in~\cite{Zhou19}, cf. Theorems 5.1 to 5.4 therein and
Lemma~\ref{lemma::pairwise}.
These bounds are sufficient to obtain convergence for the graphical 
Lasso type of estimators, cf.~\cite{Zhou14a}, but not sufficient to 
obtain the tight error bounds in the operator norm as we elaborate in
Theorem~\ref{coro::thetaDet} for inverse covariance estimation.
Similar to the discussions in  Section~\ref{sec::RE},
we do incur an {\it extra factor} $\psi_B(n) \sqrt{\log (m \vee n) }=
\tilde{O}(\sqrt{n})$ in the second term~\eqref{eq::opnorm}, namely,
$r^2_{\offd}(n) \psi_B(n) \log^{1/2} (n \vee m)$, which dictates the sample complexity.
\silent{\bens
r^2_{\offd}(n) \psi_B(n) \log^{1/2} (n \vee m)
\asymp \frac{ n \twonorm{A_0} \psi_B(n)}{\sum_{j=1}^m a_{jj} p_j^2 } \log^{1/2} (m \vee n)
 \eens}
Lemma~\ref{eq::avgrowsum} provides a refined lower bound on 
$\twonorm{\abs{B_0}}$ that depends on the average row sum of 
$B_0$.  Moreover, for non-negative symmetric matrix $\abs{B_0} \ge 0$, 
its operator norm is bounded by the maximum row sum of $B_0$. 
\begin{lemma}
  \label{eq::avgrowsum}
  Let $\psi_B(n)$ be as in Definition~\ref{def::C0}.
For symmetric matrices $B_0$ and $\abs{B_0}$, 
their spectral radii (and hence operator norms) must obey the following relations: 
  \bens
   \lambda_{\max}(\abs{B_0}) = \rho(\abs{B_0}) =  \twonorm{\abs{B_0}}
& \ge & \rho(B_0) = \twonorm{B_0} \\
 \text{ and }
\sqrt{n} \twonorm{B_0} \ge \norm{B_0}_{\infty} \ge \twonorm{\abs{B_0}} & \ge &  \onen \sum_{i, j} \abs{b_{ij}}.
 \eens
 Finally, if the rows of $\abs{B_0}$ have the same sum $r$, then $
 \rho(\abs{B_0})=\twonorm{\abs{B_0}} = r$.
 Moreover, $\psi_B(k) = 1, \forall k \in [n]$ in case $B_0 \succ 0$ is
 a diagonal matrix, or trivially, when $B_0 \ge 0$ is a matrix with
 all positive entries.
\end{lemma}
We elaborate on this extra term $\psi_B(n)$ in Section~\ref{sec::reduction}.
To ensure convergence,
 potentially we need the effective sample size to have superlinear
 dependency on $n$, rather than the linear dependency
 in~Theorem~\ref{thm::XX^T}, since by Lemma~\ref{eq::avgrowsum} and Definition~\ref{def::C0},
\ben
\label{eq::phiben}
1 \le \psi_B(n) =
{\twonorm{(\abs{b_{ij}})}}/{\twonorm{B_0}}  \le  \sqrt{n} \;
\text{ where } \; \twonorm{(\abs{b_{ij}})} \le \norm{(\abs{b_{ij}})}_{\infty}
=\norm{B_0}_{\infty} \le \sqrt{n} \twonorm{B_0};
\een 
However, the bound in \eqref{eq::phiben} can be rather crude as shown in our numerical 
examples. 

\section{Randomized quadratic forms}
\label{sec::strategy}
In this section, we provide an outline for the proof of
Theorems~\ref{thm::main-intro} and~\ref{thm::gramsparse}.
First, we present our reduction strategy to analyze the quadratic form as defined
in \eqref{eq::quadorig}.
Let $\M$ be as  defined in~\eqref{eq::entrymean}.
Let $\norm{\M}_{\diag}$ and
$\norm{\M}_{\offd} $ be as defined in~\eqref{eq::definmaxM}.
Consider $\tilde{B}_0$ as in~\eqref{eq::entrymean}.
Denote by $\Delta(B) := \X \X^T- \E \X \X^T$.
Now for all $q \in \Sp^n$,
\ben
  \nonumber
\lefteqn{
\abs{q^T  (\tilde{B}_0 - B_0) q} =\abs{q^T ( \X \X^T \oslash 
  \M -\E (\X \X^T \oslash \M) )q }} \\
\label{eq::MB1}
& = &
\abs{q^T  (\X \X^T \oslash \M -B_0) q} =
\abs{q^T   [(\X \X^T- B_0 \circ \M) \oslash \M] q} =:q_0.
\een
We emphasize that the only assumption we make in the preceding derivation is 
on the independence of mask $U$ from data matrix $X$. 
Recall for the random mask model~\eqref{eq::maskUV},
we are given $m$ dependent samples to estimate $B_0$,
namely, $(x^j \otimes x^j), j \in [m]$ with each one applied an independent
random mask $u^j \otimes u^j$.
Hence $\Delta(B) :=  \X \X^T- \M \circ B_0$,
since by definition
\bens 
\label{eq::separate}
&& \quad \X \X^T :=  (U \circ X) (U \circ X)^T =  \sum_{j=1}^m {(u^j \otimes
  u^j) \circ (x^j \otimes x^j)}\; \text{ and } \\
\nonumber 
&&\E \X \X^T = \sum_{j=1}^m {\E (u^j \otimes u^j) \circ \E 
    (x^j \otimes x^j)}  = \sum_{j=1}^m a_{jj} M_j \circ B_0 =: \M 
  \circ B_0
  \eens
for $M_j = \E u^j \otimes u^j$.
Next we break the quadratic form~\eqref{eq::MB1} into two parts and obtain
\ben
\nonumber
\lefteqn{q_0 
 \le 
 \abs{q^T   \diag[\Delta(B) \oslash \M] q}  + \abs{q^T   \offd[\Delta(B) \oslash \M] q}} \\
\label{eq::MB2}
& = &
\frac{\abs{q^T  \diag(\X  \X^T  -B_0 \circ \M)q} }{\norm{\M}_{\diag}}+
\frac{\abs{q^T  \offd(\X  \X^T  -B_0 \circ \M) q}}{\norm{\M}_{\offd}},
\een
where in \eqref{eq::MB2}, we use the fact that $\diag(\M) =\norm{\M}_{\diag} I_n$ and $\offd(\M) 
=\norm{\M}_{\offd}  (\one \one^T - I_n)$ where $\one = (1, \ldots,
1)^T \in \R^n$ by definition.
Hence, we have essentially reduced the original problem that involves quadratic form with 
mask $\M \in \R^{n \times n}$ embedded inside~\eqref{eq::MB1} to ones
not involving the masks as we pull the masks outside of the quadratic forms in~\eqref{eq::MB2}.
This decomposition in \eqref{eq::MB2} enables us to treat the diagonal
and the off-diagonal parts of the original quadratic
form~\eqref{eq::MB1} separately as we elaborate in Section~\ref{sec::mainsketch}.
In Sections~\ref{sec::reduction} and~\ref{sec::mainresult},
we provide roadmaps for proving the results in 
Theorems~\ref{thm::gramsparse}  and~\ref{thm::main} respectively.
We prove Theorem~\ref{thm::gramsparse}, Lemma~\ref{eq::avgrowsum},
and the related Corollary~\ref{coro::offdn} in Section~\ref{sec::appendcoroproof}.

\subsection{Proof sketch of Theorem~\ref{thm::main-intro}}
\label{sec::mainsketch}
For diagonal matrices, as mentioned, we can prove a 
uniform bound over all quadratic forms so long as we  prove 
coordinate-wise concentration of measure bounds as already stated in 
Theorem~\ref{thm::diagmain}.
Hence we now focus on developing proof strategies to obtain an upper 
bound on the off-diagonal component \eqref{eq::MB2} that are useful to 
prove Theorem~\ref{thm::main} (and Theorem~\ref{thm::gramsparse}).
Theorem~\ref{thm::main} is the key technical result for bounding the 
off-diagonal component of \eqref{eq::MB2} over  the class of $s_0$-sparse vectors on 
$\Sp^{p-1}$. For the rest of the section, set $1 \le s_0 \le n$. Denote by $E = \cup_{\abs{J} = s_0} E_J$.
For $s_0=1$, we control \eqref{eq::rmn} by providing the matrix max 
norm on  $\offd(\tilde{B}_0) - \offd(B_0)$ as given 
in Lemma~\ref{lemma::pairwise}. 
We then apply Lemma~\ref{lemma::quadreduction} to provide a uniform
bound for the quadratic form over all $q \in  (\sqrt{s_0} B_1^n \cap  B_2^n)$.
\begin{lemma}{\textnormal{\bf (Reduction to sparse vectors)}}
\label{lemma::quadreduction}
Let $\delta > 0$. Set $0< s_0 \le n$.
Let $\Delta$ be an $n \times n$ matrix such that
$\abs{q^T \Delta h} \le \delta,\;\forall q, h \in E \cap
\Sp^{n-1}$. Then for all $\nu \in (\sqrt{s_0} B_1^n \cap B_2^n)$, $\abs{\nu^T \Delta \nu} \le 4 \delta$.
\end{lemma}

\begin{theorem}{\textnormal{\bf (Control the quadratic form over sparse vectors)}}
  \label{thm::main}
Set $1/2 \ge \ve >0$. Suppose all conditions in
Theorem~\ref{thm::main-intro} hold. Set $0< s_0 \le n$.
Denote by $d=2s_0 \wedge n$.
Let $\ul{r_{\offd}}$ and $\norm{\M}_{\offd} = \sum_{j} a_{jj} p^2_j$ be as
in \eqref{eq::paritydual} and~\eqref{eq::definmaxM} respectively.
Let $C_{\sparse}, C, c$ be absolute constants.
Then on event $\F_4^c$, which holds with probability at least $1-C/{(n \vee m)^4} -4 \exp(-c s_0 \log 
(3en/(s_0\ve))$,
\ben
\label{eq::rmn}
  \sup_{q, h \in \Sp^{n-1} \cap E}
  \frac{\abs{q^T \offd(\X\X^T  -B_0 
      \circ \M) h}}{\twonorm{B_0}\norm{\M}_{\offd}}  \le
  C_{\sparse} \Big(\ul{r_{\offd}} \sqrt{s_0} +r_{\offd}^2 (s_0) \psi_B(d)\Big).
  \een
\end{theorem}
Theorem~\ref{thm::main-intro} follows from
Theorems~\ref{thm::diagmain} and~\ref{thm::main} in view of
Lemma~\ref{lemma::quadreduction}.  We prove
Theorem~\ref{thm::main-intro} in Section~\ref{sec::appendintromain}.
Section~\ref{sec::mainresult} further elaborates upon the proof strategies for 
Theorem~\ref{thm::main}, where we also define event $\F_4^c$, with complete proof in 
Section~\ref{sec::appendproofoffdmain}.
Lemma~\ref{lemma::quadreduction} follows from the proof of Lemma 
37~\cite{RZ17}, and hence its proof is omitted. 

\subsection{Proof Strategy for Theorems~\ref{thm::gramsparse} and~\ref{thm::main}}
\label{sec::reduction}
In this section, we will present the conditioning arguments and
some important results on a family of random matrices, which are crucial for
Theorem~\ref{thm::gramsparse} to hold.
Theorem~\ref{thm::main} involves more 
technical developments, while following a similar line of arguments, 
and hence deferred to  as Section~\ref{sec::mainresult}.
Let $\tilde{\Delta} := \offd(\X \X^T  -\M \circ B_0)$.
Denote the off-diagonal component in \eqref{eq::MB2} by
\bens 
Q_{\offd}  &  =  & 
\abs{q^T  \offd(\X  \X^T  -B_0 \circ 
  \M) q}/{\norm{\M}_{\offd}} =: \abs{q^T  \tilde{\Delta} q} /{\norm{\M}_{\offd}}.
\eens
To prove Theorem~\ref{thm::gramsparse},
we need to obtain a 
uniform bound for the quadratic form
$Q_{\offd} \norm{\M}_{\offd} = \abs{q^T  \tilde{\Delta} q}$ over all $q \in \Sp^{n-1}$ in view of 
\eqref{eq::MB2} and Theorem~\ref{thm::diagmain}.
Denote the unique symmetric square root of the positive definite matrix $A_0$ by 
\ben
\label{eq::dm}
A_{0}^{1/2} = [d_1, d_2, \ldots, d_m],  \text{ where}  \; d_1, d_2, \ldots, d_m \in \R^{m}
\een
are the column vectors of $A_0^{1/2}$, and $\ip{d_i, d_j} = a_{ij}$,
for all $i, j$.
Denote the unique symmetric square root of the positive definite matrix $B_0$ by 
\ben
\label{eq::cf}
B_{0}^{1/2} = [c_1, c_2, \ldots, c_n],  \text{ where}  \; c_1, c_2,
\ldots, c_n \in \R^{n} \text{  and} \; \ip{c_i, c_j} = b_{ij} \;
\forall i, j.
\een
For each  $q\in \Sp^{n-1}$, let $A_{qq}^{\diamond} \in \R^{mn \times mn}$
be a random matrix that can be expressed as a quadratic form over the
set of independent Bernoulli random variables $\{u^k_j, k=1, \ldots, m, j=1, \ldots, n\}$
in the mask $U$
\ben
\label{eq::tensor}
A_{qq}^{\diamond} &=&
\sum_k \sum_{i\not= j} u^k_i u^k_j 
q_i q_{j} (c_j c_i^T) \otimes  (d_k d_k^T),
\een
where the coefficient for a pair of random variables $u^k_i u^k_j$ is a 
tensor product $q_i q_j ( c_j c_i^T) \otimes (d_k d_k^T)$ that changes 
with each choice of $q \in \Sp^{n-1}$.
Now $Q_{\offd}$ has the following expression using the random matrix
$A_{qq}^{\diamond}$ and a subgaussian vector  $Z \sim \mvec{\Z^T}$:
\bens
Q_{\offd} 
&  \sim & 
\inv{\norm{\M}_{\offd}}
\abs{Z^T A^{\diamond}_{qq} Z  - \E (Z^T 
  A^{\diamond}_{qq} Z)}, \; \text{ where  $\Z$ is as defined in \eqref{eq::missingdata} }
\eens
with $K=1$ and $\sim$ represents that two vectors follow the same distribution.\\
\noindent{\bf The conditioning arguments.}
We now illustrate the general strategy we use in this work to
deal with randomness  in $X$ and $U$ simultaneously.
First, we decompose the error into two parts,
\ben 
\label{eq::decompmain}
\lefteqn{\quad \quad \quad \forall q \in \Sp^{n-1}, \quad \abs{q^T
    \tilde{\Delta} q} :=
  \abs{Z^T  A^{\diamond}_{q q} Z  - \E (Z^T     A^{\diamond}_{q q} Z )} \le  }\\
\nonumber 
& &
\abs{  Z^T  A^{\diamond}_{qq} Z- \E (Z^T  A^{\diamond}_{qq } Z |U)}  
+ \abs{\E (Z^T  A^{\diamond}_{q q} Z |U) - \E (Z^T  A^{\diamond}_{q q} Z)} =: 
{\bf I} + {\bf II}.
\een
Construct an $\ve$-net $\N \subset \Sp^{n-1}$ such that $\size{\N} \le
(3/\ve)^n$. We then present a uniform concentration bound on Part I followed by 
that of Part II for all $q  \in \N$.
Let $U$ be as in~\eqref{eq::maskUV} and condition on $U$ being
fixed. Then the quadratic form $Q_{\offd}$ can be treated as a
subgaussian quadratic form with  $A^{\diamond}_{qq}$ taken to be
deterministic for each $q$. Theorem~\ref{thm::mainop2} states
that the operator norm of $A_{qq}^{\diamond}$ is uniformly and  
deterministically bounded for all realizations  of $U$ and for all $q
\in \Sp^{n-1}$.
Theorem~\ref{thm::uninorm2} states a probabilistic 
uniform bound on the Frobenius norm of $A_{qq}^{\diamond}$
for all $s_0$-sparse vectors, where $s_0 \in [n]$.
\begin{theorem}
\label{thm::mainop2}
Let $A^{\diamond}_{qq}$ be as defined in \eqref{eq::tensor}.
Then for all  $q\in  \Sp^{n-1}$,
\bens
\twonorm{A_{qq}^{\diamond}} \le \twonorm{A_0}\twonorm{B_0}.
\eens
\end{theorem}
\begin{theorem}
\label{thm::uninorm2}
Suppose ${\sum_{j=1}^m a_{jj} p_j^2}/{\twonorm{A_0}}
=\Omega(\psi^2_B(s_0) \log (n \vee m))$, where $s_0 \in [n]$.
Let $A^{\diamond}_{qq}$ be as defined in \eqref{eq::tensor}.
Then on event $\F_0^c$, which holds with probability at least $1-{c}/{(n \vee 
  m)^4}$ for some absolute constant $c$, we have for all integer $s_0 \in [n]$,
\bens
&& \sup_{q \in \Sp^{n-1}, s_0-\sparse}
\fnorm{A^{\diamond}_{qq}} \le W \cdot \twonorm{B_0}
\twonorm{A_0}^{1/2} \;  \text{ where }\;\\
&& \label{eq::Wlocal}
W \asymp  (a_{\infty} \sum_{s=1}^m p_s^2)^{1/2}
+\psi_B(s_0) \big(a_{\infty} \twonorm{A_0}  \sum_{j=1} p_j^4 \log (n \vee m)\big)^{1/4}.
\eens
\end{theorem}
By the union bound, the Hanson-Wright inequality in
Theorem~\ref{thm::HW}~\cite{RV13},
and the preceding estimates on the operator and Frobenius norms in Theorems~\ref{thm::mainop2}
and~\ref{thm::uninorm2}, we have for $\tau_0 \asymp n \log (3e/\ve) \twonorm{A_0}  \twonorm{B_0}+ \sqrt{n \log (3e/\ve)}  \twonorm{B_0} \twonorm{A_0}^{1/2}  W$,
\ben
\nonumber
\lefteqn{\prob{\exists q \in \N, \abs{Z^T A^{\diamond}_{qq} Z - \E(Z^T 
      A^{\diamond}_{qq} Z | U) } > \tau_0} =: \prob{\F_1}} \\
\nonumber
& = & \E_U  \prob{\exists q \in \N,  \abs{Z^T A^{\diamond}_{qq} Z -
    \E(Z^T A^{\diamond}_{qq} Z | U) } > \tau_0 | U}\\
\label{eq::defineF1}
& \leq &  2 \size{\N}^2 \exp
\big(-c\min\big(\frac{\tau_0^2}{{\twonorm{B_0}^2}  \twonorm{A_0} W^2},
\frac{\tau_0}{\twonorm{B_0}\twonorm{A_0}}\big) \big)\prob{\F_0^c}
+   \prob{\F_0}  
\een
In bounding {\bf Part II}, we need the sparse Hanson-Wright inequality
as stated in Corollary~\ref{coro::offdn}, which follows from 
Theorem~\ref{thm::Bernmgf}. These are the important technical results
in this paper.
\begin{corollary}
  \label{coro::offdn}
Fix $q \in \Sp^{n-1}$. Let $\tilde{a}^{k}_{ij}(q) = a_{kk} b_{ij} q_i q_j$.
For $U$ as in~\eqref{eq::maskUV}, denote by 
\ben 
\label{eq::starship}
\quad S_{\star}(q)  = 
\E (Z^T  A^{\diamond}_{q q} Z  | U)- \E (Z^T  A^{\diamond}_{q q} Z)  = \sum_{k=1}^m \sum_{i \not=j}^n \tilde{a}^{k}_{ij}(q)(u^k_i u^k_j - p_k^2).
\een 
For any $t > 0$ and for each $q \in \Sp^{n-1}$ and $S_{\star}(q)$ as in~\eqref{eq::starship},
\bens 
  \prob{\abs{S_{\star}(q)} > t}  \le
  \exp\big(- c\min\big(\frac{t^2}{ a_{\infty} \twonorm{B_{0}}
      \twonorm{(\abs{b_{ij}})} \sum_{k=1}^m  a_{kk}p_k^2}, \frac{t}{a_{\infty} \twonorm{B_{0}}}  \big) \big).
    \eens
    \end{corollary}
Set $\tau'  \asymp a_{\infty} \twonorm{B_0}  \psi_B(n) n \log (3e/\ve)$,
where  $\psi_B(n) = O(\sqrt{n})$(cf. \eqref{eq::phiben}).
Then by Corollary~\ref{coro::offdn},~\eqref{eq::baseline} for $d =n$, and the union bound, 
\ben
\label{eq::defineF2}
\prob{\F_2}  := \prob{\exists q \in \N, \abs{S_{\star}(q)} > \tau'} 
\le \exp(-c'n).
\een
Combining the preceding bounds with a standard approximation argument,
we have on event $\F_0^c  \cap \F_1^c \cap
\F_2^c$, by \eqref{eq::decompmain},  \eqref{eq::defineF1},
\eqref{eq::starship},  and \eqref{eq::defineF2}, 
\ben
\label{eq::netb}
\sup_{q\in \Sp^{n-1}}    \abs{q^T\tilde{\Delta} q} \le  \sup_{q \in \N}
  \abs{q^T\tilde{\Delta} q}/ (1- \ve)^2
 \le C {(\tau_0 +\tau')}.
\een
By~\eqref{eq::MB2}, \eqref{eq::netb} and Theorem~\ref{thm::diagmain},
we have on event $\F_0^c  \cap \F_1^c \cap \F_2^c \cap \F_{\diag}^c$, for $p_{\max} := \max_{j} p_j$,
\ben
\nonumber
\lefteqn{\twonorm{\tilde{B}_0
    -B_0}/{\twonorm{B_0}} \le C \eta_A r_{\offd}(n) +
  C' \big(1 + {a_{\infty} \psi_B(n)}/{\twonorm{A_0}} \big) r^2_{\offd}(n)}\\
\label{eq::Bhatop}
& &\; + C_3 \eta^{1/2}_A \sqrt{p_{\max} } r^{3/2}_{\offd}(n) \psi_B(n)
\big({\log (n \vee  m) }/{n}\big)^{1/4} =: \delta_q(B);
\een
The final expression in the theorem statement for the overall rate of
convergence follows the same line of arguments from
Lemma~\ref{lemma::finalrate} with a slight variation. We leave such details in the supplementary 
Section~\ref{sec::appendcoroproof}.

\section{Inverse covariance estimation bound}
\label{sec::inverse}
In this section, we state in Theorem~\ref{coro::thetaDet} a
deterministic result about $\hat\Theta$ as defined in
Definition~\ref{def::TopHat} under Assumption (A1) as stated in Section~\ref{sec::prelim}.

\begin{theorem}{\textnormal{\bf (Deterministic error bound)}}
  \label{coro::thetaDet}
  Suppose $B_{\star} = B_0$ as in Definition~\ref{def::RMSet}.
Let  $\Theta_0 = B_0^{-1} = (\theta_{ij})$.
Suppose (A1) holds. Suppose covariance $A_0 \succ 0$ and $\twonorm{A_0}
< \infty$.  Let $\lambda, b_1$ be as in~\eqref{eq::origin} and
$\beta^{j*}, \forall j$ be as in~\eqref{eq::regress}. 
Let $\ul{r_{\offd}}$ be as defined in  \eqref{eq::paritydual}.
Let $d_0$ denote the maximum row sparsity in $\Theta_0$, satisfying
$\ul{r_{\offd}} \sqrt{d_0} =o(1)$, and suppose
 \ben
\label{eq::BHatoffd}
\shnorm{\hat{B}^{\star} - B_{0}}_{\max} \le C_{\max} b_{\infty} \ul{r_{\offd}}.
 \een
Suppose for some absolute constant $c_{\gamma}$ and
$\hat\Gamma^{(j)}$ and $\hat\gamma^{(j)}$ as in~\eqref{eq::GammaMain} and \eqref{eq::gammaMain},
\ben
\label{eq::halflambda}
\forall j \in [n],  && \norm{ \hat\gamma^{(j)}  -  \hat\Gamma^{(j)} \beta^{j*}}_{\infty}
  \le    c_{\gamma}   b_{\infty} \kappa_\rho \ul{r_{\offd}} \le
  {\lambda}/{4}, \\
  \label{eq::lambdamain}
  \text{ and set }\; 
  \lambda  &\asymp & 4 C_{\gamma}   b_{\infty} \tilde\kappa_\rho   \ul{r_{\offd}}
  \;  \text{ where} \; C_{\gamma} \ge c_{\gamma} \vee  C_{\max}.
  \een
Suppose the Lower-$\RE$ condition holds uniformly over the
matrices $\hat\Gamma^{(j)}, j \in [n]$ with $(\alpha,
\tau)$ such that  
\ben
\label{eq::lassopen}
&&\text{curvature} \; \; \alpha \asymp  \lambda_{\min}(B_0), \quad {\alpha}/{\tau} \ge
120 d_0  \;\; \text{ and }\;  \lambda /\tau \ge 4 b_1,
\een
where $\tau$ is the tolerance parameter as in
Definition~\ref{def::lowRE} and $b_1 \ge \onenorm{\beta^{j*}}  \;
\forall j \in [n]$.
Following Definition~\ref{def::TopHat}, for some absolute constant
$c_{\overall}$,
\ben
\label{eq::inverseop}
\norm{\hat\Theta - \Theta_0}_1
\le c_{\overall} \twonorm{\Theta_0}
b_{\infty} \tilde\kappa_{\rho} \theta_{\max} 
\kappa_B \ul{r_{\offd}} d_0, \; \text{ where } \;  \theta_{\max} :=
\max_{j} \theta_{jj}.
\een
\end{theorem}
\noindent{\bf The choice of $\ell_1$ radius.}
Imposing an upper bound  on the 
operator norm of $\Omega$ and $\rho(B)$ is necessary to ensure $B_0 
\succ 0$ and $\twonorm{B_0} < \infty$. 
We choose an upper bound on $\onenorm{\beta^{j*}}$ that depends on the
matrix $\ell_1$ norm of inverse correlation $\Omega = W \Theta_0 W=
(\omega_{ij}) \succ 0$:  under (A1),
\ben
\label{eq::b1omega}
\quad \quad \quad \quad
\onenorm{\beta^{j*}} := {\norm{\Theta_j}_1}/{\theta_{jj}}
\le \norm{\Omega}_1 {\sqrt{b_{\infty}}}/{\sqrt{b_{\min}}}, \; \text{
  where} \; \norm{\Omega}_1 \le \sqrt{d_0} \twonorm{\Omega}
\een
where $b_{\min} = \min_{j} b_{jj}$ and for all $j \in [n]$, $\omega_{jj} = b_{jj} \theta_{jj}
> 1$ by  Proposition~\ref{prop::projection}.
Hence on $\F_{\diag}^c$, the data-dependent choice of $b_1$
as in \eqref{eq::b1radius} provides an upper bound for 
$\onenorm{\beta^{j*}}$, for all $j \in [n]$ in view of 
\eqref{eq::b1omega} and Theorem~\ref{thm::diagmain} (see also
Theorem~\ref{thm::main-coro}); Set for $\bar{d_0} \ge d_0$,
\ben
\label{eq::b1radius}
&& b_1 :=
M_{\Omega}\sqrt{2\bar{d_0}}\big( {\max_{j}
  \hat{B}_{jj}^{\star}}/{\min_{j} \hat{B}_{jj}^{\star}}\big)^{1/2} \ge \norm{\Omega}_1{\sqrt{b_{\infty}}}/{\sqrt{b_{\min}}}.
\een
\noindent{\bf Tuning parameters.}
Set $\delta \le 1/32$.
We set $C_{\lambda} \asymp  4 (1+2\delta) C_{\ga}$
for some absolute constant $C_{\ga} \ge (c_{\gamma} \vee
C_{\max} \vee C_{\RE})$ for $c_{\ga}, C_{\max}$ are as defined in
Theorem~\ref{coro::thetaDet} and $C_{\RE}$ as defined in Theorem~\ref{thm::RE}.
Imposing an upper bound on $\kappa_{\rho} \le 
\tilde{\kappa}_{\rho}$ allows us to set $\lambda$ as in
\eqref{eq::lambdamain}.
Lemma~\ref{lemma::Gammabounds} ensures that~\eqref{eq::halflambda}
holds with high probability in view of \eqref{eq::lambdamain}.

\noindent{\bf $\RE$-conditions.}
A sufficient condition for $\hat\Gamma^{(j)}, j \in [n]$ to satisfy 
the Lower-$\RE$ condition with $(\alpha, \tau)$
is to have $\hat\Gamma$ satisfy the same condition~\cite{LW12}.
Intuitively, \eqref{eq::lassopen} states that the $\RE$ condition needs to be 
strong with respect to the row sparsity $d_0$.
Hence to ensure $\alpha/\tau> 120 d_0$,  we impose an upper bound
$\bar{d}_0$ on ${d}_0$. Loosely speaking, we assume $\bar{d}_0
={O}(s_0/\psi_B(2s_0\wedge n))$  under~\eqref{eq::forte}.
These conditions  not only guarantee $\ul{r_{\offd}}\sqrt{d_0} \to 0$
but also ensure $\lambda \ge 4 b_1 \tau$ as required 
by~\eqref{eq::lassopen} for $b_1$ as in~\eqref{eq::b1radius}.
See Theorem~\ref{thm::main-coro}.
\begin{lemma}
  \label{lemma::Gammabounds}
   Let $\abs{B_0} = (\abs{b_{ij}})$. Suppose (A1) holds and $n$ is
   sufficiently large.
   Let $c_9, c_{\gamma}$, and $C_{14} >144$ be some absolute constants.
   Suppose
  \ben
  \label{eq::forte3}
  {\sum_{j=1}^m a_{jj} p_j^2}/{\twonorm{A_0}}
  \ge C_{14} \eta_{A}^2 \log (m \vee n)
\big(d_0  \vee \big({\rho_{\max}(d_0, \abs{B_0})}/{b_{\min}}\big)^2 \big).
\een
Then with probability at least $1-{c_9}/{ (n \vee m)^4}$,
\ben
\label{eq::doublesword}
\forall j \in [n], \quad  \norm{ \hat\gamma^{(j)}  -  \hat\Gamma^{(j)} \beta{^{j*}}}_{\infty} \le
 c_{\gamma} \ul{r_{\offd}} \eta_{A}   b_{\infty} \kappa_\rho.
 \een
\end{lemma}
We prove Theorem~\ref{coro::thetaDet} and Lemma~\ref{lemma::Gammabounds} in 
the supplementary  Sections~\ref{sec::proofofinverses} and~\ref{sec::proofofGamma}. 
respectively. 
In summary, under conditions as stated in~\eqref{eq::paritydual} and Theorem~\ref{thm::RE} (also
Theorem~\ref{thm::main-coro}), following Definition~\ref{def::TopHat}, we have for $\eta_A  \asymp
1$, $\lambda$ as in \eqref{eq::lambdamain},
\ben
\label{eq::ICfinal}
\shnorm{\hat\Theta - \Theta_0}_2/\twonorm{\Theta_0}
= 
O_P\big(d_0  \theta_{\max} \tilde{\kappa}_\rho
\kappa_B\ul{r_{\offd}}\big) \;\;\text{where} \; 
 \ul{r_{\offd}} \asymp  \eta_A \Big(\frac{\twonorm{A_0} \log (m \vee 
n)}{\sum_{j=1}^m   a_{jj} \z^2_j }\Big)^{1/2}
\een
where $b_1$ is chosen as in \eqref{eq::b1radius}.
The proof is omitted as it is a direct consequence of Theorem~\ref{coro::thetaDet} and
the large deviation bounds we have derived in
Theorem~\ref{thm::main-coro} and
Lemma~\ref{lemma::Gammabounds}, which
ensure that all conditions imposed on $\hat{B}^{\star}$ hold with high
probability.

\noindent{\bf Discussions.}
Closely related to our work is that of~\cite{LW12}.
The rate of convergence in the operator norm in \eqref{eq::inverseop} is
directly comparable when we set $p_j = p, \forall j$.
In this case, our $\Theta_0 = B_0^{-1}$ is the same as their inverse covariance,
except that in their setting, samples are assumed to be independent, that is, $A_0 =I_m$ in~\eqref{eq::missingdata};
Proof of Theorem~\ref{coro::thetaDet} follows steps from Corollary 
5~\cite{LW12}, except that we now set the entrywise convergence for $\hat{B}^{\star} - B_0$ and
large deviation bounds on \eqref{eq::doublesword} at the order of
$O(\ul{r_{\offd}})$ in view of Theorem~\ref{thm::main-coro} and
Lemma~\ref{lemma::Gammabounds}.
In~\cite{LW12}, combining Corollaries 2, 5, eq. (3.1) and (3.2)
therein, and Lemmas 3 and 4~\cite{LW12supp},
the rate of convergence in the i.i.d. setting is
\ben
\label{eq::LWrate}
&& \shnorm{\hat\Theta - \Theta_0}_2/\twonorm{\Theta_0}
 =
{O}_P\Big(\kappa^2_B \twonorm{\Theta_0} d_0  r_m/p^2
\Big), \; \text{ where} \; r_m =  \sqrt{{\log  (m \vee n)}/{m}}\\
\label{eq::gammarate}
&& \text{ and } \; \forall j \in [n], \; \; 
\norm{ \hat\gamma^{(j)}  -  \hat\Gamma^{(j)} \beta^{j*}}_{\infty} =
\tilde{O}_P(r_m /{p^2}\big)
\text{ in case} \; p_j  = p;
\een
Here the $\tilde{O}(\cdot)_P$ notation hides the spectral parameters,
which are assumed to be all bounded.
The key difference between the two error bounds on $\shnorm{\hat\Theta - \Theta_0}$ are:
the relative error in~\eqref{eq::ICfinal} is bounded by
$O(\ul{r_{\offd}})$ in the present work, which is
inversely proportional to $p$, while in~\eqref{eq::LWrate}, the same error is inversely proportional to
$p^2$.
Hence our analysis results in a faster rate of convergence for
the relative error in the operator norm for
inverse covariance estimation, especially when $p$ is small.
The sub-optimal scaling in \eqref{eq::LWrate}  and \eqref{eq::gammarate} with respect to $p$ also 
results in the more restrictive scaling of
\ben
\label{eq::quadmp}
m =\Omega(\kappa_B^2 s_0 \log(m \vee n)/p^4) \; \; \text{ in case } \; p_j 
= p \; \; \forall j \in [m]
\een
for the aforementioned results, where $s_0 \asymp d_0$, even for the i.i.d. setting; cf. Corollaries 2,
5~\cite{LW12}.
\silent{Moreover, under the conditions~\eqref{eq::quadmp}, 
$\RE$ conditions are established with 
\bens
\alpha \asymp \lambda_{\min}(B_0),\;\;
\tilde\alpha \asymp \lambda_{\max}(B_0),\;\;
\text{ and tolerance } \tau =O(\kappa_B \twonorm{B_0}{\log(m \vee n)}/{(m p^2)})
\eens
respectively, in Lemma 3~\cite{LW12supp},
where $\alpha, \tilde\alpha, \tau$ are as in Definition~\ref{def::lowRE}; cf.
Corollaries 2, 5~\cite{LW12} and remarks that immediately follow.}
In the matrix variate setting, we require
\ben
\label{eq::fortemdp}
r(A_0) =\Omega(\kappa_B s_0 \big(\psi_B(2s_0 \wedge n) \vee \kappa_B \big) {\log(m \vee 
  n)}/{p^2} )\quad \; \; \text{ in case } \; p_j 
= p \; \; \forall j \in [m]
\een
in view of~\eqref{eq::forte} and \eqref{eq::forte3},
where we assume $d_0 =O(s_0/\psi_B(2s_0 \wedge n))$, which is possibly
rather conservative. Notice that for inverse covariance, we typically
deem the maximum node degree $d_0$ to be bounded or have slow growth
with respect to $n$ to benefit from the rate of convergence as
in~\eqref{eq::inverseop}.
Subsequently,
convergence in the $\ell_2^2$ error at the order of 
$\tilde{O}_P(\kappa^2_B \ul{r^2_{\offd}} d_0)$ was derived
for the corrected Lasso~\eqref{eq::origin} to recover a $d_0$-sparse 
vector in $B_2^n$ using sparse Gaussian random design with independent
entries~\cite{LW12b};  cf. Theorem 3. For this special case,
the sample lower bound was set essentially at the same order 
as~\eqref{eq::fortemdp}, where $\psi_B(k) = 1, \forall k\in
[n]$, and thus improves upon Corollary 2 in~\cite{LW12}.
Here the $\tilde{O}(\cdot)$ notation ignores differences in logarithmic factors.
To summarize, although these results are comparable when rows or
 columns, or all entries are independent, we deal with the more general
 subgaussian matrix-variate setting. Moreover,
the sample dependency on $p$ as set in the present paper is significantly more relaxed 
than~\cite{LW12} in view of \eqref{eq::quadmp} and~\eqref{eq::fortemdp}, despite the potentially superlinear dependency on $s_0$.
In essense, our primary focus is in the regime where sampling rates $p_j$s are
heterogeneous and small, so as to invoke the sparse Hanson-Wright
inequalities and the tight control we obtain on the operator and
Frobenius norm over the family of random matrices \eqref{eq::tensor}.

\silent{Moreover, under~\eqref{eq::forte},
$\hat\Gamma$ and $\hat\Gamma^{(j)}, \forall j \in [n]$
satisfy the same Lower-$\RE$ condition for \text{\it curvature }  $\alpha = \frac{5}{8}\lambda_{\min}(B_0)$ and
\text{ \it tolerance}
\bens
\tau := {3 \lambda_{\min}(B_0)}/{(8 s_0)}  \asymp
{\twonorm{B_0} (\kappa_B \vee \psi_B(2s_0 \wedge n))
  \log(m \vee n)}/{(r(A_0)  p^2)},
\eens
where $\psi_B(k) = \rho_{\max}(k, \abs{B_0})/\twonorm{B_0}, k \in [n]$ is as in Definition~\ref{def::C0}.
Hence $\tau$ depends on the condition number of $B_0$ and is inversely
proportional to $p^2$ in both papers; Moreover, in the present work,
it also depends on the sparse eigenvalue of $\abs{B_0}$, namely,
$\rho_{\max}(s_0, \abs{B_0})$, which may grow with $s_0$ at the
maximum rate of $O(\sqrt{s_0})$ in the worse case scenario.}

\section{Numerical results}
\label{sec::examples}
\begin{figure}[!tb]
    \centering
  \begin{subfigure}{.45\textwidth}
    \includegraphics[width=2.5in]{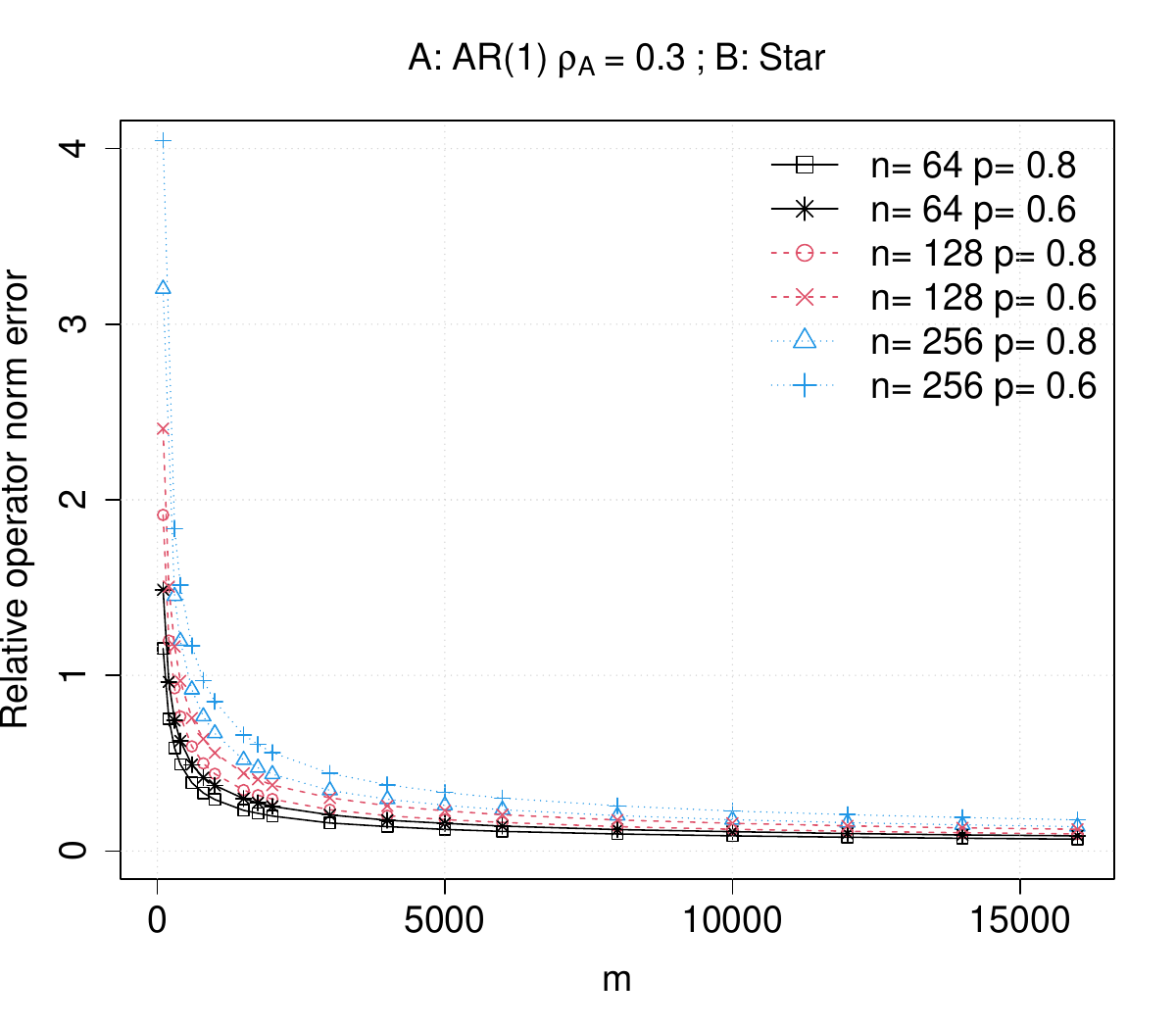}
  \end{subfigure}%
\hspace{5mm}
  \begin{subfigure}{.45\textwidth}
    \includegraphics[width=2.5in]{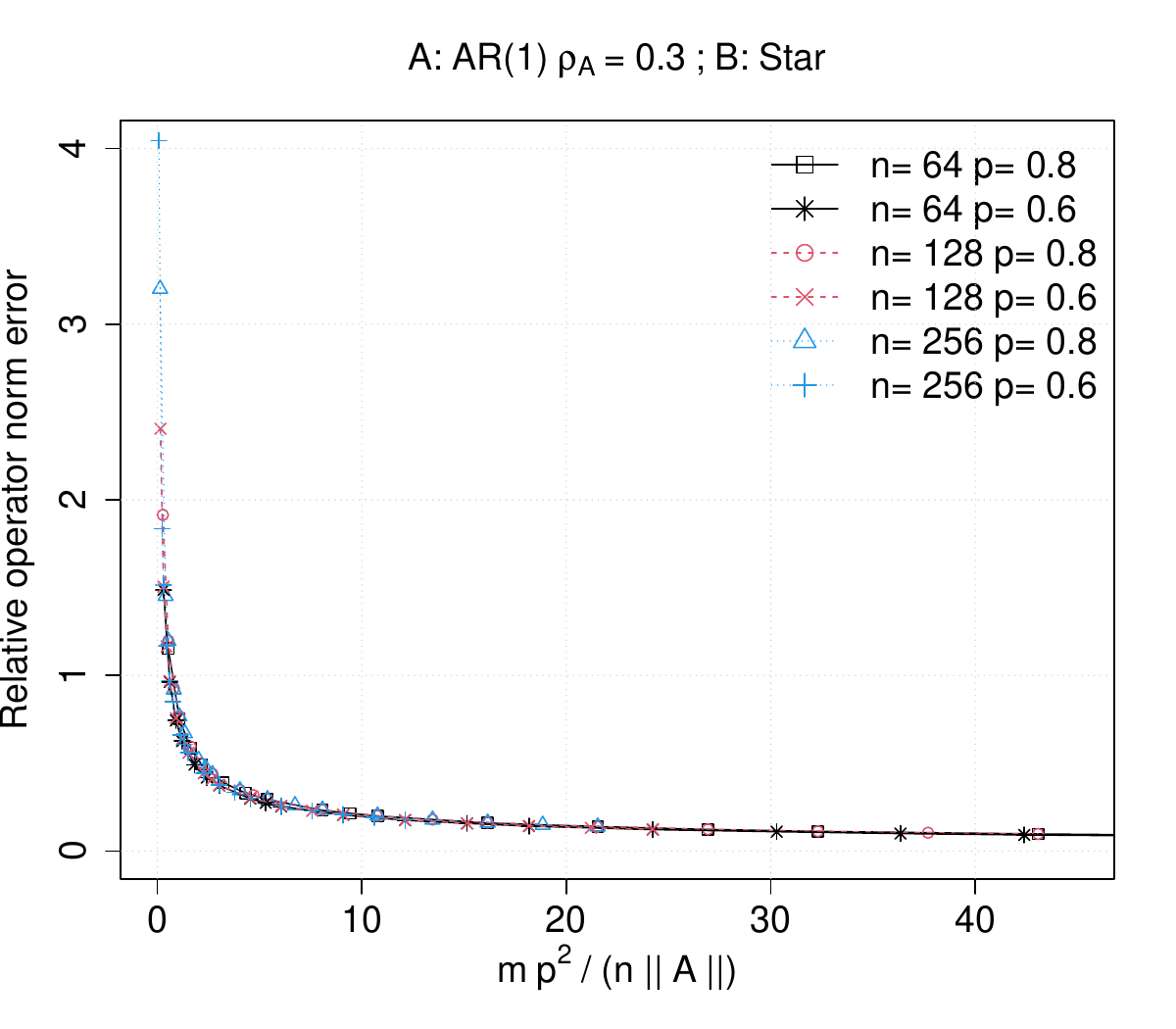}
  \end{subfigure}
  \bigskip
\begin{subfigure}{.45\textwidth}
    \includegraphics[width=2.5in]{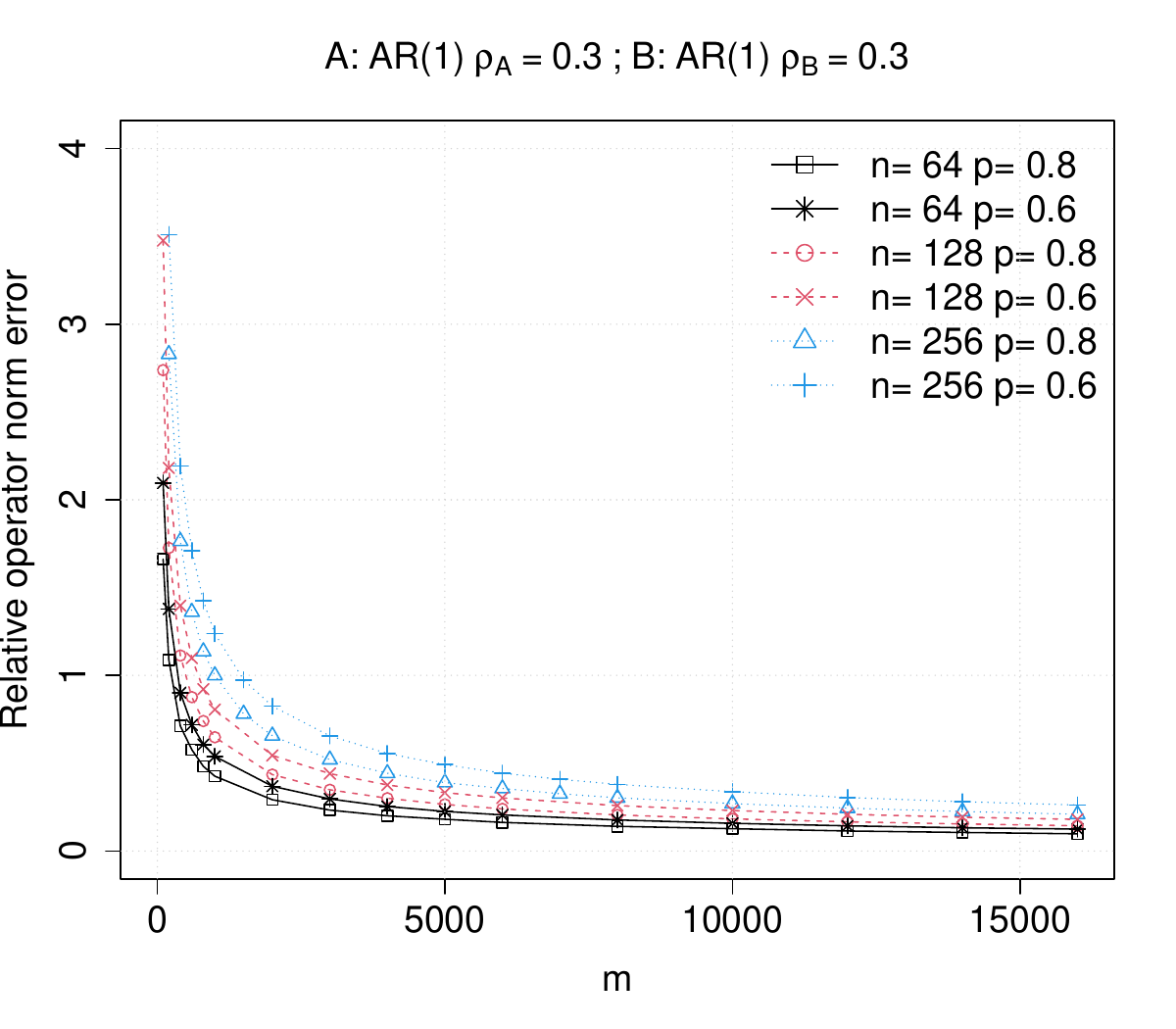}
 \end{subfigure}
\hspace{5mm}
\begin{subfigure}{.45\textwidth}
  \includegraphics[width=2.5in]{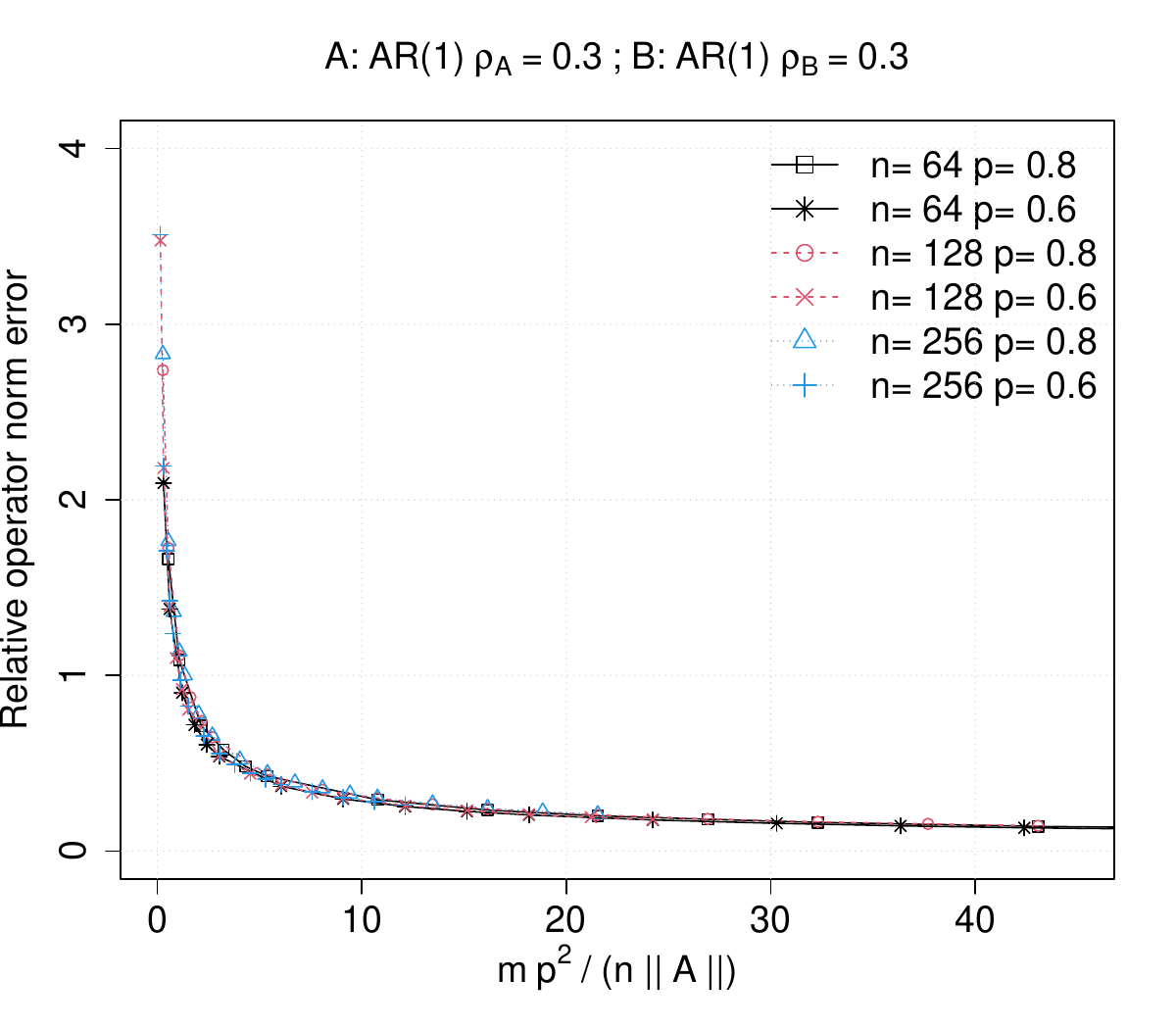}
 \end{subfigure}
\centering
  \caption{ 
    We set $\rho_A =0.3$ for two models of $B$. On the left panel,
    the average relative errors in the operator norm for estimating
    $B$  converge to $0$ as $m$ increases.
    For the same $m$, the
    errors increase with $n$ and decrease with sampling probability
    $p$ for both models of $B$.
    On the right panel, upon rescaling, all curves corresponding to different values 
  of $n=64, 128, 256$ align well for both values of $p \in \{0.6, 0.8\}$ 
  for $x \ge x_{\mbox{\scriptsize{thresh}}} \asymp  1$ and $\tilde\delta_{\mbox{\scriptsize{overall}}} \asymp r_{\mbox{\scriptsize{offd}}}(n)$.}
\label{fig::global}
\end{figure}

We use simulated data to validate the bounds as stated in Theorem~\ref{thm::gramsparse}.
We consider the following models:
(a) $\operatorname{AR}(1)$ model, where the covariance matrix is of the form
$A =\{\rho^{|i-j|}\}_{i,j}$. The graph corresponding to $A^{-1}$ is a chain;
(b) Star-Block model, where the covariance $B$ is block-diagonal with equal-sized
blocks whose inverses correspond to star structured graphs with 
$\diag(B) = I_n$.
Within each subgraph in $B^{-1}$, the only edges are
from leaf nodes to a central hub node.
The covariance for each block $S$ in $B$ is generated as in \cite{RWRY08}:
$S_{ ij} = \rho_B$ if $(i,j) \in E$ and $S_{ij} = \rho_B^2$
otherwise.
Throughout our experiments, we set $A$ to be a correlation matrix
following $\operatorname{AR}(1)$ with parameter $\rho_A = 0.3, 0.8$
and hence $\eta_A = 1$.
We choose $B$ to follow either $\operatorname{AR}(1)$ or a Star model,
that is, Star-Block model with a single block with $n-1$ edges and
$\rho_B = 1/{\sqrt{n}}$.
The dominating term in \eqref{eq::Bhatop} suggests that we rescale the effective sample size
${\sum_{j}a_{jj}   p_j^2}/{\twonorm{A} }$ by $n$ and set $x =
\sum_{j=1}^m a_{jj} p_j^2 /(\twonorm{A}  n)$. 
Hence $r_{\offd}(n) = 1/\sqrt{x}$. In our experiments, we set $p_1 =
\ldots =p_m =p$, and hence $x = m p^2/( n \twonorm{A})$.
Also, we have $m < n^4$, then $\log m  < 4 \log n$, where $n =
64, 128, 256$. As such, we will state  $\log (n)$ rather than $\log (m \vee n)$.

\begin{figure}[!tb]
  \centering
\begin{subfigure}{.45\textwidth}
  \includegraphics[width=2.5in]{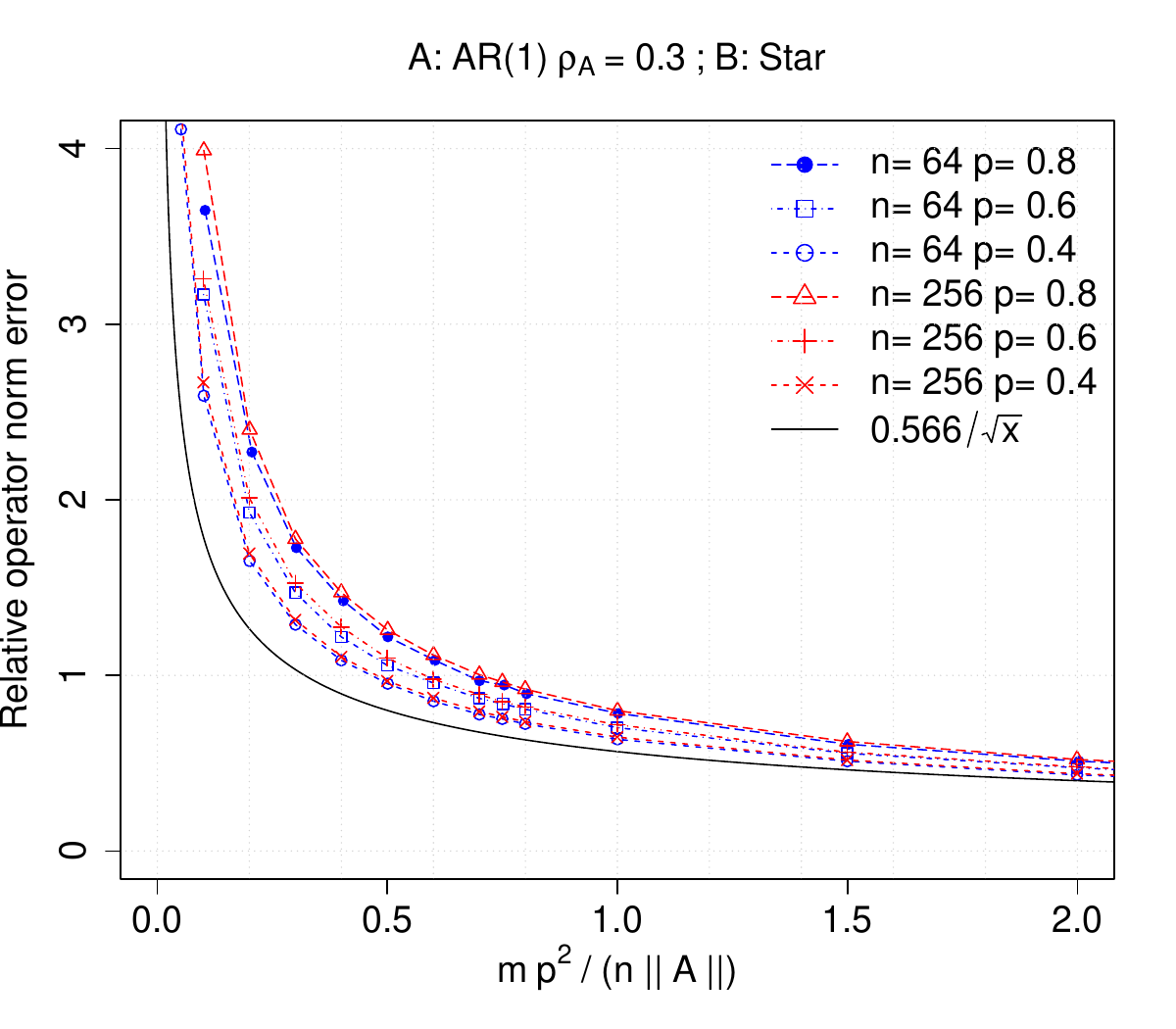}
  \end{subfigure}%
\hspace{5mm}
\begin{subfigure}{.45\textwidth}
    \includegraphics[width=2.5in]{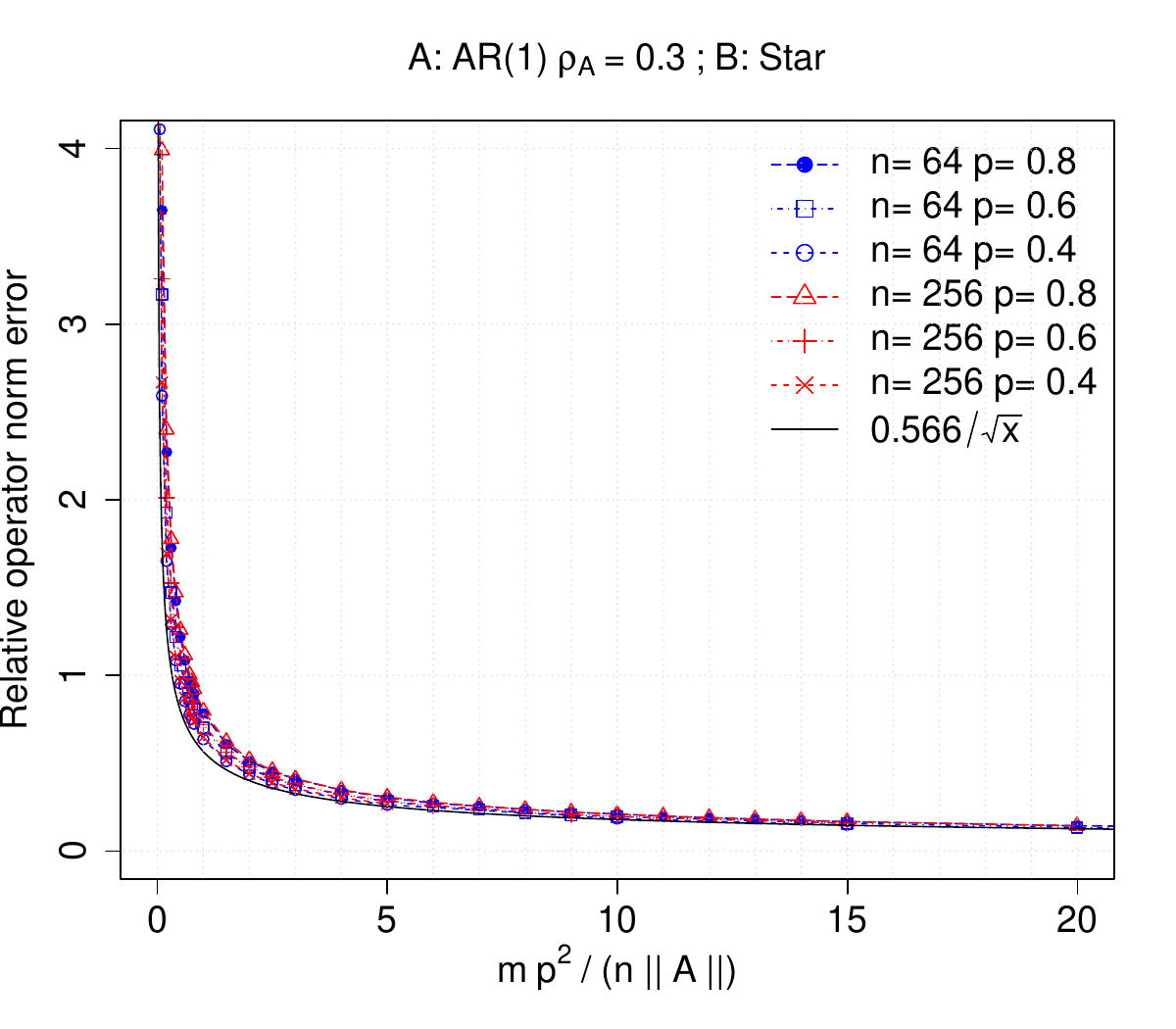}
  \end{subfigure}
  \bigskip
\begin{subfigure}{.45\textwidth}
      \includegraphics[width=2.5in]{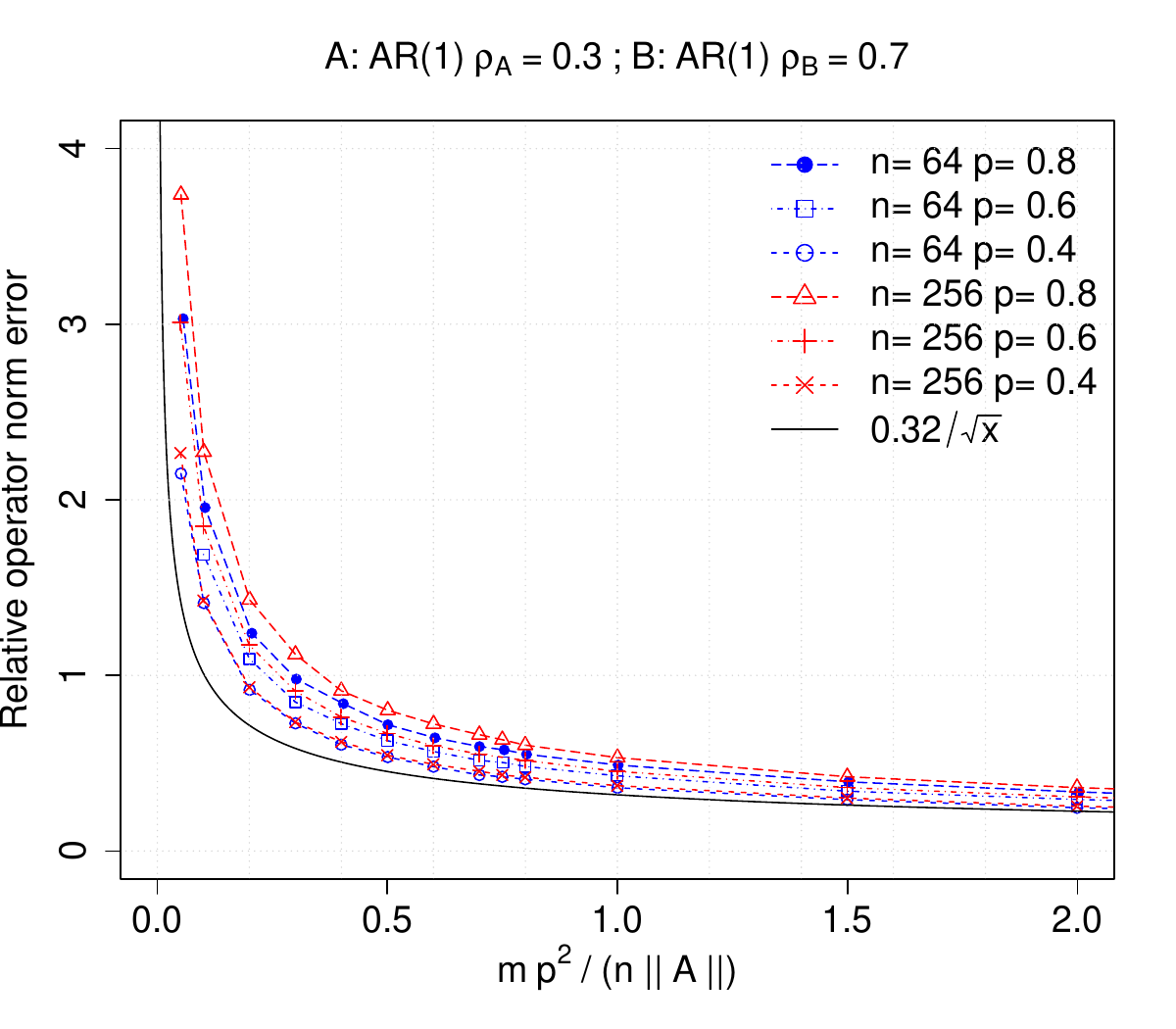}
  \end{subfigure}%
\hspace{5mm}
\begin{subfigure}{.45\textwidth}
      \includegraphics[width=2.5in]{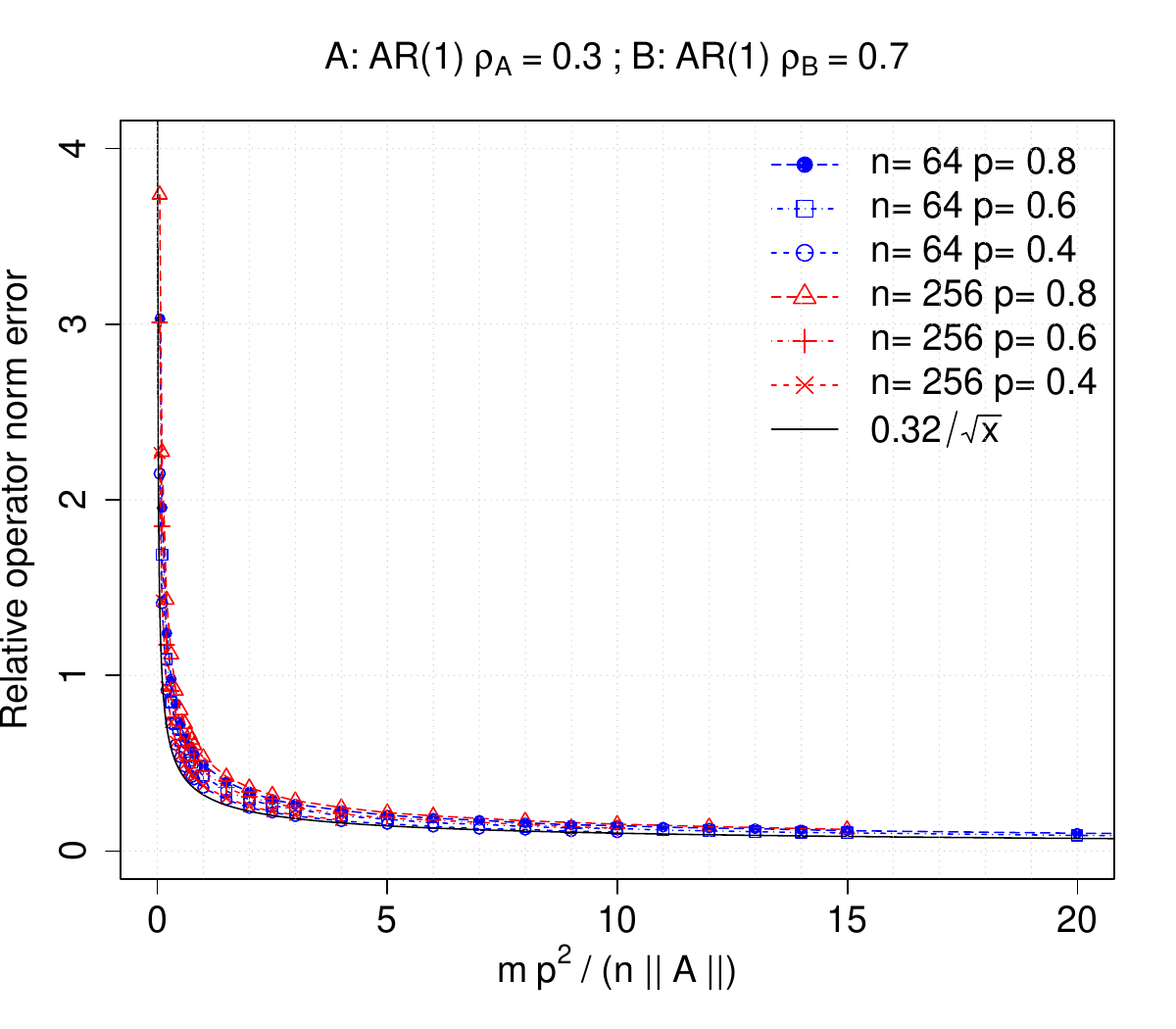}
  \end{subfigure}
  \caption{ For both models of $B$,
    we vary $p \in \{0.4, 0.6, 0.8\}$ and  $n \in \{64, 256\}$.
    On the left panel, differences between the relative errors $\tilde\delta_{\mbox{\scriptsize{overall}}}(x, 256, p)$ and $\tilde\delta_{\mbox{\scriptsize{overall}}}(x, 64, p)$ are greater for
    smaller values of $x$ and larger values of $p$ upon rescaling.
On the right panel, all error curves align with a baseline analytic curve 
of $t(x) = c_0/{\sqrt{x}}$ (solid line)
when $x > 2$.}
\label{fig::pairs}
\end{figure}

\begin{figure}[!htb]
  \centering
\begin{subfigure}{.45\textwidth}
        \includegraphics[width=2.5in]{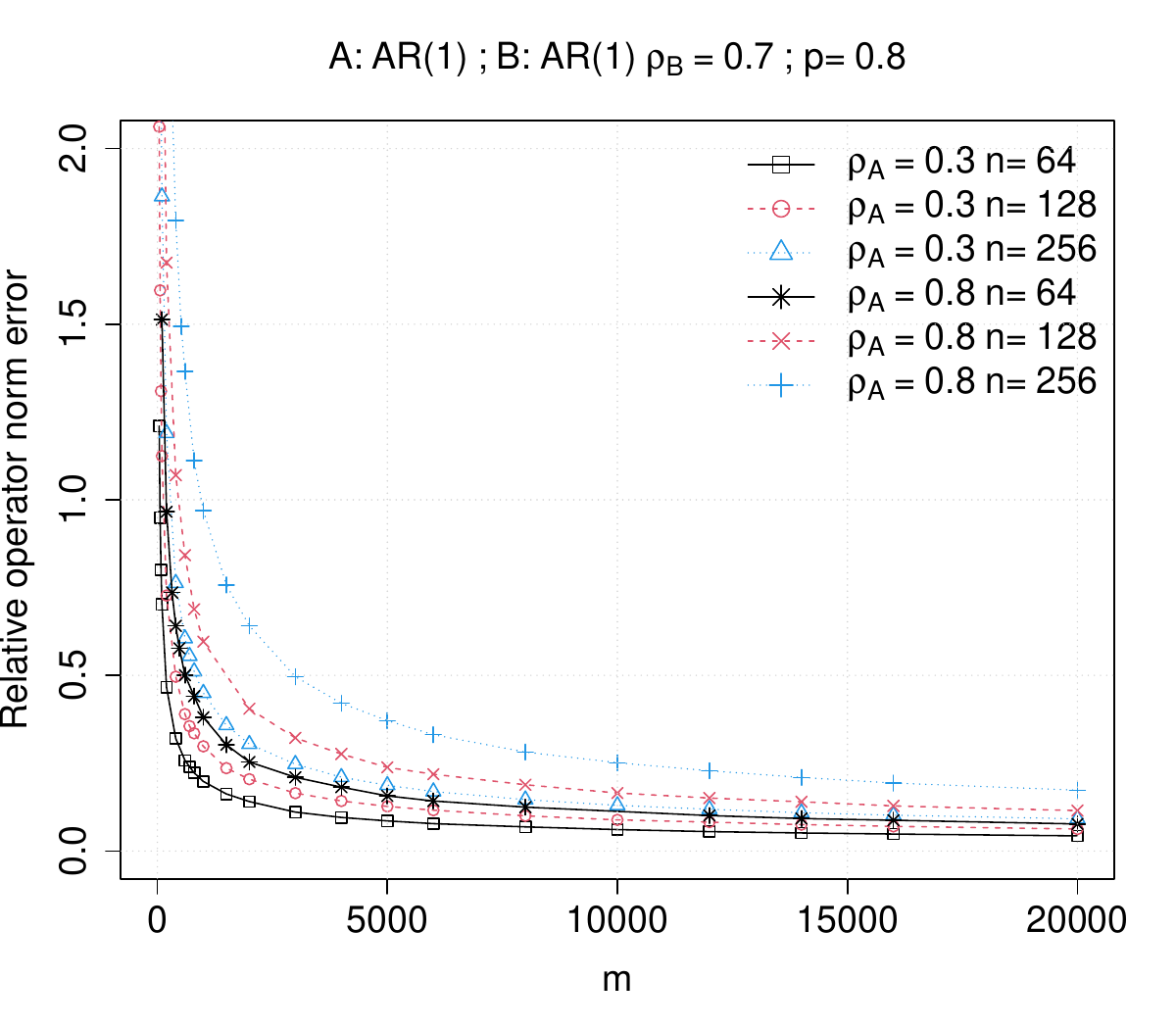}
  \end{subfigure}%
\hspace{5mm}
\begin{subfigure}{.45\textwidth}
      \includegraphics[width=2.5in]{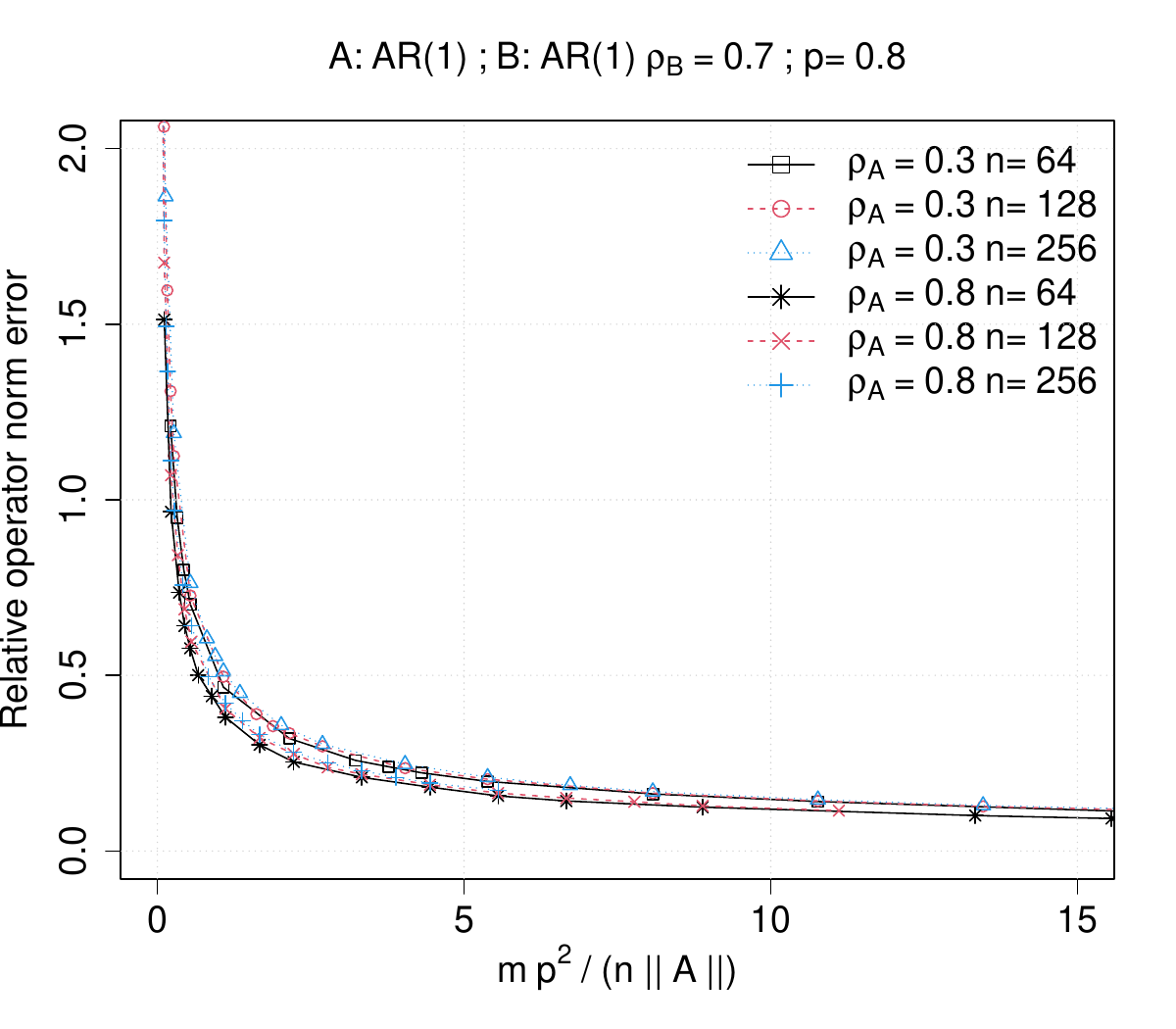}
  \end{subfigure}
\centering
\begin{subfigure}{.45\textwidth}
    \includegraphics[width=2.5in]{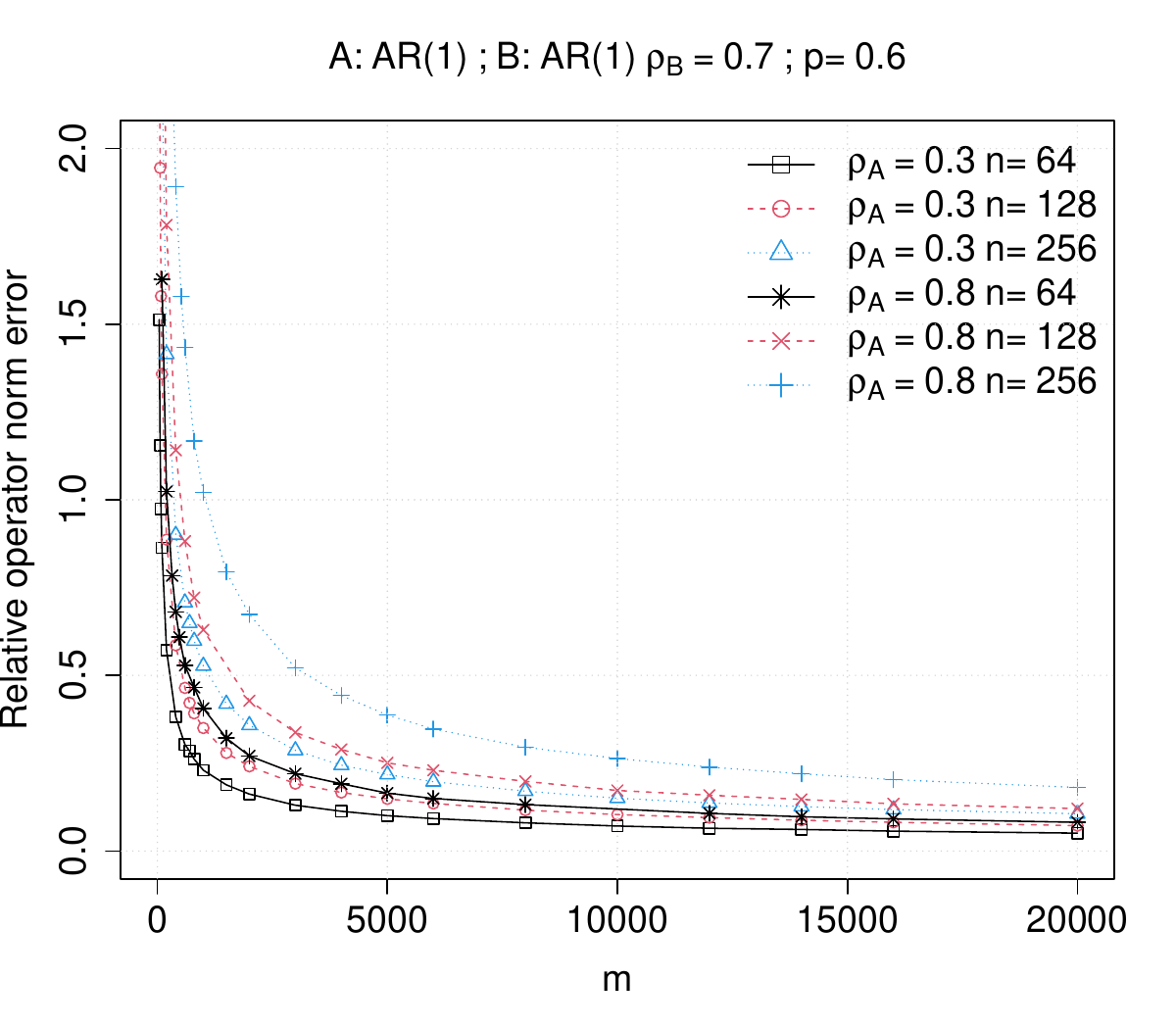}
\end{subfigure}%
\hspace{5mm}
\begin{subfigure}{.45\textwidth}
  \includegraphics[width=2.5in]{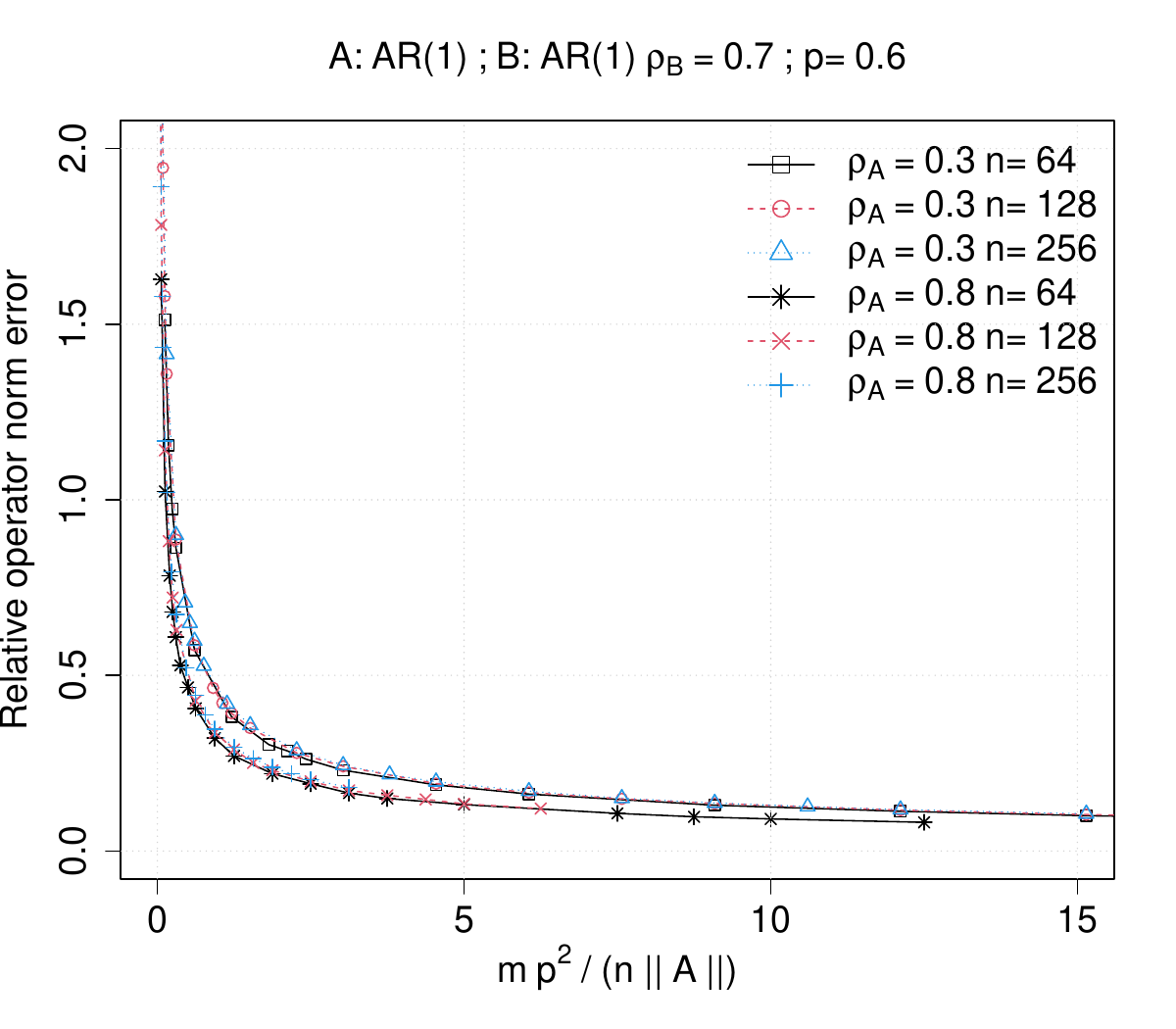}
  \end{subfigure}
\centering
\begin{subfigure}{.45\textwidth}
    \includegraphics[width=2.5in]{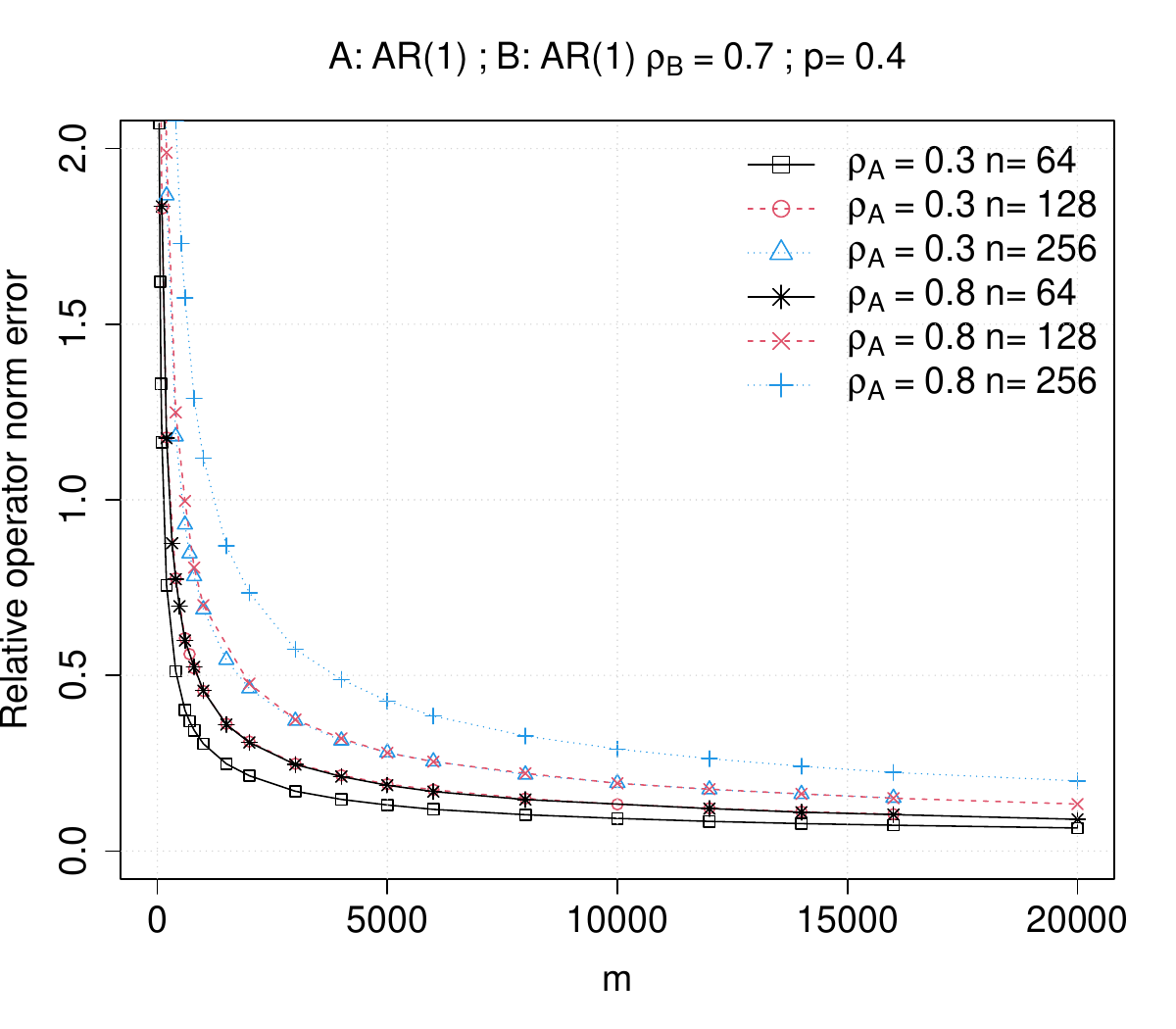}
  \end{subfigure}
\hspace{5mm}
\begin{subfigure}{.45\textwidth}
  \includegraphics[width=2.5in]{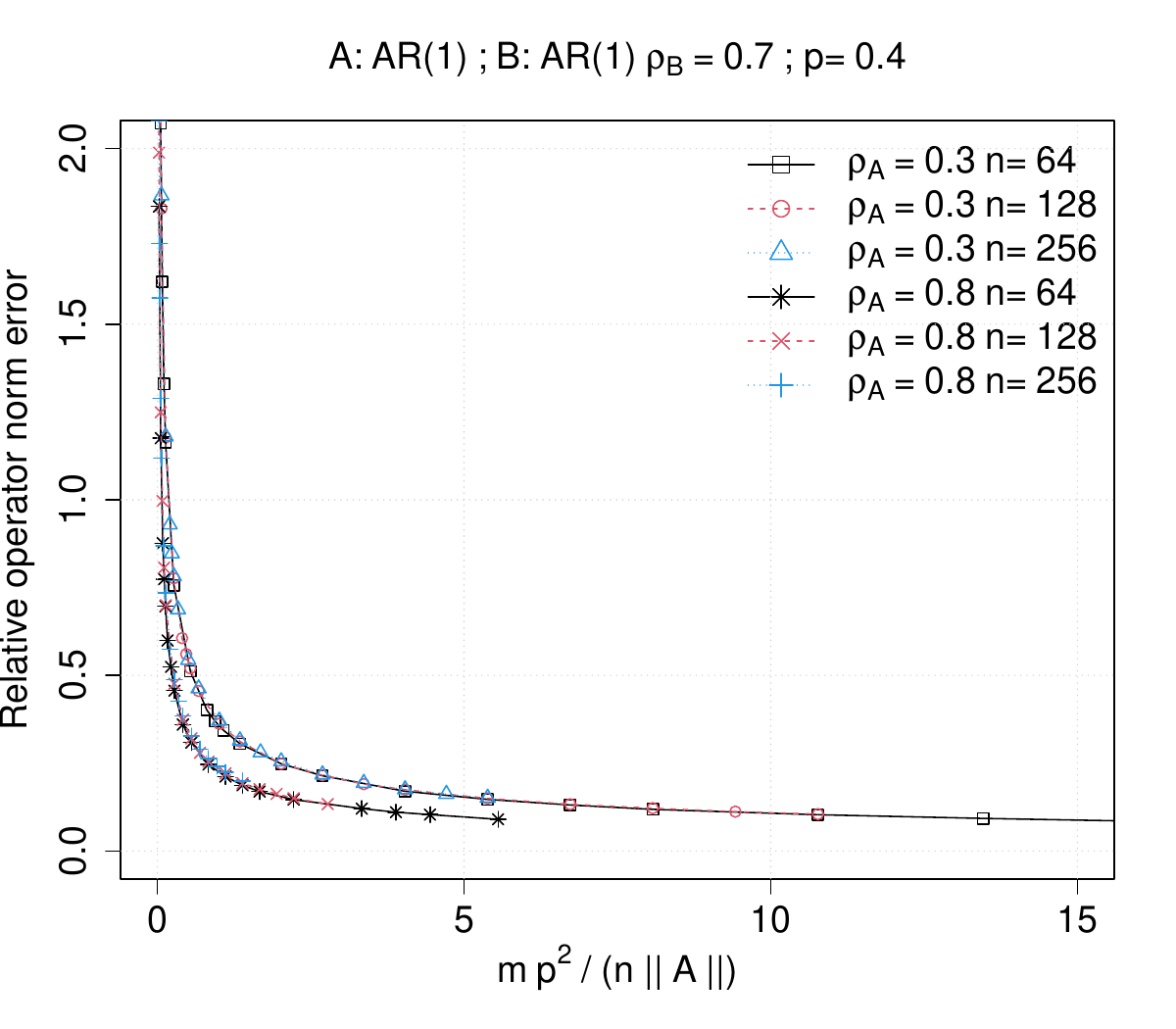}
  \end{subfigure}
  \caption{
    On the left panel, the relative errors  increase with $n$ and
    $\rho_A$ for three  choices of $p \in \{0.4, 0.6, 0.8\}$.
    On the rescaled plots (right panel), error curves $\tilde\delta(x, n, 
    p)$ across different $n$ align well for each pair of ($\rho_A$,
    $p$), with curves for $\rho_A =  0.8$ staying below those for $\rho_A = 0.3$. }
  \label{fig::twoA}
\end{figure}

\noindent{\bf Overall summary.}
We plot the relative errors in the operator norm,  
namely, $\delta_{\overall} := \shnorm{\hat{B}^{\star}
  -B}_2/\twonorm{B}$ for estimating  $B$, against $x$, for a variety  
of configurations of $(n, m, \bf{p})$.  
This is shown in Figure~\ref{fig::global} for Star and
$\operatorname{AR}(1)$ models for $B$. Each point on these curves, denoted by
$\tilde\delta_{\overall}(x, n, p)$, corresponds to an average over 100
trials. First, we note that the three cases of $B$ show the same trend when
$A$ is fixed in Figures~\ref{fig::global} and~\ref{fig::pairs}:
When $x$ is large, all error curves align with a baseline analytic curve 
of $t(x) \propto 1/{\sqrt{x}}$;
When $x$ is small, the observed average relative errors
$\tilde\delta_{\overall}(x, n, p)$  deviate from $t(x)$ in the rescaled plot, as shown
in Figure~\ref{fig::pairs}.
Both phenomena are as predicted by theory, where
the error is bounded by~\eqref{eq::Bhatop}: for $\psi_B(n) = O(\sqrt{n}/\log^{1/2} (n))$,
\ben 
 \label{eq::3terms}
   \delta_{\overall}(x, n, p) & \le &  c_0/\sqrt{x} +  f(x, n, p), \;
  \text{ where} \\
\nonumber
f(x, n, p)  & \le &  c_3 \sqrt{p \psi_B(n) } /x^{3/4} +C_4 (1 \vee (\psi_B(n) /{\twonorm{A}}))/{x};
\een
To characterize the transition between small and large values of $x$, define
$x_{\thresh}  :=   c_4 \psi^2_B(n) \Big(1   + c_5 p^2 \psi^2_B(n)\log (n)/n 
  \Big) \asymp \psi^2_B(n)$,  which holds  for all models of $B$. In fact, since all entries of $B$ are positive for the Star and $\operatorname{AR}(1)$
models we consider, we have $\psi_B(n) =1$ and hence $x_{\thresh}
\asymp 1$. For $x \le x_{\thresh}$,  the function $f(x, n, p)$ provides 
the explanation for the upward-shift of all relative error curves around the elbow region, compared to the baseline function of $t(x) := c_0/\sqrt{x}$  for a 
fixed $n$, as shown in Figure~\ref{fig::pairs}.
When $x > x_{\thresh}$,  the first term on the RHS of 
\eqref{eq::3terms} clearly dominates others, and hence all error 
curves  again align well and we have $\tilde{\delta}_{\overall} \asymp 
r_{\offd}(n)$ for all settings.
Hence, both models follow the trend set by the baseline $t(x) =
c_0/{\sqrt{x}}$, albeit with distinct  constants,
as long as the ratio $mp^2/(n\twonorm{A})$ stays to the right of $x_{\thresh} \asymp 1$. 
For the Star model, the bound $\norm{B}_{\infty}/\twonorm{B}$
increases from $3.429$ for $n=64$ to $6.488$ when  $n=256$ while
$\psi_B(n) =1$, which shows that the obvious upper bound~\eqref{eq::phiben} can be loose.\\
\noindent{\bf Dependence on $p$.}
\label{sec::dual}
Due to the influence of $f(x, n, p)$ as in~\eqref{eq::3terms},
we observe in Figure~\ref{fig::pairs} that upon rescaling
(a) the grouping of relative error curves $\{\tilde\delta_{\overall}(x, n, p), n \in \{64, 256\} \}$ by values of $p$, where the relative errors obey the order of $\tilde\delta_{\overall}(x, n, p_1) \ge \tilde\delta_{\overall}(x,
n, p_2)$ when $p_1 > p_2$  for fixed $(x, n)$;
(b) an upward shift of the error curves for $n =256$ compared to $n=64$
for the same $p$ when $x < x_{\thresh}$; and
(c) merging of all error curves when $x > 2$:
any separation due to the variations on $p, n$, is almost negligible.
Plots on the right panel of Figure~\ref{fig::pairs} (zoomed out 
versions of the left)  show that all curves for each pair $(n, p)$
again align well for $x > 2$,  as they follow the same baseline
function $\propto 1/\sqrt{x}$. For $x \le 2$, larger values of $p$ cause $f(x)$ to carry even more
weight, as evidenced by the vertical separation of the three sets of
error curves corresponding to $p \in \{0.4, 0.6,  0.8\}$
on the left panel in Figure~\ref{fig::pairs}, resulting in the
observed relative errors to further deviate from $t(x)$.
Additionally, the separation between the coupled pairs (grouped by 
the same value of $p$) for different $n$ as mentioned in (b)
is more prominent in the range of $x < 1$.
This can be explained by the functions corresponding to two values of $n_1, n_2$ for each fixed $(x, p)$: $\abs{f(x, n_1, p) -f(x, n_2, p) } \le   c_6 \sqrt{p}/x^{3/4}  + c_8 
/x$, which decreases with $x$ and increases with $p$.

\noindent{\bf Dependence on $\rho_A$.}
In Figure~\ref{fig::twoA}, $A$ and $B$ are both generated using the AR(1)
model with two choices of $\rho_A\in \{0.3, 0.8\}$ and $\rho_B=0.7$.
When $\rho_A$ increases from $0.3$ to $0.8$, we observe the relative
errors for estimating $B$ increase for the same values of $m, n, p$ on the left panel.
This is because $\twonorm{A}$ increases with the value of $\rho_A$ for a fixed value of 
$m$ and hence decreases the effective sample size by Theorem~\ref{thm::gramsparse}.
On the right panel, we observe the same phenomenon as before:
when we rescale the $x$-axis by $m 
p^2/(\twonorm{A} n)$, the error curves for all $n \in \{64, 128,
256\}$ again align well  for each configuration of $(p, \rho_A)$ when $x >
x_{\thresh}$, which shifts to the right as $p$ increases.
Moreover, upon rescaling, on the right panel,
we observe that (a) the error curves shift downwards slightly for $\rho_A = 0.8$
for fixed values of $x$; (b) the gap between the relative errors corresponding to two values of
$\rho_A$ appears to be widening as $p$ decreases,
which suggests that the influence of the operator norm 
$\twonorm{A}$ on the relative errors increases as $p$ decreases;
and (c) the error curves corresponding to all values of $n$ for
$p=0.4, 0.6$ for each $\rho_A$ align better than the case for $p =
0.8$. All phenomena can be explained through~\eqref{eq::Bhatoplocal}
and~\eqref{eq::classic}, since
the third factor in $\delta_q(B)$ as in ~\eqref{eq::classic}, namely,
$\tau_p$ becomes less influential as $p$
decreases in view of Remark~\ref{rem::threefactor}, in case $\psi_B(n) = O(\sqrt{n} /\log^{1/2} (m
\vee n))$.

\silent{
 \begin{remark}
  \label{rem::threefactor}
\ben
\nonumber
&& 
\twonorm{\hat{B}^{\star}  -B_0}/{\twonorm{B_0} }
= O_P\Big(r_{\offd}(n)  + \big(p + 1/\twonorm{A_0}\big) r^2_{\offd}(n)
\psi_B(n) \Big)
\een}



\silent{As a concrete example, Theorem~\ref{thm::gramsparse} can not be
derived using the entrywise deviation, since we
have removed the $\log (m \vee n)$ factor from the leading factor,
namely, $O(\eta_A r_{\offd}(n))$ when estimating $B_0$ as a whole,
\bens
r_{\offd}(n) := \sqrt{\frac{n \twonorm{A_0} }{\sum_{j=1}^m a_{jj} p_j^2 }}
\eens
which will not be possible to obtain using the entrywise bounds
through union bound directly, as those bounds unavoidably involve an
logarithmic term on the leading term.}

\section{Proof sketch of Theorem~\ref{thm::main}}
\label{sec::mainresult}
In view of the discussion in Section~\ref{sec::mainsketch}, we focus 
on the off-diagonal component.
Notice that $\mvec{Y} = (B_0^{1/2} \otimes A_0^{1/2}) 
\mvec{\Z^T}$, where $Y = X^T = A_0^{1/2}  \Z^T B_0^{1/2}$
for $\Z$ as in \eqref{eq::missingdata}.  Let $\tilde{\Delta} =
\offd(\X \X^T  -B_0 \circ \M)$.
For a chosen sparsity $1 \le s_0 \le n$, let $E = \cup_{\abs{J} = s_0} E_J$. 
To control the quadratic form \eqref{eq::quad0}
over all $s_0$-sparse vectors $q, h \in \Sp^{n-1} \cap E$,
we first obtain a uniform large deviation bound for
$q^T\tilde{\Delta} h$ for all $q, h \in \N$, the $\ve$-net of
$\Sp^{n-1} \cap E$ as constructed in Lemma~\ref{lemma::net}; See for exmaple~\cite{MS86}.
\begin{lemma}
  \label{lemma::net}
Let $1/2> \ve > 0$.
For a set $J \in [n]$, denote $E_J = \Span\{e_j, j \in J\}$. For each
subset $E_J$, construct an $\ve$-net $\Pi_J$, which satisfies
$$\Pi_J \subset E_J \cap \Sp^{n-1} \; \text{ and }\; \abs{\Pi_J}  \le (1+2/\ve)^{s_0}.$$
If $ \N = \bigcup_{\abs{J} = s_0} \Pi_J$, then the previous estimate implies that
\ben
\label{eq::netsize}
\abs{\N} \le (3/\ve)^{s_0}{n \choose s_0} \le \big(
  \frac{3en}{s_0 \ve}\big)^{s_0}
= \exp \big(s_0 \log \big(\frac{3en}{s_0\ve}\big) \big).
\een
Clearly, when $s_0 =n$, we have $\abs{\N} \le (3/\ve)^n$.
\end{lemma}
\noindent{\bf Symmetrization.}
We write $y^i= c_i^T \otimes A_0^{1/2} Z$, where $Z \sim \mvec{\Z^T}$, and
\ben
\label{eq::quad0}
\lefteqn{q^T \offd(\X \X^T) h  = \sum_{i \not=j} q_i h_{j} \ip{v^i
    \circ y^i, v^j \circ y^j} = } \\ 
& &
\nonumber
Z^T\big(\sum_{i \not= j} q_i h_{j}  c_i c_j^T \otimes A_0^{1/2}
\diag(v^i \otimes v^j) A_0^{1/2} \big) Z, \; \text{where } \; \; c_1,
\ldots, c_n \in \R^n 
\een
are the column vectors of $B_0^{1/2}$ as in~\eqref{eq::cf};
Now $A_0^{1/2} \diag(v^i \otimes v^j) A_0^{1/2} = \sum_{k=1}^m u^k_i
u^k_j d_k \otimes d_k$, where $d_1, d_2, \ldots, d_m \in \R^{m}$ are the column vectors of 
$A_0^{1/2}$ as in~\eqref{eq::dm}. 
Clearly, by symmetry of the gram matrix $\X \X^T$, for any $q,
h\in \Sp^{n-1}$, $q^T \offd(\X \X^T) h = h^T \offd(\X \X^T) q$;
We now symmetrize the random quadratic form~\eqref{eq::quad0} and rewrite
\ben
\lefteqn{
  \label{eq::symmetry}
\forall   q, h \in \Sp^{n-1}, \; \; q^T \offd(\X \X^T) h  := Z^T
A^{\diamond}_{q h} Z \quad \text{ where}, } \\
\label{eq::defineADM}
A_{qh}^{\diamond}
& = &
\half \sum_{k=1}^m \sum_{i\not= j} u^k_i u^k_j  (q_i h_{j}  + q_j h_i) 
(c_i c_j^T) \otimes (d_k  d_k^T)
\een
is symmetric since for any index set  $(i, j, k)$, both $(c_i c_j^T) \otimes (d_k  d_k^T)$ and its transpose
$(c_j c_i^T) \otimes (d_k  d_k^T)$ appear in the sum with the same
coefficients. 
We have by~\eqref{eq::symmetry}, for $Z =\mvec{\Z^T}$,
\ben
\nonumber
\lefteqn{
\abs{q^T \tilde{\Delta} h}  :=\abs{Z^T  A^{\diamond}_{q h} Z  - \E (Z^T 
      A^{\diamond}_{q h} Z )} \le {\bf I} + {\bf II} :=}\\
    \label{eq::decomp}
   & &
\abs{  Z^T  A^{\diamond}_{q h} Z  - \E (Z^T  A^{\diamond}_{q h} Z |U)} + \abs{\E (Z^T  A^{\diamond}_{q h} Z |U) - \E (Z^T  A^{\diamond}_{q h} Z)} 
\een
Similar to the proof of Theorem~\ref{thm::gramsparse},
we will present a uniform bound on Part I followed by that of Part II
of~\eqref{eq::decomp} for all $q, h \in \N$
and show on event $\F_0^c \cap \F_1^c \cap \F_2^c$,
cf. Theorems~\ref{thm::samplesize} and~\ref{thm::mainlights},
\ben 
\label{eq::netbounds}
\sup_{h, q \in \N}  {\abs{q^T \tilde\Delta h}}/{\big({\norm{\M}_{\offd}}
    \twonorm{B_0}\big)}
  & \le &  \delta_{q},   \; \text{where } \;\; \delta_q \asymp \underline{r_{\offd}} \sqrt{s_0} +r_{\offd}^2 (s_0) \psi_B(s_0);
\een
A standard approximation argument shows that if
\eqref{eq::netbounds} holds, then
\ben
\label{eq::approx}
\sup_{h, q \in E \cap \Sp^{n-1}}
\abs{q^T  \tilde\Delta h}/\big(\norm{\M}_{\offd} \twonorm{B_0}\big)
\le {\delta_q}/{(1-\ve)^2}.
\een
\noindent{\bf Part I.}
We first condition on $U$ being fixed. Then the quadratic form~\eqref{eq::symmetry}
can be treated as a subgaussian quadratic form with 
$A^{\diamond}_{q h}$ taken to be deterministic.
Theorem~\ref{thm::mainop} in the supplementary material,
similar to Theorem~\ref{thm::mainop2}, states that the operator norm of $A_{qh}^{\diamond}$ is uniformly and
deterministically bounded by $\twonorm{A_0} \twonorm{B_0}$ for all
realizations  of $U$ and for all $q, h \in \Sp^{n-1}$. We then state a probabilistic uniform bound on $\shnorm{A_{qh}^{\diamond}}_F$ in 
Theorem~\ref{thm::uninorm-intro} on the event that $\F_0^c$ holds, similar to
Theorem~\ref{thm::uninorm2}.
These results and their proof techniques may be of independent interests.
Applying the Hanson-Wright inequality~\cite{RV13}
(cf. Theorem~\ref{thm::HW}) with the preceding estimates on the
operator  and Frobenius norms  of 
$A_{qh}^{\diamond}, q, h \in \Sp^{n-1}$, cf.~\eqref{eq::defineADM},
and the union bound, we prove Theorem~\ref{thm::samplesize}.
\begin{theorem}
  \label{thm::samplesize}
Let $1/2> \ve > 0$. Fix $s_0 \in [n]$.
Denote by $\N$ the $\ve$-net for $\Sp^{n-1} \cap E$ as constructed in 
Lemma~\ref{lemma::net}.
Suppose
\ben
\label{eq::sample1local}
&&
{\sum_{j=1}^m a_{jj} p^2_j} \ge C_4 {\twonorm{A_0}} s_0  \log (n \vee
m) \eta_A^2\; \text{ and } \; r_{\offd}(s_0)  f_{\QA} \psi_B(s_0) <1 \\
\label{eq::QAratio}
&&\quad  \text{ where} \; \; f^2_{\QA}  := {\big(\log (n \vee m) \twonorm{A_0} a_{\infty} \sum_{s=1}^m 
  p_s^4\big)^{1/2}}/{\big(\sum_{j=1}^m a_{jj} p_j^2\big)}
\een
Then on event $\F_0^c \cap \F_1^c$, which holds with probability at least $1 -
\frac{4}{(n \vee m)^4} - 2 \exp(- c_1 s_0 \log (\frac{3e n}{(s_0 \ve)}))$,
\ben
\label{eq::omnistep1}
&& \sup_{q, h\in \N} 
\frac{ \abs{Z^T A^{\diamond}_{q h} Z-\E (Z^T A^{\diamond}_{q h} Z|
    U)}}{\norm{\M}_{\offd} \twonorm{B_0}} =O\big(\eta_{A}
r_{\offd}(s_0)   + r_{\offd}(s_0)  f_{\QA} \psi_B(s_0)\big)
\een
\end{theorem}
\noindent{\bf Part II.}
Let $\tilde{a}^{k}_{ij}$ be a shorthand for $\tilde{a}^{k}_{ij}(q, h) = \half 
a_{kk} b_{ij}(q_i h_{j}+ q_j h_i)$. 
Denote by
\ben 
\label{eq::stardust}
\quad \quad S_{\star}(q, h)  = 
\E (Z^T  A^{\diamond}_{q h} Z  | U)- \E (Z^T  A^{\diamond}_{q h} Z) 
= \sum_{k=1}^m \sum_{i \not=j}^n \tilde{a}^{k}_{ij}(u^k_i u^k_j -
p_k^2).
\een

\begin{theorem}
  \label{thm::mainlights}
Let  $S_{\star}(q, h)$ be as in~\eqref{eq::stardust}.
Suppose for $\psi_B(s_0)$ as in Definition~\ref{def::C0}, where $s_0
\in [n]$,
\ben 
\label{eq::samplecrux}
\sum_{j=1} a_{jj} p_j^2 & =& \Omega\big(a_{\infty} \psi_B(2s_0 \wedge
  n)   s_0 \log   (3en/(s_0\ve)) \big).
\een
Then for absolute constants $C_6, c_1$ and $\tau' = C_6 a_{\infty} \psi_B(2s_0 \wedge n) \twonorm{B_0} s_0 \log (3e n/(\ve 
s_0))$,
\bens
\prob{\F_2} = \prob{\exists q, h \in \N,  \abs{S_{\star}(q, h)}  \ge
  \tau'} \le 2 \exp(- c_1 s_0  (3en/(s_0\ve)) )
\eens
\end{theorem}

Similar to the proof of Theorem~\ref{thm::gramsparse},
we also use sparse Hanson-Wright inequalities and the union bound to prove
Theorem~\ref{thm::mainlights}. 
Theorem~\ref{thm::Bernmgf} is crucial in  
deriving a concentration bound for $S_{\star}(q, h)$ as
in~\eqref{eq::stardust}.
Such results may be of independent interests.
On event $\F_2^c$ as defined in Theorem~\ref{thm::mainlights} we 
have
\ben 
\label{eq::mainF2}
&& \sup_{q, h \in \N} 
{\abs{S_{\star}(q,h)} }/{(\norm{\M}_{\offd} \twonorm{B_0})}
\le  C' r_{\offd}^2(s_0) \psi_B(2s_0 \wedge n).
\een
\noindent{\bf Putting things together.}
Combining the large deviation bounds~\eqref{eq::omnistep1} and
\eqref{eq::mainF2} with~\eqref{eq::decomp} and using the
approximation argument in the sense of~\eqref{eq::netbounds}
and~\eqref{eq::approx}, we have $\text{on event } \F_0^c \cap \F_1^c
\cap \F_2^c =: \F_4^c$,
\ben
\nonumber
\lefteqn{ \sup_{q, h \in \Sp^{n-1} \cap E}
{\abs{q^T\tilde{\Delta} h}}/{\big(\twonorm{B_0}\norm{\M}_{\offd}\big)} \le
\sup_{q, h \in \N} {\abs{q^T\tilde{\Delta} h}}/
{\big((1-\ve)^2  \twonorm{B_0}\norm{\M}_{\offd}\big)} } \\
&\le  &
\label{eq::bridge2}
C \big(\eta_A r_{\offd}(s_0)   + r_{\offd}(s_0)  f_{\QA} \psi_B(s_0) +
r_{\offd}^2(s_0) \psi_B(2s_0 \wedge n)\big) =:  \delta_{\overall} 
 \een
Finally, $\prob{\F_4^c} \ge 1 - 4/(m \vee n)^4 - 4\exp(- c s_0 \log (3en/(s_0 \ve)))$ by
Theorems~\ref{thm::samplesize},~\ref{thm::mainlights} and the union bound.
Combining \eqref{eq::bridge2} and Lemma~\ref{lemma::finalrate}, we 
have~\eqref{eq::rmn}.  
Lemma~\ref{lemma::finalrate}
shows that when
$p_{\max}$ is small, the second term in the RHS of \eqref{eq::omnistep1}
will become less influential, and hence the overall rate of convergence
will be dominated by $r_{\offd}(s_0)$ and $r_{\offd}^2(s_0)
\psi_B(2s_0 \wedge n)$; cf. Remark~\ref{rem::threefactor}
in case $s_0 =n$.
\begin{lemma}
\label{lemma::finalrate}
Let $r_{\offd}(s_0)$ and $\ell_{s_0, n}$ be as defined
in~\eqref{eq::offdrate}.  Let $p_{\max} =\max_{j} p_j$.
Then
\bens
r_{\offd}(s_0)  f_{\QA} \psi_B(s_0) & \le & \sqrt{p_{\max}  } r_{\offd}(s_0)  (\ell_{s_0, n}^{1/2}  \eta_A+  r_{\offd}(s_0) \psi_B(s_0)) \\
\text{ and } \; \delta_{\overall}
& \le &
 C_{\sparse} \big(r_{\offd}(s_0) \eta_{A} \ell_{s_0, n}^{1/2} +
r^2_{\offd}(s_0)  \psi_B(2s_0 \wedge n) \big) \asymp \delta_q
\eens
where $C_{\sparse}$, $\delta_q$, and $f_{\QA}$ are as defined
in~\eqref{eq::rmn},~\eqref{eq::netbounds} and~\eqref{eq::QAratio} respectively.
\end{lemma}

In the supplementary material, we prove Theorem~\ref{thm::samplesize}
in  Section~\ref{sec::appendsamplesize}, where we define events $\F_1$
and $\F_0$.
We defer the rest of the proof on checking conditions in  
Theorems~\ref{thm::mainlights} and~\ref{thm::samplesize} to
Section~\ref{sec::appendproofoffdmain}.
We prove Theorem~\ref{thm::mainlights} and Lemma~\ref{lemma::finalrate},
and Theorem~\ref{thm::uninorm-intro}
in the supplementary 
Sections~\ref{sec::appendmainlights},~\ref{sec::bernchaos}, 
and~\ref{sec::uninormskye}, respectively.

{\bf Discussions.}
\label{rem::twofactors}
The result in Theorem~\ref{thm::mainlights} appears to be tight, since
when we ignore the logarithmic terms, the linear dependency on $s_0$
is correct in view of Theorem~\ref{thm::AD};
the extra term $\rho_{\max}(s_0, \abs{B_0})$ is unavoidable, 
because of our reliance on the sparse Hanson-Wright type of moment 
generating function bounds as stated in  Theorem~\ref{thm::Bernmgf}
(cf. Lemma \ref{lemma::Bernmgf2}). See also~\cite{RZ13, LW12b} for discussions
on the sample complexity in the i.i.d. settings and with missing
values in the linear models.
It is important that we separate the dependence on $s_0$ from
dependence on $\rho_{\max}(s_0, \abs{B_0})$ as defined
in~Definition~\ref{def::C0}.

Theorems~\ref{thm::mainop} and~\ref{thm::uninorm-intro} prove
corresponding results on matrix $A_{qh}^{\diamond}, q, h \in
\Sp^{n-1}$, cf.~\eqref{eq::defineADM}, from which
Theorems~\ref{thm::mainop2} and~\ref{thm::uninorm2} follow
respectively; we need to overcome major obstacles due to the
lack of existing tools. In particular, when we bound the Frobenius norm for
$A_{qh}^{\diamond}$ as a function of random matrix $U$,
we use a decomposition argument to express
$\shnorm{A^{\diamond}_{qh}}_F^2$ as a summation over homogeneous
polynomials of degree $2, 3$, and $4$ respectively.
We then prove the concentration of measure bounds for various homogeneous
polynomial functions in Lemmas~\ref{lemma::W2devi} to~\ref{lemma::S5devi}
respectively, in the supplementary Section~\ref{sec::uninormskye}.
The common theme is to: first, obtain an estimate on the moment
generating function of each unique polynomial function of
the non-centered Bernoulli random variables; then, obtain the
desired tail bounds by Markov’s inequality.
Finally, we use the triangle inequality and the union bound to
carefully join all pieces together to prove Theorem 18.2. These result
may be of independent interests.

Our analysis framework will extend, with suitable adaptation, to the 
general distributions of $U$ with independent nonnegative elements,
which are independent of the  (unobserved) matrix variate data $X$.
Such a result may be of 
independent interests, as more generally, the mask matrix $U$ may not 
be constrained to the family of Bernoulli random matrices.
Instead, one may consider $U$ as
a matrix with independent  rows with arbitrary positive coefficients
drawn from some distributions~\citep[cf.][]{carr:rupp:2006,HWang86,ICF99}.
For instance, one may consider $U$ as a matrix with arbitrary positive 
coefficients belonging to $[0, 1]$.
Our proof for Theorems~\ref{thm::mainop2} and~\ref{thm::mainop} will
go through if one replaces $u^k, k=1, \ldots, m$ by independent
Gaussian random vectors; however, the statement will be probabilistic
subject to an additional logarithmic factor. In particular, 
Lemma~\ref{lemma::chaosop} in supplementary material holds for general
block-diagonal matrices with bounded operator norm, which can be deterministic.


\section{Proofs}
\label{sec::mainproofs}
In this section, we prove Proposition~\ref{prop::projection} and
Theorem~\ref{thm::main-intro} in
Sections~\ref{sec::proofofprojection} and~\ref{sec::appendintromain} respectively.
We place all remaining technical proofs in the supplementary material.

\subsection{Proof of Proposition~\ref{prop::projection}}
\label{sec::proofofprojection}
\begin{proofof2}
Recall $\cov(V_j, V_j) = \E (V_j V_j^T)$ and $\E \ip{X^j, X^k} =
b_{jk} A_0$. Denote by $B_{j\cdot}$ and $\Theta_{j\cdot}$ the $j^{th}$ row
vector of $B_0 = (b_{ij})$ and $\Theta_0 = (\theta_{ij})$ respectively.
We have by linearity of covariance,
\ben
\nonumber
\lefteqn{\cov(V_j, V_j)
  =\cov(X^j- \sum_{k \not=j} X^{k} \beta_k^{j*}, X^j- \sum_{k \not=j}
X^{k} \beta_k^{j*})}\\
\nonumber
& =& 
\cov(X^j, X^j) +2 \cov(X^j,  \sum_{k \not=j} X^{k} \frac{\theta_{jk}}{\theta_{jj}}) 
+\cov(\sum_{k \not=j} X^{k} \frac{\theta_{jk}}{\theta_{jj}}, 
\sum_{\ell \not=j} X^{\ell} \frac{\theta_{j\ell}}{\theta_{jj}}) =\\
\label{eq::linearquad}
& & 
b_{jj} A_0 +2 (\sum_{k \not=j} b_{jk} 
\frac{\theta_{jk}}{\theta_{jj}}) A_0
+\big(\sum_{\ell \not=j} \sum_{k \not=j} b_{k\ell} \frac{\theta_{jk}}{\theta_{jj}}
\frac{\theta_{j \ell}}{\theta_{jj}}\big) A_0 = \\
\nonumber
& & b_{jj} A_0 -2 (b_{jj} - 1/{\theta_{jj}}) A_0+
(b_{jj}-1/{\theta_{jj}})A_0 = A_0/{\theta_{jj}} 
\een
where for the linear term and the quadratic term in
\eqref{eq::linearquad}, we have
\bens 
\sum_{k \not=j} b_{jk} {\theta_{jk}}/{\theta_{jj}} & = &
(\ip{B_{j\cdot}, \Theta_{j\cdot}} - b_{jj}\theta_{jj})/{\theta_{jj}} =
1/\theta_{jj} - b_{jj}  \; \; \text{ and } \\
\sum_{\ell\not=j} \big(\sum_{k \not=j} b_{k \ell} \frac{\theta_{jk}}{\theta_{jj}}\big) 
\frac{\theta_{j \ell}}{\theta_{jj}}  
& = &
\sum_{\ell\not=j} \inv{\theta_{jj}}\big(\ip{B_{\ell \cdot}, \Theta_{j \cdot}} -  b_{j\ell}\theta_{jj} \big)
\frac{\theta_{j \ell}}{\theta_{jj}}  = \sum_{\ell\not=j} -  b_{j\ell} \frac{\theta_{j \ell}}{\theta_{jj} }\\
& = & \ip{B_{j \cdot}, \Theta_{j\cdot}}/{\theta_{jj}}  + b_{jj}  = b_{jj}-1/{\theta_{jj}}  > 0
\eens
since $B_0 \succ 0$;
Now we verify the zero covariance condition: for all $\ell \not=j$,
\bens
\cov(X^{\ell}, V_j)  
& =& 
\cov(X^{\ell}, X^j- \sum_{k \not=j} X^{k} \beta_k^{j*})= 
\cov(X^{\ell}, X^j + \sum_{k \not=j} X^{k} {\theta_{jk}}/{\theta_{jj}}) \\
& =&
b_{\ell j} A_0 + A_0 \sum_{k \not=j} b_{\ell k} {\theta_{jk}}/{\theta_{jj}}  
= A_0 \big( b_{\ell j} {\theta_{jj}}/{\theta_{jj}} +
  \sum_{k \not=j} b_{\ell k} {\theta_{jk}}/{\theta_{jj}}  \big) \\
  & =&  \ip{B_{\ell \cdot}, \Theta_{j\cdot}} A_0/{\theta_{jj}} = 0
  \eens
  This proves Proposition~\ref{prop::projection}.
\end{proofof2}

\subsection{Proof of Theorem~\ref{thm::main-intro}}
\label{sec::appendintromain}
Let $\Delta = \tilde{B}_0  - B_0$.
Let $\overline{\Delta}_{B} = \offd(\tilde{B}_0 - B_0)/{\twonorm{B_0}}$.
By  Theorem~\ref{thm::main} and Lemma~\ref{lemma::quadreduction}, we have
on event $\F_4^c$, for all $q \in (\sqrt{s_0} B_1^n \cap B_2^n)$ and $d := 2s_0 \wedge n$,
\bens
{\abs{q^T \overline\Delta_B q}}
& = & {\abs{q^T  \offd(\X \X^T -B_0
    \circ \M) q}}/(\norm{\M}_{\offd} {\twonorm{B_0}}) \\
& \le &
4 C_{\sparse} \Big(\ul{r_{\offd}} \sqrt{s_0} +r_{\offd}^2 (s_0) 
\psi_B(d)\Big) =: 4\delta_q;
\eens
Hence by~\eqref{eq::convex},~\eqref{eq::MB2}, and~\eqref{eq::rmn},
we have on event $\F_{\diag}^c \cap \F_4^c$, $\forall q \in (\sqrt{s_0} B_1^n \cap B_2^n)$,
\bens
{\abs{q^T \Delta q}}/{\twonorm{B_0}}  & \le &
C_{\diag} r_{\diag} + {\abs{q^T \overline\Delta_B q}}
\le C_{\diag} r_{\diag} + 4 C_{\sparse} \Big(\ul{r_{\offd}}
\sqrt{s_0} +r_{\offd}^2 (s_0)  \psi_B(d)\Big) \\
& \asymp & \eta_A  r_{\offd}(s_0) \ell^{1/2}_{s_0, n} +
r_{\offd}^2(s_0) \psi_B(d),  \;\text{where } \\
{\abs{q^T \diag(\Delta) q}}/{\twonorm{B_0}}  & = &
\abs{q^T  \diag(\X  \X^T  -B_0 \circ \M) 
  q}/(\twonorm{B_0}\norm{\M}_{\diag})  \le  C_{\diag} r_{\diag} = o(\delta_q) 
\eens
The theorem thus holds.
\qed

\silent{
\subsection{Proof of Theorem~\ref{coro::thetaDet}}
\label{sec::proofofinverses}
Let $\tilde\Theta_{j \cdot}$ denote the $j^{th}$ row vector of 
$\tilde\Theta$ following Definition~\ref{def::TopHat}.
In Lemma~\ref{lemma::tidebound}, we derive error bounds for estimating the  
diagonal entries of $\Theta_0$, as well as the error bounds for constructing row vectors $\{\Theta_{j \cdot}, j\in [n]\}$ of 
$\Theta_0$ with $\{\tilde\Theta_{j \cdot}, j \in [n]\}$. Let $\kappa_B := \twonorm{B_0}/\lambda_{\min}(B_0)$.
Let  $\tilde{\kappa}_{\rho} := M_{\rho} M_{\Omega} \ge \kappa_{\rho}$ be an upper estimate on the condition number of $\rho(B)$ under 
(A1). Let $\alpha$ be as in \eqref{eq::lassopen}. Suppose
\ben 
\label{eq::parityproof}
&& {\sum_{j} a_{jj} p_j^2}/{\twonorm{A_0}}  \ge 4 C_{\overall}^2
\eta_A^2\kappa_B^2 \twonorm{\Theta_0}^2 d_0 \log (m \vee n), 
\text{ where } \\
\nonumber
&& C_{\overall}  := 16 C_{\alpha} C_{\gamma} b_{\infty} \tilde{\kappa}_{\rho}
\text{ for }  \; C_{\gamma} \ge c_{\gamma} \vee C_{\max}, \; \; C_{\alpha} := \lambda_{\min}(B_0)/{\alpha},
\een
and $C_{\max}$ and $c_{\gamma}$ are as defined in~\eqref{eq::BHatoffd} and~\eqref{eq::halflambda}, respectively.
Lemma~\ref{lemma::tidebound} is proved in supplementary
Section~\ref{sec::proofoftide},
which follows steps from Corollary 
5~\cite{LW12}. 
\begin{lemma}
  \label{lemma::tidebound}
  Suppose all conditions in Theorem~\ref{coro::thetaDet} hold. 
  Suppose~\eqref{eq::parityproof} holds.
Let $\alpha=\lambda_{\min}(B_0) /C_{\alpha}$, where $2> C_{\alpha} > 1$.
Then for each $j$, $(a) \abs{\tilde{\Theta}_{jj}} \le 2 \abs{\theta_{jj}}$;
\bens
(b) && \abs{\tilde{\Theta}_{jj} - \theta_{jj}} =\abs{\hat{a}_j - a_j}
\le  2 C_{\overall} \ul{r_{\offd}} \sqrt{d_0} \kappa_B
\theta^2_{\max},\; \; \text{ where }\; \; \theta_{\max} := \max_{jj} \theta_{jj};\\
(c)  && \onenorm{\tilde{\Theta}_{j \cdot} - \Theta_{j \cdot}} \le {2 C_{\overall} \ul{r_{\offd}} d_0 } \theta_{\max} (\kappa_B+3)/{\lambda_{\min}(B_0)}.
\eens
\end{lemma}

\begin{proofof}{Theorem~\ref{coro::thetaDet}}
Clearly, Condition~\eqref{eq::parityproof} holds under the assumption that 
$\ul{r_{\offd}} \sqrt{d_0} = o(1)$ as imposed 
in~\eqref{eq::paritydual} in Theorem~\ref{coro::thetaDet}; See~\eqref{eq::parity4}. 
  Following Lemma~\ref{lemma::tidebound}, 
\ben
\label{eq::tildeTheta}
\shnorm{\tilde{\Theta} - \Theta_{0}}_{\infty} & \le & 2 C_{\overall} \ul{r_{\offd}} d_0 \theta_{\max} (\kappa_B+3)
/{\lambda_{\min}(B_0)}\; \\
\nonumber
\; \text{ and hence } \; 
\shnorm{\hat{\Theta} - \Theta_0}_1 
& = & \shnorm{\hat{\Theta} - \Theta_0}_{\infty}  \le 
\nonumber
\shnorm{\hat{\Theta} - \tilde\Theta}_{\infty} +
\shnorm{\Theta_0 - \tilde\Theta}_{\infty}  \\
& \le & 
\nonumber
2 \shnorm{\Theta_0 - \tilde\Theta}_{\infty} =
2 \max_{j} \onenorm{\tilde{\Theta}_{j \cdot} - \Theta_{j \cdot}} 
\een
where the second inequality holds by optimality of $\hat{\Theta}$ in
minimizing $\shnorm{\Theta - \tilde\Theta}_{\infty}$ among all symmetric matrices
and the last  inequality holds by \eqref{eq::tildeTheta}.
\end{proofof}

}

\section{Conclusion}
\label{sec::conclude}
In this paper, we prove new concentration of measure bounds for quadratic
forms as defined in~\eqref{eq::quadorig} and~\eqref{eq::quadpreview}.
Convergence rate in the operator norm for estimating the inverse covariance is
also derived. 
These tools as developed in Section~\ref{sec::reduction}
are the key technical contributions of this paper for dealing
with the sparse quadratic forms such as~\eqref{eq::tensor}.
We mention in passing that the strategy we developed for dealing with 
$\offd(\X \X^T  -B_0 \circ \M)$ for~\eqref{eq::quadorig} will readily
apply when we deal with $\offd(\X^T \X - A_0 \circ \N)$ by symmetry of
the problem, while the bounds on diagonal components follow from the
theory in~\cite{Zhou19}.
One can similarly design an estimator for inverse covariance
$A^{-1}_0$ by adjusting the procedure as described in Definition~\ref{def::TopHat}, 
given the initial estimator $\hat{A}_{\star}$~\eqref{eq::Astar} as an
input. When $p_j = p, \forall j \in [n]$, the problems for estimating $B_0$
and $A_0$ are almost symmetrical, except that we need to make
suitable assumptions to ensure convergence of both inverse covariance
estimators. 
This approach has been thoroughly investigated
in~\cite{Zhou14a} in the context of complete data.
Moreover, it will be interesting to recover a low-rank mean
matrix with noise structure \eqref{eq::perturb} in the missing
value settings. For the complete matrix variate data,
this has been thoroughly explored in~\cite{Horns19},
where controlling the relative estimation error in the operator
norm for $B_0$ (covariance between and among the samples) is crucial for
establishing optimal bounds for mean estimation and for enabling subsequent
large-scale inference in the context of genomics studies. We leave this
investigation as future work.

\subsection*{Acknowledgement}
I would like to thank Tailen Hsing, Po-Ling Loh, and Mark Rudelson for helpful
discussions and my family for their support.
I thank Mark Rudelson for allowing me to present
Theorem~\ref{thm::XX^T} in this paper.
The author also thanks the Editor, the AE and two anonymous referees for
their valuable comments and suggestions.
Initial version of the manuscript entitled ``The Tensor Quadratic
Forms'' was posted as preprint arXiv:2008.03244 in August 2020.

\appendix

\section{Preliminary results}
\label{sec::appendprelim}
First, we need the following definitions and notations.
Let $\vecb^{(1)}, \ldots, \vecb^{(n)}$ denote the column (row) vectors of  
symmetric positive-definite matrix $B_0 \succ 0$.  
Let $u^s_1, \ldots, u^s_n, s=1, \ldots, m$ be independent random variables with two values
$0$ and $1$, and a polynomial $Y = \sum_{e \in \bE} w_e \prod_{(i,j) \in e}
u^i_j$, where $w_e$ is a weight which may have both positive and negative
coefficients, and $\bE$ is a collection of subsets of indices
$\{(i,j), i\in [m], j \in[n]\}$.
If the size of the largest subset in $\bE$ is $k$, $Y$ is called a polynomial of degree $k$. If
all coefficients $w_e$ are positive, then $Y$ is  called a positive polynomial of degree $k$.
A homogeneous polynomial is a polynomial whose nonzero terms all have
the same degree.
We use the following properties of the Hadamard product:
\bens
A \circ x x^T & = &D_x A D_x \;
\text { and } \; 
\tr(D_{x} A D_{x} A) =  x^T (A \circ A) x 
\eens
for $x \in \R^m$ and $A \in \R^{m \times m}$, from which a simple
consequence is $\tr(D_{x}  A D_{x}) = x^T(A \circ
I) x  =  x^T \diag(A) x$.
For matrix $A$, $r(A)$ denotes the effective rank
$\tr(A)/\twonorm{A}$.

We need the following shorthand notations.
\bnum
\item 
  Let $r_{\offd}(s_0)$ be as in~\eqref{eq::offdrate}:
 \bens
\label{eq::offdratesupp}
\quad
r_{\offd}(s_0) =
\sqrt{s_0 \log \big(\frac{3e n}{s_0\ve}\big) 
  \frac{\twonorm{A_0}}{\sum_{j} a_{jj}  p_j^2}} \; \text{ and }
 \ell_{s_0, n} = \frac{\log (n \vee m)}{\log (3en/(s_0 \ve) )};
\eens
\item
  Let $r_{\diag}$ be the as in Theorem~\ref{thm::diagmain}:
  \bens
  r_{\diag} = \eta_A \big({\twonorm{A_0} \log (m \vee n)
  }/{\norm{\M}_{\diag}}  \big)^{1/2};
  \eens
  \item 
 Let $\delta_q$  denote the rate in~\eqref{eq::rmn}:
 \bens
\delta_q = 
C_{\sparse} \Big(\eta_A \Big(s_0 \log (n \vee m)  \frac{\twonorm{A_0}}{\sum_{j} a_{jj}
  p_j^2}\Big)^{1/2} +r_{\offd}^2 (s_0) \psi_B(2s_0 \wedge n)\Big);
\eens
\item
  Let $\delta_q(B) =\Omega(r_{\offd}(n))$ be as in~\eqref{eq::Bhatop},
  for $r_{\offd}(n) \asymp \big(n \twonorm{A_0}/(\sum_{j} a_{jj}
  p_j^2)\big)^{1/2}$, where we show the rate of convergence on: $\twonorm{\tilde{B}_0 
    -B_0}/{\twonorm{B_0}}$,
  \bens
\delta_q(B) \asymp \eta_A r_{\offd}(n) +
 \Big(1 + \frac{a_{\infty} \psi_B(n)}{\twonorm{A_0}} \Big)
  r^2_{\offd}(n) + \eta^{1/2}_A \sqrt{p_{\max} } r^{3/2}_{\offd}(n) \psi_B(n)
\Big(\frac{\log (n \vee  m) }{n}\Big)^{1/4};
\eens
\item
  Let $\underline{r_{\offd}} := \eta_A  \big(\twonorm{A_0} \log  (m \vee n)/\big(\sum_{j} a_{jj} p_j^2)
  \big)^{1/2} = o(1)$ be as in~\eqref{eq::paritydual};
\item
Let $\delta_{m, \diag}=O(\ul{r_{\diag}}/\sqrt{r(B_0)})$ and  $\delta_{\mask}  =O(\ul{r_{\offd}}/\sqrt{r(B_0)})$ be as defined in Lemma~\ref{lemma::unbiasedmask}.
\enum

\subsection{Organization}
\bnum
\item
  We prove Lemma~\ref{lemma::converse} and
Theorem~\ref{thm::RE}   in Section~\ref{sec::lowerB0half};
\item
We prove Theorem~\ref{thm::AD} and Corollary~\ref{coro::tartan}  regarding the fully observed matrix variate data in Section~\ref{sec::proofofthmAD};
\item 
We prove Theorem~\ref{thm::main-coro} and 
Corollary~\ref{coro::Bhatnorm} in Section~\ref{sec::proofofBstarop};
\item 
We prove Lemma~\ref{lemma::unbiasedmask} in Section~\ref{sec::maskest}; 
\item 
  We prove Theorem~\ref{thm::gramsparse},
  Lemma~\ref{eq::avgrowsum}, and Corollary~\ref{coro::offdn}  in Section~\ref{sec::appendcoroproof};
\item 
We prove Theorem~\ref{thm::main} in  Section~\ref{sec::appendproofoffdmain};
\item
We prove key lemmas for Theorem~\ref{coro::thetaDet} in Section~\ref{sec::inverseLemma};
\item 
We prove Lemma~\ref{lemma::Gammabounds} in Section~\ref{sec::proofofGamma}; 
\item
We prove Theorem~\ref{thm::samplesize} in Section~\ref{sec::appendsamplesize};
\item
We prove Theorem~\ref{thm::mainlights} in  Section~\ref{sec::appendmainlights};
\item
We prove Theorem~\ref{thm::mainop} (Theorem~\ref{thm::mainop2}) in Section~\ref{sec::bernchaos};
\item
We prove Theorem~\ref{thm::uninorm-intro} (Theorem~\ref{thm::uninorm2}) in 
Sections~\ref{sec::uninormskye} and~\ref{sec::polyconc};
\item
We prove Theorem~\ref{thm::Bernmgf} in Section~\ref{sec::proofofbasemgf};
\item
We prove Lemma~\ref{lemma::pairwise} in 
Section~\ref{sec::proofofpairwise} for the purpose of self-containment.
\enum


\section{Proof of Lemma~\ref{lemma::converse} and Theorem~\ref{thm::RE}}
\label{sec::lowerB0half}
\silent{First, we need the following definition.
\begin{definition}
  \label{def::sparse-eigen}
For $1\le s_0 \leq  n$, we define the largest $s_0$-sparse eigenvalue
of an $n \times n$ matrix $B_0 \succ 0$ to be:  $\rho_{\max}(s_0, B_0) := 
\max_{v \in \Sp^{n-1}; s_0-\text{sparse}} \; \; v^T B_0 v$. As a consequence of the Rayleigh-Ritz theorem,
\ben 
\label{eq::eigen-Sigma}
\max_{j} b_{jj} =: b_{\infty} & \le &\rho_{\max}(s_0, B_0) \le 
\twonorm{B_0} \le \twonorm{ (\abs{b_{ij}})}
\een
\end{definition}}

\begin{proofof}{Lemma~\ref{lemma::converse}}
 Now since $\abs{B_0} \ge 0$, that is, all entries $\abs{b_{ij}}$ are either
 postive or zero, we have by \eqref{eq::eigen-Sigma},
\bens
\rho_{\max}(s_0, \abs{B_0})
\nonumber 
& :=& \max_{q \in \Sp^{n-1}, s_0-\sparse} \sum_{i=1}^n \sum_{j=1}^n \abs{q_i} \abs{q_j} \abs{b_{ij}} \\
&  = & \max_{S \subset [n]: \abs{S} = s_0}
\lambda_{\max}(\abs{B_0}_{S,S})
\le   \max_{S \subset [n]: \abs{S} = s_0} \norm{\abs{B_0}_{S,S}}_{\infty}  \\
 &  = &
 \max_{S \subset [n]: \abs{S} = s_0}
 \norm{B_{0, S,S}}_{\infty}  \le \sqrt{s_0} \twonorm{B_0}
 \eens
 \text{ and } \;by \eqref{eq::eigen-Sigma}
 \bens
\rho_{\max}(s_0, (\abs{b_{ij}}))
& \ge &
\nonumber
\max_{q \in \Sp^{n-1}, s_0-\sparse} \sum_{i=1}^n   b_{ij} q_i q_j =:
\rho_{\max}(s_0, B_0)   \ge b_{\infty}.
\eens
\end{proofof}

\subsection{Proof of Theorem~\ref{thm::RE}}
\label{sec::appendproofofRE}
First, we state Theorem~\ref{thm::REdet}, which follows from
Corollary 25~\cite{RZ17}, adapted to our settings and allows us to use Theorems~\ref{thm::diagmain} and~\ref{thm::main}
to prove the lower and upper-$\RE$-conditions.
Denote by $\kappa_B$ the condition number for $B_0$.

\begin{theorem}{\textnormal{\cite{RZ17}}}
\label{thm::REdet}
Suppose $1/8> \delta > 0$. Let $1 \le s_0 < n$. Let $B_0$ be a
symmetric positive definite covariance matrice such that $\tr(B_0)   =n$. Let $\hat{B}$ be an $n \times n$ symmetric matrix and $\hat\Delta = \hat{B} - B_0$.
Let $E = \cup_{\abs{J} \le s_0} E_J$, where $E_J = \Span\{e_j, j \in J\}$.
Suppose that for all $q, h \in E \cap \Sp^{n-1}$
\ben
\label{eq::REcond}
\abs{q^T  \hat{\Delta} h} \le \delta \le 
3 \lambda_{\min}(B_0)/32.
\een
Then the Lower and Upper $\RE$ conditions hold: for all $q \in \R^m$,
\ben
\label{eq::BDlowlocal}
q^T \hat{B} q & \ge & \frac{5}{8} \lambda_{\min}(B_0) 
 \twonorm{q}^2 -\frac{3 \lambda_{\min}(B_0)}{8 s_0} \onenorm{q}^2 \\
 \label{eq::BDuplocal}
 q^T \hat{B} q & \le & ( \lambda_{\max}(B_0) +\frac{3}{8} \lambda_{\min}(B_0)) 
 \twonorm{q}^2  +\frac{3\lambda_{\min}(B_0)}{8 s_0} \onenorm{q}^2.
 \een
\end{theorem}
\begin{proofof}{Theorem~\ref{thm::RE}}
  Let $E$ be as defined in Theorem~\ref{thm::REdet}.
  Let $d=2s_0 \wedge n$.
  Set $C_{\RE} = 2(C_{\diag} \vee C_{\sparse})$. 
 First note that
\bens 
r_{\diag}
& := & \frac{\eta_A \sqrt{ \twonorm{A_0} \log 
    (m \vee n) }}{\sqrt{\sum_{j=1}^m a_{jj}  p_j}} \le \ul{r_{\offd}}
\eens
Then on event $\F_{\diag}^c \cap \F_4^c$,
we have by Theorems~\ref{thm::diagmain} and~\ref{thm::main},
$r_{\diag} \le \ul{r_{\offd}}$,
\ben
\nonumber
\lefteqn{\forall q, h \in E \cap 
\Sp^{n-1} \; \;   \abs{q^T  (\tilde{B}_0 -B_0) h}  =  \abs{q^T ( (\X \X^T \oslash
  \M) -B_0 ) h}} \\
\label{eq::uniform3}
  & \le &   C_{\diag} b_{\infty}  r_{\diag} +   \delta_q
  \twonorm{B_0}  \le 3 \lambda_{\min}(B_0)/32
\een
where
$\delta_q$ is as defined in \eqref{eq::rmn} and $\ul{r_{\offd} }$ as
in~\eqref{eq::paritydual},
\ben
\nonumber
\delta_q & := &
C_{\sparse} \big(\ul{ r_{\offd}} \sqrt{s_0} +  r^2_{\offd}(s_0)
  \psi_B(d) \big) 
\een
Now by~\eqref{eq::forte}, we have for $\psi_B(2s_0 \wedge n) =:  \psi_B(d)$,
\bens
\frac{{\sum_{j} a_{jj} p_j^2} }{s_0 \eta_A^2 \twonorm{A_0} \log (m
  \vee n)}
&\ge&
\big(128 C_{\RE} \kappa_B \psi_B (d) \big) \vee
\big((128/11)^2 C_{\RE}^2 \kappa_B^2 \big) \text{ so that } \\
(C_{\diag}  + C_{\sparse})\ul{r_{\offd}}\sqrt{s_0}
& \le &
C_{\RE} \eta_A
\sqrt{\frac{\twonorm{A_0} s_0 \log  (m \vee n) }{\sum_{j} a_{jj}
    p_j^2}} \le {11}/{(128\kappa_B)}\;
\; \text{ and } \\ 
C_{\sparse} r^2_{\offd}(s_0) \psi_B(d)
& =& 
C_{\sparse} \frac{\twonorm{A_0}s_0 \psi_B(d) } {\sum_{j} a_{jj} p_j^2} 
\log\big(\frac{3en}{s_0 \ve}\big) \\
& \le& 
C_{\RE} \frac{\twonorm{A_0} s_0 \psi_B(d)}{\sum_{j} a_{jj} p_j^2}\log
(m \vee n) \le   1/{(128 \kappa_B)}
\eens
\text{ where }
$C_{\RE} \log (m \vee n) \ge  2 C_{\sparse} \log (m \vee n) \ge
C_{\sparse} \log\big(\frac{3en}{s_0 \ve}\big)$, which holds so long as $ (m \vee n)^2 \ge 3en/(\ve s_0)$, or simply $n
\ge 3e/(2\ve)$.
Theorem~\ref{thm::RE} follows from Theorem~\ref{thm::REdet} in view of~\eqref{eq::uniform3} immediately above. 
\end{proofof}

\section{Proof of Theorem~\ref{thm::AD}}
\label{sec::proofofthmAD}
Let $Z \in \R^{n \times m}$ be as in~\eqref{eq::missingdata}.
Denote by $r(A_0) = \tr(A_0)/\twonorm{A_0}$.
Fix $s_0 \in [n]$ and $0< \ve < 1/2$.
For a set $J \subset \{1, \ldots, n\}$,  denote by $E_J = \spin\{e_j:
j \in J\}$,
and set  $F_J=B^{1/2} E_J$.
Let $F =  B^{1/2} E$, where 
\bens 
\label{eq::spsetEF}
E = \bigcup_{\abs{J} =s_0} E_J 
\; \; \text{ and } \; \; F = \bigcup_{\abs{J} =s_0} F_J. 
\eens
For each subset $F_J \cap   \Sp^{n-1}$, construct  an $\ve$-net $\Pi_J$, which satisfies
\ben
\nonumber
\Pi_J \subset F_J \cap  \Sp^{n-1}  && \text{ and } \;\; \size{\Pi_{J}
} \le (1 + 2/\ve)^{s_0} \\
\label{eq::fishnet}
\text{ Set } \;  \Pi_F &= &\bigcup_{\abs{J} = s_0} \Pi_{J}.
\een
Then $\Pi_F \subset F \cap \Sp^{n-1}$ is an $\ve$-net of $F$; see for
example the proof of Theorem 17~\citep{RZ13}.

To prove Theorem~\ref{thm::AD},
we first state Lemmas~\ref{lemma::oneeventA} and~\ref{lemma::normA}.
Lemma~\ref{lemma::oneeventA} follows from Lemma 32~\cite{RZ17}, where
we set $K = 1$, which in turn follows from~Theorem~\ref{thm::HW}.
\begin{lemma}
\label{lemma::oneeventA}
Let $u, w \in \Sp^{n-1}$. 
Let $A \succ 0$ be an $m \times m$ symmetric positive definite matrix.
Let $Z$ be an $n \times m$ random matrix with independent entries $Z_{ij}$ satisfying
$\E Z_{ij} = 0$ and  $\norm{Z_{ij}}_{\psi_2} \leq 1$.
Then for every $t > 0$,
\bens
\prob{\abs{u^T Z A Z^T w - \E u^T Z A Z^T w } >  t}
& \le & 
2 \exp\left(-c\min\left(\frac{t^2}{\fnorm{A}^2}, \frac{t}{\twonorm{A}}\right)\right),
\eens
where $c$ is the same constant as defined in Theorem~\ref{thm::HW}.
\end{lemma}

We prove Lemma~\ref{lemma::normA} in Section~\ref{sec::proofofnormA}.
\begin{lemma}
\label{lemma::normA}
Suppose all conditions in Theorem~\ref{thm::AD} hold.
Let $c' > 1$ be the same as in~\eqref{eq::trB}.
Let $C = C_0/\sqrt{c'}$, where $C_0$ is an absolute constant chosen to satisfy $\min\{C_0, 
C_0^2\} \ge 4/c$ for $c$ as defined in Theorem~\ref{thm::HW}. 
\silent{Suppose that for some $c' > 0$ and $0< \ve \le \inv{C}$, where $C = C_0/\sqrt{c'}$}
Then on event $\G_2$, where $\prob{\G_2} \ge 1- 2 \exp\left(-c_2\ve^2 r(A_0)\right)$ for $c_2 \ge 2$, 
we have
\ben
\label{eq::RZM2}
\max_{y, w \in F \cap  \Sp^{n-1}} \abs{w^T ( Z A_0 Z^T - \E (Z A_0 Z^T)) y}
& \le & 4 C \ve \tr(A_0).
\een
\end{lemma}
We are now ready to prove Theorem~\ref{thm::AD}.

\subsection{Proof of Theorem~\ref{thm::AD}}
\begin{proofof2}
  For $u  \in E \cap \Sp^{n-1}$, denote by $h(u) := {B^{1/2}  u}/{\twonorm{B^{1/2} u}}$. 
Then clearly $h(u) \in \Sp^{n-1} \cap 
B^{1/2} E =: \Sp^{n-1} \cap  F$. 
Let 
\bens
\Delta :=  \inv{\tr(A_0)} X X^T - B_0  =: \hat{B} - B_0
& = &  \inv{\tr(A_0)}  [B_0^{1/2} Z A_0 Z^T B_0^{1/2}- B_0 \tr(A_0)]
\eens 
Now we have by definition of $\Delta$ and $X =  B_0^{1/2} Z A_0^{1/2}$,
\ben 
\nonumber
u^T \Delta v & := & u^T(\inv{\tr(A_0)} X X^T - B_0) v \\
& := &
\label{eq::defineAD}
\inv{\tr(A_0)}  \big(u^T B_0^{1/2} (Z A_0 Z^T) B_0^{1/2} v - \tr(A_0)
u^T B_0 v \big) =: Q_A(u, v)
\een
Now, we provide a uniform bound for all
$u, v \in E \cap \Sp^{n-1}$.
We have for $C$ as defined in Lemma~\ref{lemma::normA}, and
$\rho_{\max}(s_0, B_0)$ as in Definition~\ref{def::sparse-eigen},
on event $\G_2$, for all $u, v \in E \cap \Sp^{n-1}$,
\ben
\abs{Q_A(u, v)}
& := &
\nonumber
  \inv{\tr(A_0)}
\abs{h(u)^T [Z A_0 Z^T - I_n \tr(A_0)] h(v)} \shnorm{B_0^{1/2} v}_2
\shnorm{B_0^{1/2} u}_2 \\
& \le &
\nonumber
\sup_{w, y \in F \cap \Sp^{n-1}} \inv{\tr(A_0)} \abs{w^T [Z A_0 Z^T  -\E
  (Z A_0 Z^T)] y} \rho_{\max}(s_0, B_0) \\
\label{eq::uniform5}
& \le & 4 C \ve \rho_{\max}(s_0 ,B_0) 
\een
Thus so long as $\ve$ satisfies the condition in \eqref{eq::trB}, we
have by~\eqref{eq::defineAD} and~\eqref{eq::uniform5}, on event $\G_2$, 
\bens
\sup_{u, v \in E \cap \Sp^{n-1}} \abs{u^T \Delta v}
& = & \sup_{u, v \in E \cap \Sp^{n-1}}\abs{Q_A(u, v)} \le  4 C \ve \rho_{\max}(s_0, B_0) \\
& = & 
\frac{3 \lambda_{\min}(B_0)}{32\rho_{\max}(s_0, B_0)} \rho_{\max}(s_0, B_0) 
\le 3 \lambda_{\min}(B_0)/32 
\eens
Theorem~\ref{thm::RE} follows from Theorem~\ref{thm::REdet} in view
of the last inequality immediately above.
\end{proofof2}

\subsection{Proof of Corollary~\ref{coro::tartan}}
\label{sec::proofoftartan}
\begin{proofof2}
Choose an $\ve$-net $\N \subset \Sp^{n-1}$ such that $\abs{\N} \le (3/\ve)^{n}$.
First we prove concentration bounds for $u  \in \N$.

Let $t  = C \ve \tr(A_0)$.
We have by Lemma~\ref{lemma::oneeventA},
and the union bound,
\bens
\lefteqn{\prob{\exists u \in \N, \; 
\abs{u^T Z A_0 Z^T u - \E u^T Z A_0 Z^T u} >  C \ve \tr(A_0)} =: \prob{\G_3^c}} \\
& \leq &
2 \abs{\N} \exp \left[- c\min\left(\frac{t^2}{\fnorm{A_0}^2}, 
    \frac{t}{\twonorm{A_0}} \right)\right] \\
& \leq & 2  (3/\ve)^n \exp \left[- c\min\left(\frac{C^2 \ve^2
      \tr(A_0)^2}{\twonorm{A_0}\tr(A_0)},
    \frac{C \ve \tr(A_0)}{\twonorm{A_0}} \right)\right] \\
& \leq &
2  \exp(n \log(3/\ve)) \exp \left[- c\min\left(C^2, \frac{C}{\ve} \right) \ve^2
    r(A) \right]  \le  2 \exp\left(-c_3 \ve^2 r(A_0) \right),
\eens
where we use the fact that $\fnorm{A_0}^2 \le \twonorm{A_0}\tr(A_0)$, 
\bens
\abs{\N} \le (3/\ve)^n, &&  r(A_0) \ve^2 \ge c' n
\log(3/\ve) \; \text{ and for } \;  \ve \le \inv{C} \\
c \min\left(C^2, \frac{C}{\ve} \right) \ve^2 r(A_0) & \ge &
c c' C^2 n \log(3e /\ve) \ge 4 n \log(\frac{3e}{\ve}),
\eens
where $C$ is an absolute constant chosen so that $ c c' C^2 \ge 4$.
A standard approximation argument shows that if
$\max_{w \in \N} \abs{w^T \inv{\tr(A_0)}(Z A_0 Z^T- \E (Z A_0 Z^T)) w} 
\le C \ve$ for $\ve \le 1/2$, then
\ben 
\label{eq::sphereL}
\sup_{w \in \Sp^{n-1}} \abs{w^T \inv{\tr(A_0)}(Z A_0 Z^T- \E (Z A_0 Z^T)) w} 
\le \frac{ C \ve}{(1-\ve)^2} \le 4 C  \ve. 
\een
Denote by 
\bens 
\abs{Q_A(u)} 
& =: & \abs{u^T  (\inv{\tr(A_0)}  X X^T - B_0) u}
\eens 
Thus we have on event $\G_3$, for all $u \in \Sp^{n-1}$, and $h(u) :=
{B_0^{1/2}  u}/{\shnorm{B_0^{1/2} u}_2} \in \Sp^{n-1}$
\bens
\abs{Q_A(u)}
 & =: &
\inv{\tr(A_0)} \abs{h(u)^T [Z A_0 Z^T - I_n \tr(A_0)] h(u)} \twonorm{B_0^{1/2} u}^2\\
& \le &
\sup_{w \in \Sp^{n-1}} \inv{\tr(A_0)} \abs{w^T [Z A_0 Z^T  -\E (Z A_0 Z^T)] w} \lambda_{\max}(B_0) \\
& \le & 4 C \ve \lambda_{\max}(B_0) 
\eens
\silent{Thus so long as $ \ve>0$ satisfies the condition that
\bens
\ve \le \frac{3 \lambda_{\min}(B_0)}{128C  \lambda_{\max}(B_0) } = \frac{3}{128C \kappa(B_0)}
\eens
where $\kappa(B_0) =  \lambda_{\max}(B_0)/\lambda_{\min}(B_0)$ denotes
the condition number of $B_0$,}
The corollary thus holds, since on event $\G_3$, the relative error
for covariance estimation:
\bens
\twonorm{X X^T/\tr(A_0)  - B_0}/\twonorm{B_0}
& =&
\sup_{u \in \Sp^{n-1}} \abs{u^T (X X^T/\tr(A_0)  - B_0) 
  u}/\twonorm{B_0} \\
& = &
\sup_{u \in \Sp^{n-1}} \abs{Q_A(u)} /\twonorm{B_0} 
\le 4 C \ve
  \eens
\end{proofof2}

\subsection{Proof of  Lemma~\ref{lemma::normA}}
\label{sec::proofofnormA}
The proof of Lemma~\ref{lemma::normA} essentially follows the same
argument as that of Lemma 40 in~\cite{RZ17}, where we construct an $\ve$-net
for $E$ rather than $F$.
In fact, as shown in~\cite{RZ17}, the same conclusion holds for all $u, v \in E \cap \Sp^{n-1}$;
On event $\G_1$,  where $\prob{\G_1} \ge 1- 2 \exp(-c_2\ve^2 r(A_0))$ for $c_2 \ge 2$, we have
\ben
\label{eq::RZM1}
\max_{u, v \in E \cap \Sp^{n-1}} \abs{u^T (Z A_0 Z^T - \E (Z A_0 Z^T) ) v } & \le & 4 C \ve \tr(A_0). 
\een
To prove uniform concentration of measure bounds for the quadratic form 
for all pairs of $u, v \in \Pi_F$,  where $\Pi_F \subset \Sp^{n-1}$ is an $\ve$-net of $F$ as constructed 
in \eqref{eq::fishnet} with size satisfying~\eqref{eq::Fnet}:
\ben
\label{eq::Fnet}
\abs{\Pi_F} \le {n \choose {s_0}}(3/\ve)^{s_0} \le \exp({s_0} \log(\frac{3 e
n}{{s_0} \ve})).
\een
Then $\Pi_F$ is an $\ve$-net for $F \cap \Sp^{n-1}$; See~\cite{RZ13},
proof of Theorem 17.

\begin{proofof}{Lemma~\ref{lemma::normA}}
Let $C = C_0/\sqrt{c'}$, where $C_0$ is an absolute constant chosen to satisfy $\min\{C_0, 
C_0^2\} \ge 4/c$, where $c$ is the same as in Lemma~\ref{lemma::oneeventA}.
Let $t  = C \ve \tr(A_0)$.
We have by Lemma~\ref{lemma::oneeventA}, and the union bound,
\ben
\nonumber
\lefteqn{
\prob{\exists u, v \in \Pi_F, \; 
\abs{u^T Z A_0 Z^T v - \E u^T Z A_0 Z^T v} >  C \ve \tr(A_0)}} \\
& \leq &
\nonumber
2 \abs{\Pi_F}^2 \exp \left[- c\min\left(\frac{t^2}{\fnorm{A_0}^2}, 
    \frac{t}{\twonorm{A_0}} \right)\right] \\
& \leq &
\label{eq::last}
2 \abs{\Pi_F}^2 \exp \left[- c\min\left(C^2, \frac{C}{\ve} \right) \ve^2
    r(A_0) \right]  \le  2 \exp\left(-c_2\ve^2 r(A_0) \right),
\een
where we use the fact that $\fnorm{A_0}^2 \le \twonorm{A_0}\tr(A_0)$;
In the last inequality~\eqref{eq::last}, we use the fact
\eqref{eq::Fnet}, and consider two cases:
\bit
\item
Suppose $0< \ve \le \inv{C}$,
\bens 
c \min\left(C^2, \frac{C}{\ve} \right) \ve^2 r(A_0) & = & c C^2 \ve^2 r(A_0) 
\ge  c c' C^2 s_0 \log(3e n/(s_0 \ve) ) \ge 4 s_0 \log\big(\frac{3e n}{s_0 \ve}\big), 
\eens
where $c c' C^2 = c C_0^2 \ge 4$ we have by \eqref{eq::trB},
\ben 
 \label{eq::ALocalkronsum}
\ve^2 r(A_0) & \ge & c' s_0 \log(3e n/(s_0  \ve)) \; \; \text{ where }
\; c' > 1
\een
\item
  Suppose $1>\ve > \inv{C}$ and hence $\frac{C}{\ve} \le C^2$.
Then we have by \eqref{eq::ALocalkronsum} and  $C = C_0/\sqrt{c'}$, 
\bens 
c \min\left(C^2, \frac{C}{\ve} \right) \ve^2 r(A_0) & = &
c c' \frac{C}{\ve}
s_0 \log(\frac{3e n}{(s_0  \ve)}) \ge  4 s_0
\log\big(\frac{3e n}{s_0 \ve}\big), \\
\text{ since } \;\; c c' \frac{C}{\ve} & = & \sqrt{c'} \frac{c C_0}{\ve} \ge  \sqrt{c'} \frac{4}{\ve} 
\eens
where $C_0$ is an absolute constant chosen so that $ c (C_0^2 \wedge C_0) \ge 4$.
\eit
Denote by $\G_2$ the event such that 
\ben
\label{eq::A2}
\max_{u, v \in \Pi_F} \abs{v^T \Lambda u} & \le & C  \ve =: 
\delta_{s_0,n}\; \text{ for } \; \Lambda := \inv{\tr(A_0)} (Z A_0 
Z^T - I_n \tr(A_0));
\een
A standard approximation argument shows that  \eqref{eq::A2} implies
that for $\ve \le 1/2$, 
\ben 
\label{eq::sphereL}
\sup_{w, y\in \Sp^{n-1} \cap F} \abs{y^T \Lambda  w} 
\le \frac{\delta_{s_0,n}}{(1-\ve)^2} \le 4 C  \ve
\een
on event $\G_2$. The lemma is thus proved.
\end{proofof}

\silent{

\subsection{Proof of Corollary~\ref{coro::tartan}}
\label{sec::proofoftartan}

\begin{proofof}{Corollary~\ref{coro::tartan}}
\label{coro::Bfull} 
Clearly~\eqref{eq::wyFnorm} implies that \eqref{eq::ALocalkronsum}
holds for $B=I$.
Clearly~\eqref{eq::BI}  holds following the analysis of 
Lemma~\ref{lemma::normA} by setting $B = I$, while replacing event $\B_1$ with
$\B_3$, which denotes an event such that 
\bens
\sup_{u, v \in \Pi} \onen \abs{v^T (Z^T Z - I) u} 
& \le &  C \ve.
\eens
The rest of the proof follows by replacing $E$ with $F$ everywhere.
The corollary thus holds.
\end{proofof}}

\silent{
  we first state Theorem~\ref{thm::REadapt},
which follows from Corollary 25~\cite{RZ17}, adapted to our settings and allows us to use Theorems~\ref{thm::diagmain} and~\ref{thm::main}
to prove the lower and upper-$\RE$-conditions.
Denote by $\kappa_B$ the condition number for $B_0$.
\begin{theorem}{\textnormal{\cite{RZ17}}}
  \label{thm::REadapt}
Let $E=\cup_{|J| \leq k} E_J$ for $1 \le k < m/2$. 
Suppose $1/8> \delta > 0$. Let $1 \le s_0 < n$.
Suppose that for all $q, h \in E \cap \Sp^{n-1}$
\ben
\label{eq::REcond}
\abs{q^T  \Delta h} \le \delta \le 3 \lambda_{\min}(B_0)/32.
\een
Then \eqref{eq::BDloworig} and \eqref{eq::BDuporig} hold for $\hat{B} =\inv{\tr(A_0)} X X^T$.
\end{theorem}}

\section{Error bounds for $\hat{B}^{\star}$} 
\label{sec::Bstarthm}
We summarize the rates we use in Theorem~\ref{thm::main-coro} and
their relations. Recall by~\eqref{eq::starscale} and
Definition~\ref{def::RMSet}, $B_{\star} = n B_0/\tr(B_0) = B_0 \succ 0$. 
Denote by $r(B_0) ={\tr(B_0)}/{\twonorm{B_0}}$ the effective rank of 
$B_0$. 

\subsection{Elaborations on estimators for $\mathcal{M}$}
\label{sec::maskMN}
Without knowing the parameters, we need to estimate $\M$~\eqref{eq::Bmask} and $\N$ as
used in~\eqref{eq::baseN} and~\eqref{eq::baseNBer}.
Recall $v^i \sim {\bf v}, i =1, \ldots, n$ are independent, and 
\ben
\label{eq::baseM}
&& \expct{v^i \otimes v^i} = 
 \left[
\begin{array}{ccccc} 
\z_1 & \z_1 \z_2 & \z_1 \z_3  &  \ldots &  \z_1 \z_m  \\
\z_2 \z_1 & \z_2 &  \z_2 \z_3 & \ldots &  \z_2 \z_m  \\
\ldots & \ldots & \ldots & \ldots & \ldots\\
 \z_m \z_1 & \z_m \z_2 & \z_m \z_3 & \ldots &  \z_m
\end{array}\right]_{m \times m} =: M 
\een
where expectation denotes the componentwise expectation of each entry  
of $v^i \otimes v^i$, for all $i$.
Clearly we observe $v^i \in \{0, 1\}^m$, the vector of indicator variables for
nonzero entries in the $i^{th}$ row of data matrix $\X$, for all
$i=1, \ldots, n$.
That is,
\bens
v^i_j =1 && \text{ if } \; \; \X_{ij} \not=0 \text{ and } \; \;v^i_j = 0 \quad \text{ if } \; \; \X_{ij} =0
\eens
Hence we observe $M^i :=  v^i \otimes v^i, i \in [n]$, upon which we define an estimator
\ben
\label{eq::hatM}
\hat{M} & = & \onen \sum_{i=1}^n v^i \otimes v^i =  \onen \sum_{i=1}^n M^i
\een
Clearly $\hat{M}$ as defined in \eqref{eq::hatM} is unbiased since $\E
\hat{M} =\E M^i = M, \forall i$, for $M$ as defined in~\eqref{eq::baseM}.
To justify~\eqref{eq::baseN} and~\eqref{eq::baseNBer}, we have by
independence of the mask $U$
and data matrix $X$ as defined in~\eqref{eq::missingdata},
\ben
\label{eq::gramA}
&& \calX^T \calX =(U \circ X)^T (U \circ   X) =  \sum_{i=1}^n { (v^i
  \otimes v^i) \circ (y^i \otimes y^i)} \\
\label{eq::gramAexp}
&& \E \onen \calX^T \calX = \onen \sum_{i=1}^n {\E (v^i \otimes v^i) \circ \E
  (y^i \otimes y^i)} = (\tr(B_0)/n) (A_0 \circ M)
\een
where $M$ is as defined  in~\eqref{eq::baseM} and hence
\ben
\label{eq::diagmask}
 \E \hat\TM_{jj} & = &  \E \tr(\X^T \X)/n = \frac{\tr(B_0)}{n}
 \sum_j a_{jj} \z_j \; \text{ for }\; \hat\TM \text{ as in }~\eqref{eq::Bstar},
\een
where expectation in~\eqref{eq::gramAexp} denotes the componentwise expectation of each entry 
of $\calX^T \calX$.
Clearly, $\diag(\hat\TM) =\onen \tr(\X^T \X) I_n$ provides a componentwise
unbiased estimator for $\diag(\M)$ as defined
in~\eqref{eq::entrymean}.
Denote by
\ben
\label{eq::defineSc}
S_c :=  \frac{n}{n-1} \tr(\X^T \X \circ \hat{M}) - \inv{n-1}\tr(\X^T 
\X).
\een
Notice that for $M^i = v^i \otimes v^i, i=1, \ldots, n$, we have by
\eqref{eq::gramA} and~\eqref{eq::Astar}, or equivalently, $\hat{M}$ as defined immediately above in \eqref{eq::hatM}, 
\bens
\nonumber
\lefteqn{\frac{n}{n-1}
\tr(\X^T \X \circ \hat{M})
=   \inv{n-1}\tr\big(\X^T \X \circ (\sum_{j=1}^n v^j \otimes v^j)\big) }\\
\nonumber
& = &
\nonumber
\inv{n-1}\sum_{k=1}^n  \tr\big(\big(\sum_{j=1}^n {(v^j \otimes v^j)
  \circ (y^j \otimes y^j)}\big)  \circ (v^k \otimes v^k)\big) \\
& = & 
\nonumber
\inv{n-1} \sum_{j=1}^n  \tr\big( {(v^j \otimes v^j) \circ (y^j \otimes
    y^j)} \big) + \inv{n-1}\sum_{j=1}^n
\sum_{k\not=j}^n  \tr\big(M^k \circ (y^j \otimes y^j ) \circ M^j\big) \\
& = &
\inv{n-1} \tr(\X^T \X) + 
\inv{n-1} \sum_{j=1}^n \sum_{k\not=j} \tr\big(M^k \circ M^j \circ (y^j \otimes y^j) \big)
\eens
Rearranging we obtain for $\hat{\TM}$ as defined in~\eqref{eq::Bstar},
\bens
\forall k \not= \ell,
 \; \hat{ \TM}_{k \ell}  & = & \frac{S_c}{n} =
\inv{n(n-1)} \sum_{j=1}^n \sum_{k\not=j}^n (y^j)^T  \diag(M^k \circ
M^j) {y^j}
\eens
We will show that $\hat{\TM}$ is a componentwise unbiased 
estimator for $\M$ as defined in \eqref{eq::entrymean} in case 
$\tr(B_0) =n$;
Moreover, we have
\ben
\label{eq::delmask}
& &\text{ on event } \F_5^c, \; \; \delta_{\mask}
:= {\abs{\shnorm{\M}_{\offd} -\shnorm{\hat{\TM}}_{\offd}}}/{\shnorm{\M}_{\offd}}
=  O\big({\underline{r_{\offd}} }/{\sqrt{r(B_0)}}\big), \\
\label{eq::maskorder}
&& \text{ on event } \EE_8^c, \; \;
\delta_{m,\diag} :=
{\abs{\shnorm{\M}_{\diag}    -\shnorm{\hat{\TM}}_{\diag}}}/{\shnorm{\M}_{\diag}}
=  O\big( {r_{\diag}}/{ \sqrt{r(B_0)}}\big),
\een
where $r(B_0)={\tr(B_0)}/{\twonorm{B_0}} = \Omega(n)$ for $\tr(B_0) =n$.
Finally, $\prob{\EE_8^c \cap \F_5^c} \ge 1 - {C'}/{(n \vee m)^4}$.
In Section~\ref{sec::TMunbiased},  we elaborate on the plug-in estimator $\hat{\TM}$
and define events $\F_5^c$ and $\EE_8^c$ explicitly.
We also prove concentration of measure bounds in Section~\ref{sec::TMunbiased}.

\subsection{Proof of Theorem~\ref{thm::main-coro}}
\label{sec::proofofBstarop}
\begin{proofof2}
  Recall $\E (\X \X^T)= B_0 \circ \M$, where expectation is taken
  componentwise.
  Let $\EE_8^c$ and $\F_5^c$ be as in \eqref{eq::delmask}
  and \eqref{eq::maskorder}; cf. proof of Lemma~\ref{lemma::unbiasedmask}.
Let events $\F_{\diag}^c$ and $\F_4^c$ be the same 
as in Theorems~\ref{thm::diagmain} and~\ref{thm::main} respectively.

 \noindent{\bf (a) Diagonal bounds.}
 First, on event $\EE_8^c \cap \F_{\diag}^c$,
 we have  by Theorem~\ref{thm::diagmain} and
 Lemma~\ref{lemma::unbiasedmask}, Part (a) holds since  for $\delta_{m, \diag} = O(r_{\diag}/\sqrt{r(B_0)})$,
\ben 
\nonumber
   \lefteqn{\sup_{q, h \in \Sp^{n-1}}{\abs{q^T  \diag(\hat{B}^{\star} - B_0) h}
 =\twonorm{\diag\big( \X \X^T \oslash \hat{\TM} -\E (\X \X^T
     \oslash \M) \big)}}} \\
\nonumber
 & \le &
 \frac{\shnorm{\M}_{\diag}}{\shnorm{\hat{\TM}}_{\diag}}
 \frac{\norm{\diag\big( \X \X^T - (B_0 \circ \M) \big)}_{2}}{
   \shnorm{\M}_{\diag}} + 
\abs{ \frac{\shnorm{\M}_{\diag}}{\shnorm{\hat{\TM}}_{\diag}}-1} \frac{
  \twonorm{\diag(B_0 \circ  \M)}}{\shnorm{\M}_{\diag}} \\
\label{eq::Wdiag}
&& \quad \le \frac{C_{\diag}}{1-\delta_{m,\diag}}\big(b_{\infty} r_{\diag} +
  \delta_{m,\diag} b_{\infty}\big) 
  =:  C'_{\diag} b_{\infty} r_{\diag} =o(\d_q  \twonorm{B_0});
  \een
\noindent{\bf (b) Off-diagonal bounds.}
The component-wise bounds follow from Lemma~\ref{lemma::pairwise}. \\
Moreover,
we have on event $\F_{6}^c \cap \F_5^c$,  where 
$\F_{6}^c$ is the same as defined in Lemma~\ref{lemma::pairwise},
\bens
\lefteqn{\abs{b_{ij} - e_i^T \X_{\minus j} (X^{j} \circ v^j
    )/{\shnorm{\hat{\TM}}_{\offd}}} =: \abs{\hat{S}_1(e_i,
    \beta^{j*})}
\le \abs{1- {\norm{\M}_{\offd}}/{\shnorm{\hat{\TM}}_{\offd}}} \abs{b_{ij} }}\\
 && +
\abs{{\ip{X^i \circ v^i, X^{j} \circ v^j}}/\big(\sum_{i=1}^m a_{jj}
    p_i^2\big) - b_{ij}} {\norm{\M}_{\offd}}/{\shnorm{\hat{\TM}}_{\offd}} \\
&\le&
\frac{\delta_{\mask}}
{1-\delta_{\mask}} \abs{b_{ij}}  +\frac{C_{\offd}  \sqrt{b_{ii} b_{jj}}
  \ul{r_{\offd}}}{1-\delta_{\mask}}   \le
C_{\offd} \sqrt{b_{ii} b_{jj}} \ul{r_{\offd}} (1+ o(1))
 \eens
 where $\abs{b_{ij}} \le \sqrt{b_{ii} b_{jj}}$ and
 $\delta_{\mask} = \ul{r_{\offd}}/\sqrt{r(B_0)}$. \\
More generally, to bound the off-diagonal component, we have or all $q, h \in
\Sp^{n-1}$,
\ben
\nonumber
\lefteqn{\abs{q^T  \offd(\hat{B}^{\star} - B_0) h}=\abs{q^T \offd\big(
    \X \X^T \oslash \hat{\TM} - (B_0 \circ \M)      \oslash \M \big)
    h}} \\
\nonumber
  & \le &
  \abs{q^T \offd\big( \X \X^T \oslash \hat{\TM} - (B_0 \circ \M) 
      \oslash \hat{\TM} \big) h} + \\
    &  &
\label{eq::defineWlocal}
 \abs{q^T \offd\big( (B_0 \circ \M) \oslash \hat{\TM} - (B_0 \circ 
   \M) \oslash \M \big) h} =: W_1 + W_2
 \een
Now on event  $\F_5^c$, $\abs{\shnorm{\hat\TM}_{\offd} - \shnorm{\M}_{\offd} }
:=\delta_{\mask}   \shnorm{\M}_{\offd}$ and hence
\ben
 \nonumber
 \lefteqn{ \sup_{q, h \in  \Sp^{n-1}} W_2   = \sup_{q, h \in  
\Sp^{n-1}} \abs{\frac{\shnorm{\M}_{\offd}}{\shnorm{\hat{\TM}}_{\offd}} -1}
   \abs{q^T \offd((B_0 \circ  \M) \oslash \M ) h} } \\
 & \le &
  \label{eq::W2univadd}
\frac{\shnorm{\M}_{\offd} }{\shnorm{\hat{\TM}}_{\offd}}\delta_{\mask}  \twonorm{\offd(B_0)} \le
{\delta_{\mask}\twonorm{\offd(B_0)}}/{(1-\delta_{\mask})}  \\
\label{eq::maskbounds}
&& \text{where}\; {\shnorm{\M}_{\offd}}/{ \shnorm{\hat\TM}_{\offd} } \le   1/{(1-
  \delta_{\mask})}.
\een

 \noindent{\bf Overall bounds.}
For the rest of the proof, suppose event $\F_{20}^c := \F_5^c  \cap
\F_4^c \cap \EE_8^c \cap \F_{\diag}^c$ holds.
Let $E$ be as defined Lemma~\ref{lemma::quadreduction}. 
To ensure the lower and upper $\RE$
conditions hold for $\hat{B} :=  \hat{B}^{\star}$ in the sense of~\eqref{eq::BDlowlocal} and 
\eqref{eq::BDuplocal}, it is sufficient to show 
that~\eqref{eq::REcond}  holds for $\hat{\Delta}:= \hat{B}^{\star} -
B_0$ on event $\F_{20}^c$;
By Lemma~\ref{lemma::unbiasedmask},
Theorem~\ref{thm::main} and $\delta_q$ as defined therein, we have
\ben
\nonumber
\lefteqn{\quad
 \sup_{q, h \in E \cap \Sp^{n-1}}  W_1 \le  \sup_{q, h \in E \cap \Sp^{n-1}}
 {\abs{q^T \offd\big( \X \X^T - (B_0 \circ \M) \big) 
    h}}/{\shnorm{\hat{\TM}}_{\offd}} =}\\
\label{eq::sparseW1}
& &
\sup_{q, h \in E \cap \Sp^{n-1}}
\frac{\shnorm{\M}_{\offd}}{\shnorm{\hat{\TM}}_{\offd}}
\frac{\abs{q^T \offd\big( \X \X^T - (B_0 \circ \M) \big) 
    h}}{\shnorm{\M}_{\offd}} \le \frac{\delta_q \twonorm{B_0}}{1-\delta_{\mask}};
\een
Combining the error bounds on $W_1$, $W_2$,~\eqref{eq::Wdiag}, and the
effective sample size lower bound~\eqref{eq::forte},
we have for $q, h \in E \cap \Sp^{n-1}$,
\ben
\nonumber
\lefteqn{\abs{q^T (\hat{B}^{\star} - B_0) h}  \le
\abs{q^T   \diag(\hat{B}^{\star} - B_0) h} + \abs{q^T
  \offd(\hat{B}^{\star} - B_0) h} \le }\\
&  &
\label{eq::numeric}
\quad (1+o(1)) \big(C_{\diag} b_{\infty}  r_{\diag}  + \delta_{\mask}
\twonorm{\offd(B_0)} +  \delta_q \twonorm{B_0} \big) \le \frac{3}{32} \lambda_{\min}(B_0)
  \een
  following the proof of Theorem~\ref{thm::RE}, where $\delta_{\mask} = o(\delta_q)$, while  adjusting $C_{\RE} = 2(C_{\diag} \vee C_{\sparse})(1+o(1))$;
  Thus~\eqref{eq::REcond}  holds.
Hence the $\RE$ condition holds for $\hat{B}^{\star}$, following an
identical line of arguments as in the proof of Theorem~\ref{thm::RE}. 
Now, combining \eqref{eq::Wdiag},~\eqref{eq::defineWlocal},~\eqref{eq::W2univadd},
and~\eqref{eq::nearsparseW1}, by the triangle inequality, for all $q
\in (\sqrt{s_0} B_1^n \cap B_2^n)$, on event $\F_{20}^c$,
\bens
\abs{q^T (\hat{B}^{\star} - B_0) q} &  \le&
\abs{q^T \offd(\hat{B}^{\star} - B_0) q} + \abs{q^T
  \diag(\hat{B}^{\star} - B_0) q} \\
&  \le& 8 \delta_q \twonorm{B_0}  +  C'_{\diag} b_{\infty} r_{\diag},
\eens
where by Theorem~\ref{thm::main}, Lemma~\ref{lemma::quadreduction}, 
and~\eqref{eq::sparseW1}, we have on event  $\F_5^c \cap \F_4^c$, 
\ben 
\label{eq::nearsparseW1}
\sup_{q \in \sqrt{s_0} B_1^n \cap B_2^n} {\abs{q^T \offd\big( \X \X^T - (B_0 \circ \M) \big) 
    q}}/{\shnorm{\hat{\TM}}_{\offd}} 
\le  \frac{4 \delta_q \twonorm{B_0}}{(1-\delta_{\mask})}; 
\een
as desired.
Finally $\prob{\F_{20}^c} \ge 1-
{C}/{(n \vee m)^4} -4 \exp(-c s_0 \log  (3e n/(s_0 \ve)))$ by
the union bound.  
Thus the theorem is proved.
\end{proofof2}

\subsection{Proof of Corollary~\ref{coro::Bhatnorm}}
\label{sec::proofofBhatop}
\begin{proofof2}
    The proof follows identical arguments as
    above, except that we now replace the bound regarding $W_1$
in~\eqref{eq::defineWlocal} using~\eqref{eq::netb} to obtain with probability at least $1 -{C}/{(n \vee m)^4}$, 
\bens 
\nonumber 
\lefteqn{
W'_1 :=  \sup_{q \in \Sp^{n-1}} \abs{q^T \offd\big( \X \X^T \oslash \hat{\TM} - (B_0 \circ \M) 
      \oslash \hat{\TM} \big) q} }  \\
 & = & \sup_{q \in \Sp^{n-1}}  {\abs{q^T \offd\big( \X \X^T - (B_0 \circ \M) \big) 
    q}}/{\shnorm{\hat{\TM}}_{\offd}} =\\
\label{eq::W1overall}
& &
\sup_{q \in \Sp^{n-1}}
\frac{\shnorm{\M}_{\offd}}{\shnorm{\hat{\TM}}_{\offd}}
\frac{\abs{q^T \offd\big( \X \X^T - (B_0 \circ \M) \big) 
    q}}{\shnorm{\M}_{\offd}} \le 
\frac{\delta_q(B) 
  \twonorm{B_0}}{1-\delta_{\mask}}
\eens
Then for all $q \in \Sp^{n-1}$, with probability at least $1 -{C}/{(n
  \vee m)^4}$ (cf. proof of Theorem~\ref{thm::gramsparse}),
we have by \eqref{eq::Wdiag}, \eqref{eq::defineWlocal},
\eqref{eq::W2univadd}, and the bound on $W'_1$ immediately above, 
\bens
\lefteqn{\twonorm{\hat{B}^{\star} -B_0}/{\twonorm{B_0} }   =
  \sup_{q \in \Sp^{n-1}} \abs{q^T (\hat{B}^{\star} - B_0) q} }\\
&  \le&
\sup_{q \in \Sp^{n-1}} \abs{q^T \offd(\hat{B}^{\star} - B_0) q} +
\sup_{q \in \Sp^{n-1}} \abs{q^T  \diag(\hat{B}^{\star} - B_0) q} \\
&  \le& W'_1 + \frac{\delta_{\mask} }{1-\delta_{\mask}}
\twonorm{\offd(B_0)}+ C'_{\diag} b_{\infty} r_{\diag} \\
&  \le& 8 \delta_q(B) \twonorm{B_0}  +  C'_{\diag} b_{\infty} r_{\diag},
\eens
where recall that for $n= \Omega(\log (m \vee n))$ and hence
$$\delta_{\mask} \asymp \underline{r_{\offd}} /\sqrt{r(B_0)}
= \eta_A (\twonorm{A_0} \log (m \vee n))^{1/2}/\sqrt{r(B_0)\sum_{j} a_{jj}  p_j^2} = o(r_{\offd}(n))$$ while $\delta_q(B) =
\Omega(r_{\offd}(n))$.
The corollary thus holds.
\end{proofof2}

\section{Proof of Lemma~\ref{lemma::unbiasedmask}}
\label{sec::maskest}
Throughout this section, let $Y = X^T$, where $X = B_0^{1/2} \Z  A_0^{1/2}$ is as defined in~\eqref{eq::missingdata}.
Denote by $D_0 = B_0^{1/2} \otimes  A_0^{1/2}$. We have $D_0^2 = B_0
\otimes A_0$.  Let $Z = \mvec{\Z^T}$ for $\Z$ as defined in~\eqref{eq::missingdata}.
Then $Z \in \R^{mn}$ is a subgaussian random vector 
with independent components $Z_j$ that satisfy $\expct{Z_j} = 0$,
$\expct{Z_j^2} = 1$, and $\norm{Z_j}_{\psi_2} \leq  1$. 
As shown in \eqref{eq::diagmask}, $\onen\tr(\X^T \X)$ provides an unbiased estimator for entries in 
$\diag(\M)$ for $\tr(B_0) =n$.
\bens
\onen \E \tr(\X^T \X) = \onen \E \tr(\X \X^T) =\onen \E \sum_{i=1}^n \twonorm{v^i \circ y^i}^2 = \frac{ \tr(B_0)}{n} \sum_{i=1}^m a_{ii} \z_i.
\eens

\subsection{Unbiased estimator for the mask matrix: off-diagonal
  component}
\label{sec::TMunbiased}
Let $\{e_1, \ldots, e_m \} \in R^{m}$ be the canonical basis of  
$\R^m$.   Denote by
\ben
\label{eq::DV}
\D_v & = & \sum_{k=1}^n \diag(e_k) \otimes D_k \quad \text{ where}  \\
\label{eq::DK}
D_k & = & \inv{n-1}\sum_{j \not=k} \diag(v^k \otimes v^j), \forall 
k=1, \ldots, n
\een
Then for $S_c$ as defined in \eqref{eq::defineSc}, we have for $M^i = v^{i} \otimes v^{i}$ and hence $\diag(M^i) =\diag(v^{i})$,
\bens
\lefteqn{
S_c := \inv{n-1}\sum_{j=1}^n \sum_{k\not=j}^n \tr\big(M^k \circ M^j \circ (y^j
  \otimes y^j) \big)} \\
& = &
\inv{n-1}\sum_{j=1}^n \sum_{k\not=j}^n (y^j)^T  \diag(v^k \circ v^j)
{y^j}  =  \sum_{j=1}^n (y^j)^T D_j y^j = \mvec{Y}^T \D_v  \mvec{Y} 
\eens
Hence we rewrite for $\mvec{Y} =B_0^{1/2} \otimes A_0^{1/2} Z =: D_0
Z$ and $ A_{v}  \; := \; D_0 \D_v D_0$,
\bens
  S_c = Z^T  B_0^{1/2} \otimes A_0^{1/2} \D_v B_0^{1/2}
\otimes A_0^{1/2} Z  & =: &
Z^T A_v Z
\eens
where $Z = \mvec{\Z^T}$ for $\Z$ as defined in
\eqref{eq::missingdata}. It is not difficult to verify that $\E S_c = \tr(B_0) \sum_{i=1}^m
a_{ii} \z^2_i$. Indeed, we can compute for $\E Z^2_j =1$ and $\E Z_j =0$,  
\bens
\E S_c & = & \E \E (Z^T D_0 \D_v D_0 Z|U) = \E \tr(D_0 \D_v D_0) \\
& = & \inv{n-1}\sum_{j=1}^n b_{jj} \sum_{i\not=j}^n \tr(A_0  \E \diag(v^i \circ v^j))  = \tr(B_0) \sum_{j=1}^m a_{jj} p_j^2
\eens
\noindent{\bf Decomposition.}
We now decompose the error into two parts:
\ben
  \nonumber 
  \lefteqn{\abs{S_c - \E S_c} = \abs{Z^T  A_{v} Z  - \E (Z^T 
      A_{v} Z )} \le    \abs{  Z^T  A_{v} Z  - \E (Z^T A_{v} Z |U)}} \\
    \label{eq::decompAV}
   & &
+ \abs{\E (Z^T  A_{v} Z |U) - \E (Z^T  A_{v} Z)} =:  I + II.
 \een
\noindent{\bf Part I.}
Since $\D_v$~\eqref{eq::DV} is a block diagonal matrix with the $k^{th}$
block along the diagonal being $D_k, \forall \; \; k=1, \dots,
m$~\eqref{eq::DK}, with entries in $[0, 1]$, we have its
its operator norm, row and column $\ell_1$-norms all bounded by $1$ and thus 
\ben
\label{eq::AVop}
\twonorm{A_v} & := & \twonorm{D_0 {\D}_v D_0} \le \twonorm{D_0}^2
\twonorm{\D_v} \le \twonorm{A_0}\twonorm{B_0} 
\een
Lemma~\ref{lemma::EAV} shows that  tight concentration of measure 
bound on $\fnorm{A_v}$ can be derived under a mild condition;
Since the proof follows a similar line of arguments  as in the proof of Theorem~\ref{thm::uninorm-intro}, we omit it here.
\begin{lemma}
\label{lemma::EAV}
Suppose that $\sum_{s=1}^m a_{ss}^2 p_s^2 =\Omega(a_{\infty}^2 \log  (n
\vee m))$.
Then on event $\F_0^c$ as defined in Theorem~\ref{thm::uninorm-intro},  we have
\bens 
\fnorm{A_v}^2 
& \le &
\frac{C_1}{(n-1)} \fnorm{\diag(B_0)}^2 \sum_{s=1}^m a_{ss}^2  p_s^2 +
C_2 \fnorm{B_0}^2 \sum_{s=1}^m a_{ss}^2  p_s^3 \\
&+&
C_3 \fnorm{B_0}^2 a_{\infty} \twonorm{A_0} \sum_{s=1}^m p_s^4 +
C_4 \fnorm{B_0}^2 \twonorm{A_0}^2 \log (n \vee m)    =: W_v^2
\eens
\end{lemma}
Hence we set the large deviation bound to be
\bens
\tau_0 
& = & C_6 \log^{1/2}(n \vee m)\big(\frac{2}{\sqrt{n-1}}
  \fnorm{\diag(B_0)}  \big(\sum_{s=1}^m a_{ss}^2     p_s^2\big)^{1/2}
  +\fnorm{B_0}\big(\sum_{s=1}^m a_{ss}^2 p_s^3\big)^{1/2} \big)
  + \\
&&
C_7 \log^{1/2}(n \vee m) \fnorm{B_0} \big(\big(a_{\infty}
\twonorm{A_0} \sum_{s=1}^m p_s^4 \big)^{1/2}+ C  \twonorm{A_0} \log^{1/2}(n \vee m)\big)
\eens
By Lemma~\ref{lemma::EAV} and Theorem~\ref{thm::HW},  for absolute
constants $c$ and $C_6, C_7$ sufficiently large,
\ben
\nonumber
\lefteqn{
\prob{\abs{Z^T A_{v} Z - \E(Z^T A_{v} Z | U) } > \tau_0 } =: \prob{\event_7}}\\
\nonumber
& = & 
\E_{V} \prob{\abs{Z^T A_{v} Z - \E(Z^T A_{v} Z | U) } > \tau_0 \big| U} \\
& \leq &
\nonumber
2 \exp \big(- c\min\big(\frac{\tau_0^2}{W^2_{v}},    \frac{\tau_0}{\twonorm{A_0}\twonorm{B_0}
} \big)\big)  \prob{\F^c_0} + \prob{\F_0}\\
& \leq &
\label{eq::tau0step}
2 \exp\big(- c\min\big(\frac{\tau_0^2}{W^2_{v}},
    \frac{\tau_0}{\twonorm{A_0}\twonorm{B_0}}  \big)\big)  + \prob{\F_0}  \leq  \frac{c}{(n \vee m)^4}
\een
where by Theorem~\ref{thm::HW} and~\eqref{eq::AVop}, we have for 
any $t>0$,
\bens
&& \prob{\abs{Z^T A_{v} Z - \E(Z^T A_{v} Z | U) } > t \big|  U \in \F_0^c} \leq 
2 \exp \big(- c\min\big(\frac{t^2}{ W_{v}^2},
\frac{t}{\twonorm{B_0}\twonorm{A_0}} \big) \big) \\
&&
\text{ since on event $\F_0^c$, where $U$ is being fixed,} \; \; \fnorm{A_{v}}^2 \le W_{v}^2.
\eens
\noindent{\bf Part II.}
Denote by 
\bens 
S_{\star} &:= &
\E (Z^T D_0  \D_v D_0 Z |U)  -\E \tr(D_0 {\D}_v D_0) \\
 &:= &\tr(D_0 {\D}_v D_0) - \E \tr(D_0 {\D}_v D_0) 
 \eens
Denote by
\bens
\nonumber
\D_{\star} & := &B' \otimes \diag(A_0) \text{  where } \; \; 
B' :=  \inv{n-1} \left[
\begin{array}{ccccc} 
0 & b_{11} & b_{11}  & \ldots &  b_{11}  \\
b_{22}  & 0 & b_{22} &  \ldots &  b_{22}  \\
\ldots & \ldots & \ldots & \ldots & \ldots\\
b_{nn} &b_{nn} &  \ldots & b_{nn}  &  0 \\
\end{array}\right]_{n \times n} 
\eens
It is straightforward to check the following holds:
\bens 
\norm{D_{\star}}_{\infty}  \le  b_{\infty} a_{\infty}
\quad\text{and} \quad \norm{D_{\star}}_1 \le  b_{\infty} a_{\infty}.
\eens 
\begin{corollary}
  \label{coro::expmgfS2II}
 Let $A_0 = (a_{ij})$ and $B_0 = (b_{ij})$.
Let $U$ be as defined in~\eqref{eq::maskUV} and $V =U^T$.
  Then for $D_{\star}$ as defined immediately above,
  \bens 
S_{\star} 
& =&  \mvec{V}^T \D_{\star} \mvec{V} -  \tr(B_0) \sum_{s=1}^m a_{ss} p_s^2 
\eens 
Let $b_{\infty} := \max_j b_{jj}$ an $a_{\infty} = \max_i a_{ii}$.
For all $t > 0$,
\bens  
\prob{\abs{S_{\star}} > t} & \le &   
2 \exp\big(- c\min\big(
\frac{t^2}{\lambda^2 a_{\infty} b_{\infty} \tr(B_0) \sum_{j =1}^m a_{jj} p_j^2},\frac{t}{a_{\infty} b_{\infty}} \big)\big)
\eens  
\end{corollary}
Set for $C$ large enough, $\tau_{\star} = C \lambda \log^{1/2} (n \vee m) \sqrt{a_{\infty}  b_{\infty} \tr(B_0) \sum_{j=1}^m a_{jj} p_j^2}$;
We have by Corollary~\ref{coro::expmgfS2II}, and the by assumption in
Lemma~\ref{lemma::unbiasedmask},
\bens 
\lefteqn{\prob{\abs{S_{\star}} > \tau_{\star}} =:  \prob{\event_6} \le }\\
&  & 2 \exp\big(- c\min\big(\frac{\tau_{\star}^2}{\lambda^2 a_{\infty}
      b_{\infty} \tr(B_0) \sum_{j =1}^m a_{jj}      p_j^2},\frac{\tau_{\star}} {a_{\infty} b_{\infty}}  \big)\big) \le \inv{(n \vee m)^4}.
\eens
Putting these two parts together, we have on event
$\F_0^c \cap \event_6^c \cap 
\event_7^c$,
\ben
\label{eq::SCtwo}
\abs{S_c - \E S_c} \le \abs{S_c - \E[S_c | U]} + \abs{S_{\star}} \le
\tau_0 + \tau_{\star}
\een
\subsection{Proof of Corollary~\ref{coro::expmgfS2II}}
\begin{proofof2}
Denote by $B_0^{1/2} \otimes A_0^{1/2} =: D_0$.  Then
\bens
\lefteqn{ \E [S_c|U] = \tr(D_0 {\D}_v D_0)  =  \inv{n-1}\sum_{t=1}^n
  b_{tt} \tr( A_0 D_t )}\\
& = &
\inv{n-1} \sum_{t=1}^n b_{tt}  \sum_{s \not=t} \tr( A_0
\diag(v^t \otimes v^s) ) \\
& = &
\inv{n-1}\sum_{t=1}^n b_{tt}  \sum_{s \not=t} (v^t)^T (A_0 \circ I)  v^s =  \mvec{V}^T \D_{\star} \mvec{V}
\eens
Notice that by definition of $D_{\star}$ we have for $\E(\mvec{V})
=:\vecp   \in \R^{mn}$,
\bens
\sum_{i\not=j} \abs{D_{\star,ij}} \vecp_j \vecp_i 
& = & 
\sum_{i=1}^n \inv{n-1} \sum_{j\not=i} \abs{b_{ii}} \sum_{q=1}^m \abs{a_{jj}} p_j^2  = 
\tr(B_0)\sum_{j=1}^m a_{jj} p_j^2
\eens
By Corollary~\ref{coro::Bernmgf}, where we substitue $A$ with
$D_{\star}$, and set $D_{\max} := a_{\infty} b_{\infty}$, we have the following estimate on the moment generating function for $S_{\star}$.
We have for
$\abs{\lambda} < \inv{16 a_{\infty}  b_{\infty}}$ and $$e^{8\abs{\lambda} a_{\infty} b_{\infty}} \le  e^{1/2} \le 1.65,$$
\bens 
\E (\exp( \lambda S_{\star})) 
&\le &
\exp\big(36.5 \lambda^2 D_{\max} e^{1/2}   \tr(B_0) \sum_{j =1}^m a_{jj} p_j^2 \big).
\eens
The rest of the proof follows that of Lemma~\ref{lemma::Bernmgf2} and
hence omitted.
\end{proofof2}

\subsection{Proof of Lemma~\ref{lemma::unbiasedmask}}
\begin{proofof2}
We use the following bounds:
 $\fnorm{\diag(B_0)}^2 \le b_{\infty} \tr(B_0)$, $\tr(B_0) \le \sqrt{n}\fnorm{\diag(B_0)}$ and 
$\fnorm{B_0} \le \sqrt{\twonorm{B_0} \tr(B_0)}$.\\
  \noindent{\bf Off-diagonal component}
  Clearly,   On event $\event_6^c \cap \event_7^c$,
  \bens
\; \; \forall k \not= \ell, \; \; \; 
\delta_{\mask} :=
\frac{\abs{\hat{\TM}_{k \ell} -  \M_{k \ell} }}{\M_{k \ell}}=
\frac{\abs{S_c - \E S_c}}{\E S_c} \le \frac{\tau_0 + \tau_{\star}}{\E S_c}
\eens
Putting things together we have on event
$\F_5^c := \F_0^c \cap \event_6^c \cap 
\event_7^c$,
which holds with probability at least $1- \frac{C}{(n \vee 
  m)^4}$,  where $C  \le 6$, by \eqref{eq::SCtwo},~\eqref{eq::tau0step} and Corollary~\ref{coro::expmgfS2II},
\bens
\delta_{\mask}
  := \frac{\abs{S_c - \E S_c}}{\E S_c}
  \le  \frac{\tau_0 + \tau_\star}{\tr(B_0)\sum_{s=1}^m a_{ss} p_s^2} 
\asymp
\underline{r_{\offd}}\frac{\twonorm{B_0}^{1/2}}{\sqrt{\tr(B_0)}}(1+o(1)) \; 
\eens
\noindent{\bf Diagonal component}
Recall $V = U^T$.
Let $\xi = \mvec{V} :=(\xi_1, \ldots, \xi_{m n}) \in \{0, 1\}^{m n}$ be a random vector
independent of $X$, as defined in \eqref{eq::maskUV}.
Let $A_{\xi} = D_0 D_{\xi} D_0$ where $D_0 = B_0^{1/2} \otimes  A_0^{1/2}$.
Thus we can write for $\mvec{Y} = B_0^{1/2} \otimes A_0^{1/2} Z$
and ${\bf v} := \mvec{V}$,
\bens
\lefteqn{
\tr(\X^T \X)  = 
\mvec{Y}^T \diag(\mvec{V} \otimes \mvec{V}) \mvec{Y} } \\
 & = & 
Z^T B_0^{1/2} \otimes A_0^{1/2} \diag(\bv \otimes \bv) 
B_0^{1/2} \otimes A_0^{1/2} Z =: Z^T D_0 D_{\xi} D_0 Z
\eens
Now
$$D_0^2 \circ D_0^2 = (B_0 \otimes A_0) \circ (B_0 \otimes A_0) =(B_0 \circ B_0) \otimes (A_0 \circ A_0).$$
We can apply Theorem~1.2~\citep{Zhou19} here directly to argue that for every $t > 0$,
\ben
\label{eq::deviQ}
\lefteqn{
\prob{\abs{Z^T A_{\xi} Z  - \expct{Z^T A_{\xi} Z}} > t} \leq  2\exp \big(-c \min\big(\frac{t^2}{K^4 Q}, \frac{t}{K^2
  \twonorm{A_0}\twonorm{B_0}} \big)\big) }\\
\nonumber
\; \text{ where }  Q & =& \sum_{j=1}^n  b_{jj}^2 \sum_{i=1}^m a_{ii}^2 p_i + 
\E \big(\xi^T \offd((B_0 \circ B_0) \otimes (A_0 \circ A_0))  \xi\big)^T  \\
&  \le&
\nonumber
\fnorm{B_0}^2 a_{\infty} \twonorm{A_0} \onenorm{\vecp} \; \text{
  where} \; \vecp = (p_1, p_2, \ldots, p_m), \text{ and } \;  \onenorm{\vecp} =\sum_{j=1}^m p_j
\een
Thus by choosing for some absolute constants $C_1, C_2, c$,
\bens
\tau_{\diag} 
& = & C_1 \log (n \vee m)
\twonorm{A_0} \twonorm{B_0} + C_2 \log^{1/2} (n \vee m) \sqrt{a_{\infty}
  \twonorm{A_0} \onenorm{\vecp}} 
\fnorm{B_0} 
\eens
We have by \eqref{eq::deviQ}
\bens
\lefteqn{\prob{\abs{Z^T A_{\xi} Z  - \expct{Z^T A_{\xi} Z}} > \tau_{\diag}} =:
\prob{\event_8} }\\
&\le & 2\exp \big(-c \min\big(\frac{\tau_{\diag}^2}{Q},
    \frac{\tau_{\diag}}{\twonorm{A_0}\twonorm{B_0}} \big)\big) \leq
\frac{c}{(n \vee m)^4}
\eens
Hence on event $\event_8^c$, for all $\ell$,
\bens
\lefteqn{
\frac{\abs{\hat{\TM}_{\ell \ell} -
    \M_{\ell \ell} }}{\M_{\ell \ell}} =  \frac{\abs{\fnorm{\X}^2 - \E \fnorm{\X}^2}}{\E \fnorm{\X}^2}
    = \frac{\abs{\tr(\X^T \X) - \E \tr(\X^T \X) } }{ \E \tr(\X^T \X) }}\\
  &\le &   \frac{\tau_{\diag}}{{\tr(B_0) \sum_{j=1}^m  a_{jj} p_j} } \asymp 
  \frac{r_{\diag} \sqrt{\twonorm{B_0}} }{\sqrt{\tr(B_0)} } (1+o(1))
  \eens
where  $r_{\diag}$ is as defined in~Theorem~\ref{thm::diagmain}. The lemma is proved.
\end{proofof2}

\section{Proof of Theorem~\ref{thm::gramsparse}}
\label{sec::appendcoroproof}
We now give a complete proof of Theorem~\ref{thm::gramsparse}.
Throughout this section, we denote by $Z \in \R^{mn}$  a subgaussian random vector 
with independent components $Z_j$ that satisfy $\expct{Z_j} = 0$, 
$\expct{Z_j^2} = 1$, and $\norm{Z_j}_{\psi_2} \leq 1$. 
Recall 
\bens
A_{qq}^{\diamond} &=&
\sum_k \sum_{i\not= j} u^k_i u^k_j 
q_i q_{j} (c_j c_i^T) \otimes  (d_k d_k^T).
\eens
Recall that we denote by $Q_{\offd}$ the off-diagonal component in \eqref{eq::MB2},
\bens
Q_{\offd} 
&  =  & \inv{\norm{\M}_{\offd}} \abs{q^T  \offd(\X  \X^T  -B_0 \circ  \M) q} =:\inv{\norm{\M}_{\offd}} \abs{q^T  \tilde{\Delta} q} \\
&  \sim & 
\inv{\norm{\M}_{\offd}} \abs{Z^T A^{\diamond}_{qq} Z  - \E (Z^T A^{\diamond}_{qq} Z)}, \; \text{ where }\; Z \sim \mvec{\Z^T}
\eens
for $\Z$ as defined in \eqref{eq::missingdata} with $K=1$.

\begin{proofof}{Theorem~\ref{thm::gramsparse}}
Following~\eqref{eq::MB2}, we decompose the error into two parts:
for $\tilde{\Delta} = \offd(\X \X^T  -B_0 \circ \M)$ and $q \in \Sp^{n-1}$,
\ben
  \nonumber 
  \lefteqn{\abs{q^T \tilde{\Delta} q} = \abs{Z^T  A^{\diamond}_{q q} Z  - \E (Z^T 
      A^{\diamond}_{q q} Z )} \le    \abs{  Z^T  A^{\diamond}_{qq} Z
      - \E (Z^T  A^{\diamond}_{qq } Z |U)}} \\
    \label{eq::decompcoro}
   & &
+ \abs{\E (Z^T  A^{\diamond}_{q q} Z |U) - \E (Z^T  A^{\diamond}_{q q} Z)} =:
{\bf I} + {\bf II} 
\een
\noindent{\bf Part I}
Denote by $\N$ the $\ve$-net for $\Sp^{n-1}$, for example, as constructed in 
Lemma~\ref{lemma::net} with $\size{\N} \le (3/\ve)^n$.

Applying the Hanson-Wright inequality (cf. Theorem~\ref{thm::HW}) with the preceding estimates on the operator
and Frobenius norms and the union bound, we have by Theorems~\ref{thm::mainop2} and~\ref{thm::uninorm2}, and the union bound, for any $t > 0$,
\bens
\lefteqn{\prob{\left\{\exists q \in \N,  \abs{Z^T A^{\diamond}_{qq} Z - \E(Z^T A^{\diamond}_{qq}    Z | U) } > t \right\} | U \in \F_0^c}  }\\
&\le &
 2 \size{\N}^2 \exp \left(-  c\min\left(\frac{t^2}{{\twonorm{B_0}}^2
       \twonorm{A_0} W^2}, \frac{t}{\twonorm{A_0}\twonorm{B_0}} \right)\right)
 \eens
 Hence we set 
\bens 
\tau_0 \asymp n \log (3e/\ve) \twonorm{A_0}  \twonorm{B_0}+
\sqrt{n \log (3e/\ve)}  \twonorm{B_0} \twonorm{A_0}^{1/2}  W 
\eens 
Finally, we have by Theorems~\ref{thm::HW},~\ref{thm::mainop2} and~\ref{thm::uninorm2},
\bens
\lefteqn{\prob{\exists q \in \N, \abs{Z^T A^{\diamond}_{qq} Z - \E(Z^T 
        A^{\diamond}_{qq} Z | U) } > \tau_0} =: \prob{\F_1}} \\
& = & \E_U  \prob{\exists q \in \N,  \abs{Z^T A^{\diamond}_{qq} Z - \E(Z^T A^{\diamond}_{qq} Z | U) } > \tau_0 | U}\\
& \leq &  2 \size{\N}^2 \exp
\left(-c\min\left(\frac{\tau_0^2}{{\twonorm{B_0}^2}  \twonorm{A_0} W^2},  \frac{\tau_0}{\twonorm{B_0}\twonorm{A_0}}\right) \right)\prob{\F_0^c} +   \prob{\F_0} \\
   & \le & \exp(- c_1 n \log (3e / \ve))) + \prob{\F_0}.
   \eens
   \noindent{\bf Part II}
Fix $q \in \Sp^{n-1}$. 
Let $\tilde{a}^{k}_{ij}$ be a shorthand for $\tilde{a}^{k}_{ij}(q) = a_{kk} b_{ij} q_i q_j$. 
Let 
\bens 
S_{\star}(q)  = 
\E (Z^T  A^{\diamond}_{q q} Z  | U)- \E (Z^T  A^{\diamond}_{q q} Z)  = \sum_{k=1}^m \sum_{i \not=j}^n \tilde{a}^{k}_{ij}(u^k_i u^k_j - p_k^2) 
\eens
Corollary~\ref{coro::offdn} follows from Theorem~\ref{thm::Bernmgf}, which is the  
main tool to deal with the sparse quadratic forms. Its proof appears
in Section~\ref{sec::proofofpartII}.

Now suppose that for $C_4$ large enough, set
\bens
\tau' & = & C_4 a_{\infty} \twonorm{B_0}  \psi_B(n) n \log (3e/\ve)
\asymp  a_{\infty} \twonorm{\abs{B_0}} n \log (3e/\ve).
\eens
Then by~\eqref{eq::baseline}, Corollary~\ref{coro::offdn} and the union
bound,
\bens
\lefteqn{\prob{\exists q \in \N, \abs{S_{\star}(q)} > \tau'} =: \prob{\F_2}  } \\
&\le &  \abs{\N} \exp\left(- C_5 \min\left(\frac{\psi_B(n)(n   \log(3e/\ve) )^2 }{ \sum_{k=1}^m a_{kk}p_k^2}, \frac{ \twonorm{\abs{B_0}} n\log(3e/\ve)}{\twonorm{B_0}}
  \right)\right) \\
& \le & (3/\ve)^{n} \exp\left(-  C' s_0 \log   (3e/\ve) \right) \le \exp\left(-2n \log (3e/\ve)\right),
\eens
where $1\le \psi_B(n) = O(\sqrt{n})$.
On event $\F_0^c  \cap \F_1^c \cap \F_2^c$, for $W$ as defined in 
\eqref{eq::Wlocal}, we have by the bounds immediately above,
\bens
\lefteqn{\sup_{q\in \N}  \inv{\twonorm{B_0}}
  \abs{q^T\tilde{\Delta} q} =  \sup_{q \in \N}  \inv{\twonorm{B_0}}
  \abs{Z^T A^{\dm}_{q, q} Z-\E Z^T A^{\dm}_{q, q} Z}  }\\
& \le &
\sup_{q \in \N} \inv{\twonorm{B_0}} \left(\abs{Z^T A^{\dm}_{qq} Z-\E
    (Z^T A^{\dm}_{qq} Z | U)}  + \abs{\E (Z^T A^{\dm}_{qq} Z | U)-\E
    Z^T A^{\dm}_{q, q} Z} \right)\\
& \le &  \frac{\tau_0 +\tau'}{\twonorm{B_0}}  \asymp \sqrt{n \log
  (3e/\ve)}  \twonorm{A_0}^{1/2}  W +
n \log (3e/\ve) (\twonorm{A_0} +  a_{\infty}  \psi_B(n))
\eens
where $W  \asymp  (a_{\infty} \sum_{s=1}^m p_s^2)^{1/2}
+\psi_B(n) \big(a_{\infty} \twonorm{A_0} \log (n \vee m) \sum_{j=1}
p_j^4 \big)^{1/4}$. Denote by
\bens
r_{\offd}(n) := \sqrt{{n \twonorm{A_0}}/{(\sum_{j}
     a_{jj}  p_j^2)}}, \text{ and } \eta_A =
 \sqrt{a_{\infty}/a_{\min}}.
 \eens
 \noindent{\bf Putting things together}
 The rest of the proof follows from that of 
Theorem~\ref{thm::main-intro} using~\eqref{eq::MB1}
and~\eqref{eq::MB2}. 
By a standard approximation argument, we have on event  $\F_0^c  \cap 
\F_1^c \cap \F_2^c$, 
\ben
\label{eq::offdfinal}
&&\sup_{q \in \Sp^{n-1}} {\abs{q^T\tilde{\Delta}
    q}}/{\big(\twonorm{B_0}\norm{\M}_{\offd}\big)}
\le  \sup_{q \in \N}  {\abs{q^T\tilde{\Delta}
    q}}/{\big((1-\ve)^2\twonorm{B_0}\norm{\M}_{\offd}\big)} \\
\nonumber
  &\le &    C' \frac{\sqrt{n} \twonorm{A_0}^{1/2} W}{\sum_{j=1}^m a_{jj} p_j^2} +
C_2 \frac{n \twonorm{A_0} (1 +  a_{\infty}  \psi_B(n)/\twonorm{A_0})}{\sum_{j=1}^m a_{jj} p_j^2} \\
\nonumber
&\le &   C \eta_A r_{\offd}(n) + C' \psi_B(n)  r_{\offd}(n) f_{\QA}
+ C_4 r^2_{\offd}(n) \big(1+ {a_{\infty} \psi_B(n)}/{\twonorm{A_0}}  \big)
\een
where $f_{\QA} :=  \big(\log (n \vee m) \twonorm{A_0} a_{\infty} \sum_{s=1}^m 
  p_s^4\big)^{1/4}/\big(\sum_{j=1}^m a_{jj} p_j^2\big)^{1/2}$ and
\bens
\lefteqn{\frac{\sqrt{n} \twonorm{A_0}^{1/2}
  W}{\sum_{j=1}^m a_{jj} p_j^2}  =: \eta_A  r_{\offd}(n) +
\psi_B(n)  r_{\offd}(n) f_{\QA} } \\
 && \asymp \eta_A  r_{\offd}(n) +  \frac{ \psi_B(n) n^{1/2}\twonorm{A_0}^{3/4}
   a_{\infty}^{1/4}  \big(\log (n \vee m)  \sum_{s=1}^m 
p_s^4\big)^{1/4}}{\sum_{j=1}^m a_{jj} p_j^2}.
\eens
Obviously $\sum_{s=1}^m  p_s^4 \le p_{\max}^2  \sum_{s=1}^m
  p_s^2$, where $p_{\max}  = \max_{j} p_j$ and hence 
\ben
\nonumber
\psi_B(n)  r_{\offd}(n) f_{\QA}
& \le & \psi_B(n) \big(\frac{n
\twonorm{A_0}}{\sum_{j}a_{jj}  p_j^2} \big)^{3/4}
\big(\frac{a_{\infty} p_{\max}^2 \sum_{j} p_j^2}{a_{\min} \sum_{j}
  p_j^2}\big)^{1/4} \big(\frac{\log (n \vee m)}{n}\big)^{1/4} \\
& \asymp &
\label{eq::rate3}
\eta^{1/2}_A \sqrt{p_{\max} } \psi_B(n) r^{3/2}_{\offd}(n) \big({\log (n \vee   m) }/{n}\big)^{1/4},
\een
where we use the fact that $\big(\sum_{j=1}^m a_{jj} p^2_j \big)^{1/4}
\ge a_{\min}^{1/4} \big(\sum_{i=1}^m p_i^4 \big)^{1/4}$.
By \eqref{eq::MB2},~\eqref{eq::offdfinal}, and~\eqref{eq::rate3}, we have on event $\F_0^c  \cap \F_1^c 
 \cap \F_2^c \cap \F_{\diag}^c$,
 \ben
 \nonumber
\lefteqn{\twonorm{\tilde{B}_0 -B_0}/{\twonorm{B_0}
    \le C \eta_A  r_{\offd}(n) + C_4  r^2_{\offd}(n) \big(1+ {a_{\infty} \psi_B(n)}/{\twonorm{A_0}} \big) }}\\
& &
\label{eq::ratelocal}
+ C_3 \eta^{1/2}_A \sqrt{p_{\max} } r^{3/2}_{\offd}(n) \psi_B(n)  \big(\log (n \vee  m)/n\big)^{1/4}
\een
where for the diagonal component, we have for
$n = \Omega(\log (m \vee 
n))$
\bens
\lefteqn{{\maxnorm{\diag(\X \X^T) -  \E\diag(\X
      \X^T)}}/(\twonorm{B_0} \norm{\M}_{\diag})\le  C_{\diag}
  r_{\diag}}
\\
& \le &
\eta_A \big({\twonorm{A_0} \log (m \vee n) 
}/{\norm{\M}_{\diag}}  \big)^{1/2} =o(\eta_A r_{\offd}(n)). 
\eens
Thus~\eqref{eq::Bhatop} holds.
 The final expression in the theorem statement 
 for the overall rate of convergence follows  the same line of 
 arguments from Lemma~\ref{lemma::finalrate} with a slight variation.
Now by~\eqref{eq::rate3} and the Cauchy-Schwarz  inequality,
\ben
\nonumber
\psi_B(n)  r_{\offd}(n) f_{\QA} & \le &
  C \eta^{1/2}_A   \sqrt{p_{\max}} r^{3/2}_{\offd}(n) \psi_B(n) \big(\frac{\log (n \vee 
    m) }{n}\big)^{1/4}\\
  \label{eq::ratelocal2}
  & \le & c \eta_A r_{\offd}(n) + c' p_{\max}  r_{\offd}^2 (n) \psi_B(n)
 \big(\frac{ \psi_B(n)\log^{1/2} (n \vee m)}{\sqrt{n}}\big) 
 \een
 By \eqref{eq::MB2},~\eqref{eq::rate3}~\eqref{eq::ratelocal},~\eqref{eq::ratelocal2},
 we have on event $\F_0^c  \cap \F_1^c  \cap \F_2^c \cap \F_{\diag}^c$,
 \ben
 \nonumber
\lefteqn{\twonorm{\tilde{B}_0 -B_0}/{\twonorm{B_0} }\le C \big(\eta_A
    r_{\offd}(n) +   r^2_{\offd}(n) \big) + } \\ 
   \nonumber
  &  &
  C'   r^2_{\offd}(n) \psi_B(n)
  \Big(\frac{c_2 p_{\max} \psi_B(n)\log^{1/2} (n \vee m)
  }{\sqrt{n}}+\frac{a_{\infty}}{\twonorm{A_0}}\Big)  \\
  \label{eq::oprate5}
  &  \le &  C_1
 \eta_A r_{\offd}(n)  + C_2 r^2_{\offd}(n) \psi_B(n) \log^{1/2}(m \vee n)
 \een
where  $\psi_B(n)=O(\sqrt{n})$.
Finally, we have by the union bound, for some absolute constants $c, C$,
\bens
\prob{\F_0^c  \cap \F_1^c  \cap \F_2^c \cap \F_{\diag}^c} & \ge & 1- c /(m \vee n)^4 - 2
\exp\big(-c n \log (3e/\ve)\big) \\
& \ge &  1- C /(m \vee n)^4
\eens
since $\exp(\log (m \vee 
n) ) \le \exp(c' n \log (3e/\ve))$ since $\log (m \vee n) = o(n)$.
 Theorem~\ref{thm::gramsparse} thus holds.
\end{proofof}

\begin{remark}
Assume that $\log^{1/2} (m \vee n) \psi_B(n) /\sqrt{n} = o(1)$,
 we can get rid of the extra $\log^{1/2}(m \vee n)$ factor
 in~\eqref{eq::oprate5}.
  Then we have by \eqref{eq::Bhatop}, 
\bens 
 \twonorm{\tilde{B}_0 -B_0}/{\twonorm{B_0} }= O_P\Big(\eta_A 
 r_{\offd}(n)  + \big(p_{\max} + a_{\infty}/\twonorm{A_0}\big) 
 r^2_{\offd}(n) \psi_B(n) \Big) 
 \eens
which increases with $p_{\max}$ and decreases with $\twonorm{A_0}$ for 
a fixed value of $r_{\offd}(n)$, as we emprically observed in Section
\ref{sec::examples}.
\end{remark}

In order for~\eqref{eq::baseline} to hold for $s_0 = n$,
Lemma~\ref{eq::avgrowsum} suggests that we need
\bens 
 {\sum_{j=1}^m a_{jj} p_j^2}/{\twonorm{A_0}} = \Omega\big( \onenorm{B_0}
 \eta^2_A \log (n \vee m)/\twonorm{B_0} \big)\; \text{where} \; \onenorm{B_0} =
 \sum_{i,j} \abs{b_{ij}}
 \eens
On the other hand, in order for~\eqref{eq::baseline} to hold, it is sufficient to have 
\ben 
 \label{eq::baselinecoro}
 \quad \quad 
 {\sum_{j=1}^m a_{jj} p_j^2}/{\twonorm{A_0}}
 =O\big(\eta^2_A \log (n \vee m) n \norm{B_0}_{\infty}/{\twonorm{B_0}} \big) 
 \een
 using the upper bound of $\psi_B(n) \le \norm{B_0}_{\infty}/{\twonorm{B_0}}$.
Table~\ref{tab2} suggests that for the $\operatorname{AR}(1)$ models,
the upper and lower bounds are closing up; however, for the Star
model, the gap can be large. \\
\begin{table}
\caption{Metrics for models of $B$}
\label{tab2}
\begin{tabular*}{\textwidth}{l|ccc|ccc|ccc}
\hline
  \textbf{Metric}&\multicolumn{3}{c}{AR(1) $\rho_B=0.3$}& \multicolumn{3}{c}{AR(1) $\rho_B=0.7$}& \multicolumn{3}{c}{Star $\rho_B=\frac{1} {\sqrt{n}}$}  \\ \hline
  $n$                          & 64    & 128    & 256    &  64   & 128    & 256     &  64 & 128 & 256 \\ \hline
  $\norm{B}_{\infty}$                        &  1.857 & 1.857  & 1.857   & 5.667 & 5.667  & 5.667  &8.875 &  12.176 & 16.938 \\
  $\twonorm{B}$                        &  1.855 & 1.856  & 1.857   & 5.578 & 5.642  & 5.660  &2.588 &  2.593 &  2.611 \\
  $\onenorm{B}/\twonorm{B}$  &  63.43 & 127.39 & 255.36  & 62.22 & 125.79 & 253.54 &54.39 & 105.77 & 207.18 \\
  $\norm{B}_{\infty}/\twonorm{B}$          &  1.001 & 1.000  & 1.000   & 1.016 & 1.004  & 1.001  &3.429  & 4.696 & 6.488 \\
\hline
\end{tabular*}
\end{table}

\subsection{Proof of Lemma~\ref{eq::avgrowsum}}
\begin{proofof2}
For the lower bound, we have by the Rayleigh-Ritz theorem,
\bens
 \lambda_{\max}(\abs{B_0}) & \ge & \frac{\ip{\abs{B_0}  \vecone, 
     \vecone}}{\twonorm{\vecone}^2} =\inv{n} \sum_{i, j} \abs{b_{ij}}
 \eens
where  $\vecone = (1, 1, \ldots, 1)$;
By definition, the largest eigenvalue of symmetric non-negative
matrix $\abs{B_0}$ coincides with the spectral radius
\bens
\lambda_{\max}(\abs{B_0})  
& := & \max_{v \in \Sp^{n-1}} \abs{v}^T \abs{B_0} \abs{v} =
\max_{v \in \Sp^{n-1}} \abs{v^T \abs{B_0} v} =: \rho(\abs{B_0}); \text{and}, \\
\twonorm{B_0} & = & \rho(B_0) 
\le  \rho(\abs{B_0})   =\lambda_{\max}(\abs{B_0})  =
\twonorm{\abs{B_0}}  \le   \norm{\abs{B_0}}_{\infty} =
\norm{B_0}_{\infty}
\eens
in view of the previous derivations, where the upper bound follows
from the fact that the matrix operator norm is upper bounded by its $\ell_{\infty}$ norm.
\end{proofof2}

\subsection{Proof of Corollary~\ref{coro::offdn}}
\label{sec::proofofpartII}
Corollary~\ref{coro::Bernmgf} follows from Theorem~\ref{thm::Bernmgf}.
\begin{corollary}
  \label{coro::Bernmgf}
Suppose all conditions in Theorem~\ref{thm::Bernmgf} hold.
Let $A = (a_{ij})$ be an $m \times m$ matrix with 
$0$s along its diagonal. Then,
for every $\abs{\lambda} \le \inv{16(\norm{A}_1 \vee \norm{A}_{\infty})}$ and 
$D_{\max} := \norm{A}_{\infty} + \norm{A}_1$
\bens
\E \exp\big(\lambda\big( \sum_{i,j} a_{ij} \xi_i \xi_j - \E \sum_{i,j} a_{ij} \xi_i \xi_j \big)\big) 
  \le   \exp\big(36.5 \lambda^2 D_{\max}  e^{8 \abs{\lambda}
     D_{\max}} \sum_{i\not=j} \abs{a_{ij}}  p_i   p_j\big).
\eens
\end{corollary}

\begin{proofof}{Corollary~\ref{coro::offdn}}
Let $\tilde{A} = \E \DD(q) = (\tilde{a}^{k}_{ij})_{k=1, \ldots, m}$ 
be the block-diagonal matrix with $k^{th}$ block along the 
diagonal being $\tilde{A}^{(k)} := (\tilde{a}^{k}_{ij}(q))_{i,j \le
  n}, k \in [m]$, where $\tilde{a}^{k}_{ij}(q) = a_{kk} b_{ij} q_i q_j$ for
$i \not=j$ and  $\tilde{a}^{k}_{jj} = 0$.
Thus $\tilde{A}  =  \diag(A_0) \otimes \offd(B_0 \circ  (q \otimes q))$,
where the expectation is taken elementwise;
Then for $\abs{\lambda} \le  1/{(16 D_{\max})}$, where $D_{\max} :=
\shnorm{\tilde{A}}_{\infty} \vee \shnorm{\tilde{A}}_1 \le a_{\infty} \twonorm{B_0}$,
 \bens
\E \exp\big(\lambda S_{\star}(q) \big)
 & \le &  \exp\big(60 \lambda^2 D_{\max}  
 \sum_k \sum_{i\not=j} \abs{\tilde{a}^{k}_{ij}(q)}  p_i  p_j\big) \\
& \le & \exp\big(60 \lambda^2 D_{\max}  \twonorm{\abs{B_0}}  \sum_k a_{kk} p_k^2\big)
\eens
by   Corollary~\ref{coro::Bernmgf}.
The rest of the proof follows that of Lemma~\ref{lemma::Bernmgf2}.
\end{proofof}

\section{Proof of Theorem~\ref{thm::main}}
\label{sec::appendproofoffdmain}
\begin{proofof2}
  To prove Theorem~\ref{thm::main}, it remains
  to check conditions in 
Theorems~\ref{thm::mainlights} and~\ref{thm::samplesize}.
Now~\eqref{eq::baseline} ensures~\eqref{eq::samplecrux}
and~\eqref{eq::sample1local} hold, where we show
\ben 
\nonumber 
\lefteqn{
  \sum_{j=1}^m a_{jj} p^2_j
\ge \big(\sum_{j=1}^m a_{jj} p^2_j\big)^{3/4}
\big(a_{\min} \sum_{j=1}^m p^4_j \big)^{1/4} }\\
\label{eq::sample2local}
  &&  \ge C_5  \twonorm{A_0}^{3/4} a_{\infty}^{1/4}    \psi_B(s_0) 
  \big(s_0 \log \big(\frac{3en}{s_0\ve}\big)\big)^{1/2}  \big(\log (n \vee m)  \sum_{s=1}^m p_s^4\big)^{1/4}
  \een
  so that $r_{\offd}(s_0)  f_{\QA} \psi_B(s_0) < 1$.
Since $\psi_{B}(s_0) \le \psi_{B}(2s_0 \wedge n) \le \sqrt{2s_0}$, we
have by~\eqref{eq::baseline}, for $S_a := {\sum_{j=1}^m a_{jj} p^2_j}/{\twonorm{A_0}}$ and $d = 2s_0 \wedge n$,
  \ben
 \label{eq::propcond}
 \quad \quad S_a \ge  C_4 \eta_A^2 s_0\psi_B(d) \log (n \vee m)
  \ge C_6 (\eta_A s_0)^{2/3} \psi^{4/3}_B(s_0) \log (n \vee m)
  \een
Hence it is sufficient to have \eqref{eq::medium} in order for
\eqref{eq::sample2local}  to hold:
\ben
\label{eq::medium}
&& \big({\sum_{j=1}^m a_{jj} p^2_j}/{\twonorm{A_0}}\big)^{3/4} 
\ge C_5 \psi_B(s_0) \big(s_0 \log (\frac{3en}{s_0 \ve})\big)^{1/2}
\big(\frac{a_{\infty} \log (n \vee m)}{a_{\min}}\big)^{1/4},
\een
for which~\eqref{eq::propcond} in turn is a sufficient condition;
This completes the proof of Theorem~\ref{thm::main}.
\end{proofof2}

\silent{\label{eq::sample2local}
&&
\sum_{j=1}^m a_{jj} p^2_j \ge C_5 \psi_B(s_0) (N \log (n \vee m))^{1/4}
\big(s_0 \log \big(\frac{3e n}{s_0\ve} \big) \twonorm{A_0}\big)^{1/2}}

\silent{
  Alternatively, we may consider
  \bens
\lefteqn{  \sqrt{\eta_A}   (\psi_B(s_0))^{1/2} \ell_{s_0, n}^{1/4}
  r^{3/2}_{\offd}(s_0) =  r_{\offd}(s_0) \sqrt{\eta_A r_{\offd}(s_0)
    \psi_B(s_0) \ell_{s_0, n}^{1/2}} }\\
  & \le & \eta_A r_{\offd}(s_0)  + \ell_{s_0, n}^{1/2}  r^2_{\offd}(s_0) \psi_B(s_0)
\eens
While for $s_0 < n$, we have for 
\bens
r_{\diag}
\asymp \eta_A \big(\frac{\twonorm{A_0} \log (m \vee n) 
}{\sum_{j} a_{jj}  p_j}\big)^{1/2} 
& \le &
\eta_{A}
r_{\offd}(s_0) =
\eta_{A} \sqrt{s_0 \log \big(\frac{3e n}{s_0\ve}\big) 
  \frac{\twonorm{A_0}}{\sum_{j} a_{jj}  p_j^2}} \\
\text{while } \;  r_{\offd}^2 (s_0)  \psi_B(s_0) \ell_{s_0, n}^{1/2}
& \asymp &
\frac{s_0 \psi_B(s_0) \twonorm{A_0}}{(\sum_{j}a_{jj}  p_j^2)}
\big({\log (n \vee m)}{\log (3en/(s_0 \ve) )} \big)^{1/2} \\
& = &
\eta_{A} r_{\offd}(s_0) =
\eta_{A} s_0 \log^{1/2} \big(\frac{3e n}{s_0\ve}\big) 
  \frac{\twonorm{A_0}}{\sum_{j} a_{jj}  p_j^2} \\
\text{so long as } \;
\log (m \vee n) \sum_{j} a_{jj}  p_j^2 & \le & s_0 \log \big(\frac{3e n}{s_0\ve}\big) (\sum_{j} a_{jj}  p_j)
\eens
since $n \log (6e) \ge s_0 \log \big(\frac{3e n}{s_0\ve}\big) \ge \log
(6e n)$ since $f(s_0) = s_0 \log \big(\frac{3e n}{s_0\ve}\big)$ is
monotonically increasing with $s_0$, while clearly $\log (m \vee n)
\in [\log n, n \log (6e))$ can be anywhere in this range so long as $m
= o(e^n)$.
Recall
\bens 
r_{\offd}(s_0) = \sqrt{s_0 \log \big(\frac{3e n}{s_0  \ve}\big) 
  \frac{\twonorm{A_0}}{\sum_{j} a_{jj}  p_j^2}} \; \text{ and }
 \sqrt{\ell_{s_0, n}} = \sqrt{\frac{\log (n \vee m)}{\log (3en/s_0 \ve)}}. 
 \eens
\bnum
\item
Suppose that $\log (6 en) <s_0 \log \big(\frac{3e n}{s_0  \ve}\big) \le \log
(m \vee n)$; That is, we have many samples
$m = \Omega\big(\big(\frac{3e n}{s_0  \ve}\big) e^{s_0} \big) $, where clearly
  $s_0 \le \log (m)/\log (3en/(s_0 \ve))$.
  Then for $\psi_B(s_0) \le \sqrt{s_0}$
  \bens
  r_{\offd}^2 (s_0)  \psi_B(s_0) \ell_{s_0, n}^{1/2}
& \asymp &
\frac{\psi_B(s_0) \twonorm{A_0}}{(\sum_{j}a_{jj}  p_j^2)}
{\sqrt{s_0} \log^{1/2} (n \vee m)}
\big({s_0 \log (3en/(s_0 \ve) )} \big)^{1/2} \\
& \le &
\frac{s_0 {\log (n \vee m)} \twonorm{A_0}}{(\sum_{j}a_{jj}  p_j^2)}
\asymp r^2_{\offd}(s_0)
\eens
In other words, when $s_0$ is small, then the error is dominated by $\eta_A r_{\offd}(s_0)$.
\item
  Otherwise, suppose that $s_0 \log \big(\frac{3e n}{s_0  \ve}\big) \ge
 \log (m \vee n)$ so that  $r_{\offd}(s_0) \ge r_{\diag}$
   and hence,
 \bens
 r_{\offd}(s_0) \psi_B(s_0)   \sqrt{\ell_{s_0, n}}
 &= & \psi_B(s_0)\sqrt{\frac{\log (n \vee m)}{\log (3en/s_0 \ve)}}
 \sqrt{s_0 \log \big(\frac{3e n}{s_0\ve}\big) 
   \frac{\twonorm{A_0}}{\sum_{j} a_{jj}  p_j^2}} \\
 &\le &  \psi_B(s_0) \sqrt{\log (n \vee m)}  \sqrt{\frac{s_0
     \twonorm{A_0}}{\sum_{j} a_{jj}  p_j^2} } \le \sqrt{ \psi_B(s_0)}
\eens
where recall that we assume
\bens
{\sum_{j=1}^m a_{jj} p_j^2}/{ \twonorm{A_0}} \ge C_4 \eta_A^2 ( \psi_B(2s_0 \wedge n) \vee 1)
 s_0 \log (n \vee m)
 \eens
 \enum
}



\silent{\bens
\nonumber
\sup_{q, h \in \Sp^{n-1} \cap E} \frac{\abs{q^T\tilde{\Delta}
      h}}{\twonorm{B_0}\norm{\M}_{\offd}} \le
C_{\sparse} 
\big(r_{\offd}(s_0) \eta_{A} \ell_{s_0, n}^{1/2} +
r^2_{\offd}(s_0)  \psi_B(2s_0 \wedge n) \big)
\eens}



  \silent{
    \bens 
\frac{\sum_{j=1}^m a_{jj}
  p_j^2}{ \twonorm{A_0}}
& \ge & 
\psi_B(s_0) s_0 \log (n/s_0) 
\big(\log (n \vee m) \sum_{s=1}^m 
  p_s^4\big)^{1/4}
  \eens 
Suppose 

     \bens 
 Z^T  A^{\diamond}_{q,h} Z 
& =: &
Z^T \big(\half \sum_{k=1}^m \sum_{i=1}^n
  \sum_{j \not= i}^n  u^k_i u^k_j \big(q_i h_j + q_j h_i) (c_i
    c_j^T) \otimes (d_k d_k^T) \big) \big) Z
    \eens }
\silent{
  since $ s_0 \ge \psi_B(s_0)^2$.
\ben
\label{eq::middleterm}
\frac{a_{\infty} \sum_{j} p_j^2}{\twonorm{A_0} \log (n \vee m) } \ge
\frac{\sum_{j} a_{jj} p_j^2} {\twonorm{A_0} \log (n \vee m) } \ge
\een
\begin{corollary}
  \label{coro::uninorm}
  Suppose~\eqref{eq::sample1local} holds.
  Then we have
\bens
&& \sup_{q, h \in \Sp^{n-1}, s_0-\sparse}  
\fnorm{A^{\diamond}_{qh}} \le C \twonorm{B_0}
\twonorm{A_0}^{1/2} (V_1 + V_2)
\eens
\end{corollary}
\begin{proof}
\bens
&& \sup_{q, h \in \Sp^{n-1}, s_0-\sparse}  
\fnorm{A^{\diamond}_{qh}} \le W \cdot \twonorm{B_0}
\twonorm{A_0}^{1/2}
\; \text{ where }\; \\
&& W  \asymp  \big(a_{\infty} \sum_{s=1}^m p_s^2\big)^{1/2} + 
\psi_B(s_0)\sqrt{\twonorm{A_0} \log (n \vee m)}
+ \psi_B(s_0) \big(N \log (n \vee m)\big)^{1/4}
\eens
The corollary thus holds while adjusting the constant $C$. 
\end{proof}}

\section{Proofs for Theorem~\ref{coro::thetaDet}}
\label{sec::inverseLemma}
We use $\hat{a}_{j} := -\tilde{\Theta}_{jj}$ as shorthand 
notation, where $\tilde{\Theta}_{jj}$ is as defined in 
\eqref{eq::digest}. Denote by $a_{j} := - \theta_{jj}$.

\subsection{Proof of Theorem~\ref{coro::thetaDet}}
\label{sec::proofofinverses}
Let $\tilde\Theta_{j \cdot}$ denote the $j^{th}$ row vector of 
$\tilde\Theta$ following Definition~\ref{def::TopHat}.
In Lemma~\ref{lemma::tidebound}, we derive error bounds for estimating the  
diagonal entries of $\Theta_0$, as well as the error bounds for constructing row vectors $\{\Theta_{j \cdot}, j\in [n]\}$ of 
$\Theta_0$ with $\{\tilde\Theta_{j \cdot}, j \in [n]\}$. Let $\kappa_B := \twonorm{B_0}/\lambda_{\min}(B_0)$.
Let  $\tilde{\kappa}_{\rho} := M_{\rho} M_{\Omega} \ge \kappa_{\rho}$ be an upper estimate on the condition number of $\rho(B)$ under 
(A1). Let $\alpha$ be as in \eqref{eq::lassopen}. Suppose
\ben 
\label{eq::parityproof}
&& {\sum_{j} a_{jj} p_j^2}/{\twonorm{A_0}}  \ge 4 C_{\overall}^2
\eta_A^2\kappa_B^2 \twonorm{\Theta_0}^2 d_0 \log (m \vee n), 
\text{ where } \\
\nonumber
&& C_{\overall}  := 16 C_{\alpha} C_{\gamma} b_{\infty} \tilde{\kappa}_{\rho}
\text{ for }  \; C_{\gamma} \ge c_{\gamma} \vee C_{\max}, \; \; C_{\alpha} := \lambda_{\min}(B_0)/{\alpha},
\een
and $C_{\max}$ and $c_{\gamma}$ are as defined in~\eqref{eq::BHatoffd}
and~\eqref{eq::halflambda}, respectively. Lemma~\ref{lemma::tidebound} is proved in supplementary
Section~\ref{sec::proofoftide},
which follows steps from Corollary 
5~\cite{LW12}. 
\begin{lemma}
  \label{lemma::tidebound}
  Suppose all conditions in Theorem~\ref{coro::thetaDet} hold. 
  Suppose~\eqref{eq::parityproof} holds.
Let $\alpha=\lambda_{\min}(B_0) /C_{\alpha}$, where $2> C_{\alpha} > 1$.
Then for each $j$, $(a) \abs{\tilde{\Theta}_{jj}} \le 2 \abs{\theta_{jj}}$;
\bens
(b) && \abs{\tilde{\Theta}_{jj} - \theta_{jj}} =\abs{\hat{a}_j - a_j}
\le  2 C_{\overall} \ul{r_{\offd}} \sqrt{d_0} \kappa_B
\theta^2_{\max},\; \; \text{ where }\; \; \theta_{\max} := \max_{jj} \theta_{jj};\\
(c)  && \onenorm{\tilde{\Theta}_{j \cdot} - \Theta_{j \cdot}} \le {2 C_{\overall} \ul{r_{\offd}} d_0 } \theta_{\max} (\kappa_B+3)/{\lambda_{\min}(B_0)}.
\eens
\end{lemma}

\begin{proofof}{Theorem~\ref{coro::thetaDet}}
Clearly, Condition~\eqref{eq::parityproof} holds under the assumption that 
$\ul{r_{\offd}} \sqrt{d_0} = o(1)$ as imposed 
in~\eqref{eq::paritydual} in Theorem~\ref{coro::thetaDet}; See~\eqref{eq::parity4}. 
  Following Lemma~\ref{lemma::tidebound}, 
\ben
\label{eq::tildeTheta}
\shnorm{\tilde{\Theta} - \Theta_{0}}_{\infty} & \le & 2 C_{\overall} \ul{r_{\offd}} d_0 \theta_{\max} (\kappa_B+3)
/{\lambda_{\min}(B_0)}\; \\
\nonumber
\; \text{ and hence } \; 
\shnorm{\hat{\Theta} - \Theta_0}_1 
& = & \shnorm{\hat{\Theta} - \Theta_0}_{\infty}  \le 
\nonumber
\shnorm{\hat{\Theta} - \tilde\Theta}_{\infty} +
\shnorm{\Theta_0 - \tilde\Theta}_{\infty}  \\
& \le & 
\nonumber
2 \shnorm{\Theta_0 - \tilde\Theta}_{\infty} =
2 \max_{j} \onenorm{\tilde{\Theta}_{j \cdot} - \Theta_{j \cdot}} 
\een
where the second inequality holds by optimality of $\hat{\Theta}$ in
minimizing $\shnorm{\Theta - \tilde\Theta}_{\infty}$ among all symmetric matrices
and the last  inequality holds by \eqref{eq::tildeTheta}.
\end{proofof}

\subsection{Proof of Lemma~\ref{lemma::tidebound}}
Lemma~\ref{lemma::detLassonew}  is identical to Theorem
16~\cite{RZ17} upon adjustment of constants, which we elaborate
in Section~\ref{sec::proofofdetLasso}.

\begin{lemma}
  \label{lemma::detLassonew}
Suppose all conditions in Theorem~\ref{coro::thetaDet} hold,
where we set $\lambda$ as in~\eqref{eq::lambdamain}.

Let $\beta^{j*} = \{\beta_k^{j*}, k \not=j, k \in  [n]\}$.
Let $\hat\beta^{(j)} = \{\hat\beta^{(j)}_k, k \not=j, k \in 
[n]\}$ be as in~\eqref{eq::origin} with $\hat\Gamma^{(j)}$ and 
$\hat\gamma^{(j)}$ given in~\eqref{eq::GammaMain} and~\eqref{eq::gammaMain}.
Then for all $j$,
\ben
\label{eq::twofix}
    \twonorm{\hat{\beta}^{(j)} - \beta^{j*}}
    & \le & 3\lambda \sqrt{d_0}/\alpha \le 12 C_{\gamma} C_{\alpha} b_{\infty} \tilde{\kappa}_{\rho}
    {\ul{r_{\offd}}\sqrt{d_0} }/{\lambda_{\min}(B_0)}, \\
\label{eq::onefix}
\onenorm{\hat{\beta}^{(j)} - \beta^{j*}}
& \le & 4 \sqrt{d_0}   \shnorm{\hat{\beta}^{(j)} - \beta^{j*}}_2 \le {12 \lambda d_0}/{\alpha} 
\een
\end{lemma}

\begin{corollary}
  \label{coro::oracle}
  Suppose all conditions in Lemma~\ref{lemma::detLassonew} hold;
Suppose~\eqref{eq::parityproof} holds.   Let $a_j := -\theta_{jj}$.
Then for all $j$, we have the following bounds
\ben
\label{eq::oneabs}
\onenorm{\hat{\beta}^{(j)} - \beta^{j*}}/{\sqrt{d_0}}
  & \le & 4  \shnorm{\hat{\beta}^{(j)} - \beta^{j*}}_2 \le 2/\kappa_B <
  2 \\
\label{eq::betaoracle}
\abs{a_j} \onenorm{\hat\beta^{(j)}}
  & \le &
  c' \sqrt{d_0} \twonorm{\Theta_{j \cdot}} \; \text{ for } \; \; c' \le 2 \sqrt{2}.
\een
\end{corollary}
ction{On estimating the diagonal entries of $\Theta_0$}

\begin{proof}
   Under \eqref{eq::parityproof}, we have 
\ben 
\label{eq::parity4}
16 C_{\alpha}  C_{\gamma} b_{\infty} \tilde\kappa_{\rho}\ul{r_{\offd}}
\sqrt{d_0} & := & C_{\overall} \ul{r_{\offd}} \sqrt{d_0} \le 
{\lambda_{\min}(B_0)}/{(2  \kappa_B )}
\een
Now, we have by \eqref{eq::twofix}, \eqref{eq::onefix} and \eqref{eq::parity4}, and $\kappa_B \ge 1$
  \bens
\onenorm{\hat{\beta}^{(j)} - \beta^{j*}}/{\sqrt{d_0}}
& \le &
4  \shnorm{\hat{\beta}^{(j)}  - \beta^{j*}}_2
\le {48 C_{\alpha} C_{\gamma} b_{\infty}  \tilde{\kappa}_{\rho}   \ul{r_{\offd}}\sqrt{d_0}}/{\lambda_{\min}(B_0)} \\
& \le &
3/ (2\kappa_B) < 3/2\; \;   \text{ and hence}\\
\nonumber 
\abs{a_j} \onenorm{\hat\beta^{(j)}}
  & \le &
\sqrt{d_0} \abs{\theta_{jj}}  \big(\onenorm{\beta^{j*}} /\sqrt{d_0} +
  \onenorm{\hat\beta^{(j)} - \beta^{j*}}/\sqrt{d_0}  \big)\\
        \nonumber 
& \le &
\sqrt{d_0} \big(\big(\sum_{k\not=j}^n\theta_{jk}^2\big)^{1/2} + {3\theta_{jj}}/{2} \big)\le c' \sqrt{d_0} \twonorm{\Theta_{j \cdot}}
\eens
where $c' \le 2 \sqrt{2}$ and the last step holds by~\eqref{eq::oneabs}.
\end{proof}

As an initial step, we show Proposition~\ref{prop::aratio}, which
shows Part (a) of Lemma~\ref{lemma::tidebound}.
To simplify notation,  we suppress $0$ from the subscripts when we
refer to column ( row) vectors of $B_0$ (and $\Theta_0$). 
\begin{proposition}
  \label{prop::aratio}
For $\tilde{\Theta} := (\tilde\theta_{jk})$, we have by definition
\bens
\forall k \not=j, \; \;
\tilde\theta_{jk} = \hat{a}_j \hat\beta^{(j)}_k,  \; \; \text{ where } \; \; \tilde\theta_{jj} = -\hat{a}_j
\eens
Suppose all conditions in Lemma~\ref{lemma::tidebound} hold. Then
\ben
  \label{eq::thetaratio}
\abs{{a_j}/{\hat{a}_j} - 1}
&  = & \abs{a_i}  \abs{\hat{a}_j^{-1} - a_j^{-1}} 
<   C_{\overall}  \ul{r_{\offd}} \sqrt{d_0} \theta_{\max} \kappa_B <1/2
\een
\end{proposition}

\begin{proof}
Now by Definition~\ref{def::TopHat}, we have
\bens
\lefteqn{\abs{\frac{{a}_{j} }{\hat{a}_j} -1}
 = \abs{a_{j} }\abs{a_j^{-1} - \hat{a}_j^{-1} } 
 = \abs{a_{j} } \abs{ (\hat{B}^{\star}_{jj} - \hat{B}^{\star}_{j, \minus j} \hat{\beta}^{(j)}) -  (B_{jj} - {B}_{j, \minus j} \beta^{j*})}} \\
& \le & \abs{a_{j} } \abs{\hat{B}^{\star}_{jj} -B_{jj}} +
 \abs{a_{j} }\abs{  \hat{B}^{\star}_{j, \minus j} \hat{\beta}^{(j)}  - {B}_{j, \minus j} \beta^{j*}} = I + II;
 \eens
 By~\eqref{eq::BHatoffd} and the triangle inequality, we have for $\kappa_B \ge 1$,
\bens
 (I) & := & 
 \abs{a_{j} } \abs{\hat{B}^{\star}_{jj} -B_{jj}} \le
 C_{\max} b_{\infty}  \theta_{\max}\ul{r_{\offd}} \le C_{\max}
 \kappa_B  \theta_{\max} \ul{r_{\offd}} \\
\lefteqn{ \text{ while } \; \;  (II)
 :=  \abs{a_{j}}  \abs{\hat{B}^{\star}_{j, \minus j} \hat{\beta}^{(j)}   - {B}_{j, \minus j}  \beta^{j*}} }\\
& \le&  \abs{a_{j}} \abs{  \hat{B}^{\star}_{j, \minus j}
  \hat{\beta}^{(j)} -  {B}_{j, \minus  j} \hat\beta^{(j)}}
+\abs{a_{j}} \abs{{B}_{j, \minus  j} \hat\beta^{(j)}  - {B}_{j, \minus
    j} \beta^{j*}}  =:    \Pi_{\alpha} +   \Pi_{\beta} \\
& \le & 
\abs{a_{j}} \norm{  \hat{B}^{\star}_{j, \minus j} -  {B}_{j, \minus
    j} }_{\infty} \onenorm{\hat\beta^{(j)}} + \abs{a_{j}}
\twonorm{B_{j, \minus  j}} \twonorm{\hat\beta^{(j)} -  \beta^{j*}} \\
& \le & \abs{a_{j}}\onenorm{\hat\beta^{(j)}}   C_{\max} b_{\infty}
\ul{r_{\offd}} + \abs{a_{j}}\twonorm{B_0} \twonorm{\hat{\beta}^{(j)} - \beta^{j*}} 
  \eens
where by ~\eqref{eq::BHatoffd}, $\shnorm{\hat{B}^{\star}_{j, \minus j} -  {B}_{j, \minus  j} }_{\infty} \le 
  C_{\max} b_{\infty} \ul{r_{\offd}}$ and by~\eqref{eq::betaoracle},~\eqref{eq::parityproof} and~\eqref{eq::twofix},
\ben
\nonumber
\Pi_{\alpha} & \le & {c' C_{\max} b_{\infty} \ul{r_{\offd}}
  \sqrt{d_0}} \twonorm{\Theta_{0}} \le
 c' C_{\max} b_{\infty} \ul{r_{\offd}} \sqrt{d_0} \theta_{\max}
 \kappa_B   \; \text{ and} \\
\label{eq::2bproof}
\nonumber
\Pi_{\beta} & \le &\abs{a_j}  \twonorm{B_0}
\twonorm{\hat{\beta}^{(j)} - \beta^{j*}} \le 
\abs{a_j}  \twonorm{B_0} 3 \lambda \sqrt{d_0} /{\alpha}\\
&    \asymp  &
\nonumber
\underline{r_{\offd}} \sqrt{d_0} \abs{a_j}  \twonorm{B_0} {12 C_{\alpha} C_{\gamma} b_{\infty}
  \tilde\kappa_{\rho}}/{\lambda_{\min}(B_0)}
\le (12 C_{\gamma} C_{\alpha} b_{\infty} \tilde\kappa_{\rho})
 \ul{r_{\offd}} \sqrt{d_0}  \theta_{\max} \kappa_B 
\een
where $\theta_{\max}  \kappa_B =  (\max_{j}\abs{\theta_{jj}})
{\twonorm{B_0}}  \twonorm{\Theta_0}
\ge (\max_{j} b_{jj} \theta_{jj}) \twonorm{\Theta_0} >
\twonorm{\Theta_0}$ since $ b_{jj} \theta_{jj} > 1$ by Proposition~\ref{prop::projection}.
Thus we have for $C_{\gamma} \ge
c_{\gamma} \vee C_{\max}$,
\bens
\nonumber
\lefteqn{\abs{\frac{{a}_{j} }{\hat{a}_j} -1}\le 
  I + \Pi_{\alpha} + \Pi_{\beta} \le 
  \big(1+ c' + 12 C_{\alpha} \tilde\kappa_{\rho}\big)
  b_{\infty} C_{\gamma} \ul{r_{\offd}} \sqrt{d_0} \kappa_B \theta_{\max}  }\\ 
&< &
C_{\overall} \ul{r_{\offd}} \sqrt{d_0} \theta_{\max}
\kappa_B  \le {C_{\overall} \ul{r_{\offd}} \sqrt{d_0}  \kappa_B}/{\lambda_{\min}(B_0)} \le 1/2
\eens
where $c' < 2 \sqrt{2}$, $C_{\alpha} > 1$ and the last step holds by~\eqref{eq::parityproof}
and~\eqref{eq::parity4}.
\end{proof}

\subsection{Proof of Lemma~\ref{lemma::tidebound}}
\label{sec::proofoftide}
\begin{proofof2}
 Clearly, by Proposition~\ref{prop::aratio},
\ben
 \label{eq::solratio}
1/{2} & < & \abs{{{a}_{j} }/{\hat{a}_j}} < {3}/{2}
\text{ and hence } \; \;\abs{\hat{a}_{j}} < 2 \abs{a_j}
\een
 Part (b) also holds for all $j \in [n]$, as by~\eqref{eq::thetaratio} and
\eqref{eq::solratio},
\ben
\nonumber
\lefteqn{\abs{\tilde{\Theta}_{jj} - \theta_{jj}}
 =  \abs{\hat{a}_j - a_j} =\abs{\hat{a}_j}
\abs{1-{a_j}/{\hat{a}_j} } \le
2 \abs{a_j} \abs{a_j} \abs{a_j^{-1} - \hat{a}_j^{-1} } }\\
\nonumber
& \le  &
2 C_{\overall} \ul{r_{\offd}} \sqrt{d_0}  \theta_{\max}^2 
\kappa_B \le { 2 C_{\overall} \ul{r_{\offd}} \sqrt{d_0}}
\theta_{\max} \kappa_B/{\lambda_{\min}(B_0)} 
\een
\ben
\nonumber
\text{Finally,} \;
\onenorm{\tilde{\Theta}_{j \cdot} - \Theta_{j \cdot}} & \le &  \abs{\hat{a}_j - a_j} +   \onenorm{\hat{a}_j \hat{\beta}^{(j)} - a_j 
  \beta^{j*}} \\
\nonumber
& = &
\abs{\hat{a}_j - a_j} +   \onenorm{\hat{a}_j \hat{\beta}^{(j)} - \hat{a}_j   \beta^{j*} +   \hat{a}_j   \beta^{j*} -  a_j 
  \beta^{j*}}\\
\nonumber
  & \le &
  \abs{\hat{a}_j - a_j} + \abs{\hat{a}_{j}} \onenorm{\hat{\beta}^{(j)} - 
    \beta^{j*}} + \abs{\hat{a}_{j} - a_j}  \onenorm{\beta^{j*}} \\
  & \le &
  \label{eq::defineW}
 \abs{\hat{a}_j - a_j} \big(1 +\onenorm{\beta^{j*}}\big)
 +\abs{\hat{a}_{j}} \onenorm{\hat{\beta}^{(j)} -   \beta^{j*}} =: Y_1
 + Y_2 \\
 \nonumber
 \text{ where } \;&&
 1 +\onenorm{\beta^{j*}} = \onenorm{\Theta_{j \cdot}}/{\theta_{jj}}
 \le \sqrt{d_0}  \twonorm{\Theta_0}/{\theta_{jj}}
 \een
Now, by~\eqref{eq::thetaratio} and~\eqref{eq::solratio}, we have for
$C_{\overall} = 16 C_{\alpha} C_{\gamma}
b_{\infty}\tilde\kappa_{\rho}$ and $ \abs{a_j} = \theta_{jj}$,
 \ben
 \label{eq::W1a}
 &&  Y_1    =  \abs{\hat{a}_j - a_j} \big(1 +\onenorm{\beta^{j*}}\big)
 =\abs{\hat{a}_j} \abs{a_j} \abs{a_j^{-1} - \hat{a}_j^{-1} } {\onenorm{\Theta_{j \cdot}}}/{\theta_{jj}} \\
  \nonumber
&\le &
 2 \abs{a_j} \abs{a_j^{-1} - \hat{a}_j^{-1} } \sqrt{d_0}\twonorm{\Theta_{0}}
 \le 2 C_{\overall}  \theta_{\max} \kappa_B {\ul{r_{\offd}}
   d_0}/{\lambda_{\min}(B_0)}, \\
 \nonumber
\text{and }  & & Y_2 = \abs{\hat{a}_{j}} \onenorm{\hat{\beta}^{(j)} - 
   \beta^{j*}}  \le 24 \abs{a_j} d_0 \lambda /\alpha \\
 \label{eq::W2a}   
 &&\le
 96 \abs{a_j} C_{\alpha} C_{\gamma}  \ul{r_{\offd}} d_0 
    b_{\infty} \tilde\kappa_{\rho}/{\lambda_{\min}(B_0)}     \le 6 C_{\overall} \theta_{\max} {  \ul{r_{\offd}}
    d_0}/{\lambda_{\min}(B_0)}
  \een
by~\eqref{eq::onefix}. The lemma holds by \eqref{eq::defineW},
\eqref{eq::W1a} and \eqref{eq::W2a}, since
\ben
  \nonumber 
\onenorm{\tilde{\Theta}_{j \cdot} - \Theta_{j \cdot}}
     &\le &
 Y_1 + Y_2   < 2 C_{\overall} \theta_{\max} (\kappa_B+3) {\ul{r_{\offd}} d_0 }/{\lambda_{\min}(B_0)}.
 \een
\end{proofof2}

\subsection{Proof of Lemma~\ref{lemma::detLassonew}}
\label{sec::proofofdetLasso}
The proof follows that of~\cite{RZ17}, cf. Theorem~15;
Let $S_j := \supp(\beta^{j*})$, $d_0 \ge \size{S_j} \forall j \in [n]$.

\begin{lemma}{\textnormal~\cite{BRT09,LW12}}
  \label{lemma:magic-number}
  Suppose all conditions in Lemma~\ref{lemma::detLassonew} hold.
  Fix $j \in [n]$. Denote by $S = S_j$. Denote by  $\beta = \beta^{j*}$ and 
  $$\upsilon = \hat{\beta} - \beta,$$
  where $\hat\beta := \hat{\beta}^{(j)}$ is as defined in \eqref{eq::origin}.

Then for all $j \in [n]$,
\ben
\label{eq::onenorm}
\onenorm{\upsilon_{S^c}} & \leq & 3 \onenorm{\upsilon_{S}}, \; \text{
  and hence} \; \onenorm{\upsilon} \le \; \; 4 \onenorm{\upsilon_{S}} \le 4 \sqrt{d_0}
\twonorm{\upsilon}.
\een
\end{lemma}

\begin{proof}
We use the shorthand notation $\hat\gamma = \hat\gamma^{(j)}$ and   $\hat\Gamma = \hat\Gamma^{(j)}$.
By the optimality of $\hat{\beta}$, we have
\begin{eqnarray}
  \nonumber
\lambda \onenorm{\beta} - 
\lambda \onenorm{\hat\beta}
& \geq & 
\inv{2} \hat\beta \hat\Gamma  \hat\beta - \inv{2} \beta \hat\Gamma \beta -
         \ip{\hat\gamma, v} \\
   \nonumber
& = & 
\inv{2} \up^T \hat\Gamma \up +\ip{\up, \hat\Gamma \beta} -
\ip{\up, \hat\gamma}  = \inv{2} \up^T \hat\Gamma \up -\ip{\up,
      \hat\gamma - \hat\Gamma \beta};
      \een
      Hence
      \ben
  \label{eq::precondition}
      \lefteqn{\half \up^T \hat\Gamma \up 
\leq \ip{\up, \hat\gamma - \hat\Gamma \beta} +
 \lambda \left(\onenorm{\beta} -  \onenorm{\hat\beta} \right)}\\
  \nonumber
  & \leq &       \lambda
\left(\onenorm{\beta}-  \onenorm{\hat\beta} \right) +
         \norm{\hat\gamma - \hat\Gamma \beta}_{\infty}
           \onenorm{\upsilon} \text{ and thus by
           \eqref{eq::halflambda}}\\
  \label{eq::finalupperbound}
\quad \up^T \hat\Gamma \up 
& \leq &  \lambda \left(2 \onenorm{\beta} - 2 \onenorm{\hat\beta} \right) +
  \frac{\lambda}{2} \onenorm{\upsilon} \leq \lambda \half \left(5 \onenorm{\upsilon_{S}} - 3\onenorm{\upsilon_{S^c}}\right),
\end{eqnarray}
where by the triangle inequality, and $\beta_{\Sc} = 0$, we have
\begin{eqnarray} 
\nonumber
2 \onenorm{\beta} -  2 \onenorm{\hat\beta} + \half \onenorm{\upsilon}
& = &
2 \onenorm{\beta_S} - 2 \onenorm{\hat\beta_{S}} -
2 \onenorm{\upsilon_{\Sc}} + \half\onenorm{\upsilon_S} +\half \onenorm{\upsilon_{S^c}} \\
& \leq &
\nonumber
2 \onenorm{\up_{S}} -2 \onenorm{\up_{\Sc}} + \half \onenorm{\up_S}
+ \half \onenorm{\up_{S^c}} \\ 
& \leq & 
\label{eq::magic-number-2}
 \half \left(5 \onenorm{\upsilon_{S}} - 3\onenorm{\upsilon_{S^c}}\right).
\end{eqnarray}
We now give a lower bound on the LHS of~\eqref{eq::precondition},
applying the lower-$\RE$ condition as in Definition~\ref{def::lowRE},
\ben
\nonumber
\up^T \hat\Gamma \up
& \ge &
\alpha \twonorm{\up}^2 - \tau \onenorm{\up}^2
\ge  - \tau \onenorm{\up}^2\\
\text{ and hence } \;
- \up^T \hat\Gamma \up
& \le & 
\nonumber
\onenorm{\up}^2 \tau \le \onenorm{\up} 2 b_1 \tau\\
& \le &
\label{eq::magic-number-neg}
\frac{\lambda}{2} \onenorm{\up}
=\half \lambda (\onenorm{\up_S} + \onenorm{\up_{\Sc}}),
\een
where we use the assumption that 
$\onenorm{\up} \le \onenorm{\hat{\beta}} + \onenorm{\beta}
\le 2 b_1  \quad \text{ and } \;\;
2 b_1 \tau \le \half \lambda$,
which holds by the triangle inequality and the fact that both 
$\hat{\beta}$ and $\beta$ have $\ell_1$ norm being bounded by
$b_1$ in view of~\eqref{eq::origin} and the assumption.
Hence 
\bens
0 
& \le &
 - \up^T \hat\Gamma \up + \frac{5}{2} \lambda \onenorm{\upsilon_{S}} -
\frac{3}{2} 
\lambda \onenorm{\upsilon_{S^c}} \\
\nonumber
& \le &
\half \lambda \onenorm{\upsilon_{S}} + \half \lambda \onenorm{\upsilon_{S^c}}
+ \frac{5}{2}\lambda \onenorm{\upsilon_{S}} - \frac{3}{2} \lambda
\onenorm{\upsilon_{S^c}} \le 3 \lambda \onenorm{\upsilon_{S}} - \lambda \onenorm{\upsilon_{S^c}}.
\eens
by~\eqref{eq::finalupperbound} and~\eqref{eq::magic-number-neg}; we have $\onenorm{\upsilon_{S^c}} \le 3 \onenorm{\upsilon_{S}}$
and the lemma holds.
\end{proof}

\begin{proofof}{Lemma~\ref{lemma::detLassonew}}
  The proof is deterministic and works for all $j \in [n]$.
  Hence fix $j \in [n]$ and denote by $\up = \hat{\beta}^{(j)} - \beta^{j*}$. 
By \eqref{eq::onenorm}, we have $\onenorm{\up}^2 \le 16 d_0
\twonorm{\up}^2$ for all $j$.
Moreover, we have by the lower-$\RE$ condition as in
Definition~\ref{def::lowRE}, for all $j \in [n]$,
\ben
\label{eq::prelow}
\up^T \hat\Gamma \up
& \ge & 
\alpha \twonorm{\up}^2 - \tau \onenorm{\up}^2 \ge 
(\alpha  - 16 d_0 \tau) \twonorm{\up}^2 \ge  \frac{13 \alpha}{15} \twonorm{\up}^2,
\een
where the  last inequality follows from the assumption that 
$16 d_0 \tau \le \frac{2\alpha}{15}$.
Combining the bounds in \eqref{eq::prelow}, \eqref{eq::onenorm} and \eqref{eq::finalupperbound},
we have
\bens
 \frac{13 \alpha}{15} \twonorm{\up}^2
&\le &
\up^T \hat\Gamma \up
\le \frac{5}{2} \lambda  \onenorm{\upsilon_S}
\le \frac{5}{2}  \lambda \sqrt{d_0} \twonorm{\up}.
\eens
Thus we have for all $j$,
\bens
\twonorm{\up} \le \frac{75}{26 \alpha} \lambda  \sqrt{d_0}  \le
\frac{3 \lambda \sqrt{d_0}}{\alpha} \; \;
\text{ and } \;\;
\onenorm{\up} \le  4 \sqrt{d_0}\twonorm{\up} \le
\frac{12 \lambda d_0 }{\alpha}
\eens 
The Lemma is thus proved.
\end{proofof}

\section{Proof of Theorem~\ref{thm::samplesize} }
\label{sec::appendsamplesize}
We first state results on matrix $A_{qh}^{\diamond}, q, h \in
\Sp^{n-1}$, from which Theorems~\ref{thm::mainop2} and~\ref{thm::uninorm2} follow 
respectively.
Recall
\ben
\label{eq::defineADMlocal}
A_{qh}^{\diamond}
& = &
\half \sum_{k=1}^m \sum_{i\not= j} u^k_i u^k_j  (q_i h_{j}  + q_j h_i) 
(c_i c_j^T) \otimes (d_k  d_k^T)
\een
\begin{theorem}
  \label{thm::mainop}
  Let $A^{\diamond}_{q, h}$ be as defined in \eqref{eq::defineADMlocal}.
Then for all  $q, h \in  \Sp^{n-1}$,
\bens
\twonorm{A_{qh}^{\diamond}} \le \twonorm{A_0}\twonorm{B_0}.
\eens
\end{theorem}

\begin{theorem}
  \label{thm::uninorm-intro}
  Let $s_0 \in[ n]$.  
Suppose $\sum_{j=1}^m a_{jj}^2 p_j^2 =\Omega\big(a^2_{\infty} \log (n
\vee m)\big)$.
Then on event $\F_0^c$,  where $\prob{\F_0^c} \ge 1-{4}/{(n \vee
  m)^4}$, for $N  \le \twonorm{A_0}  a_{\infty} \sum_{s=1}^m  p_s^4$,
\ben
\nonumber
&& \sup_{q, h \in \Sp^{n-1}, s_0-\sparse}  
\fnorm{A^{\diamond}_{qh}} \le W \cdot \twonorm{B_0}
\twonorm{A_0}^{1/2}
\; \text{ where }\; W \asymp \big(a_{\infty} \sum_{s=1}^m p_s^2\big)^{1/2} \\
\label{eq::Aqbounds}
&&\quad \quad \quad  
+ \psi_B(s_0)\big(\twonorm{A_0} \log (n \vee m) \big)^{1/2} + \psi_B(s_0)\big(N
\log (n \vee m)\big)^{1/4}
\een
where $\psi_B(s_0) = {\rho_{\max}(s_0, (\abs{b_{ij}}))}/{\twonorm{B_0}}$ is as in Definition~\ref{def::C0}. 
\end{theorem}

\begin{proofof2}
Denote by $\N$ the 
$\ve$-net for $\Sp^{n-1} \cap E$ as constructed in
Lemma~\ref{lemma::net}.
Hence on event $\F_0^c$, by Theorem~\ref{thm::uninorm-intro}, we have
for $N = a_{\infty} \twonorm{A_0}  \log (n \vee m) \sum_{s=1}^m p_s^4$
\bens
&& \sup_{q, h \in \Sp^{n-1} \cap E} \fnorm{A^{\diamond}_{qh}} \le \twonorm{B_0}   \twonorm{A_0}^{1/2}  W
=: \mu_f \;\; \text{ where} \\
&& W \asymp \big(a_{\infty} \sum_{s=1}^m p_s^2\big)^{1/2} +
\psi_B(s_0) \big(N \log (n \vee m)  \big)^{1/4} =: V_1 + V_2
\eens
since for the second term in~\eqref{eq::Aqbounds}, we have
$\psi_B(s_0) \big(\twonorm{A_0} \log (n \vee m) \big)^{1/2} \le \big(s_0 \twonorm{A_0} \log (n \vee m)
\big)^{1/2} \le \big(a_{\infty} \sum_{s=1}^m p_s^2\big)^{1/2}$
by condition~\eqref{eq::sample1local}, where $\psi_B(s_0) =
O(\sqrt{s_0})$.

By Theorem~\ref{thm::mainop}, we have 
$\sup_{q, h \in \Sp^{n-1}} \norm{A_{q, h}^{\diamond}}_2 \le
\twonorm{A_0}\twonorm{B_0}$.

Hence, we have by Theorem~\ref{thm::HW} and the union bound, for any $t > 0$,  
\bens
\lefteqn{\prob{\left\{\exists q,h \in \N,  \abs{Z^T A^{\diamond}_{qh} Z - \E(Z^T A^{\diamond}_{qh}
        Z | U) } > t \right\} | U \in \F_0^c}  }\\
&\le &
2 \size{\N}^2 \exp \big(-c\min\big(\frac{t^2}{\mu_f^2}, \frac{t}{\twonorm{A_0}\twonorm{B_0}} \big)\big) 
     \eens
Set for some absolute constants $C_1, C_2$,
\ben
\label{eq::tauF1}
&& \tau_0 = C_1  \mu_f \sqrt{s_0  \log\big(\frac{3e n}{s_0\ve}\big)}  
+ C_2 s_0 \log \big(\frac{3e n}{s_0\ve}\big) 
\twonorm{A_0} \twonorm{B_0}
\een
Then for some absolute constant $c_1$ and $\F_0$ as defined in
Theorem~\ref{thm::uninorm-intro},
\bens
\lefteqn{\prob{\exists q,h \in \N, \abs{Z^T A^{\diamond}_{qh} Z - \E(Z^T 
        A^{\diamond}_{qh} Z | U) } > \tau_0} =: \prob{\F_1}} \\
& = & \E_U  \prob{\exists q, h \in \N,  \abs{Z^T A^{\diamond}_{qh} Z - \E(Z^T A^{\diamond}_{qh} Z | U) } > \tau_0 | U}\\
& \leq &  
\prob{\left\{\exists q,h \in \N,  \abs{Z^T A^{\diamond}_{qh} Z - \E(Z^T A^{\diamond}_{qh}
    Z | U) } > \tau_0 \right\} \cap \F_0^c} +  \prob{\F_0}  \\
& \le & \exp(- c_1 s_0 \log (3e n/(s_0 \ve))) + \prob{\F_0}.
\eens
On event $\F_1^c \cap \F_0^c$, we have by~\eqref{eq::sample1local} and~\eqref{eq::sample2local}, for $\norm{\M}_{\offd} = \sum_{k=1}^m a_{kk} p_k^2$,
\bens 
&& \lefteqn{\forall q, h \in \N \quad \abs{Z^T A^{\diamond}_{q h} Z-\E (Z^T A^{\diamond}_{q h} Z|
    U)}/ \big(\norm{\M}_{\offd} \twonorm{B_0} \big) \le
  \frac{\tau_0}{\big(\norm{\M}_{\offd} \twonorm{B_0}\big)} } \\
& \asymp &
\frac{(V_1 + V_2)\big(s_0 \log \big(\frac{3e n}{s_0\ve} \big)
  \twonorm{A_0}\big)^{1/2}}{\sum_{k=1}^m a_{kk} p_k^2}
+ \frac{s_0 \log\big(\frac{3e n}{s_0\ve} \big) \twonorm{A_0}}{\sum_{k=1}^m
a_{kk} p_k^2}\\
& \asymp & 
r_{\offd}(s_0) \eta_{A} + r_{\offd}(s_0) f_{\QA}\psi_B(s_0) +
r_{\offd}^2(s_0)=o(1) \text{ where}  \\
&& 
\big(a_{\infty} \sum_{s=1}^m p_s^2\big)^{1/2} \big(s_0 \log \big(\frac{3e n}{s_0\ve} \big) 
\twonorm{A_0}\big)^{1/2}
/{\big(\sum_{k=1}^m  a_{kk} p_k^2\big)}
\asymp \eta_{A} r_{\offd}(s_0), \\
&&
\frac{\psi_B(s_0) \big(N \log (n \vee m)\big)^{1/4} }{\big(\sum_{k=1}^m  a_{kk} p_k^2\big)}
\big(s_0 \log \big(\frac{3e n}{s_0\ve} \big) 
\twonorm{A_0}\big)^{1/2} 
\asymp  r_{\offd}(s_0) f_{\QA} \psi_B(s_0),
\eens
\text{ and }
${C_2 s_0 \log\big(\frac{3e n}{s_0\ve} \big)
  \twonorm{A_0}}/{\big(\sum_{k=1}^m a_{kk} p_k^2\big)} \asymp
r_{\offd}^2(s_0) = o(r_{\offd}(s_0))$. Furthermore, we have $\prob{\F_1} \le \prob{\F_0} +
2\exp(- c_1 s_0 \log (3en/(s_0 \ve)))$. This completes the proof of Theorem~\ref{thm::samplesize}
\end{proofof2}

\section{Proof of Theorem~\ref{thm::mainlights}}
\label{sec::appendmainlights}
Let  $\DD(q,h)$ be the block-diagonal matrix with $k^{th}$ block along the 
diagonal being $(\DD_{ij}^{(k)}(q, h))_{i,j \le n}$, where 
\bens  
\DD^{(k)}_{ij}(q, h) 
& = & 
\half Z^T \big((q_i h_{j} +q_j h_{i} )  c_i c_j^T  \otimes (d_k 
  d_k^T)\big) Z  \; \; \text{ and hence} \\
\E \DD^{(k)}_{ij}(q,h) & = & \half \tr\big((q_i h_{j} + q_j h_i) c_i c_j^T  \otimes (d_k 
  d_k^T)\big)=
\half  a_{kk} (q_i h_j + q_j h_i) b_{ij}
\eens
Hence for $\tilde{a}^{k}_{ij}(q,h)  = \half  a_{kk} (q_i h_j + q_j
h_i) b_{ij}$ and $\E (u^k_i) = p_k$ for all $i$,
\bens
 \lefteqn{\E (Z^T  A^{\dm}_{q,h} Z |U) - \E (Z^T  A^{\dm}_{q,h} Z)}\\
 &= & 
   \mvec{U}^T \E \DD(q,h) \mvec{U}  - \E (\mvec{U}^T \E 
   \DD(q, h) \mvec{U}) \\
   &=: &
\sum_{k=1}^m \sum_{i=1}^n \sum_{j \not= i}^n 
 (u^k_i u^k_j - \E(u^k_i u^k_j))  \tilde{a}^{k}_{ij}(q,h) =:
 S_{\star}(q,h)
 \eens
for $S_{\star}(q, h)$ as defined in \eqref{eq::stardust}.
First we consider $q, h$ being fixed.
We state in Lemma~\ref{lemma::Bernmgf2} an estimate on the moment generating function of 
$S_{\star}(q, h)$ as defined in~\eqref{eq::stardust}, from which a
large deviation bound immediately follows.
We emphasize that Lemma \ref{lemma::Bernmgf2} holds for all vectors
$q, h \in \Sp^{n-1}$, rather than for sparse vectors only; When $q$
and $h$ are indeed $s_0$-sparse, then  $
\tilde{\rho}(s_0, \abs{B_0}) \le 2\rho_{\max}((2s_0) \wedge n,
(\abs{b_{ij}}))$ is used to replace $\sum_{i, j} \abs{b_{ij}}
\abs{q_i} \abs{h_j}$ appearing in~\eqref{eq::2sparse}.
\begin{lemma}
  \label{lemma::Bernmgf2}
  Let $E = \cup_{\abs{J} \le s_0} E_J$ for $0 < s_0 \le n$.
  Let $\abs{B_0} =  (\abs{b_{ij}})$.
  Denote by
  \ben
  \nonumber
 \breve\rho(s_0, \abs{B_{0}}) & := &
  \tilde\rho(s_0, \abs{B_{0}}) \wedge \twonorm{B_0}, \; \text{ where } \\
\label{eq::C0proof}
\tilde\rho(s_0, \abs{B_{0}})
& := &
\max_{q, h \in \Sp^{n-1}, s_0-\sparse} \sum_{i=1}^n
\sum_{j=1}^n\abs{b_{ij}} \abs{q_i}\abs{h_j}
\een
Then for any $q, h \in E \cap \Sp^{n-1}$, for
$\abs{\lambda} \le 1/\big(16 a_{\infty} \breve\rho(s_0, \abs{B_{0}})\big)$, we have
\bens 
&& \E \exp\big(\lambda S_{\star}(q, h)\big)
\le \exp\big(60 \lambda^2 a_{\infty}   \breve\rho(s_0,
\abs{B_{0}}) \tilde\rho(s_0, \abs{B_{0}})  \sum_k a_{kk} p_k^2 \big);\\
&& \lefteqn{\text{and for any } \; \; t > 0, \quad\prob{\abs{S_{\star}(q,h)} > t}  \le }\\
&&2\exp\big(- c\min\big(\frac{t^2}{a_{\infty} \breve\rho(s_0, \abs{B_{0}})
  \tilde\rho(s_0, \abs{B_{0}})  \sum_{k=1}^m  a_{kk}p_k^2},
\frac{t}{a_{\infty} \breve\rho(s_0, \abs{B_{0}})}  \big) \big)
\eens 
\end{lemma}

\begin{proof}
Let $\tilde{A} = \E \DD(q,h) = (\tilde{a}^{k}_{ij})_{k=1, \ldots, m}$ 
be the block-diagonal matrix with $k^{th}$ block along the 
diagonal being $\tilde{A}^{(k)} := (\tilde{a}^{k}_{ij})_{i,j \le n}, k
=1, \ldots, m$, where
$\tilde{a}^{k}_{ij}:= \tilde{a}^{k}_{ij}(q,h) = \half a_{kk} b_{ij}(q_i h_{j}+q_j h_i)$;
Then 
\ben
\label{eq::tildeA}
\tilde{A} = \E \DD(q, h)  & = & \half \diag(A_0) \otimes \offd(B_0 \circ 
((q \otimes h) + (h \otimes q))),
\een
where the expectation is taken componentwise.

Now we compute for
$\tilde{A}$ as defined in \eqref{eq::tildeA}, the quantity $D_{\max}
:= \shnorm{\tilde{A}}_{\infty} \vee \shnorm{\tilde{A}}_1$; First,
 for all $q, h \in \Sp^{n-1}$, 
\bens
D_{\max} \le \max_{k} \frac{a_{kk}}{2} \max_{i}  \big(\abs{q_i} \sum_{j\not=i}
\abs{h_j} \abs{ b_{ij}} + \abs{h_i} \sum_{j\not=i}
\abs{q_j} \abs{ b_{ij}} \big) \le  a_{\infty} \max_{i}
\norm{\vecb^{(i)}}_2
\eens
where for $B_0  = [\vecb^{(1)}, \ldots, \vecb^{(n)}]$  and $q, h \in
\Sp^{n-1}$, $\abs{q_i} \le 1$ for all $i$ and
\bens  
\sum_ {j\not= i}^m  \abs{b_{ij}}\abs{h_j}
& \le &
\big(\sum_{j \not= i} b_{ij}^2\big)^{1/2}
\big(\sum_{j\not=i} h_j^2\big)^{1/2} \le \max_{i} \norm{\vecb^{(i)}}_2  \le \twonorm{B_0};
\eens
On the other hand, $\max_{i} \abs{q_i} \sum_{j\not=i}\abs{h_j} \abs{b_{ij}} \le 
\sum_{i} \abs{q_i} \sum_{j\not=i}\abs{h_j} \abs{b_{ij}}$ and hence
\bens
\forall h, q \in \Sp^{n-1} \cap E, \; \half \big(\max_{i} \abs{q_i} \sum_{j\not=i} \abs{h_j} \abs{ b_{ij}}
  + \max_{i} \abs{h_i} \sum_{j\not=i}
  \abs{q_j} \abs{ b_{ij}} \big) \le  \tilde\rho(s_0, \abs{B_{0}})
  \eens
  \text{ and hence}
$D_{\max} \le a_{\infty} (\twonorm{B_0} \wedge  \tilde\rho(s_0, \abs{B_{0}}))
  =:  a_{\infty} \breve\rho(s_0, \abs{B_{0}})$.
Hence by Theorem~\ref{thm::Bernmgf}, we have for 
$\abs{\lambda} \le \inv{16a_{\infty} \breve\rho(s_0, \abs{B_{0}})}\le \inv{16 D_{\max}}$,
\bens
\lefteqn{\E \exp(\lambda S_{\star})
\le \prod_{k=1}^m   \exp\big(36.5 \lambda^2 D_{\max} e^{8 \abs{\lambda} D_{\max}}
  \sum_{i\not=j} \abs{\tilde{a}^{k}_{ij}} \E(u^k_j)\E(u^k_i) \big)} \\
& =&
\exp\big(60 \lambda^2 D_{\max} \sum_{k=1}^m \sum_{i\not=j}^n 
  \abs{\tilde{a}^{k}_{ij}} p_k^2\big)  \le \exp( \lambda^2 60D_{\max}
  T)  \; \; \text{ where }  \\
&& \lefteqn{
\sum_{k=1}^m \sum_{i\not=j}^n \abs{\tilde{a}^{k}_{ij}} p_k^2  \le
\sum_{k=1}^m  \frac{a_{kk} p_k^2}{2}  \sum_{i\not=j}^n\abs{b_{ij}}(\abs{q_i}\abs{h_j}
+\abs{q_j}\abs{h_i}) }\\
& \le & \tilde\rho(s_0, \abs{B_{0}}) \sum_{k=1}^m  a_{kk}p_k^2 =: T
\eens
Let $t>0$. Optimizing over $0 <  \lambda <  \inv{16a_{\infty} (\twonorm{B_0} \wedge
  \tilde\rho(s_0, \abs{B_{0}}))}$, we have
\bens 
\lefteqn{\prob{S_{\star} > t}  \le \frac{\E \exp(\lambda
    S_{\star})}{e^{\lambda t}} \le \exp(-\lambda t+ \lambda^2  60 D_{\max} T)} \\
&\le & \exp\big(- c\min\big(\frac{t^2}{a_{\infty} \breve\rho(s_0,
  \abs{B_{0}})  \tilde\rho(s_0, \abs{B_{0}}) \sum_{k=1}^m
  a_{kk}p_k^2}, \frac{t}{a_{\infty} \breve\rho(s_0,\abs{B_{0}})}  \big) \big) =: q_{\star};
\eens 
Repeating the same arguments, we have for $t>0$,
\bens 
\prob{S_{\star} < -t} & = &  \prob{-S_{\star} > t} \le q_{\star}
\eens 
hence the lemma is proved by combining these two events.
\end{proof}
We now derive an upper bound on $\tilde{\rho}(s_0, \abs{B_0})$ in 
Lemma~\ref{lemma::S2bounds}.
\begin{lemma}
  \label{lemma::S2bounds}
  Denote by $\abs{q} = (\abs{q_1}, \ldots, \abs{q_n})$ the vector with 
  absolute values of $q_j, j =1, \ldots, n$.
  Let $$\rho_{\min}(s_0, \abs{B_0}) := \min_{q \in \Sp^{n-1};
    s_0-\text{sparse}} \; \; \abs{q}^T \abs{B_0} \abs{q}.$$
  Let $q, h \in E \cap \Sp^{n-1}$ be $s_0$-sparse.
  Then for $\tilde\rho(s_0, \abs{B_{0}})$ as in Definition~\eqref{eq::2sparse}
  \ben
\label{eq::2sparse}
  \lefteqn{  \tilde\rho(s_0, \abs{B_{0}})
    := \sup_{h, q \in E \cap \Sp^{n-1}}  \abs{h}^T \abs{B_0}\abs{q}} \\
  \nonumber
& \le & [2 \rho_{\max}((2s_0) \wedge n, (\abs{b_{ij}}))- \rho_{\min}(s_0, \abs{B_0}) ] \wedge
[\sqrt{s_0} \twonorm{B_0}].
\een
\end{lemma}

\begin{proof}
Note that
$\twonorm{\abs{q} + \abs{h}} \le \twonorm{\abs{q} } +\twonorm{\abs{h}} = 2$.
Thus we have
\ben
\nonumber
(\abs{q} + \abs{h})^T \abs{B_0}(\abs{q} + \abs{h}) & = &
  \abs{q}^T \abs{B_0}\abs{q} + 2 \abs{h}^T \abs{B_0} \abs{q} +
  \abs{h}^T \abs{B_0}\abs{h} \\
  \nonumber
\text{and hence}
\abs{h}^T \abs{B_0}\abs{q}
& = &
\half\big( (\abs{q} + \abs{h})^T \abs{B_0}(\abs{q} + \abs{h}) - \abs{q}^T 
\abs{B_0}\abs{q} -\abs{h}^T \abs{B_0}\abs{h} \big)\\
\label{eq::0sparse}
& \le & 2 \rho_{\max}((2s_0) \wedge n, (\abs{b_{ij}}))- \rho_{\min}(s_0, \abs{B_0})
\een
for $q, h \in \Sp^{n-1}$ that are $s_0-\sparse$;
On the other hand, $\forall h \in \Sp^{n-1}$ and  $\forall i$,
$\sum_{j=1}^n  \abs{b_{ij}}\abs{h_j} \le  \twonorm{\vecb^{(i)}}  \twonorm{h} =  \twonorm{\vecb^{(i)}}$
and hence
\ben
\nonumber
\max_{q, h\in \Sp^{n-1}, s_0-\sparse}
\sum_{i=1}^n  \abs{ q_i} \sum_{j=1}^n\abs{b_{ij}}\abs{h_j}
& \le &
\max_{q \in \Sp^{n-1}, s_0-\sparse} \sum_{i=1}^n  \abs{q_i}
\twonorm{\vecb^{(i)}} \\
\label{eq::1sparse}
& \le & \sqrt{s_0} \twonorm{B_0}
\een
where $\vecb^{(1)}, \ldots, \vecb^{(n)}$ are column (row) vectors of  
symmetric matrix $B_0 \succ 0$. 
Thus combining \eqref{eq::0sparse} and \eqref{eq::1sparse}, we have
\eqref{eq::2sparse}.
\end{proof}

\subsection{Proof of Theorem~\ref{thm::mainlights}}
Denote by $\N$ the 
$\ve$-net for $\Sp^{n-1} \cap E$ as constructed in
Lemma~\ref{lemma::net}.
Now suppose $q, h \in \Sp^{n-1}$ are $s_0$-sparse.
  Denote by $\breve\rho(s_0) = \breve\rho(s_0, \abs{B_{0}}) = 
\tilde\rho(s_0, \abs{B_{0}}) \wedge \twonorm{B_0}$; we use the 
shorthand notation $\tilde\rho(s_0) = \tilde\rho(s_0, \abs{B_0}) \le 2 
\rho_{\max}((2s_0 )\wedge n, \abs{B_0})$ as 
defined in \eqref{eq::C0proof}. Let 
$$\tau'  = C_4 a_{\infty} (\twonorm{B_0} \psi_B((2s_0) \wedge n)) s_0 \log
(3en/(s_0 \ve)).$$s
Then by Lemma \ref{lemma::Bernmgf2}, we have for $C_4$ large enough,
$\psi_B((2s_0) \wedge n) =O(\sqrt{2s_0})$, and letting $d := (2s_0) \wedge n$,
\bens
\lefteqn{\prob{\exists q, h \in \N, \abs{S_{\star}(q, h)} \ge \tau'}
  =: \prob{\F_2}
  \le \abs{\N}^2 \cdot} \\
&&  \exp\big(- c\big(\frac{\big(C_4 a_{\infty} \twonorm{B_0}
        \psi_{B}(d)  s_0 \log \big(\frac{3en}{s_0 \ve}\big) \big)^2 }
    {a_{\infty}(\breve\rho(s_0)) \tilde\rho(s_0)\sum_{k=1}^m
      a_{kk}p_k^2} \wedge
    \frac{C_4 \twonorm{B_0} \psi_{B}(d) s_0 \log \big(\frac{3en}{s_0
        \ve}\big) }{\breve\rho(s_0) }
    \big)\big)
\eens
where by Lemma~\ref{lemma::S2bounds}
and condition~\eqref{eq::samplecrux}, 
\bens
\lefteqn{\prob{\exists q, h \in \N, \abs{S_{\star}(q, h)} \ge \tau'} \le } \\
& &
\abs{\N}^2 \exp\big(- C \big(s_0 \log\big(\frac{3en}{s_0 \ve}\big)\big)
\wedge \big(\psi_{B}(d) s_0 \log\big(\frac{3en}{s_0 \ve}\big)\big)\big) \\
& \le &
(3/\ve)^{2s_0} {n \choose s_0}^2 \exp\big(- C's_0
\log\big(\frac{3en}{s_0 \ve} \big) \big)
\le  \exp\big(-C'' s_0 \log\big(\frac{3en}{s_0 \ve} \big) \big)
 \hskip2pt \;\;\scriptstyle\Box
\eens

\subsection{Proof of Lemma~\ref{lemma::finalrate}}
\begin{proofof2}
For $\psi_B(s_0) \le \sqrt{s_0}$ as in
Definition~\ref{def::C0} and by Lemma~\ref{lemma::converse},
\bens 
\lefteqn{\psi_B(s_0)  r_{\offd}(s_0) f_{\QA} 
  \asymp
  \frac{\twonorm{A_0}^{3/4} a_{\infty}^{1/4}
    \psi_B(s_0) 
  \big(s_0 \log \big(\frac{3en}{s_0\ve}\big)\big)^{1/2}
  \big(\log (n \vee m)  \sum_{s=1}^m 
    p_s^4\big)^{1/4}}{\sum_{j=1}^m a_{jj} p_j^2}} \\
 & \le & \sqrt{p_{\max}}\sqrt{\eta_A}   (\psi_B(s_0))^{1/2}  r^{3/2}_{\offd}(s_0) \ell_{s_0, n}^{1/4}
 =  \sqrt{p_{\max}} r_{\offd}(s_0) \sqrt{\eta_A r_{\offd}(s_0)  \psi_B(s_0) \ell_{s_0, n}^{1/2}} \\
 & \le & \sqrt{p_{\max}} r_{\offd}(s_0)  (\ell_{s_0, n}^{1/2}  \eta_A+  r_{\offd}(s_0) 
 \psi_B(s_0)) \asymp \sqrt{p_{\max}} \big(\underline{r_{\offd}} \sqrt{s_0} +r_{\offd}^2 (s_0) 
\psi_B(s_0)\big)
\eens
where $r_{\offd}(s_0)$ and $\ul{r_{\offd}}$ are as in~\eqref{eq::offdrate} and~\eqref{eq::paritydual} respectively.
\end{proofof2}

\section{Proof of Theorem~\ref{thm::mainop}}
\label{sec::bernchaos}
We will  prove a uniform deterministic bound on
$\norm{A^{\dm}_{q, h}}_2$ in this section. The proof techniques
developed in this section may be of independent  interests for
analyzing tensor quadratic forms.
As mentioned, our proof  for Theorem~\ref{thm::mainop} will go through
if one replaces $u^k, k=1, \ldots, m$ by independent Gaussian random vectors; however, the
statement will be probabilistic subject to an additional logarithmic
factor. Such a result may be of independent interests, as more
generally, the mask matrix $U$ may not be constrained to the family of
Bernoulli random matrices.  For example, one may consider $U$ as a
matrix with arbitrary positive coefficients belonging to $[0, 1]$. In
particular, Lemma~\ref{lemma::chaosop} holds for general
block-diagonal matrices with bounded operator norm, which can be deterministic.

Let $\{u^{k}, k=1, \ldots, m\}$ be the column vectors of the mask matrix 
$U = [u^{1} | \ldots | u^{m}]$. 
Recall that the nuclear norm or trace norm of $d \times d$ matrix $X$ is defined as
$\norm{X}_{*} :=\onenorm{s(X)} = \sum_{i=1}^d s_i(X) = \tr(\sqrt{X^TX})$,
where $\sqrt{X^TX}$ represents the unique positive-semidefinite matrix
$C$ such that $C^2 = X^T X$ and $s(X) := (s_i(X))_{i=1}^d$ denotes the vector of singular values of $X$.
For positive-semidefinite matrix $A \succeq 0$, clearly, $\sqrt{A^TA}
= \sqrt{A^2} = A$, and
\bens 
\norm{A}_{*} :=\onenorm{s(A)} = \tr(A) = \sum_{i=1}^m \lambda_i(A),
\eens
where $\lambda_i(A)$ are eigenvalues of $A$.
First, denote by $ \upsilon_{ij}^{(k)} = c_i c_j^T
\otimes (d_k   d_k^T) \in \R^{mn \times mn}$.
Denote by $\DD_0(w)$ the block diagonal matrix such that
on the $k^{th}$ block, we have for a fixed $w \in \Sp^{mn-1}$
\ben
\nonumber
\DD_{0,ij}^{(k)}(w)  & = & w^T \big(c_i c_j^T  \otimes (d_k 
  d_k^T)\big) w  =: w^T  \upsilon_{ij}^{(k)} w \text{ and hence} \\
\label{eq::defineDD0}
\DD^{(k)}_{0}(w) & = &
\big(\DD^{(k)}_{0, ij}(w)\big) \in \R^{n \times n}
\een
As a preparation, we first state a general result in
Lemma~\ref{lemma::chaosop} involving sum of tensor products,
specialized to our settings, as well as properties of matrix 
$\DD_0$ in Lemma~\ref{lemma::nuclearnorm}. We defer proofs of Lemmas~\ref{lemma::nuclearnorm}
  and~\ref{lemma::UKbound} to Section~\ref{sec::proofofnuclearnorm}.
\begin{lemma}
  \label{lemma::chaosop} 
Denote by $\tilde{\UU}$ a block-diagonal matrix with $0$s along 
 the diagonal, and on the $k^{th}$ block, we have a symmetric matrix
 $\tilde{\UU}^{(k)}= (\tilde{\UU}^{(k)}_{ij}) \in \R^{n \times n}$ with bounded operator norm.
Consider $\HH_0^{(k)} = \sum_{i \not=j} \tilde{\UU}^{(k)}_{ij} c_i
c_j^T$. Then for $\DD_0(w)$ as defined in~\eqref{eq::defineDD0},
 \bens 
 \norm{\sum_{k=1}^m \HH_0^{(k)} \otimes d_k  d_k^T}_2
 = \sup_{w \in \Sp^{mn-1}} \abs{\ip{\tilde{\UU},  \DD_0(w)}} \le
 \twonorm{\tilde\UU}
 \sup_{w \in \Sp^{mn-1}}\norm{\DD_0(w)}_{*}.
 \eens
\end{lemma}

\begin{lemma}
\label{lemma::UKbound}
Let $\tilde{\UU}$ be a block-diagonal matrix with $0$s along the 
diagonal, and $\tilde{\UU}^{(k)}(q, h)  \in \R^{n \times n}$ on the 
$k^{th}$ block along the diagonal such that
\bens
&& \forall k, \; \forall i \not=j, \quad \tilde{\UU}_{ij}^{(k)}(q, h) = 
\half (u^k_i u^k_j (q_i h_{j} + h_i  q_j)); \\
\text{Then,} &&
\forall q, h \in \Sp^{n-1}, \forall k, \quad 
\shnorm{\tilde{\UU}^{(k)}(q,h)}_2  \le  1.
\eens
\end{lemma}

\begin{lemma}
  \label{lemma::nuclearnorm}
 Let $\DD_0$ be defined as in \eqref{eq::defineDD0}.
Then $\DD_0 \succeq 0$ for all $w \in \R^{mn}$ and
$\norm{\DD_0(w)}_{*} =  w^T (B_0 \otimes A_0) w$; Moreover,
\bens
\sup_{w \in \Sp^{mn-1}} \norm{\DD_0(w)}_{*} & = & \twonorm{B_0} \twonorm{A_0}.
\eens
\end{lemma}

\begin{proofof}{Theorem~\ref{thm::mainop}}
Consider for arbitrary $q, h \in  \Sp^{n-1}$ and $\tilde{\UU} := \tilde{\UU}(q,h)$ as defined in Lemma~\ref{lemma::UKbound},
\bens 
\HH_0^{(k)}(q, h)
& = & \sum_{i=1}^n \sum_{j\not= i} \tilde{\UU}_{ij}^{(k)}(q, h) c_i
c_j^T.
\eens
Let $\DD_0(w)$ be as defined in \eqref{eq::defineDD0}.
Now  for arbitrary $q, h \in \Sp^{n-1}$, 
\bens  
A_{q, h}^{\dm}
& = & \sum_{k=1}^m \HH_0^{(k)}(q, h) \otimes d_k d_k^T
= \sum_{k=1}^m \sum_{j \not= i} \tilde{\UU}^{(k)}_{ij}(q, h) c_i c_j^T \otimes d_k d_k^T,
\eens
we have by Lemmas~\ref{lemma::chaosop}  and~\ref{lemma::UKbound}, 
\bens
\twonorm{A_{qh}^{\diamond}}
& = &
 \twonorm{\sum_{k=1}^m \HH_0^{(k)}(q,h) \otimes d_k  d_k^T} = \sup_{w \in  \Sp^{mn-1}}\abs{\ip{\tilde{\UU}(q, h), \DD_{0}(w)}} \\
 &\le &
 \twonorm{\tilde\UU(q,h)} \sup_{w \in \Sp^{mn-1}}\norm{\DD_0(w)}_{*} \le \sup_{w \in    \Sp^{mn-1}}\norm{\DD_0(w)}_{*},
\eens
where $\forall q, h \in \Sp^{n-1}$, we have $\twonorm{\tilde{\UU}(q, h) }  \le  1$;
Hence by Lemma~\ref{lemma::nuclearnorm},
\bens 
\sup_{q, h \in \Sp^{n-1}} \twonorm{A_{q, h}^{\dm}} 
&  \le &  \sup_{w \in \Sp^{mn-1}}\norm{\DD_0(w)}_{*} = \twonorm{A_0}\twonorm{B_0}.
\eens
\end{proofof}

\subsection{Proof of Lemmas~\ref{lemma::chaosop} to~\ref{lemma::nuclearnorm}} 
\label{sec::proofofnuclearnorm}
\begin{proofof}{Lemma~\ref{lemma::chaosop}}
Clearly $\HH_0^{(k)}$ is symmetric.
Denote by $M = \sum_{k=1}^m \HH_0^{(k)}\otimes d_k d_k^T$, then $M$ is
also symmetric.
Recall the operator norm of the symmetric matrix $M = \sum_{k=1}^m
\HH_0^{(k)} \otimes d_k d_k^T$ is the same as the spectral radius of
$M$, denoted by $\rho(M) := \{\max{\abs{\lambda}}, \lambda \text{ eigenvalue of }
M\}.$
Let $\DD_0(w)$ be as defined in \eqref{eq::defineDD0}.
Hence by definition,
\ben
\nonumber
\lefteqn{\twonorm{M}
  = \sup_{w \in \Sp^{mn-1}}
\abs{w^T \big(\sum_{k=1}^m\sum_{i\not=j}\tilde{\UU}^{(k)}_{ij} c_i
    c_j^T \otimes d_k d_k^T\big) w} =} \\
\label{eq::specradius}
 &  & \quad \sup_{w \in \Sp^{mn-1}} \abs{\sum_{k=1}^m
   \sum_{i\not=j}\tilde{\UU}^{(k)}_{ij}w^T (c_i c_j \otimes d_k d_k^T)
   w}   =  \sup_{w \in \Sp^{mn-1}} \abs{\sum_{k=1}^m
  \ip{\tilde{\UU}^{(k)},  \DD_0^{(k)}(w)}} \\
\label{eq::specradius2}
& & \quad\quad\quad =\sup_{w \in \Sp^{mn-1}}
 \abs{\ip{\tilde{\UU},  \DD_0(w)}}  \le \sup_{w \in \Sp^{mn-1}} \shnorm{\tilde{\UU}}_2
 \norm{\DD_0(w)}_{*} 
\een
where \eqref{eq::specradius} follows since $\diag(\tilde{\UU}) = 0$,
and in \eqref{eq::specradius2}, we use the fact that
 \bens 
 \forall w \in \Sp^{mn-1} \; \;
 \abs{\ip{\tilde{\UU},  \DD_0(w)}} = \abs{\sum_{k=1}^m 
   \ip{\tilde{\UU}^{(k)},  \DD_0^{(k)}(w)}}  
 \le \twonorm{\tilde{\UU}} \norm{\DD_0(w)}_{*}
 \eens
 by H\"{o}lder's inequality.
 See for example Exercise 10.4.2 in~\cite{Vers18}.
\end{proofof}

\begin{proofof}{Lemma~\ref{lemma::UKbound}}
  Denote by $\tilde{u}^k(q)= u^k \circ q$, where $q \in  \Sp^{n-1}$.
  Hence $\forall i \not=j$, we have $\tilde{\UU}_{ij}^{(k)}(q, h)  =
  \half \big(\tilde{u}^k(q)\otimes \tilde{u}^k(h)  +\tilde{u}^k(h)
  \otimes \tilde{u}^k(q)\big)_{i,j}$,
  and $\forall q, h \in  \Sp^{n-1}$, and $\forall k$,
\bens 
\twonorm{\tilde{\UU}^{(k)}(q,h)} & \le &
\half\twonorm{\offd(\tilde{u}^k(q) \otimes \tilde{u}^k(h))} +
\half\twonorm{\offd(\tilde{u}^k(h) \otimes \tilde{u}^k(q))} \\
& = &
\twonorm{\offd(\tilde{u}^k(q) \otimes \tilde{u}^k(h))} =
\fnorm{\offd(\tilde{u}^k(q) \otimes \tilde{u}^k(h))} \\
& \le &
\fnorm{\tilde{u}^k(q) \otimes \tilde{u}^k(h)}
= \sqrt{\tr\big((\tilde{u}^k(q) \otimes \tilde{u}^k(h))^T (\tilde{u}^k(q)
  \otimes \tilde{u}^k(h))\big)}\\
& = &
\sqrt{\tilde{u}^k(h)^T \tilde{u}^k(h)}
\sqrt{\tilde{u}^k(q)^T \tilde{u}^k(q)}
= \twonorm{u^k \circ h}  \twonorm{u^k \circ q} \le 1 
\eens
where for all $u^k, k=1, \ldots, m$ and $q \in \Sp^{n-1}$, we have for 
$\twonorm{\tilde{u}^k(q)}  = \twonorm{u^k \circ q} \le \twonorm{q} = 1$.
\end{proofof}

\begin{proofof}{Lemma~\ref{lemma::nuclearnorm}}
Fix $w \in \R^{mn}$. 
We break a vector $w$ on the sphere into  $n$ vectors $w^1, w^2, \ldots, w^n$, each of which has 
size $m$. We show that each block $(\DD_0^{(k)}(w)_{ij})_{i,j \le n}$ is
positive semidefinite (PSD) and hence $\DD_0(w)$ is PSD for all $w \in \R^{mn}$; Indeed,
\bens
\forall h \in \R^n \quad
h^T \DD_0^{(k)} h & = &
\sum_{ij} h_{i} h_j \sum_{s, t} c_{i,s}
\ip{w_s, d_k} c_{j, t} \ip{w_t, d_k} \\
& = & 
(\sum_{i} h_{i} \sum_{s} c_{i,s}
\ip{w_s, d_k}) (\sum_{j} h_{j} \sum_{t} c_{j, t} \ip{w_t, d_k}) \ge 0
\eens
For the nuclear norm, we use the fact that 
for all $w \in \R^{mn}$, $\DD_0(w) \succeq 0$, and
\bens
\norm{\DD_0(w)}_{*} & = &  \tr(\DD_0(w)) =\sum_{k=1}^m
\tr(\DD^{(k)}_0(w))  \\
& = & \sum_{k=1}^m \sum_{j=1}^n \DD^{(k)}_{0, jj}(w) 
=\sum_{k=1}^m \sum_{j=1}^n w^T (c_j c_j^T)  \otimes 
  (d_k   d_k^T) w \\
&= & w^T \big(\sum_{j=1}^n (c_j c_j^T)  \otimes \sum_{k=1}^m (d_k
d_k^T) \big) w = w^T (B_0 \otimes A_0) w; \text{ and hence} \\
\sup_{w \in \Sp^{mn-1}}
\norm{\DD_0(w)}_{*} & = & \sup_{w \in \Sp^{mn-1}} w^T 
(B_0 \otimes 
A_0) w= \twonorm{B_0} \twonorm{A_0}
\eens
and the last statement follows from the definition of operator
norm.
\end{proofof}

\section{Proof sketch for Theorem~\ref{thm::uninorm-intro}}
\label{sec::uninormskye}
Our analysis framework will work beyond cases considered in the
present work, namely, it will work in cases where random matrix $U$ follows
other distributions; for example, one may consider $U$ as a matrix
with positive coefficients, rather than $0, 1$s.
First, we rewrite the off-diagonal part of the quadratic form as follows:
\ben
\nonumber
\lefteqn{q^T \offd(\X \X^T) h  = \sum_{i \not=j} q_i h_{j} \ip{v^i \circ y^i, v^j \circ y^j}} \\ 
\nonumber
&= &
\label{eq::quad1}
Z^T\big(\sum_{i \not= j} q_i h_{j}  c_i c_j^T \otimes A_0^{1/2} 
\diag(v^i \otimes v^j) A_0^{1/2} \big) Z
\een
and recall for each row vector $y^i$ of $X =
B_0^{1/2} \Z A_0^{1/2}$, for $\Z$ as defined in \eqref{eq::missingdata}, 
we observe in $\X$ its sparse instance:  $\forall i=1,  \ldots, n,$
\ben 
\label{eq::rowmissing}
v^i \circ y^i, \text{ where } v^i_k  \sim \Ber(p_k), \forall k = 1, \ldots, m, 
\een 
and for two vectors $v^i, y^i \in \R^m$,  $v^i \circ y^i$ denote 
their Hadamard  product such that $(v^i \circ y^i)_{k} = v^{i}_k 
y^i_k = u^k_i x^k_i$.
Recall the symmetric matrix $A_{qh}^{\diamond}$ as defined in
\eqref{eq::defineADM}
is the average of the asymmetric versions:
$A_{qh}^{\diamond} = \half(A_{qh}^{\diamond}(\ell) + A_{qh}^{\diamond}(r))$
where we denote by 
\bens 
A^{\diamond}_{q,h}(\ell) 
& = &
\sum_{i =1}^n \sum_{j\not= i}^n
q_i h_{j} (c_i c_j^T) \otimes \big( A_0^{1/2}
  \diag(v^i \otimes v^j) A_0^{1/2} \big) \text{ and } \\
A^{\diamond}_{q,h}(r)  & = &
\sum_{i =1}^n \sum_{j\not= i}^n q_i h_j (c_j c_i^T) \otimes 
\big( A_0^{1/2} \diag(v^i \otimes v^j) A_0^{1/2} \big) = (A^{\diamond}_{q,h}(\ell) )^T
\eens
Recall  $\fnorm{A^{\diamond}_{q,h}}  \le \half
\big(\fnorm{A^{\diamond}_{q,h}(\ell)} + \fnorm{A^{\diamond}_{q,h}(r)}
\big)$ \; \text{ and hence} 
\bens
\fnorm{A^{\diamond}_{q,h}}^2   & \le &  \inv{4} \big(2 \fnorm{A^{\diamond}_{q,h}(\ell)}^2
  + 2 \fnorm{A^{\diamond}_{q,h}(r)}^2\big) =\fnorm{A^{\diamond}_{q,h}(r)}^2
\eens
First, we use the decomposition argument to express $\fnorm{A^{\diamond}_{q,h}(r)}^2$
as a summation over homogeneous polynomials of degree $2, 3, 4$
respectively. Hence
\ben
\nonumber
\lefteqn{\fnorm{A^{\diamond}_{q,h}(r)}^2 =}\\
\nonumber
& &
\sum_{i\not= j} \sum_{k \not=\ell}
  q_i h_{j} q_k h_{\ell} \tr\big((c_i c_j^T c_{\ell} c_k^T)\big) 
    \tr\big(A_0^{1/2} \diag(v^i \circ v^j) A_0 
      \diag(v^{\ell} \circ v^k) A_0^{1/2} \big) \\
      &  &
      \label{eq::summands}
\quad \quad \quad \quad
  =    \sum_{i, k}b_{ki} q_i q_k 
\sum_{j \not= i, k \not=\ell}
b_{j \ell} h_{j}  h_{\ell} \big( (v^i \circ v^j)^T (A_0 \circ A_0 ) 
(v^{\ell} \circ v^k)\big).
\een
For now, we have a quick summary of these random functions and their expectations.
We then prove the concentration of measure bounds for each homogeneous
polynomial respectively.
We now characterize the sums that involve all unique pairs, triples, and quadruples of Bernoulli 
random variables. It is understood that $(i,j)$ and $(k,\ell)$ are 
allowed to overlap in one or two vertices, but $i \not=j$ and $k \not=
\ell$. Here and in the sequel, we use $j\not= i \not= \ell \not=k$ to
denote that $i, j, k, \ell$ are all distinct, while $i\not=j\not=k$ denotes that
indices $i, j, k$ are distinct, and so on.
On the other hand, the conditions $i \not=j$ and $k \not= \ell$ do not exclude the
possibility that $k=j$ or $k=i$, or  $i = \ell$, or $j = \ell$, or
some combination of these.
We first introduce the following definitions.
\bit
\item
Fix $i \not=j$. In \eqref{eq::summands}, when an unordered pair of indices $(k,
 \ell)$ is chosen to be identical with an unordered pair $(i, j)$, we
 add an element in $W_2^{\diamond}$ resulting in a homogeneous positive polynomial of degree 2
\ben 
\label{eq::defineW2}
W_2^{\diamond}
 & = & \sum_{(i, j) }^{{n\choose 2}}
  w^{\be}_{(i,j)} (v_i \circ  v_j)^T \diag(A_0 \circ 
  A_0)(v_i\circ v_j) 
  \een 
  where the weight $w^{\be}_{(i,j)} \ge 0$ for all $i \not=j$ is to be
  defined in~\eqref{eq::defineWbe}.
\item
 In forming polynomial $W_3^{\diag}$, it is understood that the quadruple $(i, j, k, 
\ell)$~\eqref{eq::summands}, where $i \not= j$ and $k \not=\ell$ will 
collapse into a triple with three distinct indices, say, $i, j, k$. 
Indeed, suppose we first fix a pair of indices $(i, j)$, where $i
\not=j$ and for the second pair $(k, \ell)$, we pick a single new 
coordinate $k \not=i, j$, while, without loss of generality, fixing $\ell = i$;
We then add a set of elements in $W_3^{\diag}$ with 4 coefficients
denoted by $\Delta_{(i,j), (i,k)}$ for $i \not=j
\not=k$:
\ben
\label{eq::azerb}
&& \lefteqn{\quad \quad \Delta_{(i,j), (i,k)}  := 
b_{ii}  q_i^2 b_{jk} h_{j} h_{k} +  b_{jk} q_{j} q_{k} b_{ii}  h_i^2 +} \\
\nonumber 
&& \quad q_i h_{i} (b_{ik} q_k b_{ij} h_{j} + b_{ij} q_j b_{ik} h_{k}),\; \text{ and hence}\\
\nonumber
&& W_3^{\diag}  :=  \sum_{i =1}^n \sum_{j \not= i}^n \sum_{k \not= i,j}^n \Delta_{(i,j),
  (i,k)} \cdot (v_i \circ  v_j)^T \diag(A_0 \circ A_0)(v_i \circ v_k).
\een
In Section~\ref{sec::polyexpo}, we explain the counting strategy for 
this case and will analyze $W_3^{\diamond}$ in Lemma~\ref{lemma::triangleweight}. 
\item
In $W_4^{\diag}$, we have all indices $i, j, k, \ell$ being distinct, namely,
\bens
\label{eq::defineW4}
W_4^{\diag} :=
\sum_{i\not= k}  b_{ki}  q_i q_k 
\sum_{j\not= i \not= \ell \not=k}
b_{j\ell} h_{j} h_{\ell}
(v^i \circ  v^j)^T \diag(A_0 \circ A_0)(v^k \circ v^\ell);
\eens
\eit
Thus we have $\fnorm{A^{\diamond}_{q,h}(r)}^2  =
  W_2^{\diamond} +   W_3^\diag+W_4^{\diag} + W_4^{\diamond}$, where
\bens
\label{eq::W4dm}
W_4^{\diamond} = 
\sum_{i=1}^n \sum_{k=1}^n  b_{ki}  q_i q_k \sum_{i\not= j, k\not=\ell} 
b_{j\ell} h_{j} h_{\ell}(v^i \circ  v^j)^T \offd(A_0 \circ  A_0)(v^k
\circ v^\ell)
\eens
In summary, we have the following bounds for the expected values of each 
component in Lemma~\ref{lemma::mean}, which follows immediately from
Lemmas~\ref{lemma::W4dm},~\ref{lemma::W2diag},~\ref{lemma::triangleweight}, and~\ref{lemma::W4distinct}.
We prove Lemma~\ref{lemma::mean} in Section~\ref{sec::proofofmean}.
\begin{lemma}
  \label{lemma::mean}
  Denote by $b_{\infty} = \max_{j} b_{jj}$.  For all $q, h \in
  \Sp^{n-1}$, we have
\bens
 \abs{\E W_2^{\diamond} }& \le & 2 b_{\infty}^2 \sum_{i=1}^m  a_{ii}^2
 p_i^2, \quad
 \abs{\E W_4^{\diamond}} \le  \twonorm{B_0}^2 \sum_{i\not=j}  a_{ij}^2 p_i^2 p_j^2,\\
\abs{\E W_3^{\diag}} & \le & (2 b_{\infty}^2 + 2 b_{\infty} \twonorm{B_0}+ 2 \twonorm{B_0}^2)\sum_{i=1}^m  a_{ii}^2 p_i^3, \\
\abs{\E W_4^{\diag}} & \le & (2 b_{\infty}^2 +2 b_{\infty} \twonorm{B_0}+
6\twonorm{B_0}^2)\sum_{i=1}^m  a_{ii}^2 p_i^4.
\eens
\end{lemma}

\subsection{Proof of Theorem~\ref{thm::uninorm-intro}}
\label{sec::proofofuninorm}
We first state the large deviation bounds for $W_2^{\diamond}$,
$W_3^{\diag}$ and $W_4^{\diag}$, which are bounded in a similar manner.
As a brief summary of the number of unique events (or unique
polynomial of order 2, 3, 4) in the summation
$\sum_{i, k} \sum_{i \not=j, k \not= \ell} w_{i\not= j, k \not= \ell}
\sum_{t=1}^m a^2_{tt} u^t_i u^t_j
u^t_{k}  u^t_{\ell},$ where the labels are allowed to repeat, we end up with the following
categories:
\bens
\# W_2^{\diamond} = \half n(n-1) && \text{ with coefficients}\; \;
w^{\be}_{(i,j)} = b_{ii} q_i^2 b_{jj} h_{j}^2 \\
&&+b_{jj} q_j^2 b_{ii} h_{i}^2 + 2 b_{ij}^2 q_i q_j h_{i}  h_{j} \\ 
\# W_3(\diag) = 
\frac{n(n-1)(n-2)}{3!} &&
\; \text{with coefficients} \;
b_{ii} q_i^2 b_{jk}h_{j} h_{k} +
b_{jk} q_{j} q_{k} b_{ii} h_i^2 \\
&& + q_i h_i (b_{ik} q_k b_{ij} h_{j} + b_{ij} b_{ik} h_k q_{j}) \\
&& \text{ and their $3!$ permutations} \\
\# W_4(\diag) = \frac{(n)(n-1)(n-2)(n-3)}{4!} \;
&&
\text{with 24 coefficients which} \\
&&\text{we do not enumerate here}
\eens
For a unique pair $(i, j), i \not=j$, the coefficient 
corresponding to the unique quadratic term
$\big((v^i \circ v^j)^T \diag(A_0 \circ A_0) (v^i \circ
  v^j)\big)$ is
\ben
\label{eq::defineWbe}
0 \le w^{\be}_{(i, j)} & =& b_{ii} b_{jj} q_i^2  h_{j}^2 +b_{ii} b_{jj} q_j^2 h_{i}^2 +
2 b_{ij}^2 q_i q_j h_i h_j\\
\nonumber
& \le & 2 b_{ii} b_{jj} q_i^2  h_{j}^2 + 2b_{ii} b_{jj} q_j^2 h_{i}^2
\; \text{for } \; B_0 \succeq 0.
\een

Exploiting symmetry, we rewrite
\ben
\label{eq::defineW2diag}
W_2^{\diamond} := 
\sum_{(i\not= j)}^{{n \choose 2}}  w^{\be}_{(i,j)} \sum_{s=1}^m a_{ss}^2 u^s_i u^s_j 
 \;\; \text{ where}\;\; 0 \le \sum_{(i\not= j)}^{{n \choose 2}}  w^{\be}_{(i,j)} \le 2 b_{\infty}^2.
 \een
Thus our starting point is to obtain the large deviation bound for each
such linear term $\sum_{s=1}^m a_{ss}^2 u^s_i u^s_j$,
and then put together a large deviation bound for
$W_2^{\diamond}$ from its mean using  the union bound, as 
well as the upper bound on the total weight $\sum_{i
  \not=j}w^{\be}_{(i,j)} \le 2 b_{\infty}^2$ as in
\eqref{eq::defineW2diag}; cf. Lemma~\ref{lemma::W2diag} 
\begin{lemma}{\textnormal{($W_2^{\diamond}$ bound)}}
\label{lemma::W2devi}
Let $W_2^{\diamond}$ be as defined in~\eqref{eq::defineW2diag}.
Denote by 
\ben 
\label{eq::defS2star}
\quad \abs{ S_2^{\star}} := \max_{i \not=j} \abs{S_2^{\star}(i, j)}
\; \text{ where }  \; S_2^{\star}(i, j) := \sum_{s=1}^m a_{ss}^2 (u^s_i u^s_j  - p_s^2).
\een
Suppose that for some absolute constant $C_a$, $\sum_{s=1}^m a_{ss}^2 p_s^2 \ge 16 C_a^2 a^2_{\infty} \log (n \vee m)$
and  $\tau_2 =  C_a a_{\infty} \big(\log (n \vee m) \sum_{s=1}^m a_{ss}^2 p_s^2\big)^{1/2}$.
Then $\prob{ \abs{S_2^{\star}} > \tau_2} := \prob{\event_2} \le  \inv{(n \vee m)^4}$.
On event $\event_2^c$, we have $\forall q, h \in \Sp^{n-1}$,
\bens
\abs{W_2^{\diamond} - \E W_2^{\diamond}} \le \sum_{(i\not= j)}  w^{\be}_{(i,j)} \abs{S_2^{\star}(i,j)} \le 2
b_{\infty}^2 \tau_2 \le  \half b_{\infty}^2 \sum_{s=1}^m a_{ss}^2 p_s^2
\eens
\end{lemma}

It follows from the proof of Lemma~\ref{lemma::W2devi} that for
some $0< \ve < 1/2$,
$$(1-\ve) \E W_2^{\diamond}
\le W_2^{\diamond} \le (1+\ve)\E W_2^{\diamond},$$
so long as $\sum_{s=1}^m a_{ss}^2 p_s^2 =\Omega( a^2_{\infty} \log (n
\vee m))$;
This is not surprising given that $W_2^{\diamond}$ is a positive polynomial
of degree $2$, which is known to have strong concentration.
Unfortunately, although the dominating term 
$W_2^{\diamond}$ has non-negative coefficients $w^{\be}_{(i,j)}$ with respect to each unique
$(i, j)$ pair and their corresponding linear term $\sum_{s=1}^m a_{ss} u^s_i u^s_j$,
the same property does not hold for others.
We exploit crucially an upper bound on the sum of absolute values of 
coefficients (including many possibly non-positive) to derive the
corresponding large deviation bounds for $W_3(\diag)$, $W_4(\diag)$ and $W_4^\diamond$.

We next prove large deviation bounds and obtain an upper bound on the
following polynomial functions:
$\abs{W_3^\diag - \E W_3^\diag}$, $\abs{W_4^\diag - \E W_4^\diag}$, and $\abs{W_4^{\diamond}
- \E W_4^\diamond}$ in Lemmas~\ref{lemma::S3devi} to~\ref{lemma::S5devi} respectively.
In combination with the absolute value bounds on their
expected values in Lemma~\ref{lemma::mean}, we obtain an upper bound for each of the following terms: $W_2^{\diamond}, \abs{W_3(\diag)}, \abs{W_4(\diag)}$, 
and $\abs{W_4^{\diamond}}$ using the triangle inequality, which
collectively leads to a large deviation bound on the Frobenius norm as stated in
Theorem~\ref{thm::uninorm-intro}.

Throughout the rest of this section, it is understood that  for $\abs{B_0} =
(\abs{b_{ij}})$ and when $q, h \in \Sp^{n-1}$ are $s_0$-sparse, 
we replace $\twonorm{\abs{B_0}} = \twonorm{(\abs{b_{ij}})}$ with its
maximum $s_0$-sparse eigenvalue $\rho_{\max}(s_0, (\abs{b_{ij}})) \le \sqrt{s_0} \twonorm{B_0}$,
as in Definition~\ref{def::C0} and bounded in
Lemma~\ref{lemma::converse}. Moreover, we choose constants large
enough so that all probability statements hold. 
Recall that to extract the cubic polynomial $W_3^{\diag}$ from
 \eqref{eq::summands},
 we first allow $v^i$ to appear on both sides of the quadratic form 
 $(v^i \circ  v^j)^T \diag(A_0 \circ A_0)(v^i \circ v^k)$; we then
 add an element in $W_3^{\diag}$ with weight
 \bens
 \Delta_{(i,j), (i,k)} \; 
\text{ for } \; (v_i \circ  v_j)^T \diag(A_0 \circ A_0)(v_i \circ
v_{k})  =\sum_{s=1}^m a_{ss}^2 u^s_i u^s_j u^s_k
\eens
where it is understood that index pairs $(i, j)$ and $(i, k)$ on both sides of
$\diag(A_0 \circ A_0)$ remain unordered, resulting in four
coefficients in $\Delta_{(i,j), (i,k)}$ (cf. \eqref{eq::azerb}).
We crucially exploit an upper bound on the sum over absolute 
values of coefficients corresponding to each polynomial function 
$S_3^{\star}(i, j, k)$ as stated in Lemma~\ref{lemma::triangleweight}
to derive their corresponding large deviation bounds.
\begin{lemma}
\label{lemma::S3devi}
Denote by
$$\tau_3 = C_2 a_{\infty} \big(a_{\infty} \log (n \vee m) \big) \vee \big(\log  (n \vee m) \sum_{j=1}^m a_{jj}^2  p_j^3\big)^{1/2}.$$
Let $\abs{S_3^{\star}} := \max_{i \not=j\not=k} \abs{S_3^{\star}(i, j,
  k)}\; \text{ where } \;
\abs{S_3^{\star}(i, j, k)} := \abs{\sum_{s=1}^m a_{ss}^2 (u^s_i u^s_j u^s_{k} - p_s^3)}$.
Then $$\prob{ \abs{S_3^{\star}} > \tau_3} =: \prob{\event_3} \le    \frac{1}{3 (n \vee m)^4}.$$
Under event $\event_3^c$, we have
\bens
\lefteqn{\forall q, h \in \Sp^{n-1}, \quad
  \abs{W_3^{\diag} -\E W_3^{\diag}} \le  2 (\twonorm{\abs{B_0}}
  b_{\infty} + \twonorm{B_0}^{2}) \tau_3}  \\
& \le & \twonorm{B_0}^2 \sum_{s=1}^m a_{ss}^2   p_s^3+ C_3
a_{\infty}^2 \log (n \vee m) (\twonorm{\abs{B_0}}^2 + \twonorm{B_0}^2)
\eens
where $C_3$ is an absolute constant; When we consider $s_0$-sparse vectors $q, h \in
\Sp^{n-1}$, we replace $\twonorm{\abs{B_0}} = \twonorm{(\abs{b_{ij}})}$ with 
$\rho_{\max}(s_0, (\abs{b_{ij}}))$ as in Definition~\ref{def::C0}.
\end{lemma}

\begin{lemma}
\label{lemma::S4devi}
Denote by
\bens
\abs{S_4^{\star}} = \max_{(i \not= j \not= k \not= \ell)} \abs{S_4^{\star}(i,j,k,\ell)} \; 
\text{ where } \; 
S_4^{\star}(i,j,k,\ell) := \sum_{s=1}^m a_{ss}^2 (u^s_i u^s_j u^s_k u^s_{\ell}  - p_s^4)
\eens
Then for $\tau_4  = C_4 a_{\infty}( a_{\infty} \log (n \vee m) \vee \big(\log (n \vee m) \sum_{s=1}^m  a_{ss}^2 p_s^4\big)^{1/2})$, we have
$$\prob{\abs{S_4^{\star}} > \tau_4} =: \prob{\event_4}  \le \inv{12 (n \vee m)^4}.$$
Under event $\event_4^c$, we have for all $q, h \in \Sp^{n-1}$,
\ben 
\label{eq::largeW4}
\abs{W_4^{\diag} - \E W_4^{\diag}}
& \le &
\twonorm{\offd(\abs{B_0})}^2 \abs{S_4^{\star}} \le \twonorm{(\abs{b_{ij}})}^2 
\tau_4. 
\een
When we consider $s_0$-sparse vectors $q, h \in
\Sp^{n-1}$, we replace $\twonorm{(\abs{b_{ij}})}$ with
$\rho_{\max}(s_0, (\abs{b_{ij}}))$ as in Definition~\ref{def::C0}. 
\end{lemma}
\begin{lemma}
  \label{lemma::S5devi}
Let $\ol{S}_{\diamond}(i,j,k,\ell)  := 
\sum_{s \not=t} a_{st}^2 
(u^s_i u^s_j u^t_k u^t_{\ell} -p_s^2 p_t^2) \; \; \forall i\not=j,  k
\not=\ell.$
Denote by
$$\tau_5 = C_5 \twonorm{A_0} \big( \twonorm{A_0} \log (n \vee 
m) \vee \big(\log (n \vee m) \sum_{s\not=t}^m     a_{st}^2 p_s^2
p_t^2\big)^{1/2} \big).$$
On event $\event_5^c$, we have $\abs{S_5^{\star}} :=  \max_{(i\not= j,  k\not=\ell)}
\abs{\ol{S}_{\diamond}(i,j,k,\ell)} \le \tau_5$, and hence
\ben
\label{eq::W4db}
\forall q, h \in \Sp^{n-1}, &&
\abs{W_4^{\diamond} - \E W_4^{\diamond}} \le   \twonorm{\abs{B_0}}^2
\tau_5,\\
\nonumber
\text{and }\;\; W_4^{\diamond} - \E W_4^{\diamond} 
& = & 
\sum_{i=1}^n \sum_{k =1}^n (b_{ik} q_i q_k) 
\big(\sum_{j \not=i}^n \sum_{\ell \not=k}^n b_{j\ell} h_{j}h_{\ell}
\big) \ol{S}_{\diamond}(i,j,k,\ell)  \\
\label{eq::defineSdm}
& =: &
\sum_{i \not =j,  k \not= \ell} w_{\square}(i, j, k, \ell)
(\ol{S}_{\diamond}(i, j, k, \ell))
\een
 \text{ where } $\abs{B_0} = (\abs{b_{ij}})$;
Then $\prob{\abs{S_5^{\star}} > \tau_5} =: \prob{\event_5} \le  \inv{2 (n \vee m)^4}$.
When we consider $s_0$-sparse vectors $q, h \in
\Sp^{n-1}$, we replace $\twonorm{\abs{B_0}} = \twonorm{(\abs{b_{ij}})}$ with 
$\rho_{\max}(s_0, (\abs{b_{ij}}))$ as in Definition in \ref{def::C0}.
in \eqref{eq::W4db}.
\end{lemma}
Lemma~\ref{lemma::S3devi},~\ref{lemma::S4devi} and~\ref{lemma::S5devi}
are proved in Sections~\ref{sec::W2dmproof},~\ref{sec::W3devi}, \ref{sec::W4diagdevi}, and~\ref{sec::W4offddevi} respectively.

\subsection{Proof of Theorem~\ref{thm::uninorm-intro}}
\begin{proofof2}
  Throughout this proof, it is understood that when we consider $s_0$-sparse vectors $q, h \in
  \Sp^{n-1}$,  we replace $\twonorm{\abs{B_0}} =
  \twonorm{(\abs{b_{ij}})}$ with $\rho_{\max}(s_0, (\abs{b_{ij}})) \le
  \sqrt{s_0} \twonorm{B_0}$ as  in Definition~\ref{def::C0} and shown in 
Lemma~\ref{lemma::converse}.
First, we have by Lemma~\ref{lemma::mean}, for all $q, h \in \Sp^{n-1}$, 
\bens
\E \fnorm{A^{\diamond}_{q,h}}^2 & = & \E W_2^{\diamond} + \E W_3(\diag)+  \E 
W_4^{\diag} + \E W_4^{\diamond} \\
& \le &
\E W_2^{\diamond} + \abs{\E W_3(\diag)} +  \abs{\E W_4(\diag)} + \abs{\E
  W_4^{\diamond}} \\
& \le &
2 b_{\infty}^2\sum_{j=1}^m  a_{jj}^2 (p_j^2 +p_j^3 +p_j^4 ) + 
4 \twonorm{B_0}^2  (\sum_{j=1}^m  a_{jj}^2 p_j^3 + 2\sum_{j=1}^m
a_{jj}^2 p_j^4 + \sum_{i\not=j}  a_{ij}^2 p_i^2 p_j^2 ) \\
& \le &
 2 b_{\infty}^2 \sum_{i=1}^m  a_{ii}^2 p_i^2 + 
6\twonorm{B_0}^2 \big(\sum_{j=1}^m  a_{ij}^2 p_j^3 + \sum_{i=1}^m 
  a_{ii}^2 p_i^4\big)   + 4 \twonorm{B_0}^2 N  \\
\text{where} \; \; 
N & = &  \sum_{j=1}^n a_{jj}^2 p_j^4  + \sum_{i \not=j}^m a_{ij}^2 
p_i^2 p_j^2 \le \lambda_{\max}(A_0 \circ A_0)  \sum_{s=1}^m p_s^4 
 \le a_{\infty} \twonorm{A_0} \sum_{s=1}^m p_s^4 
 \eens
Let events $\event_2, \event_3, \event_4, \event_5$ be as defined in  Lemmas~\ref{lemma::W2devi},~\ref{lemma::S3devi},  
~\ref{lemma::S4devi}, and~\ref{lemma::S5devi} respectively.
We have on event $\event^c_2$, by Lemma~\ref{lemma::W2devi} 
for all $q, h \in \Sp^{n-1}$, $\abs{W_2^{\diamond} -\E W_2^{\diamond}}
\le b_{\infty}^2 \sum_{s=1}^m a_{ss}^2 p_s^2$. Denote by 
\bens 
S_2 & = &
\sum_{s=1}^m  a^2_{ss} p_s^2 \text{ and } \quad S_4 =\sum_{s=1}^m  a_{ss}^2 p_s^4 \quad 
\text{and} \quad 
S_3  :=  \sum_{s=1}^m  a_{ss}^2 p_s^3 \le \sqrt{S_2 S_4} 
\eens
On event $\event_3^c$, we have by 
Lemma~\ref{lemma::triangleweight} and Lemma~\ref{lemma::S3devi}, for
all $q, h \in \Sp^{n-1}$,
\bens
\abs{W_3^{\diag} } & \le &
\abs{\E W_3^{\diag}} + \abs{ W_3^{\diag} -\E W_3^{\diag}}\\
& \le &
8 \twonorm{B_0}^2 S_3
+ C_3 (\twonorm{\abs{B_0}}^2 + \twonorm{B_0}^2) a_{\infty}^2  \log (n \vee m)  
\eens
where $S_3 \le (S_2 S_4)^{1/2} \le \half(S_2  + S_4)$ and $C_3 a_{\infty}^2 \log (n \vee m) \le c' S_2$ by assumption on the
lower bound on $S_2$.
Finally, on event $\event_4^c \cap \event_5^c$, by
Lemmas~\ref{lemma::S4devi} and~\ref{lemma::S5devi}, we have for some absolute constants $C_1, C_1'$,
\bens 
\lefteqn{\abs{W_4^{\diag} - \E W_4^{\diag}} + \abs{W_4^{\diamond} - \E W_4^{\diamond}} }\\
& \le &  
C_1 \log (n \vee m) \twonorm{A_0}^2 \twonorm{\abs{B_0} }^2+ C'_1 \twonorm{\abs{B_0}}^2 \twonorm{A_0}
  \log^{1/2}(n \vee m) \sqrt{N}
\eens
Hence, on event $\event_2^c \cap \event_3^c \cap \event_4^c \cap
\event_5^c$, by Lemmas~\ref{lemma::mean},~\ref{lemma::W2devi},~\ref{lemma::S4devi},~\ref{lemma::S3devi} and~\ref{lemma::S5devi}, we have for all $q, h \in \Sp^{n-1}$,
\bens
\lefteqn{\fnorm{A^{\diamond}_{q,h}}^2 \le \abs{W_2^{\diamond} } +\abs{W_3^{\diag}} + \abs{W_4^{\diag}} +
  \abs{W_4^{\diamond}} \le  \E W_2^{\diamond} + \abs{W_2^{\diamond} -\E W_2^{\diamond}}} \\
& + &
\abs{W_3(\diag)} +  \abs{\E W_4(\diag)} + \abs{\E  W_4^{\diamond}} +
\abs{W_4^{\diag} -\E W_4^{\diag}} + \abs{W_4^{\diamond} -  \E  W_4^{\diamond}} \\
& \le &
C_6 \twonorm{B_0}^2 S_2 +
C_7 \twonorm{B_0}^2 a_{\infty} \twonorm{A_0} \sum_{s=1}^m p_s^4  + C_8 \twonorm{\abs{B_0}}^2  \twonorm{A_0}^2 \log (n \vee m) \\
& +&  C_9 \twonorm{\abs{B_0}}^2 \twonorm{A_0} \big(\log (n \vee m) a_{\infty} \twonorm{A_0}  \sum_{j=1}^m p_j^4\big)^{1/2}
\eens 
for some absolute constants $C_6, C_7, C_8, \ldots$.
The theorem statement thus holds on event  $\F_0^c := \event^c_2 \cap \event^c_3 \cap
\event_4^c \cap \event_5^c$,  which holds with probability  at least $1 -
\frac{4}{(n \vee m)^4}$ by the union bound. 
\end{proofof2}

\subsection{Proof of Lemma~\ref{lemma::mean}}
\label{sec::proofofmean}
\subsubsection{Case $W_4^{\diamond}$}
We prove Lemma~\ref{lemma::W4dm} in
Section~\ref{sec::proofofW4dm}.
\begin{lemma}
  \label{lemma::W4dm}
  Let $W_4^{\diamond}$ be as defined in~\eqref{eq::W4dm}.
  \ben 
\label{eq::W4dm}
W_4^{\diamond}
& = & 
\sum_{i=1}^n \sum_{k =1}^n (b_{ik} q_i q_k) 
\big(\sum_{j \not=i}^n \sum_{\ell \not=k}^n(b_{j\ell} h_{j} h_{\ell}) 
\big) \sum_{s\not=t}  a_{st}^2  u^s_i u^s_j u^t_k u^t_{\ell}
\een 
  Then
  $\forall q, h  \in \Sp^{n-1}$, 
  $\abs{\E W_4^{\diamond}} \le   4 \twonorm{B_0}^2 \sum_{i\not=j}
  a_{ij}^2 p_i^2 p_j^2$
  \;\text{ where }
  \ben
  \label{eq::offdsum}
\;\; \abs{\sum_{i, k}  b_{ki}  q_i q_k
  \sum_{i\not= j, k\not=\ell} b_{j\ell} h_{j} h_{\ell}} \le   4 \twonorm{B_0}^2
\een
\end{lemma}

\subsubsection{Case $W_2^{\diamond}$}
\label{sec::polyexpo}
We prove Lemma~\ref{lemma::W2diag} in Section~\ref{sec::proofofW2mean}.
\begin{lemma}
\label{lemma::W2diag}
Fix   $q, h \in \Sp^{n-1}$. Let $ w^{\be}_{(i, j)}$ be as defined in~\eqref{eq::defineWbe}.
Then for all $q, h \in \Sp^{n-1}$ and $W_2^{\diamond}$ as defined in \eqref{eq::defineW2diag},
$\E W_2^{\diamond}  \le 2 b_{\infty}^2 \sum_{s=1} a_{ss}^2 p_s^2.$
\end{lemma}

\subsubsection{Counting strategy for unique triples}
We prove Lemma~\ref{lemma::triangleweight} in 
Section~\ref{sec::proofofW3dm}. 
Recall for $\Delta_{(i,j),(i,k)}$ as defined in~\eqref{eq::azerb}, 
\bens
W_3^{\diag} & := &
\sum_{i =1}^n \sum_{j \not= i}^n \sum_{k \not= i,j}^n \Delta_{(i,j),
  (i,k)} \cdot (v_i \circ  v_j)^T \diag(A_0 \circ A_0)(v_i \circ v_k).
\eens
\begin{lemma}
  \label{lemma::triangleweight}
For all $q, h\in \Sp^{n-1}$, 
\bens 
\abs{\E W_3^{\diag}} & \le & 
\big(2 b_{\infty}^2 + 2  \twonorm{B_0} b_{\infty} + 
2 \twonorm{B_0}^2\big) \sum_{s=1}^m a_{ss}^2 p_s^3; 
\eens
where  we have for $\Delta_{(i,j),(i,k)}$ as 
defined in~\eqref{eq::azerb}, 
\bens
\abs{\sum_{ i \not= j\not= k} \Delta_{(i,j),(i,k)}}
& \le & 2 \twonorm{B_0}^2 + 2  \twonorm{B_0} b_{\infty} + 2
b_{\infty}^2;
\eens
Moreover, for $\twonorm{\abs{B_0}} = \twonorm{(\abs{b_{ij}})}$ and for all $q, h\in \Sp^{n-1}$, 
\bens
\sum_{ i \not= j\not= k} \abs{\Delta_{(i,j), (i,k)}}
& \le & 2 \twonorm{(\abs{b_{ij}})} b_{\infty} + 2 \twonorm{B_0}^{2}; 
\eens
where  it is understood that when we consider $s_0$-sparse vectors $q, h \in
\Sp^{n-1}$, we replace $\twonorm{\abs{B_0}} = \twonorm{(\abs{b_{ij}})}$ with 
$\rho_{\max}(s_0, (\abs{b_{ij}}))$ as in Definition~\ref{def::C0} and bounded in Lemma~\ref{lemma::converse}.
\end{lemma}

\subsubsection{$W_4^{\diag}$: distinct $i, j, k, \ell$}
The proof of Lemma~\ref{lemma::W4distinct} follows the arguments in
Lemma 2.1 by~\cite{BVZ19}, which we defer to Section~\ref{sec::proofofW4distinct}.
\begin{lemma}
\label{lemma::W4distinct}
We have $\abs{\E W_4^{\diag}} \le  \big(2 b_{\infty}^2 +  2
b_{\infty}\twonorm{B_0}+ 6 \twonorm{B_0}^2\big) \sum_{i=1}^m  a_{ii}^2
p_i^4$, where for any $q, h  \in \Sp^{n-1}$,
\ben
\label{eq::W4diag}
&& \quad W_4^{\diag} = \sum_{i \not=k \not=j \not=\ell}^n 
b_{ki} q_k q_i b_{\ell j} h_{\ell}  h_{j} (v_i \circ v_j)^T \diag(A_0
\circ A_0)(v_k \circ v_\ell) \\
\nonumber
\text{ and} && \abs{ \sum_{i, j, k, \ell  \; \text{\distinct}} b_{ki}  q_i q_k b_{j\ell} h_{j}
  h_{\ell}}  \le 6 \twonorm{B_0}^2 + 2 b_{\infty}\twonorm{B_0} + 2 b_{\infty}^2
\een
\end{lemma}

\subsection{Proof of Lemma~\ref{lemma::W4dm}}
\label{sec::proofofW4dm}
Recall $\{\vecb^{(1)}, \ldots, \vecb^{(n)}\}$ denotes the set of column (row) vectors of  
symmetric positive-definite matrix $B_0 \succ 0$.  
\begin{proofof2}
For $i \not=j$ and $k \not=\ell$, $(v^i \circ  v^j)^T \offd(A_0 \circ
A_0)(v^k \circ v^\ell) =   \sum_{s \not=t} a_{st}^2 u^s_{i} u^s_{j} 
u^t_{k}  u^t_{\ell}$; Hence by linearity of expectation,
\bens 
\lefteqn{\abs{\E W_4^{\diamond} }
=\abs{\sum_{i=1}^n \sum_{k =1}^n (b_{ik} q_i q_k) 
\big(\sum_{j \not=i}^n \sum_{\ell \not=k}^n(b_{j\ell} h_{j} h_{\ell}) 
\big) \E \sum_{s\not=t}  a_{st}^2  u^s_i u^s_j u^t_k u^t_{\ell} }} \\
& = &
\abs{\sum_{i=1}^n \sum_{k =1}^n 
(b_{ki} q_k q_i)  \big(\sum_{j \not=i}^n \sum_{\ell \not=k}^n( b_{j\ell} h_{j} h_{\ell}) 
\big) }
\sum_{s\not=t}  a_{st}^2 p_s^2 p_t^2  \le  4 \twonorm{B_0}^2 \sum_{i\not=j}  a_{ij}^2 p_i^2 p_j^2 
\eens
To see the last step, we now examine the coefficients for $W_4^{\diamond}$: 
for each fixed $(i, k)$ pair,
\ben
\nonumber
\sum_{j \not= i, \ell \not=k} b_{j\ell} h_{j} h_{\ell}
& = & \sum_{j, \ell} b_{j\ell} h_{j} h_{\ell} -
\sum_{j, \ell=k} b_{j\ell} h_{j} h_{\ell} -
\sum_{j=i, \ell} b_{j\ell} h_{j} h_{\ell} +
\sum_{j= i, \ell=k} b_{j\ell} h_{j} h_{\ell} \\
& = &
\nonumber
\sum_{j, \ell} b_{j\ell} h_{j} h_{\ell} -
h_k \sum_{j=1}^n b_{j k} h_{j} - 
h_{i} \sum_{\ell=1}^n b_{i \ell} h_{\ell} + b_{i k} h_{i} h_{k}  \\
\text{ and hence }\; \;
&& \nonumber
\sum_{i, k}  b_{ki}  q_i q_k 
\sum_{j\not= i, \ell \not=k}
b_{j\ell} h_{j} h_{\ell} = 
\sum_{i, k}  b_{ki}  q_i q_k \sum_{j, \ell} b_{j\ell} h_{j} h_{\ell} -
\\
&&
\label{eq::foursum}
\sum_{i, k}  b_{ki}  q_i q_k h_k \sum_{j=1}^n b_{j k} h_{j} - \sum_{i, k}  b_{ki}  q_i q_k 
h_{i} \sum_{\ell=1}^n b_{i \ell} h_{\ell} + \sum_{i, k}  b_{ki}  q_i q_k  b_{i k} h_{i} h_{k},
\een
where due to symmetry, we bound the middle two terms in an identical
manner: denote by $\abs{q} = (\abs{q_1}, \ldots, \abs{q_n})$,
\bens
\lefteqn{\abs{\sum_{i, k}  b_{ki}  q_i q_k h_k \sum_{j=1}^n b_{j k} h_{j} }
  = \abs{\sum_{k=1}^n q_{k}h_{k}  \sum_{i=1}^n \sum_{j=1}^n
    b_{ki} b_{kj} q_{i} h_{j}   } }\\
  & = & \abs{\sum_{k=1}^n q_{k}h_{k} \big( q^T (\vecb^{(k)} \otimes 
    \vecb^{(k)})h \big)}  \le   \max_{k} \twonorm{\vecb^{(k)} \otimes
    \vecb^{(k)}}\abs{\ip{\abs{q}, \abs{h}}} \le \twonorm{B_0}^2
  \eens
  \bens
\lefteqn{ \text{Similarly}, \abs{\sum_{i, k}  b_{ki}  q_i q_k 
  h_{i} \sum_{\ell=1}^n b_{i \ell} h_{\ell}} = 
\abs{\sum_{i=1}^n q_{i}h_{i}
  \sum_{k=1}^n  b_{ik}  q_k \sum_{\ell=1}^n b_{i \ell} 
  h_{\ell}} }\\
& = &
\abs{\sum_{i=1}^n q_{i}h_{i}
    \big( q^T (\vecb^{(i)} \otimes 
    \vecb^{(i)})h\big)} \le \twonorm{B_0}^2 \\
  \text{ finally } & & \sum_{i, k}  b_{ki}  q_i q_k  b_{i k} h_{i} h_{k}
= \sum_{i, k}  b_{ki}^2  (q \circ h)_i (q \circ h)_{k} \le \twonorm{B_0 \circ B_0} \le b_{\infty} \twonorm{B_0}
\eens
where $q \circ h = (q_1 h_1, \ldots, q_n h_n)$; Hence by
\eqref{eq::foursum} and the inequalities immediately above, we have
\bens
\abs{\sum_{i, k}
  b_{ki}  q_i q_k 
  \sum_{i\not= j, k\not=\ell} b_{j\ell} h_{j} h_{\ell}}
& \le &
\abs{\sum_{i, k}  b_{ki}  q_i q_k}\abs{ \sum_{j, \ell} b_{j\ell} h_{j}
  h_{\ell}}
+
\abs{\sum_{i, k}  b_{ki}  q_i q_k h_k \sum_{j=1}^n b_{j k} h_{j} } +  \\
&&
\abs{\sum_{i, k}  b_{ki}  q_i q_k 
h_{i} \sum_{\ell=1}^n b_{i \ell} h_{\ell} }+
\abs{\sum_{i, k}  b_{ki}  q_i q_k  b_{i k} h_{i} h_{k}}  \le  4 \twonorm{B_0}^2.
\eens
\end{proofof2}

\subsection{Proof of Lemma~\ref{lemma::W2diag}}
\label{sec::proofofW2mean}
  We have  for all $q, h \in \Sp^{n-1}$,
  \bens
  0 \le \E W_2^{\diamond}
  & = & \sum_{i\not= j}^{{n \choose 2} } w^{\be}_{(i, j)} \E \big((v_i \circ  v_j)^T \diag(A_0 \circ 
  A_0)(v_i\circ v_j) \big)\\
& = &
\sum_{(i, j), i\not=j} w^{\be}_{(i,j)} 
\sum_{s=1} a_{ss}^2 p_s^2  \le 2 b_{\infty}^2 \sum_{s=1} a_{ss}^2
p_s^2 \; \text{ where} \; \\
0 \le \sum_{i\not=j} w^{\be}_{(i,j)}  & \le &
\sum_{i \not=j}  (2 b_{ii} b_{jj} q_i^2 h_{j}^2 + 2 b_{jj} b_{ii} q_j^2  h_{i}^2 )  \\
& \le & 
b_{\infty}^2 (\sum_{i =1}^n q_i^2 \sum_{j \not=i} h_{j}^2  +
\sum_{i=1}^n h_i^2 \sum_{j \not= i}^n q_{j}^2) \le 2 b_{\infty}^2 
\eens 
where we use the fact that  for $B_0 = (b_{i,j}) \succ 0$,  $b_{ii}
b_{jj} \ge b_{ij}^2$ and
\bens
(\sqrt{b_{ii} b_{jj}} q_i  h_{j})^2 + (\sqrt{b_{ii} b_{jj}} h_{i}
q_j)^2 \ge 2 b_{ii} b_{jj}\abs{q_i q_j h_i h_j} \ge 2 b_{ij}^2 \abs{q_i 
  q_j h_i h_j}.
\eens 

\subsection{Proof of  Lemma~\ref{lemma::triangleweight}}
\label{sec::proofofW3dm}
\begin{proofof2}
  Throughout this proof, we assume that $q, h \in
  \Sp^{n-1}$. Denote by
  $$\vecb^{(i)}_{\setminus i} = (b_{i, 1}, \ldots, b_{i,   i-1}, b_{i, i+1}, \ldots, b_{i, n}).$$ 
Summing over $ \Delta_{(i,j), (i,k)}$ over all unique triples $i \not=j\not=k$, we have 
\bens
\lefteqn{\sum_{i=1}^n \sum_{j\not= i}^n \sum_{k \not=j, i}^n 
  \Delta_{(i,j), (i,k)} =  \sum_{i =1}^n b_{ii}  q_i^2   \sum_{j \not= k \not= i} b_{jk} h_{j} h_{k}} \\
&  &
+ \sum_{i =1}^n b_{ii}  h_i^2   \sum_{j \not= k \not= i} b_{jk} q_{j}
q_{k}  +  2 \sum_{i=1}^n q_i h_{i} \sum_{k\not=j\not=i}^n b_{ik} q_k b_{ij} h_{j}
\eens
where for a fixed index $i$ and column $u^s$, when we sum over all $j,k$, due to symmetry, we have 
\bens 
\sum_{(j \not= k) \not=i} b_{ik} q_k b_{ij} h_{j} & = & \sum_{(j \not= k) \not=i} b_{ij} q_j b_{ik} h_{k}
=q_{\cdot\setminus i}^T \offd(\vecb^{(i)}_{\cdot \setminus i} \otimes 
\vecb^{(i)}_{\cdot \setminus i}) h_{\cdot \setminus i};
\eens
Now, this leads to the following equivalent expressions for
$W_3^{\diag}$:
\ben
\label{eq::defineW3}
\lefteqn{ W_3^{\diag}  :=  \sum_{i =1}^n \sum_{j \not= i} \sum_{k \not= i,j}
\Delta_{(i,j),(i,k)} \cdot (v^i \circ  v^j)^T \diag(A_0 \circ A_0)(v^i \circ v^k) }\\
\nonumber
& = &
\sum_{s=1}^m a_{ss}^2
\sum_{i =1}^n  u^s_i \sum_{j \not= i} \sum_{k \not= i,j} \big(
  b_{ii} q_i^2 (b_{jk} h_{j} h_{k}) +(b_{jk} q_j q_k) b_{ii} h_i^2
\big) u^s_j u^s_k \\
\nonumber
&&   +
\sum_{s=1}^m a_{ss}^2 \sum_{i =1}^n 2 (q_i h_{i})  u^s_i \sum_{j \not= i} \sum_{k
    \not= i,j}  (b_{ki} q_k b_{ij} h_{j})   u^s_j u^s_k
\een
 where in~\eqref{eq::defineW3}, we have for each of the $n(n-1)(n-2)$
 uniquely ordered triple $(i, j, k)$, a total weight defined by 
 ${\Delta}_{(i,j),(i,k)}$, and the second expression holds by symmetry of
 the quadratic form; in more details,
\bens
\sum_{(j \not= k) \not=i} b_{ik} q_k b_{ij} h_{j} u^s_j u^s_k 
& = &
\sum_{(j \not= k) \not=i} b_{ij} q_j b_{ik} h_{k} u^s_j u^s_k  \\
& = &
(q \circ u^s)_{\cdot\setminus i}^T \offd(\vecb^{(i)}
\otimes  \vecb^{(i)})_{\cdot\setminus i, {\cdot \setminus i}} (h\circ u^s)_{\cdot\setminus i}.
\eens
Therefore, we only write $2 b_{ik} q_k b_{ij}h_{j}$
rather than the sum over $(b_{ik} q_k b_{ij} h_{j} + b_{ij} q_j b_{ik}
h_{k})$ in the summation in the sequel.
Notice that
\bens\E (v_j \circ  v_i)^T \diag(A_0 \circ A_0)(v_j \circ v_k)
& = & \sum_{s=1}^m 
\E (a_{ss}^2 u^s_i u^s_j u^s_k) = \sum_{s=1}^m a_{ss} p_s^3.
\eens
Hence by linearity of expectations, we have
\bens
\abs{\E W_3^{d}}
& = & \abs{\sum_{ i \not= j\not= k} \Delta_{(i,j),(i,k)} \sum_{s=1}^m a_{ss}^2 
p_s^3 } \le \abs{\sum_{ i \not= j\not= k} \Delta_{(i,j),(i,k)}} \sum_{s=1}^m a_{ss}^2 
p_s^3 \\
& \le &
2 \big(\twonorm{B_0} b_{\infty} + b_{\infty}^2 + \twonorm{B_0}^2 \big) \sum_{s=1}^m a_{ss}^2 p_s^3.
\eens
where we will show that
\bens
  \lefteqn{
\abs{\sum_{ i \not= j\not= k} \Delta_{(i,j), (i,k)}} \le 
\sum_{ i=1}^n b_{ii} q_i^2 \abs{\sum_ {j\not= i} \sum_{k \not= i, j}
  b_{jk} h_j h_k} +
\sum_{ i=1}^n b_{ii} h_i^2 \abs{\sum_ {j\not= i} \sum_{k \not= i, j}
  b_{jk} q_j q_k} }\\
&& + 2 \sum_{ i=1}^n \abs{q_i} \abs{h_i} \abs{\sum_{j\not= i}^n 
  (b_{ij} h_j) (\sum_{k \not= i \not=j} b_{ik} q_{k})}
=: I + II + III 
\eens
It remains to show $I, II \le \twonorm{B_0}
b_{\infty} + b_{\infty}^2 $ and $III \le  2 \twonorm{B_0}^2$.
The first two terms are bounded similarly and with the same upper
bound:
\bens
&& I  =  \sum_{ i=1}^n b_{ii} q_i^2 \abs{\sum_ {j\not= i} \sum_{k \not= i, j}
  b_{jk} h_j h_k} \le \twonorm{B_0} b_{\infty} + b_{\infty}^2, \text{ where for a fixed $i$} \\
&& \abs{\sum_{k, j, k \not= j \not=i} b_{jk} h_j h_k} \le 
\abs{\sum_{k, j\not=i} b_{jk} h_j h_k} + 
 \abs{ \sum_{k= j\not=i}^n b_{jj} h_j^2} \le \twonorm{B_0} +
 b_{\infty}, \\
&& \text{and similarly,} \;\; II = \sum_{ i=1}^n b_{ii} h_i^2 \abs{\sum_ {j\not= i} \sum_{k \not= i, j}
  b_{jk} q_j q_k}  \le \twonorm{B_0} b_{\infty}  + b_{\infty}^2.
\eens
We can rewrite the sum for $III$ as follows.
Let $\abs{q} = (\abs{q_1}, \ldots, \abs{q_n})$. Then
\bens
\nonumber
  III & = & 2\sum_{ i=1}^n \abs{q_i} \abs{h_i} \abs{\sum_{j\not= i}^n 
    (b_{ij} h_j) (\sum_{k \not= i \not= j} b_{ik} q_{k})} \\
& \le &
2 \sum_{i=1}^n \abs{q_i}\abs{ h_i }
  \sum_{j \not=i}^n \sum_ {k \not=j\not= i}^n 
  \abs{b_{ij}}  \abs{h_j}\abs{b_{ik}}\abs{q_{k}} =: IV
  \eens
where for $\vecb^{(i)}_{\cdot \setminus i} = (b_{i,1}, \ldots,
b_{i, i-1}, b_{i, i+1}, \ldots, b_{i,n})$, 
\ben
\nonumber
IV  & := & 2 \sum_{i=1}^n \abs{q_i}\abs{ h_i }
  \sum_{j\not=i}^n \sum_{k\not=i}^n \abs{b_{ij}}
  \abs{h_j}\abs{b_{ik}}\abs{q_{k}} \le 
  2 \sum_{i=1}^n \abs{q_i}\abs{ h_i } \twonorm{\vecb^{(i)}_{\cdot
      \setminus i}}^2 \\
\label{eq::IVbound}
& \le & 2 \max_{i} \twonorm{\vecb^{(i)}_{\cdot \setminus i}}^2
\ip{\abs{q}, \abs{h}}  \le 2 \twonorm{B_0}^2;
\een
Similarly, we have for $\abs{B_0} = (\abs{b_{ij}})$, 
\bens 
\lefteqn{\sum_{ i \not= j\not= k} \abs{\Delta_{(i,j),(i,k)}} 
 \le 
2 \sum_{i=1}^m \abs{q_i}\abs{h_i}
\sum_{k\not=i}  \sum_{j \not=k \not=i}
\abs{b_{ik}}\abs{q_{k}}\abs{b_{ij}}  \abs{h_j} +} \\
&&
\sum_{i=1}^n b_{ii} q_i^2
\sum_ {j\not= i} \sum_{k \not= i,
  j}\abs{b_{jk}}\abs{h_j} \abs{h_k} + 
\sum_{ i=1}^n b_{ii} h_i^2 \sum_ {j\not= i} \sum_{k \not= i, j}
\abs{b_{jk}}\abs{q_j} \abs{q_k} \\
&& =: IV + V + VI \le  2  \twonorm{B_0}^2+ 2  b_{\infty}\twonorm{\abs{B_0}},
\eens
where the term $IV$ is as bounded in~\eqref{eq::IVbound}, and $V +VI
\le 2  b_{\infty}\twonorm{\abs{B_0}}$ since
\bens 
V & =  & \sum_{i=1}^n b_{ii} q_i^2 \sum_{k, j, k \not= j \not=i}
\abs{b_{jk}} \abs{h_j} \abs{h_k} \le 
\sum_{i=1}^n b_{ii} q_i^2 \sum_{k, j\not=i}^n \abs{b_{jk}} \abs{h_j}
\abs{h_k} \le b_{\infty} \twonorm{\abs{B_0}}
\eens
where $\twonorm{\abs{B_0}} := \twonorm{(\abs{b_{ij}})}$ in the
argument above  is understood to denote $\rho_{\max}(s_0, \abs{B_0})
\le  \sqrt{s_0} \twonorm{B_0}$  in case $h$ is $s_0-\sparse$, and  $VI  :=\sum_{ i=1}^n b_{ii} h_i^2 \sum_ {j\not= i} \sum_{k \not= i, j}
\abs{b_{jk}}\abs{q_j} \abs{q_k}$  is bounded in a similar manner.
\end{proofof2}

\subsection{Proof of  Lemma~\ref{lemma::W4distinct}}
\label{sec::proofofW4distinct}
\begin{proofof2}
First we have by \eqref{eq::offdsum}, 
\bens
\abs{ \sum_{k\not=i} b_{ki}  q_i q_k  \sum_{j \not=i,  \ell \not=k}  b_{j\ell}
  h_{j} h_{\ell}}  \le 4 \twonorm{B_0}^2
\eens
  Denote by
  \bens
  E_1 & = &
  \{(i, j, k, \ell): (i=k),  i\not=j, k \not= \ell\}, \quad  E_2 = \{(i, j, k, \ell): (i=\ell),  i\not=j,  k \not= \ell\},\\
E_3 & = & \{(i, j, k, \ell):  (j=k), i\not=j,  k \not= \ell\}, \quad E_4 =
\{(i, j, k, \ell): (j=\ell), i\not=j, k\not=\ell\}
  \eens
  Now by the inclusion-exclusion principle,
  \bens
  \lefteqn{
\sum_{i, j, k, \ell  \; \text{\distinct}} b_{ki}  q_i q_k b_{j\ell}
h_{j}  h_{\ell}  =  \sum_{i, k}  b_{ki}  q_i q_k 
 \sum_{i\not= j, k\not=\ell} b_{j\ell} h_{j} h_{\ell} }\\
   & - &
  \sum_{i = k}  b_{ki}  q_i q_k 
  \sum_{j\not= i, \ell \not=k} b_{j\ell} h_{j} h_{\ell}  (E_1) -
  \sum_{i \not=k}   b_{ki}  q_i q_k \sum_{i\not= j, k\not=\ell =i}
  b_{j\ell} h_{j} h_{\ell} (E_2) \\
    & - &
   \sum_{i \not= k}  b_{ki}  q_i q_k 
   \sum_{i\not= j, k\not=\ell, j=k} b_{j\ell} h_{j} h_{\ell} (E_3) 
   -  \sum_{i, k}  b_{ki}  q_i q_k   \sum_{\ell \not=k, j 
     =\ell \not=i} b_{j\ell} h_{j} h_{\ell} (E_4) \\
  & +&
  \sum_{i = k}  b_{ki}  q_i q_k 
   \sum_{i\not= j, j=\ell} b_{j\ell} h_{j} h_{\ell}(E_4 \cap E_1)  
   + \sum_{i \not= k}  b_{ki}  q_i q_k   \sum_{\ell =i, j 
     =k \not=i} b_{j\ell} h_{j} h_{\ell} (E_2 \cap E_3)\\
&& \text{where clearly} \;\; 
E_1 \cap E_4  =  \{(i, j, k, \ell): (i=k), (j=\ell), i\not=j\} \\
&& \text{ and } \; \; E_2 \cap E_3 = \{(i, j, k, \ell): (i=\ell), (j=k), i\not=j\}
\eens
while all other pairwise intersections are empty and hence all three-wise
and four-wise intersections are empty.
Simplifying the notation, we have
\bens
\lefteqn{
\sum_{i, j, k, \ell  \; \text{\distinct}} b_{ki}  q_i q_k b_{j\ell} h_{j}  h_{\ell} = 
 \sum_{i, k}  b_{ki}  q_i q_k 
 \sum_{i\not= j, k\not=\ell} b_{j\ell} h_{j} h_{\ell} }\\
 & - &
\sum_{i}  b_{ii}  q_i^2\sum_{j \not= i, \ell \not=i} b_{j\ell} h_{j} h_{\ell}  (E_1) -
\sum_{i \not=k}   b_{ki}  q_i q_k \sum_{i\not= j} b_{j i} h_{j} h_{i} (E_2) \\
  & - &
  \sum_{i, k}  b_{ki}  q_i q_k   \sum_{j \not=k, j \not=i} b_{jj} h_{j}^2 (E_4) -
   \sum_{i \not=j}  b_{ji}  q_i q_j  \sum_{ j\not=\ell} b_{j\ell} h_{j} h_{\ell} (E_3) \\
  & +&
  \sum_{i = 1}^n  b_{ii}  q_i^2 \sum_{i\not= j} b_{jj} h_{j}^2 
  \; (E_4 \cap E_1)     + \sum_{i=1}^n \sum_{j \not=i}  b_{ji}  q_i q_j b_{ji} h_{j} h_{i} \;(E_2 \cap E_3)  
   \eens
   where for indices $(i, j, k, \ell)$ in $E_1$ and $E_2$,
   \bens
&&   (E_1) 
   \abs{\sum_{i}  b_{ii}  q_i^2\sum_{j \not= i, \ell \not=i} b_{j\ell}
     h_{j} h_{\ell}} \le b_{\infty} \sum_{i=1}^n  q_i^2 \abs{\sum_{j \not= i, \ell \not=i} b_{j\ell} h_{j}
     h_{\ell}} \le b_{\infty} \twonorm{B_0} \text{ and } \\
&& (E_2) \abs{\sum_{i \not=k}   b_{ki}  q_i q_k \sum_{j \not =i}
  b_{ji} h_{j} h_{i}}
=  \sum_{i=1}^n \abs{ q_i} \abs{ h_i} \abs{\sum_{i \not=k}   b_{ki}  q_k \sum_{i\not=
    j} b_{j i} h_{j}} \le  \twonorm{B_0}^2
\eens
where for a fixed index $i$,
\bens
\abs{ \sum_{i \not=k}   b_{ki}  q_k \sum_{i\not=
    j} b_{j i} h_{j}} \le \abs{\sum_{i \not=k}  b_{ki}  q_k } \abs{\sum_{i\not= j} b_{j i} h_{j}}
\le \twonorm{\vecb^{(i)}}^2 \le \twonorm{B_0}^2
\eens
Similarly, we bound for indices $(i, j, k, \ell)$ in $E_3$ and $E_4$,
\bens
 (E_4) \quad
\abs{\sum_{i, k}  b_{ki}  q_i q_k   \sum_{j \not=k, j \not=i} b_{jj}
  h_{j}^2} \le  b_{\infty} \twonorm{B_0} \text{ and }  (E_3) \quad
\abs{ \sum_{i \not=j}  b_{ji}  q_i q_j  \sum_{ j\not=\ell} b_{j\ell} h_{j}
 h_{\ell}} \le \twonorm{B_0}^2
\eens
Finally, for indices $(i, j, k, \ell)$ in $(E_3 \cap E_2) \cup (E_4 \cap
E_1)$, we have
\bens
\sum_{i = 1}^n \sum_{j \not= i} (b_{ii}  q_i^2 b_{jj} h_{j}^2
+ b_{ji}  q_i q_j b_{ji} h_{j} h_{i}) \le 2 b_{\infty}^2
\eens
following Lemma~\ref{lemma::W2diag}.
Now as we have shown in  Lemma~\ref{lemma::W4dm},
\bens
\abs{\sum_{i, k}  b_{ki}  q_i q_k 
  \sum_{i\not= j, k\not=\ell} b_{j\ell} h_{j} h_{\ell}} \le 4
\twonorm{B_0}^2
\eens
Thus we have by  the bounds immediately above,
\bens
\abs{\sum_{i, j, k, \ell  \; \text{\distinct}} b_{ki}  q_i q_k
  b_{j\ell} h_{j}  h_{\ell}}
\le 6 \twonorm{B_0}^2 + 2 \twonorm{B_0} b_{\infty} +  2 b_{\infty}^2
\eens
Finally, for $w^\square_{i,j,k,\ell} =b_{ki}q_iq_k b_{j\ell} h_j
h_{\ell}$,
\bens 
\abs{\E W_4^{\diag}} =  \abs{\sum_{ i \not= j\not= k\not=\ell} w^\square_{i,j,k,\ell}
  \sum_{s=1}^m a_{ss}^2 p_s^4 }
\le (8 \twonorm{B_0}^2 + 2b_{\infty}^2 ) \sum_{s=1}^m a_{ss}^2 
p_s^4.
\eens
Thus the lemma holds.
\end{proofof2}

\section{Concentration of measure bounds}
\label{sec::polyconc}
We need the following result which follows from Proposition 3.4~\cite{Tal95}.
\begin{lemma}{\textnormal{(\cite{Tal95})}}
\label{lemma::bern-sumlocal}
Let $A = (a_{ij})$ be an $m \times m$ matrix.
Let $a_{\infty} := \max_i \abs{a_{ii}}$.
Let $\xi = (\xi_1, \ldots, \xi_m) \in \{0, 1\}^m$ be a random vector
with independent Bernoulli random variables $\xi_i$ such that
$\xi_i = 1$ with probability $p_i$ and $0$ otherwise.
Then for $\abs{\lambda} \le \inv{4 a_{\infty}}$,
\bens
\E\exp\big(\lambda \sum_{i=1}^m a_{ii} (\xi_i -p_i)\big)  \le 
\exp\big(\inv{2} \lambda^2 e^{\abs{\lambda} a_{\infty}} \sum_{i=1}^m
a_{ii}^2 \sigma^2_i \big).
\eens
\end{lemma}

\begin{corollary}
\label{coro::expmgfW3}
Let $A_{\infty} = \max_{i} a^2_{ii} = a_{\infty}^2$.
  Let $\xi = (\xi_1, \ldots, \xi_m) \in \{0, 1\}^m$ be a random vector 
with independent Bernoulli random variables $\xi_i$ such that 
$\xi_j = 1$ with probability $\E \xi_j$ and $0$ otherwise. 
Let $S_{\star} =\sum_{j=1}^m a^2_{jj} (  \xi_j - \E \xi_j)$. 
Then for $t>0$,  $\prob{\abs{S_{\star}} > t} \le 2 \exp\big(-
c\min\big(\frac{t^2}{M}, \frac{t}{a^2_{\infty} } \big) \big)$,
where $M = a^2_{\infty} \sum_{s=1}^m a^2_{ss}  \E \xi_s$.
\end{corollary}

We are going to apply Lemma~\ref{lemma::bern-sumlocal} and its
Corollary~\ref{coro::expmgfW3}
to obtain concentration of measure bounds on the diagonal components
corresponding to degree-$2, 3, 4$  polynomials in $W_2^{\diamond}, W_3(\diag),$
and $W_4^{\diag}$ respectively, where $i \not=j \not=k \not=\ell$ are
being fixed. Moreover, we choose constants large enough so that all probability  
statements hold. We prove Corollary~\ref{coro::expmgfW3} in 
Section~\ref{sec::proofofexpmgfW3}.

\subsection{Proof of Lemma~\ref{lemma::W2devi}}
\label{sec::W2dmproof}
Denote by $M_2 = a_{\infty}^2 \sum_{s  =1}^m a_{ss}^2 p_s^2$.
Notice that  by assumption,
\bens
2C_{\alpha}^2 a^2_{\infty} \log (n \vee m) 
\le \tau_2 := C_a \sqrt{\log (n \vee m) M_2} \le \inv{4} \sum_{s=1}^m
a_{ss}^2 p_s^2
\eens
By Corollary~\ref{coro::expmgfW3} and the union bound,
\bens 
\lefteqn{\prob{ \abs{S_2^{\star}} > \tau_2}
  :=\prob{\max_{i\not=j} \abs{S_2^{\star}(i, j)} > \tau_2}  := \prob{\event_2} } \\
& \le &  
{n \choose 2} 2\exp\big(- c\min\big({\tau_2^2}/{M_2},
   {\tau_2}/{a_{\infty}^2} \big)\big) \le \frac{1}{(n \vee m)^4}
\eens
which holds for $C_\alpha^2$ sufficiently large.
Then on event $\event_2^c$, for positive weights $w^{\be}_{(i,j)} \ge 0$,
we have for all $q, h \in \Sp^{n-1}$,
\bens
\abs{W_2^{\diamond} - \E W_2^{\diamond}}
\le  \sum_{j \not=i}^n w^{\be}_{(i,j)} \abs{\sum_{s=1}^m a_{ss}^2 u^s_i u^s_j - \sum_{s=1}^m a_{ss}^2 
  p_s^2 }  \le  \tau_2 \sum_{j \not=i}^n w^{\be}_{(i,j)} 
\le 2 b_{\infty}^2 \tau_2
\hskip2pt \;\;\scriptstyle\Box
\eens

\subsection{Proof of Lemma~\ref{lemma::S3devi}}
\label{sec::W3devi}
By Corollary~\ref{coro::expmgfW3} and the union bound,
\bens 
\lefteqn{\prob{ \abs{S_3^{\star}} > \tau_3} =
  \prob{\max_{i\not=j\not=k} \abs{S_3^{\star}(i, j, k)} > \tau_3} := \prob{\event_3} 
} \\
& \le &  
{n \choose 3} 2 \exp\big(- c\min\big(\frac{\tau_3^2}{ a_{\infty}^2 \sum_{s =1}^m a_{ss}^2 p_s^3},  \frac{\tau_3}{a_{\infty}^2}\big) \big)
 \le  1/{(3 (n \vee m) ^4)}
 \eens
where recall $\tau_3 = C_2 a_{\infty}^2 \log (n \vee m) \vee (a_{\infty} \sqrt{\log  (n \vee m) \sum_{j=1}^m a_{jj}^2 
   p_j^3})$ for  absolute constant $C_2$.  By~\eqref{eq::defineW3} and Lemma~\ref{lemma::triangleweight}, we have on event $\event_3^c$ 
\bens
\nonumber
\lefteqn{
\forall q, h \in \Sp^{n-1}, \; \; \abs{W_3^{\diag} - \E W_3^{\diag}} \le
\sum_{i \not= j \not=k} \abs{\Delta_{(i,j), (i,k)}}
\abs{S_3^{\star}(i, j, k)}} \\
 & \le &
\label{eq::W3star}
2 C_2 (\twonorm{\abs{B_0}}  b_{\infty} + \twonorm{B_0}^{2}) 
\big(  a_{\infty} \sqrt{\log (n \vee m)  S_3} \vee a_{\infty}^2
  \log (n \vee m)  \big) \\
 & \le &
\twonorm{B_0}^2 S_3  + C_3 a_{\infty}^2 \log (n \vee m)  (\twonorm{\abs{B_0}}^2+
\twonorm{B_0}^2)
\eens
where for $b_{\infty} \le \twonorm{\abs{B_0}} \wedge \twonorm{B_0}$,
where $\abs{B_0} =(\abs{b_{ij}})$,
\bens
\lefteqn{2 C_2 (\twonorm{\abs{B_0}}  b_{\infty} + \twonorm{B_0}^2)
  a_{\infty}\sqrt{\log (n \vee m)  S_3}}\\
& \le &
4 C_2 (\twonorm{\abs{B_0}} \vee \twonorm{B_0}) \twonorm{B_0}
a_{\infty}\sqrt{\log (n \vee m)  S_3} \\
& \le &\twonorm{B_0}^2 S_3  + 4 C_2^2 a_{\infty}^2  \log (n \vee m)  (\twonorm{\abs{B_0}}^2 \vee 
\twonorm{B_0}^2)
\eens
\bens
\lefteqn{\text{ while }
2 C_2 ( \twonorm{\abs{B_0}}  b_{\infty} + \twonorm{B_0}^{2}) 
a_{\infty}^2 \log (n \vee m) } \\
& \le &
2 C_2 \twonorm{\abs{B_0}}  b_{\infty} a_{\infty}^2 \log (n \vee m) +
2C_2 \twonorm{B_0}^{2} a_{\infty}^2 \log (n \vee m).  \hskip1pt \;\;\scriptstyle\Box
\eens

\subsection{Proof of Lemma~\ref{lemma::S4devi}}
\label{sec::W4diagdevi}
Denote by $M_4 \asymp a_{\infty}^2 \sum_{s=1}^m  a_{ss}^2 p_s^4$.
By Corollary~\ref{coro::expmgfW3}, we have
for $\tau_4   = C_4( \sqrt{M_4 \log (n \vee m) } \vee a^2_{\infty}
\log (n \vee m)) $,
\bens 
\prob{\abs{S_4^{\star}} > \tau_4}
& = &  
\prob{\exists \distinct \; i, j, k, \ell: \abs{S_4^{\star}(i, j, k,\ell)} > \tau_4 }\\
& \le &  {n \choose 4} 2\exp\big(-c \min\big(\frac{\tau_4^2}{M_4},
\frac{\tau_4}{a_{\infty}^2} 
\big)\big) \le 1/{(12 (n \vee m)^4)}
\eens
where $C_4$ is an absolute constant; Denote the above exception event as $\event_4$. 
Now, we have on event $\event_4^c$ and for all $q, h \in \Sp^{n-1}$, 
\bens
\abs{W_4^{\diag} - \E W_4^{\diag}}
& \le &
\sum_{i \not =k} \abs{b_{ki}}\abs{q_k} \abs{q_i}
\sum_{ j \not= k \not= \ell \not=i} \abs{b_{j\ell}} \abs{h_{j}}\abs{h_{\ell}} \abs{S_4^{\star}(i,j,k,\ell)}\\
& \le &  \twonorm{\offd(\abs{B_0})}^2 \abs{S_4^{\star}} \le
\twonorm{\offd(\abs{B_0})}^2 \tau_4.    \hskip1pt \;\;\scriptstyle\Box
\eens

\subsection{Proof of Corollary~\ref{coro::expmgfW3}}
\label{sec::proofofexpmgfW3}
\begin{proofof2}
  Let $A = (A_0 \circ A_0) = (a_{ij}^2)$ be an $m \times m$ matrix. 
  Let $A_{\infty} := \max_s a_{ss}^2$. Denote by $M = A_{\infty}  \sum_{s=1}^m a_{ss}^2 \E \xi_s$.
Thus we have by Lemma~\ref{lemma::bern-sumlocal},  for $\abs{\lambda}
\le 1/(4 A_{\infty})$,  
\bens
\E \exp\big(\lambda S_{\star} \big) 
& \le &
\exp\big(\half \lambda^2 e^{\abs{\lambda} A_{\infty}}
  \sum_{s=1}^m a_{ss}^4 \E \xi_s \big) \le
\exp\big(\lambda^2 M \big)
 \eens
 where $e^{\abs{\lambda} A_{\infty}} \le  e^{1/4}$.
 Now for $0< \lambda \le 1/{(4 A_{\infty})}$, by Markov's inequality,
\bens 
\prob{S_{\star} > t}
& = &  \prob{\exp(\lambda S_{\star}) > \exp(\lambda t)} \le \E
\exp(\lambda S_{\star})/  \exp(\lambda t) \\
& \le &  \exp\big(-\lambda t +  \lambda^2  M \big) 
\eens
Optimizing over $0< \lambda < 1/{(4 A_{\infty})}$, we have for $t>0$,
\bens 
\prob{S_{\star} > t} & \le &  
\exp\big(- c\min\big({t^2}/{M}, {t}/{A_{\infty} }  \big)\big) =: q_{\diag}
\eens
Repeating the same arguments above, for $0< \lambda < 1/{(4 a_{\infty})}$ and $t>0$,
\bens
\prob{S_{\star} < -t} & = &  \prob{\exp(-\lambda S_{\star}) >
  \exp(\lambda t)} \\
& \le & \frac{\E \exp(-\lambda S_{\star})}{e^{\lambda t}}
\le  \exp\big(-\lambda t +  \lambda^2 M  \big) \le q_{\diag}
\eens
The corollary is thus proved by combining these two events since
$\prob{\abs{S_{\star}} > t} \le 2  q_{\diag}$.
\end{proofof2}

\subsection{Large deviation bound on $W_4^{\diamond}$}
\label{sec::W4offddevi}
Our goal in this section is to prove Lemma~\ref{lemma::S5devi}.
Let $ i\not=j$ and  $k \not=\ell$.
For $W_4^{\diamond}$, we need to derive large deviation 
bound on polynomial function $\ol{S}_{\diamond}(i, j, k, \ell)$ for 
each quadruple $(i, j, k,\ell)$  such that $i \not=j, k \not=\ell$. 
More precisely, we have Theorem~\ref{thm::expaqua}.
We prove Theorem~\ref{thm::expaqua} in Section~\ref{sec::proofofaqua}.
Denote by $\forall  i \not= j,  \; \; k \not=\ell$,
\ben
\label{eq::Sdm}
S_{\diamond}(i, j, k, \ell)  & := &
(v_i \circ v_j)^T \offd(A_0 \circ A_0)(v_k 
\circ v_\ell)  :=  \sum_{s\not=t}^m a_{st}^2 u_i^s u_j^s u^t_k 
u^t_{\ell}     \\
 \label{eq::Sdm0}
 \quad  \quad 
     \overline{S}_{\diamond}(i, j, k, \ell) & := &
     S_{\diamond}(i,j,k,\ell)   - \E S_{\diamond}(i,j,k,\ell) \\
     \nonumber
 \text{ where} \quad \E S_{\diamond}(i, j, k, \ell)
 &   = & \sum_{s \not=t} a_{st}^2  p_t^2   p_s^2 \; \;
 \forall 
 i \not= j,  \; \; k \not=\ell.
 \een
\begin{theorem}
\label{thm::expaqua}
Let $a_{\infty} = \max_{i} a^2_{ii}$ and $c$ be an absolte constant.
Let $\overline{S}_{\diamond}(i, j, k, \ell)$ be as defined in  \eqref{eq::Sdm0}.
For any $t > 0$, and quadruple $(i, j, k, \ell)$ such that $i
\not=j, k \not=\ell$, 
\bens 
\prob{\abs{\ol{S}_{\diamond}(i, j, k, \ell)} > t} & \le &  
2\exp\big(- c \min\big(\frac{t^2}{\twonorm{A_0}^2  \sum_{s \not=t}
    a_{st}^2  p_t^2   p_s^2 }, \frac{t}{\twonorm{A_0}^2 } \big) \big)
\eens
\end{theorem}

\subsubsection{Proof of Lemma~\ref{lemma::S5devi}}
  We can now apply Theorem~\ref{thm::expaqua} with
  $\tau_5 = C_5 \twonorm{A_0} \big( \twonorm{A_0} \log (n \vee m)
    \vee  \sqrt{\log (n \vee m) \sum_{s\not=t}^m a_{st}^2 p_s^2 p_t^2}
    \big)$, for $C_5$ large enough,
\bens 
\lefteqn{\prob{\abs{S_5^{\star}} > \tau_5}
=\prob{\exists i \not=j, k \not=\ell, \; \;\ol{S}_{\diamond}(i, j, k, \ell) \ge \tau_5} 
  =: \prob{\event_5}}\\
& \le &
{n \choose 2}^2 2\exp\big(- c\min\big(\log (n \vee m)\big) \big) \le  1/{(2 (n \vee m)^4)} 
\eens
Hence on event $\event_5^c$, for $\abs{B_0} = (\abs{b_{ij}})$ and all $q, h
\in \Sp^{n-1}$,
 \bens
\lefteqn{\abs{W_4^{\diamond} - \E W_4^{\diamond}}
\le 
\sum_{i \not =j,  k \not= \ell} 
\abs{w_{\square}(i, j, k, \ell)} \abs{\ol{S}_{\diamond}(i, j, k,
  \ell)}} \\
& \le &
\sum_{i,k}\sum_{ j \not= i,\ell \not=k}
(\abs{b_{ki}}\abs{q_k} \abs{ q_i} \abs{b_{j\ell}} \abs{h_{j}}
\abs{h_{\ell}} )\abs{\overline{S_{\diamond}}(i,j,k,\ell)} \le
\twonorm{\abs{B_0}}^2 \tau_5.
\hskip1pt \;\;\scriptstyle\Box
\eens

\subsubsection{Proof of Theorem~\ref{thm::expaqua}}
\label{sec::proofofaqua}
We obtain the following estimate on the moment generating function of
$\ol{S}_{\diamond}(i, j, k, \ell)$.
Although we give explicit constants here, they are by no means optimized. This
result may be of independent interests. In fact, we present a proof
aiming for clarity rather than optimality of the constants being involved.
\begin{lemma}
  \label{lemma::Sdm}
For all $\abs{\lambda} \le 1/{(32 \twonorm{A_0}^2)}$ and
$C_{12} = 65
e^{1/4}$, we have for all quadruple $(i, j, k, \ell)$, where $i \not=j$ and $k \not=\ell$,
\bens
\E \exp(\lambda (S_{\diamond}(i,j,k,\ell)  - \E S_{\diamond}(i, j, k,\ell)))
& \le &
\exp\big(C_{12}  \lambda^2   \twonorm{A_0}^2 \sum_{s \not=t } a_{st}^2  p_t^2   p_s^2
\big) 
\eens
\end{lemma}

\begin{proofof}{Theorem~\ref{thm::expaqua}}
  Fix $i\not=j$ and $k \not=\ell$.
  Denote by $\ol{S}_{\diamond}=  \ol{S}_{\diamond}(i, j, k, \ell)$.
  By Lemma~\ref{lemma::Sdm}, we have for  $t> 0$ and $0< \lambda \le 1/{(32 \twonorm{A_0}^2)}$,
  \bens
  \prob{\ol{S}_{\diamond}  \ge t}
  & = &
\prob{\lambda(\ol{S}_{\diamond})  \ge \lambda t}  =
\prob{\exp(\lambda(\ol{S}_{\diamond}))  \ge \exp(\lambda t)} \\
& \le &
\frac{\E \exp(\lambda \ol{S}_{\diamond})}{e^{\lambda t}} \le 
\exp\big(-\lambda t + \lambda^2 C_{12}  \twonorm{A_0}^2 \sum_{s\not=t} a_{st}^2  p_s^2 
  p_t^2\big).
\eens
Optimizing over $\lambda$, we have for  $\ol{S}_{\diamond} = \ol{S}_{\diamond}(i, j, k, \ell)$,
\bens
\prob{\ol{S}_{\diamond}  
 \ge t} & \le & 
\exp\big(-c\min \big(\frac{t^2}{\twonorm{A_0}^2 \sum_{s\not=t} a_{st}^2
      p_s^2 p_t^2}, \frac{t}{\twonorm{A_0}^2}\big) \big) =: p_1;
    \eens
Repeating the same arguments above, for $0< \lambda \le 1/{(32 \twonorm{A_0}^2)}$ and $t>0$, we have 
\bens 
\prob{\ol{S}_{\diamond} < -t} & = &  \prob{\exp(\lambda(-\ol{S}_{\diamond}) )> \exp(\lambda
  t)} \le
\frac{\E \exp(-\lambda \ol{S}_{\diamond})}{e^{\lambda t}} \le p_1;
\eens 
Hence $\prob{\abs{\ol{S}_{\diamond}} > t} \le 2p_1$ and the theorem is thus proved.
\end{proofof}

\subsubsection{Proof of Lemma~\ref{lemma::Sdm}}
Denote by 
$\xi =v^i \circ v^j$ and $\xi' = v^k \circ v^{\ell}$ in  the following 
steps, where  $i\not= j, k \not=\ell$.
Then
\bens
S_{\diamond}(i, j, k, \ell)  & := & 
(v_i \circ v_j)^T \offd(A_0 \circ A_0)(v_k     \circ v_\ell) \\
& = &
\sum_{s\not=t}^m a_{st}^2 u_i^s u_j^s u^t_k 
u^t_{\ell} = \sum_{t=1}^m  \xi'_t \sum_{s\not=t}^m a_{st}^2 \xi_s
\eens
First, we compute the expectation: $\E S_{\diamond}(i,j,k,\ell) =
\sum_{i\not=j}  a_{ij}^2 p_i^2 p_j^2, \forall i\not= j, k \not=\ell$.
Notice that the two vectors $\xi$ and $\xi'$ may not be independent, since $(i, j)$
and $(k, \ell)$ can have overlapping vertices; however, when we
partition $U$ into disjoint submatrices
$(U)_{\Lambda} :=\{u^s\}_{s \in \Lambda}$ and
$(U)_{\Lambda^c} :=\{u^t\}_{t  \in \Lambda^c}$,
each formed by extracting columns of $U$ indexed by $\Lambda \subset
[m]$ and its complement set $\Lambda^c$ respectively, then
\bens
\forall i \not=j, \; k \not=\ell,
u^s_i, u^s_j, u^t_k, u^t_{\ell}, \; \; 
s \in  \Lambda, \; \; t \in \Lambda^c
\eens
are mutually independent Bernoulli random variables; and hence each
monomial $\xi'_t := u^t_k u^t_{\ell}, t \in \Lambda^c$ is independent
of the sum $\sum_{s \in 
  \Lambda} a_{st}^2 \xi_s =\sum_{s \in 
  \Lambda} a_{st}^2 u^s_i u^s_j$. 

 \noindent{\bf Decoupling.}
Consider independent Bernoulli random variables $\delta_i \in \{0,
1\}$ with $\E \delta_i =1/2$. 
Let $\E_{\delta}$ denote the expectation with respect to random vector $\delta 
= (\delta_1, \ldots, \delta_m)$. Since $\E \delta_i (1-\delta_j)=1/4$ for
$i\not=j$ and $0$ for $i=j$, we have
\ben
\label{eq::meandelta}
S_{\diamond} & := & 4 \E_{\delta} S_{\delta} \; \text{ where } \; 
S_{\delta} := \sum_{s ,t} \delta_s (1-\delta_t) a_{st}^2 
u^s_i u^s_j u^t_k u^t_{\ell}
\een
Let $\Lambda_{\delta} =\{i \in [m]: \delta_i =1\}$ and 
$\Lambda_{\delta}^c$ 
be the complement of $\Lambda_{\delta}$. 
First notice that we can express
\bens 
S_{\delta} :=
\sum_{s \in \Lambda_{\delta}} \sum_{t \in \Lambda_{\delta}^c} a_{st}^2  
u^s_i u^s_j u^t_k u^t_{\ell}
=
\sum_{t \in \Lambda_{\delta}^c} u^t_k u^t_{\ell}
\sum_{s \in \Lambda_{\delta}} a_{st}^2 u^s_i u^s_j  
\eens  
Hence
\ben
\exp( \lambda \E S_{\diamond}) & = &
\label{eq::mean}
\exp(4 \lambda \E_{U, \delta}  S_{\delta}) =: \exp(4\lambda
\E_{\delta} \E(S_{\delta} | \delta)) \; \text{ where } \; \\
\label{eq::ESupper}
\E (S_{\delta} | \delta) & = &
\E \big(\sum_{s \in \Lambda_{\delta}}
  \sum_{t \in \Lambda_{\delta}^c} a_{st}^2 u^s_i u^s_j u^t_k u^t_{\ell}  | \delta \big) =
 \sum_{s \in \Lambda_{\delta}} \sum_{t \in \Lambda_{\delta}^c}
 a_{st}^2    p_s^2  p_t^2 
 \een
and in \eqref{eq::mean},
$\E_{\delta}$ denotes the expectation with respect to the random vector $\delta = (\delta_1, \ldots, \delta_m)$, or equivalently, the
random set of indices in $\Lambda_{\delta}$, and  $\E_{U, \delta}$
denotes expectation with respect to both $U$ and $\delta$.
Now by~ \eqref{eq::meandelta} and~\eqref{eq::mean}, we have for all $\lambda \in \R$, 
\ben
\nonumber
\E \exp(\lambda (S_{\diamond}  - \E S_{\diamond}))
& = &
\nonumber
\E_U \exp\big(4\lambda \big(\E_{\delta} (S_{\delta}) -   \E_{\delta}
    \E (S_{\delta} | \delta)
  \big) \big)\\
& = &
\label{eq::step0}
\E_U \exp\left\{4 \lambda \E_{\delta} [S_{\delta} - \E (S_{\delta} |  \delta) ]\right\} =: \E_U \exp\left\{4 \lambda g(U)\right\}
\een
where the random function $g(U)$ is defined as follows: 
\ben
\label{eq::definegU}
g(U) & = & \E_{\delta} [S_{\delta} -\E (S_{\delta} |  \delta) ] \\
\nonumber
& = &
\E_{\Lambda_{\delta}} \big(\sum_{t \in \Lambda_{\delta}^c} u^t_k u^t_{\ell}
\sum_{s \in \Lambda_{\delta}} a_{st}^2 u^s_i u^s_j 
- \sum_{s \in \Lambda_{\delta}} \sum_{t \in \Lambda_{\delta}^c}
a_{st}^2    p_s^2  p_t^2  \big| U\big)\\
\nonumber
\text{ while }\;\;
\E_{\delta} S_{\delta} & = & \E_{\Lambda_{\delta}} \big(\sum_{t \in \Lambda_{\delta}^c} u^t_k u^t_{\ell}
\sum_{s \in \Lambda_{\delta}} a_{st}^2 u^s_i u^s_j  \big| U\big)
\een
Hence we have reduced the original problem to the new problem of computing the moment
generating function for $g(U)$. \\
\noindent{\bf Centering.}
Denote by
\bens
Z'_t := \xi'_t - p_t^2 = u^t_k u^t_{\ell} - p_t^2 \text{ and } Z_s :=
\xi_s - p_s^2 = u^s_i u^s_j -p_s^2
\eens
Fix $\delta$.
First, we express the decoupled quadratic form  involving centered random variables with
\bens
    \sum_{s \in \Lambda_{\delta}} 
    \sum_{t \in \Lambda_{\delta}^c} a_{st}^2 Z_s Z'_t
    & := &
  \sum_{s \in \Lambda_{\delta}} 
  \sum_{t \in \Lambda_{\delta}^c} a_{st}^2 ({\xi}_s - p_s^2 ) 
  ({\xi'}_t - p_t^2) \\
 & = &
  \sum_{s \in \Lambda_{\delta}} 
  \sum_{t \in \Lambda_{\delta}^c} a_{st}^2\big(({\xi}_s {\xi'}_t)- 
    ({\xi}_s - p_s^2) p_t^2- p_s^2({\xi'}_t - p_t^2) - p_s^2 p_t^2\big); 
  \eens
  where
$\{\xi' -p_t^2\}_{\Lambda_{\delta}^c} \text{ and }
\{\xi-p_s^2\}_{\Lambda_{\delta}}$ are each centered and mutually
independent random vectors.
Hence we can now express $S_{\delta} - \E (S_{\delta} \big| \delta)$ as sum of 
quadratic and linear forms based on centered random vectors with independent mean-zero coordinates: 
\ben
 \nonumber
S_{\delta} - \E (S_{\delta} \big| \delta)
& = &
 \nonumber
\sum_{t \in \Lambda_{\delta}^c} \xi'_t \sum_{s \in \Lambda_{\delta}}
a_{st}^2 \xi_s - \sum_{t \in \Lambda_{\delta}^c} p_t^2 
\sum_{s \in \Lambda_{\delta}} a_{st}^2 p_s^2  \\ 
& = &
 \nonumber
  \sum_{s \in \Lambda_{\delta}} 
  \sum_{t \in \Lambda_{\delta}^c} a_{st}^2\big({\xi}_s {\xi'}_t - p_s^2 p_t^2\big) \\
  & = &
   \nonumber
\sum_{s \in \Lambda_{\delta}} \sum_{t \in \Lambda_{\delta}^c}
a_{st}^2 Z_s Z'_t 
+ \sum_{s \in \Lambda_{\delta}} \sum_{t \in \Lambda_{\delta}^c} p_t^2 
a_{st}^2 Z_s +
\sum_{s \in \Lambda_{\delta}} p_s^2 
\sum_{t \in \Lambda_{\delta}^c}
a_{st}^2  Z'_t  \\
 \label{eq::step4}
& = & Q_1 + L_1 + L_2
\een
where  for each fixed 
$\delta$, $Q_1, L_1, L_2$ are quadratic and linear terms involving mean-zero
independent random variables in $\{Z_s\}_{s \in \Lambda_{\delta}}$ and  $\{Z'_t\}_{t \in 
  \Lambda_{\delta}^c}$, which are in turn mutually independent.
 Hence by~\eqref{eq::step0}, \eqref{eq::definegU} and~\eqref{eq::step4},
 \ben
 \nonumber
 \E_U \exp\big(4 \lambda g(U) \big) & = & \E_U \exp\big(4 \lambda \E_{\delta} [S_{\delta} - \E (S_{\delta} | \delta)]\big) \\
 & = &
 \nonumber
 \E_U \exp\big(4\lambda \E_{\delta} [Q_1 + L_1 + L_2] \big) \\
 & = &
 \nonumber
 \E_U\big( \exp\big(4\lambda \E_{\delta} [Q_1]\big) \exp\big( 4\lambda \E_{\delta}[L_1 + L_2] \big)\big) \\
 & \le &
 \label{eq::step1}
\big(\E_U \exp\big(8\lambda \E_{\delta} [Q_1] \big) \big)^{1/2}
\big(\E_U \exp\big(8\lambda \E_{\delta}[L_1 + L_2]  \big) \big)^{1/2}
\een
where $\E_{\delta}$ denotes the conditional expectation with respect 
to randomness in $\delta$ for fixed $U$; The second equality
 holds by linearity of expectations, $\E_{\delta}[Q_1 + L_1 + L_2]=\E_{\delta}[Q_1] +\E_{\delta}[L_1 +
 L_2]$, and~\eqref{eq::step1} follows from the Cauchy-Schwartz inequality.

\noindent{\bf Computing moment generating functions.}
We have by Jensen's inequality and Fubini theorem,  for $\abs{\lambda} \le 1/{(32 \twonorm{A_0}^2)}$,
\ben
\nonumber 
 \E_U \exp\big(8\lambda \E_{\delta} [Q_1] \big) & \le &   \E_U
 \E_\delta \exp\big(8\lambda Q_1 \big)=
\E_\delta  \E_U \exp\big(8\lambda  \big( \sum_{s \in \Lambda_{\delta}} Z_s 
   \sum_{t \in \Lambda_{\delta}^c} a_{st}^2 {Z'}_t \big)\big)\\
\label{eq::step2}
 &\le &
 \E_\delta \exp\big(65 \lambda^2  \twonorm{A_0}^2 e^{1/4}
   \big( \sum_{s \in \Lambda_{\delta}} p_s^2      \sum_{t \in \Lambda_{\delta}^c} a_{st}^2 p_t^2 \big)\big) \\
 \nonumber
 &\le &
  \exp\big(65 \lambda^2  \twonorm{A_0}^2 e^{1/4}
    \big( \sum_{s \not=t} a_{st}^2 p_s^2 p_t^2 \big) \big) 
 \een
 where~\eqref{eq::step2} follows from the proof of Theorem~\ref{thm::Bernmgf} (cf. \eqref{eq::mgS1} and \eqref{eq::precur}), where we replace
 $A$ with $A_0 \circ A_0$ while adjusting constants; The second inequality holds since for all $\delta$ and
 \bens
\forall  t>0, \quad \exp \big( t \sum_{s \in \Lambda_{\delta}} p_s^2 
   \sum_{t \in \Lambda_{\delta}^c} a_{st}^2 p_t^2 \big) \le
\exp\big(t \sum_{s  \not=t} a_{st}^2 p_s^2 p_t^2 \big)  
 \eens
Now for any fixed $\delta$, $L_1$ and $L_2$ are independent random
variables with respect to the randomness in $U$.
By Jensen's inequality and Fubini theorem again,   for all $\lambda
\in \R$,
\bens
\lefteqn{\E_U \exp\big(8\lambda \E_{\delta}  (L_1 + L_2)  \big)
\le \E_U \E_{\delta} \exp\big(8\lambda (L_1 + L_2)  \big) }\\
& = &
\E_{\delta} \E_U  \exp\big(8\lambda (L_1 + L_2) \big) = 
\E_{\delta} \big(\E_U  \exp\big(8\lambda (L_1)  \big) \E_U  \exp\big(8\lambda
    (L_2)  \big) \big) 
\eens
Thus conditioned on $\delta$, we denote by $d_s$ and $d_t$, the
following fixed constants:
\ben 
\label{eq::defDmax}
\forall s \in \Lambda_{\delta},
\quad 0 \le d_s & := & \sum_{t \in  \Lambda_{\delta}^c}  a_{st}^2 p_t^2 
\le \sum_{t =1}^m  a_{st}^2 p_t^2  \le D_{\max} = \twonorm{A_0}^2 \\
\forall t \in \Lambda_{\delta}^c,
\quad 0 \le d_t & := & \sum_{s \in  \Lambda_{\delta}}  a_{st}^2 p_s^2 
\le \sum_{s=1}^m  a_{st}^2 p_s^2  \le D_{\max} = \twonorm{A_0}^2 
\een
where $\E \xi_s = \E u^s_i u^s_j = p_s^2$ and recall that
\bens
L_1 
& := &
\sum_{s \in \Lambda_{\delta}} \sum_{t \in \Lambda_{\delta}^c} p_t^2 
a_{st}^2 Z_s = \sum_{s \in  \Lambda_{\delta}} d_s (\xi_s -p_s^2) \\
L_2
& := &
\sum_{s \in \Lambda_{\delta}} p_s^2 
\sum_{t \in \Lambda_{\delta}^c}
a_{st}^2  Z'_t  =\sum_{t \in  \Lambda_{\delta}^c} d_t (\xi'_t -p_t^2) 
\eens
Now by Lemma~\ref{lemma::bern-sumlocal}, we have for  $\tau = 8
\lambda \le 1/{(4D_{\max})}$ where $D_{\max} = \twonorm{A_0}^2$,
\bens
\lefteqn{\E_{U_{\Lambda_{\delta}}} \exp\big(\tau \sum_{s \in
    \Lambda_{\delta}} d_s(\xi_s - p_s^2) \big)
 = \prod_{s \in \Lambda_{\delta}}  \E \exp\big(\tau d_s \big( u^s_i
 u^s_j - \E u^s_i u^s_j
\big) \big) }\\
 & \le & \exp\big(\half \tau^2 e^{\abs{\tau} D_{\max}} \sum_{s \in 
     \Lambda_{\delta}}  d_s^2   p_s^2 \big)  \le 
\exp\big(\half \tau^2 D_{\max} e^{\abs{\tau} D_{\max}}
  \sum_{s \in  \Lambda_{\delta}}    \sum_{t \in  \Lambda_{\delta}^c} a_{st}^2 p_s^2 p_t^2\big)
 \eens
 Similarly, for $d_t := \sum_{s \in  \Lambda_{\delta}}  a_{st}^2
 p_s^2$ and  $\tau = 8 \lambda$
\bens
\lefteqn{
 E_{U_{\Lambda_{\delta}^c}}
 \exp\big(8\lambda (L_2)  \big) := 
 E_{U_{\Lambda_{\delta}^c}}   \exp\big(8\lambda \sum_{t \in \Lambda_{\delta}^c} d_t Z'_t \big)} \\
 & \le &
 \exp\big(\half \tau^2 e^{\abs{\tau} D_{\max}}
   \sum_{t \in  \Lambda_{\delta}^c}  d_t^2   p_t^2 \big)  \le 
\exp\big(\half \tau^2 D_{\max} e^{\abs{\tau} D_{\max}}
\sum_{t \in  \Lambda_{\delta}^c}  \sum_{s \in  \Lambda_{\delta}}  a_{st}^2 p_s^2 p_t^2\big)
\eens
Hence for $\abs{\lambda } \le 1/{(32 \twonorm{A_0}^2)}$ and
$\abs{\tau} \le 1/{(4 D_{\max})}$ 
\ben
\nonumber
\E_U \exp\big(8\lambda \E_{\delta}  (L_1 + L_2)  \big) 
& \le &
\nonumber
\E_{\delta}\big(\E_U  \exp\big(8\lambda (L_1) \big)  \E_U  \exp\big(8\lambda
    (L_2)  \big) \big)\\
& \le &
\nonumber
 \E_{\delta}
 \big(\exp\big(\tau^2 D_{\max} e^{\abs{\tau} D_{\max}}   \sum_{s \in
       \Lambda_{\delta}}    \sum_{t \in  \Lambda_{\delta}^c} a_{st}^2
     p_s^2 p_t^2 \big) \big)\\
 \label{eq::step3}
 & \le & \exp\big(64 e^{1/4} \lambda^2 D_{\max}  \sum_{s \not=t} a_{st}^2 p_t^2  p_s^2 \big)
 \een
\noindent{\bf Putting things together.}
Hence by~\eqref{eq::step0},~\eqref{eq::step1},~\eqref{eq::step2} and
\eqref{eq::step3}, we have for all $\abs{\lambda } \le 1/{(32 \twonorm{A_0}^2)}$,
\bens
\lefteqn{\E \exp(\lambda (S_{\diamond}  - \E S_{\diamond})) = \E_U \exp\left\{4 \lambda g(U)\right\}}\\
& \le &
\big(\E_U \exp\big(8\lambda \E_{\delta}[ Q_1] \big) \big)^{1/2}
\big(\E_U \exp\big(8\lambda \E_{\delta}  [L_1 + L_2]  \big)\big)^{1/2} \\
& \le &  
\exp\big(65 \lambda^2  \twonorm{A_0}^2 e^{1/4}  \big( \sum_{s
      \not=t} a_{st}^2 p_s^2 p_t^2 \big) \big)  \le  \exp\big(65 e^{1/4} \lambda^2 D_{\max}   \sum_{s \not=t} a_{st}^2 p_t^2  p_s^2 \big)
 \eens 
 \qed

\section{Proof of Theorem~\ref{thm::Bernmgf}}
\label{sec::proofofbasemgf}
We first state the following Decoupling Theorem~\ref{thm::decouple}, which follows from
Theorem 6.1.1~\cite{Vers18}.
\begin{theorem}{\textnormal{(\cite{Vers18})}}
\label{thm::decouple}
Let $A$ be an $m \times m$ matrix.
Let $X =(X_1, \ldots, X_m)$ be a random vector with independent
mean zero coordinates $X_i$. Let $X'$ be an  independent copy of $X$.
Then, for every convex function $F: \R \mapsto \R$, one has
\ben
\label{eq::decoupled}
\E F(\sum_{i\not=j} a_{ij}X_i X_j)
 \le \E F(4 \sum_{i\not=j} a_{ij}X_i X'_j).
\een
\end{theorem}
We use the following bounds throughout our paper.
For any $x \in \R$, 
\ben 
\label{eq::elem}
e^x \le 1 + x + \half x^2 e^{\abs{x}}.
\een
Let $Z_i := \xi_i - p_i$.
Denote by $\sigma_i^2 = p_i(1-p_i)$.
For all  $Z_i$, we have $\abs{Z_i} \le 1$, $\E Z_i = 0$ and 
\ben
\label{eq::Z2}
\E Z_i^2 & = & (1-p_i)^2 p_i + p_i^2 (1-p_i) = p_i(1-p_i) =\sigma^2_i,
\\
\label{eq::Zabs}
\E \abs{Z_i} & = & (1-p_i) p_i + p_i(1-p_i) = 2p_i(1-p_i) = 2 \sigma^2_i.
\een
\begin{proofof}{Theorem~\ref{thm::Bernmgf}}
  Denote by $\breve{a}_{i} := \sum_{j\not=i} (a_{ij} + a_{ji}) p_j +
a_{ii} \le 2 D $.
We express the quadratic form as follows:  
\bens
\sum_{i=1}^m a_{ii} (\xi_i - p_i) 
+ \sum_{i\not=j} a_{i j} (\xi_i  \xi_{j}- p_i p_{j})
& = &
\sum_{i\not=j} a_{ij} Z_i Z_j + \sum_{j=1}^m Z_j \breve{a}_{i} =: S_1
+ S_2.
\eens
We will show the following bounds on  the moment generating functions
of $S_1$ and $S_2$: for every $\abs{\lambda} \le \inv{16 (\norm{A}_1 \vee 
  \norm{A}_{\infty})}$,
\ben
\label{eq::mgS1}
\quad \E \exp(\lambda 2S_1) \le  \exp \big(65\lambda^2 \norm{A}_{\infty}
 e^{8 \abs{\lambda} \norm{A}_{\infty} } \sum_{i\not=j} 
 \abs{a_{ij}} \sigma^2_i  \sigma^2_j \big) \; \text{ and} \\
 \label{eq::mgS2}
 \quad \quad
\E \exp(\lambda 2 S_2)  \le
 \exp\big(4 \lambda^2  D e^{4 \abs{\lambda} D}
\big(2 \sum_{i\not=j}^m \abs{a_{ij}} p_i p_j+\sum_{i=1}^m \abs{a_{ii}} \sigma_i^2\big)
\big).
\een
The estimate on the moment generating function for $\sum_{i,j} a_{ij}
\xi_i \xi_j$ then follows  immediately from the Cauchy-Schwartz
inequality. \\
\noindent{\bf Bounding  the moment generating function for $S_1$.}
In order to bound the moment generating function for $S_1$, 
we start by a decoupling step following
Theorem~\ref{thm::decouple}. Let  $Z'$ be an  independent copy of
$Z$. \\
\noindent{\bf Decoupling.} 
Now consider random variable $S_1  :=  \sum_{i\not=j } a_{ij} (\xi_i
-p_i)(\xi_j -p_j)  = \sum_{i\not=j } a_{ij} Z_i Z_j$ and
\bens 
S'_1 := \sum_{i \not=j} a_{ij} Z_i Z'_j, 
\; \text{ we have } \; \;
\E \exp(2\lambda S_1)  \le \E \exp(8\lambda S'_1) =: f
\eens
by~\eqref{eq::decoupled}. Thus we have by independence of $Z_i$,
\ben
\label{eq::precur}
\quad f  := \E_{Z'}\E_{Z}\exp\big(8 \lambda  \sum_{i=1}^m Z_i  \sum_{j
    \not=i} a_{ij} Z'_j\big)  =  \E_{Z'} \prod_{i=1}^m \E \big(\exp\big(8 \lambda Z_i \tilde{a}_{i} \big) \big),
\een
where $Z'_j, \forall j$ satisfies $\abs{Z'_j} \le 1$, and 
\ben
\label{eq::anorm}
\forall i, \quad \tilde{a}_{i} := \sum_{j\not=i} a_{ij} Z'_j \; \; \text{ and hence} \; \; 
\abs{\tilde{a}_{i}} \le \norm{A}_{\infty}. 
\een
First consider $Z'$ being fixed.
Recall $Z_i, \forall i$ satisfies: $\abs{Z_i} \le 1$, $\E Z_i = 0$ and 
$\E Z_i^2 = \sigma_i^2$. 
Then by~\eqref{eq::elem}, for all $\abs{\lambda} \le \inv{16
  \norm{A}_{\infty} }$ and $t_i := 8 \lambda \tilde{a}_i$,
\ben
\nonumber
\E\exp\big(8 \lambda  \tilde{a}_{i} Z_i\big)  
& := & \E\exp\big(t_i Z_i\big) \le 
1 + \inv{2} t_i^2 \E Z_i^2 e^{\abs{t_i}} \le
\exp\big(\inv{2} t_i^2 e^{\abs{t_i}} \E Z_i^2 \big) \\
  \nonumber
& \le &  \exp\big(32 \lambda^2 \norm{A}_{\infty} e^{8 \abs{\lambda} \norm{A}_{\infty}}
 \abs{\tilde{a}_{i}} \sigma^2_i  \big) \\
\label{eq::oneterm}
& =: &  \exp\big(\tau'\abs{ \tilde{a}_{i}} \sigma^2_i \big)
;\; \text{ for } \; \;
\tau'  :=  32 \lambda^2 \norm{A}_{\infty} e^{8 \abs{\lambda}
  \norm{A}_{\infty} } \ge 0,
\een
where by \eqref{eq::anorm}
\ben
\nonumber
\half t_i^2 e^{\abs{t_i}}\le 
\inv{2} (8)^2 \lambda^2 \tilde{a}_{i}^2  \sigma_i^2 
e^{\abs{8 \lambda \tilde{a}_{i}}} \le 
32 \lambda^2 \norm{A}_{\infty} \abs{\tilde{a}_i } e^{8
  \abs{\lambda}\norm{A}_{\infty} }; \\
\text{Denote by} \; \;
\label{eq::abar}
\abs{\bar{a}_{j}} := \sum_{i \not=j}
\abs{a_{ij}} \sigma^2_i \le \norm{A}_1/4, \; \; j =1, \ldots, m.
\een
Thus by \eqref{eq::precur} and \eqref{eq::oneterm},
\bens
\quad
f & \le &  
\E_{Z'} \prod_{i=1}^m
\exp\big(\tau' \abs{\tilde{a}_{i}}  \sigma^2_i\big) \le 
\E_{Z'} \exp\big(\tau'\sum_{i=1}^m \sigma^2_i \sum_{j\not=i} \abs{a_{ij}}\abs{Z'_j}\big) \\
& = & 
\prod_{j=1}^m \E \exp\big(\tau' \abs{Z'_j} 
\sum_{i \not=j}^m \abs{a_{ij}} \sigma^2_i \big) 
=: \prod_{j =1}^m \E\exp\big(\tau' \abs{\bar{a}_{j}}\abs{Z'_j} \big)
\eens
where $\E (Z'_j)^2 =\sigma^2_j$ and $\E \abs{Z'_j} = 2 \sigma^2_j$
following~\eqref{eq::Z2} and~\eqref{eq::Zabs}.
Denote by
\bens
\breve{t}_j
:= \tau' \abs{\bar{a}_{j}} = 
32 \lambda^2 \norm{A}_{\infty} e^{8 \abs{\lambda} \norm{A}_{\infty} }
\abs{\bar{a}_{j}} > 0,
\eens
we have by \eqref{eq::abar} and for $\abs{\lambda} \le1/{(16(\norm{A}_1 \vee \norm{A}_{\infty}))}$ ,
\bens
\breve{t}_j := 
32 \lambda^2 \norm{A}_{\infty} \abs{\bar{a}_{j}}  e^{8 \abs{\lambda} \norm{A}_{\infty} }
\le 
\frac{\norm{A}_{\infty} \norm{A}_1 e^{8 \abs{\lambda} \norm{A}_{\infty} } }{32 (\norm{A}_1 \vee \norm{A}_{\infty})^2}
\le \frac{e^{1/2}}{32} \approx 0.052;
\eens
Thus we have by the elementary approximation~\eqref{eq::elem},~\eqref{eq::mgS1} holds,
\bens
\lefteqn{
 \E\exp\big(\breve{t}_j \abs{Z'_j} \big)
\le  1 + \E \big(\breve{t}_j \abs{Z'_j} \big) + 
\inv{2} (\breve{t}_j)^2 \E (Z'_j)^2 e^{\abs{\breve{t}_j}} }\\
& \le &
 \exp\big(2 \breve{t}_j \sigma^2_j + \inv{2} (\breve{t}_j)^2  
\sigma^2_j e^{0.052}\big) \; \text{  using the inequality of $x \le e^x$,}\\ 
& \le &  \exp\big(2 \breve{t}_j \sigma^2_j + 0.026 \breve{t}_j \sigma^2_j  e^{0.052} \big) \le 
 \exp\big(2.03 \breve{t}_j \sigma^2_j\big)\\
& \le & \exp\big(65 \lambda^2 \norm{A}_{\infty} e^{8 \abs{\lambda}
    \norm{A}_{\infty} } \sum_{i \not=j} \abs{a_{ij}} \sigma^2_i
  \sigma^2_j \big) \; \; \text{ so long as } \abs{\lambda} \le 1/{16 (\norm{A}_1 \vee 
  \norm{A}_{\infty})}.
  \eens
\noindent{\bf Bounding  the moment generating function for $S_2$.} 
Recall
\bens
S_2 & := & 
\sum_{i=1}^m Z_i 
\big(\sum_{j\not=i} (a_{ij} + a_{ji}) p_{j} +
a_{ii}\big) =: \sum_{i=1}^m Z_i \breve{a}_{i}.
\eens
Let $a_{\infty} := \max_{i} \abs{\breve{a}_{i}} \le \norm{A}_{\infty} +
\norm{A}_1 \le 2 D$. 
Thus we have by Lemma~\ref{lemma::bern-sumlocal}, for all
$$\abs{\lambda} \le 1/{(16  (\norm{A}_{\infty} \vee \norm{A}_1))},$$
\bens
\E\exp\big(2\lambda\sum_{i=1}^m Z_i \breve{a}_{i}\big) 
\label{eq::s2bound}
\le \exp\big(2 \lambda^2   e^{2\abs{\lambda} a_{\infty}} 
\sum_{i=1}^m  \breve{a}_{i}^2 \sigma^2_i \big) \le  \exp\big(2 \lambda^2 a_{\infty}  e^{2\abs{\lambda} a_{\infty}} 
  \sum_{i=1}^m  \abs{\breve{a}_{i}} \sigma^2_i \big) 
\eens
where $\forall i, \abs{\breve{a}_{i}} =  \abs{\sum_{j\not=i} (a_{ij} +
  a_{ji}) p_{j} + a_{ii}}$ and hence
\bens
  \sum_{i=1}^m  \abs{\breve{a}_{i}} \sigma^2_i
  & \le &
  2  \sum_{i=1}^m \sum_{j \not=i}^m \abs{a_{ij}} p_i p_j+  \sum_{i=1}^m 
    \abs{a_{ii}} \sigma_i^2  \quad \text{ where } \; \ \sigma_i^2  =p_i(1-p_i)
\eens
Thus \eqref{eq::mgS2} holds for all $\abs{\lambda} \le 
\inv{16  (\norm{A}_{\infty} \vee \norm{A}_1)}$.
Hence by the Cauchy-Schwartz inequality, in view of \eqref{eq::mgS1}
and \eqref{eq::mgS2},
\bens
\lefteqn{
\E \exp\big(\lambda \big(\sum_{i=1}^m a_{ii} (\xi_i - p_i) 
+ \sum_{i\not= j} a_{ij}  (\xi_i \xi_j - p_i p_j )\big)\big)}\\
& =& \E \exp\big(\lambda ( S_1 + S_2)\big) \le
(\E\exp(2 \lambda  S_1))^{1/2} (\E\exp(2\lambda  S_2))^{1/2}
\eens
for all  $\abs{\lambda} \le \inv{16 (\norm{A}_{\infty} \vee  \norm{A}_1)} =
\inv{16 D}$.
The theorem is thus proved since $ \norm{A}_{1} \vee
\norm{A}_{\infty} =: D$.
\end{proofof}

\section{Proof of Lemma~\ref{lemma::pairwise}}
\label{sec::proofofpairwise}
We need some preliminary results.
Let $\EE_8^c$ and $\F_5^c$ be as in Lemma~\ref{lemma::unbiasedmask}.
We first define event $\F_9^c$ in Lemma~\ref{lemma::eventF9}, followed
by event $\F_8^c$ in Lemma~\ref{lemma::butterfly}.
The proof of Lemma~\ref{lemma::eventF9}
follows that of Lemma~\ref{lemma::W2devi} while 
properly adjusting constants and hence is omitted.
\begin{lemma}
  \label{lemma::eventF9}
  For all $i \not=k$, denote by
$T_{ik}:=  \abs{\tr\big(A_0 \diag(v^i \otimes v^k)\big)   -  \sum_{j=1}^m a_{jj} p_j^2}$.
Under the  conditions in Lemma~\ref{lemma::pairwise},
we have on event $\F_9^c$, which  holds with probability at least  $1- \frac{1}{2(m \vee n)^4}$, 
$T_{ik} \le C_2 \big(\log (m \vee n) a_{\infty} {\sum_{j=1}^m a_{jj}    p_j^2}\big)^{1/2}$.
\end{lemma}

\begin{lemma}
  \label{lemma::butterfly}
  Suppose all conditions in Lemma~\ref{lemma::pairwise} hold.
  Let $C_{10} = 2 e^{1/8} < 2.27$ and $C_{13} = 32.5 e^{1/4} + 4e^{1/8} < 46.27$.
  Let $W_4 := C_{10} \sum_{s=1}^m a^2_{ss} p^2_s + C_{13}\sum_{s\not=t} a_{st}^2  p_s^2 p_t^2$.
Denote by
 $$\abs{S_6^{\star}} := \max_{i \not=j}
  \abs{{S}_{\star}(i,j)},\;  \text{ where} \; S_{\star}(i,j) :=
  \sum_{s,t} a^2_{st} \xi_s \xi_t - \E \sum_{s,t} a^2_{st} \xi_s
  \xi_t$$
  for $\xi := (v^i \circ v^j), \forall i \not=j$.
  Then on event $\F_8^c$, which holds with probability at least
$1-c_2/{ (n \vee m)^4}$, for some absolute constants $C_8, C_{11}$ and $C_{12}$,
\ben
\nonumber
\abs{S_6^{\star}} & \le& \tau_6 \le C_8 \twonorm{A_0} \big( \twonorm{A_0} \log (n \vee m) 
\vee  \sqrt{\log (n \vee m) W_4}\big) \\
& &
\label{eq::tau6up}
\le  C_{11} W_4 + C_{12} \twonorm{A_0}^2 \log(m \vee n) 
\een
\end{lemma}

\begin{proof}
Recall $\xi = (\xi_1, \ldots, \xi_m) \in \{0, 1\}^m$ is a random vector 
with independent Bernoulli random variables $\xi_k$ such that 
$\xi_k = 1$ with probability $p^2_k$ and $0$ otherwise.
Hence  $\sigma_k^2 = p^2_k(1-p_k^2) \le p_k^2$.
Denote by  $D_0 := \norm{A_0 \circ A_0}_{\infty} \vee \norm{A_0 \circ
  A_0}_1$. Then $D_0 = \max_{j=1}^m \sum_{i=1}^m a_{ij}^2 \le \twonorm{A_0}^2$.
Then by Theorem~\ref{thm::Bernmgf}, we have for
every $\abs{\lambda} \le 1/{(32 \twonorm{A_0}^2)}$, for all $i \not=j$,
\bens
\lefteqn{\E \exp(\lambda S_{\star}(i,j))
\le   \exp\big(32.5 \lambda^2 D_0 e^{8 \abs{\lambda} D_0}
  \sum_{s \not=t} a^2_{st} \sigma^2_s \sigma^2_t \big)  \cdot}\\
&& \exp\big(\lambda^2 D_0 e^{4\abs{\lambda} D_0} \big(2\sum_{s=1}^m
a^2_{ss} \sigma^2_s + 4 \sum_{s\not= t} a^2_{st} p^2_s p^2_t \big)  \big)
\le  \exp\big(\lambda^2 \twonorm{A_0}^2  W_4 \big)
\eens
Fix $t>0$. We have for $\lambda > 0$ and ${S}_{\star}: 
={S}_{\star}(i,j)$, by Markov’s inequality,
    \bens
  \prob{S_{\star} \ge t}
  & = &
  \prob{\lambda S_{\star} \ge \lambda t}   =   \prob{\exp(\lambda S_{\star} ) \ge \exp(\lambda t)}  \\
& \le & {\E \exp(\lambda S_{\star})}/{e^{\lambda t}}
\le \exp\big(-\lambda t +\lambda^2 \twonorm{A_0}^2 W_4 \big);
\eens
Optimizing over $\lambda$, we have for $t>0$, 
\bens
\prob{{S}_{\star} \ge t} & \le & \exp\big(-c_4\min \big(\frac{t^2}{\twonorm{A_0}^2 W_4},
    \frac{t}{\twonorm{A_0}^2}\big) \big) =: p_1
    \eens
Repeating the same arguments for $0< \lambda \le 1/{(32 \twonorm{A_0}^2)}$ and $t>0$,  we have
\bens 
\prob{{S}_{\star} < -t}
 =  \prob{\exp(\lambda(-{S}_{\star}) )> \exp(\lambda  t)} \le
{\E \exp(-\lambda {S}_{\star})}/{e^{\lambda t}} \le p_1;
\eens 
Hence $\prob{\abs{{S}_{\star}(i,j)} > t} \le 2p_1$.
Now set $\tau_6 = C_8 \twonorm{A_0} \big( \twonorm{A_0} \log (n \vee
m)   \vee  \sqrt{\log (n \vee m) W_4}\big)$; Then by the union bound,
for $C_8$ sufficiently large,
\bens 
\lefteqn{\prob{\abs{S_6^{\star}} > \tau_6}
=\prob{\exists i \not=j, \abs{{S}_{\star}(i, j)} \ge \tau_6}   =: \prob{\F_8}}\\
& \le &
{n \choose 2} 2\exp\big(- c\min\big(\frac{\tau_6^2}{\twonorm{A_0}^2
  W_4}, \frac{\tau_6}{\twonorm{A_0}^2} \big)\big) \le c_2/{ (n \vee m)^4}
\eens
Using the fact that $2\sqrt{ab} \le a + b$ for $a, b>0$, we have
\eqref{eq::tau6up}.
\end{proof}

\subsubsection{Proof of Lemma~\ref{lemma::pairwise}}
\begin{proofof2}
Throughout the rest of this section, denote by $Z \in \R^{mn}$ a
subgaussian random vector with independent components $Z_j$ that satisfy $\expct{Z_j} = 0$, $\expct{Z_j^2} = 1$, and $\norm{Z_j}_{\psi_2} \leq 1$. 
Denote by $A(i, j) = c_i c_j^T \otimes\big( A_0^{1/2} \diag(v^i \otimes v^j) A_0^{1/2} \big)$; 
Then $\ip{X^i \circ v^i, X^{j} \circ v^j} =  Z^T \big( c_i c_j^T \otimes\big( A_0^{1/2} \diag(v^i \otimes v^j) 
A_0^{1/2} \big) \big)  Z = Z^T A(i, j) Z$. We have $\forall i  \not=j$,
  \ben
     \label{eq::S1op}
     && \quad \quad { \twonorm{A(i, j)}}/{{\sqrt{b_{ii} b_{jj}}}}\le
     \twonorm{A_0},\; \text{ where} \; \; \shnorm{c_i c_j^T}_2 
     =\shnorm{c_i c_j^T}_F =      \sqrt{b_{ii}b_{jj}}, \\
 \label{eq::S1Fnorm}
 &&  \quad \text{ and on event $\F_8^c$ },  
{\fnorm{A(i, j)}}/{\sqrt{b_{ii} b_{jj}}} = \\
\nonumber
&  &
\norm{A_0^{1/2} \diag(v^i \otimes v^j)     A_0^{1/2} }_F =
O \big(\big(a_{\infty} \twonorm{A_0} \sum_{j=1}^m 
   p_j^2\big)^{1/2}\big) \; \text{by Lemma~\ref{lemma::butterfly}}.
   \een
\noindent{\bf  Step 1.} 
On event $\F_8^c$, \eqref{eq::S1Fnorm} holds $\forall i \not=j$ and
$\xi := (v^i \circ v^j)$, since by \eqref{eq::tau6up}
\ben 
\nonumber 
\lefteqn{
  \fnorm{A_0^{1/2}\diag(v^i \otimes v^j) A_0^{1/2}}^2
  = \sum_{i,j} a^2_{ij} \xi_i \xi_j \le \E \sum_{i,j} a^2_{ij} \xi_i \xi_j 
+  \tau_6  }\\
\label{eq::F1}
& & \quad \quad \quad \le \E \sum_{i,j} a^2_{ij} \xi_i \xi_j + C_{11} W_4
+ C_{12} \twonorm{A_0}^2  \log(m \vee n)   \le C_6 A_{\vecp}^2 \text{
  where} \\
\nonumber
&& \lefteqn{A_{\vecp}^2 := a_{\infty}\twonorm{A_0}  \sum_{i=1}^m p_i^2
=\Omega(\twonorm{A_0}^2  \log(m \vee n)) \; \text{ and } \;  \E
\sum_{i,j} a^2_{ij} \xi_i \xi_j =}\\
\nonumber
&&  \sum_{i=1}^m a^2_{ii} p^2_i +\sum_{i\not= j} a^2_{ij}  p^2_i p^2_j 
\le  \lambda_{\max}(A_0 \circ A_0) (\sum_{i=1}^m p_i^2) \le A_{\vecp}^2
\een
Hence by Theorem~\ref{thm::HW},~\eqref{eq::S1op},~\eqref{eq::S1Fnorm}, and~\eqref{eq::F1}, for any $\tau > 0$ and $i\not=j$,
\ben 
\nonumber 
\lefteqn{\prob{\left\{\abs{Z^T A(i,j) Z - \E(Z^T A(i, j) Z | U) }/{\sqrt{b_{ii} b_{jj}}} > \tau \right\} | U \in \F_8^c}  }\\
& &
\label{eq::neighbor}
\quad \quad \quad \le 2\exp \big(-  c\min\big(\frac{\tau^2}{C_6 A_{\vecp}^2 },\frac{\tau}{\twonorm{A_0}} \big)\big) 
\een 
\noindent{\bf  Step 2.}
On event $\F_8^c$, $A(i, j)$ can be treated as being deterministic
satisfying~\eqref{eq::S1op} and~\eqref{eq::S1Fnorm}. Hence we set $\tau_0  \asymp A_{\vecp} \log^{1/2} (m \vee n)$;
By {\bf Step 1},
\bens
\lefteqn{\prob{\exists i, j, i \not=j, \;\abs{Z^T A(i, j)
      Z - \E(Z^T A(i, j) Z | U) }/\sqrt{b_{ii}b_{jj}} > \tau_0} =: \prob{\F_7}} \\
& \leq &  \prob{\left\{\exists i, j, i \not=j,  \;\frac{\abs{Z^T A(i, j) Z - \E(Z^T 
        A(i, j) Z | U) }}{\sqrt{b_{ii}b_{jj}}} > \tau_0 \right\} \cap \F_8^c} +  \prob{\F_8}  \\
& \leq &  n^2 \exp \big(-c\min\big(\frac{\tau_0^2}{C_6 A_{\vecp}^2},
\frac{\tau_0}{\twonorm{A_0}}\big) \big) \prob{\F_8^c} + \prob{\F_8} \le {c_7}/{(m \vee n)^4}
\eens
\noindent{\bf  Step 3.}
By Lemma~\ref{lemma::eventF9}, we have on
event $\F_9^c$, $\forall i \not=j$,
\bens
\lefteqn{{\abs{\E (Z^T A(i, j) Z) - \E(Z^T A(i, j) Z | U)
    }}/{\sqrt{b_{ii}b_{jj}}} \le  T_{ij}\abs{b_{ij}}/{\sqrt{b_{ii}b_{jj}}} }\\
& \le &
C_2 \big(a_{\infty} \log (m \vee n) \sum_{i=1}^m a_{jj} p_j^2\big)^{1/2} \text{ where } 
\E Z^T A(i, j) Z
= b_{ij} \sum_{i=1}^m a_{jj} p_j^2
\eens
\noindent{\bf  Step 4.}
Putting things together,~\eqref{eq::pairwise} follows from similar
steps in the proof of Lemma~\ref{lemma::S3main}.

\end{proofof2}

\section{Proof of Lemma~\ref{lemma::Gammabounds}}
\label{sec::proofofGamma}
Let $c, c_0, c_1, c_2, c', C, C', C_3, C_{\offd}, C_{\diag} \ldots$ be some absolute
constants.
First we have $\forall j \in [n]$,
\ben
\nonumber
 && \norm{\hat\Gamma^{(j)}    \beta^{j*} -\hat\gamma^{(j)} }_{\infty} =
\norm{ \big(\X_{\minus j}  \X_{\minus j}^T \oslash  \hat{\TM}_{\minus j,   {\minus j}} \big)\beta^{j*}-
  \X_{\minus j} \big(X^j \circ v^j\big)/{\shnorm{\hat{\TM}}_{\offd}}}_{\infty} \le \\
& &
\nonumber
 \max_{i \in [n], i\not=j}
 \big(\abs{e_i \big(\X_{\minus j}  \X_{\minus j}^T 
    \oslash \hat{\TM}_{\minus j, {\minus j}} \big)\beta^{j*} -b_{ij}}+
  \abs{b_{ij} -  e_i^T \X_{\minus j} (X^{j} \circ v^j )
    /{\shnorm{\hat{\TM}}_{\offd}}}\big) =:\\
  & &
\label{eq::lineareq}
\quad\quad\quad\quad
\max_{i \in [n], i\not=j} \left(\abs{(\hat{S}_2+\hat{S}_3)(e_i,
      \beta^{j*})} + \abs{\hat{S}_1(e_i, \beta^{j*})}\right),
  \een
where by Lemma~\ref{lemma::pairwise}, on event $\F_6^c \cap \F_5^c$
for $\ul{r_{\offd}}$ as defined in~\eqref{eq::paritydual},
\bens
  \label{eq::hatS1}
 \nonumber
\abs{\hat{S}_1(e_i, \beta^{j*}) }
:= \abs{b_{ij} - e_i^T \X_{\minus j} (X^{j} \circ v^j )/\shnorm{\hat{\TM}}_{\offd}} \le C_{\offd} (1+o(1)) \ul{r_{\offd}} {\sqrt{b_{ii}b_{jj}}}
\eens
while for the first component in
\eqref{eq::lineareq}, we will treat the diagonal and the off-diagonal components separately as follows:
\ben
\nonumber
\lefteqn{
  \abs{e_i  \big(\X_{\minus j}  \X_{\minus j}^T 
     \oslash  \hat{\TM}_{\minus j, {\minus j}} \big) \beta^{j*} -
     b_{ij}
   } \le 
 \abs{\big({e_i \diag(\X_{\minus j}  \X_{\minus j}^T)\beta^{j*}}/{\shnorm{\hat{\TM}}_{\diag}}\big) -  b_{ii} \beta_i^{j*} }+
} \\
&&
 \label{eq::hatS23}
 \abs{b_{ii} \beta_i^{j*} - b_{ij}+ \frac{e_i \offd(\X_{\minus j}  \X_{\minus 
       j}^T)\beta^{j*}}{\shnorm{\hat{\TM}}_{\offd}}} 
     =:\abs{\hat{S}_2(e_i, \beta^{j*} )} + \abs{\hat{S}_3(e_i, \beta^{j*}) }.
\een
To prove Lemma~\ref{lemma::Gammabounds}, 
we need Lemmas~\ref{lemma::S2main} and~\ref{lemma::S3main},
where we assume (A1) holds and $n$ is sufficiently large.
\begin{lemma}
  \label{lemma::S2main}
On event $\F_{\diag}^c \cap \event_8^c$, which is the same as defined in
Theorem~\ref{thm::main-coro},
we have for all $i, j \in [n], i\not=j$,
\ben
\label{eq::S2hatbounds}
\lefteqn{\abs{\hat{S}_2(e_i, \beta^{j*}) }
  :=  \abs{e_i^T \diag([\X \X^T  \oslash \hat{\TM}]_{\minus j, \minus
    j})\beta^{j*} - b_{ii} \beta_i^{j*}  } }\\
\nonumber
& \le & \big({C_{\diag}}/{(1- \delta_{m, \diag})} \big)\sqrt{b_{ii} b_{jj}}\abs{{\omega_{ij}}/{\omega_{jj}}}
\big(r_{\diag} + \delta_{m, \diag} \big)
\een
where $\omega_{jj} \ge 1$ and $\delta_{m, \diag} = O(r_{\diag}/\sqrt{r(B_0)})$ is as defined in Lemma~\ref{lemma::unbiasedmask}.
\end{lemma}

\begin{lemma}
\label{lemma::S3main}
Suppose~\eqref{eq::forte3} holds. Then on event $\F_{10}^c \cap \F_5^c$,
which holds with probability at least $1- {c_3}/{(n \vee m)^4}$,
\ben
\nonumber
\forall i \not= j, \quad 
\abs{\hat{S}_3(e_i, \beta^{j*} ) } & :=  & \abs{ (b_{ii} \beta_i^{j*}
  -b_{ij}) + {e_i \offd(\X_{\minus j}  \X_{\minus
      j}^T)\beta^{j*}}/{\shnorm{\hat{\TM}}_{\offd}}} \\
\label{eq::S3hat}
  &\le& C_3 (1+o(1)) \kappa_\rho  \ul{r_{\offd}} {\sqrt{b_{ii} b_{jj}}}/{\omega_{jj} }
\een
where $\E (e_i \offd(\X_{\minus j}  \X_{\minus   j}^T) \beta^{j*} ) =
(b_{ij} - b_{ii} \beta_i^{j*} )\sum_{j} a_{jj} p_j^2$.
\end{lemma}

\begin{proofof}{Lemma~\ref{lemma::Gammabounds}}
Suppose event $\F_{\diag}^c \cap \event_{8}^c \cap \F_{5}^c \cap \F_{6}^c \cap \F_{10}^c =: \F_{22}^c$ holds.
Then we have by~\eqref{eq::lineareq},~\eqref{eq::hatS23},
~\eqref{eq::pairwise},~\eqref{eq::S2hatbounds}, and~\eqref{eq::S3hat},
for all $j \in [n]$, 
\bens 
\lefteqn{\norm{\hat\gamma^{(j)}  -  \hat\Gamma^{(j)} \beta^{j*}}_{\infty} \le
 \max_{i \in [n], i\not=j} \big(\abs{\hat{S}_1(e_i,  \beta^{j*})} +
   \abs{\hat{S}_2(e_i,  \beta^{j*})} +\abs{\hat{S}_3(e_i,
     \beta^{j*})} \big) }\\ 
 &\le &  (1+o(1)) \ul{r_{\offd}} ({\sqrt{b_{ii}b_{jj}}}/{\omega_{jj}}) \big(C_{\offd} \omega_{jj}
+ C_{\diag}  \abs{\omega_{ij}} +C_3  \kappa_\rho\big) \\
 &\le &
 C_4  \ul{r_{\offd}} b_{\infty} \kappa_{\rho} =:   c_{\gamma}
 b_{\infty} \kappa_\rho \ul{r_{\offd}}
\; \text{ where} \; \kappa_{\rho} \ge 1,
 \eens
 and $C_4 \le (C_{\offd} + C_{\diag} +C_3)(1+o(1))$.
 Finally,  by the union bound, $\prob{\F_{\diag}^c \cap \event_8^c \cap \F_{10}^c \cap
   \F_{6}^c \cap \F_5^c} \ge 1 - {c_9}/{(n \vee
   m)^4}$. The lemma thus holds.
We defer the proof of Lemma~\ref{lemma::S3main} to Section~\ref{sec::proofofS3main}. 
\end{proofof}

\begin{proofof}{Lemma~\ref{lemma::S2main}}
On event $\F_{\diag}^c$, by Theorem~\ref{thm::diagmain}, 
\bens
\lefteqn{\abs{S_2(e_i, \beta^{j*})}:= 
\abs{b_{ii}  \beta_i^{j*} - {\big(e_i \diag(\X_{\minus j}  \X_{\minus 
      j}^T)\beta^{j*}\big)}/{\norm{\M}_{\diag}} }} \\
& \le &  C_{\diag} b_{ii} \abs{\beta_i^{j*} } r_{\diag}
\le C_{\diag} r_{\diag} \big( \sqrt{b_{ii} b_{jj}}
{\abs{\omega_{ji} }}/{\omega_{jj}} \big) \; \;
\text{where }\;\; \\
b_{ii} \abs{\beta_i^{j*} } & := &
b_{ii} \abs{\theta_{ji}} /{\theta_{jj}}
  = (\abs{b_{ii}}/{\theta_{jj}})({\abs{\omega_{ji}}}/{\sqrt{b_{jj}
      b_{ii}}} )= \sqrt{b_{ii} b_{jj}}{\abs{\omega_{ji}
    }}/{\omega_{jj}}.
  \eens
  The rest of the proof follows that of Theorem~\ref{thm::main-coro}.
\end{proofof}

\subsection{Proof of Lemma~\ref{lemma::S3main}}
\label{sec::proofofS3main}
Denote by $Z \in \R^{mn}$ a subgaussian random vector 
with independent components $Z_j$ that satisfy $\expct{Z_j} = 0$, 
$\expct{Z_j^2} = 1$, and $\norm{Z_j}_{\psi_2} \leq 1$.
 Denote by
\bens
\label{eq::S3norm}
&& \abs{S_3(e_i, \beta^{j*} )} := {\abs{e_i \offd(\X_{\minus j}  \X_{\minus 
      j}^T) \beta^{j*} -    \E (e_i \offd(\X_{\minus j}  \X_{\minus 
      j}^T) \beta^{j*} )}}/{\norm{\M}_{\offd}} \\
&& \; \text{ where } \E\big(e_i \offd(\X_{\minus j}  \X_{\minus j}^T) \beta^{j*}\big) 
= \E \big( \sum_{k \not= i, j} \beta_k^{j*} \ip{X^i \circ v^i, X^{k}
  \circ v^k}\big) = \\
& &
\big(\sum_{k \not=i, j}- {b_{ik} \theta_{jk}}/{\theta_{jj}} \big)
\sum_{j=1}^m a_{jj} p_j^2 \; \; \text{ where } \; 
\sum_{k \not=i \not=j}^n-  b_{ik}{\theta_{jk}}/{\theta_{jj}}  = b_{ij}
+ b_{ii} {\theta_{ji}}/{\theta_{jj}}. 
\eens
First we rewrite $S_3(e_i, \beta^{j*})$ as
\bens
\forall i \not= j \in [n],  \; 
S_3(e_i, \beta^{j*}) & := & \big(Z^ T A^{\diamond}(R_j, e_i) Z -\big(b_{ij} +  b_{ii}{\theta_{ji}}/{\theta_{jj}}\big) \big)/ {\norm{\M}_{\offd}} \\
\;\text{ where} \; 
A^{\diamond}(R_j, e_i) 
& := &
\sum_{k \not= i, j} \beta_k^{j*} (c_i c_k^T)  \otimes 
A_0^{1/2} \diag(v^i \otimes v^k) A_0^{1/2}
\eens
Lemmas~\ref{lemma::ABsum}  and~\ref{lemma::AFbounds} show a
deterministic bound on the operator norm and a probabilistic uniform
bound for the Frobenius norm of matrix $A^{\diamond}(R_j, e_i)$ for
all $i \not=j$.
\begin{lemma}
  \label{lemma::ABsum}
Let $\Omega = \rho(B)^{-1}$ where $\rho(B) =
({b_{jk}}/{\sqrt{b_{jj} b_{kk}}})$, 
\ben
\label{eq::ARopnorm}
\forall i \not=j, \; \; 
\twonorm{A^{\diamond}(R_j, e_i)}
& \le & 
\twonorm{A_0} \sqrt{d_0} {\sqrt{b_{ii} b_{jj}}}
\twonorm{\Omega^{(j)}} /{\omega_{jj}}
\een
\end{lemma}

\begin{lemma}
  \label{lemma::AFbounds}
Let event $\event_2^c$, $\event_3^c$ and $\event_5^c$ and absolute  
constants  $C_{\alpha}$, $C_2$, and $C_5$ be as defined in  
Lemmas~\ref{lemma::W2devi},~\ref{lemma::S3devi} 
and~\ref{lemma::S5devi} respectively.  
Suppose
$${\sum_{j=1}^m a_{jj} p_j^2}/{ \twonorm{A_0}} \ge \big(C_8 {\rho_{\max}(d_0, \abs{B_0})}/{b_{\min}}\big)^2 
\frac{a_{\infty}}{a_{\min}} \log (m \vee n)$$ for $C_8 = 8 C_2 \vee C_5 \vee C_{\alpha}$.
Then on event $\F_{15}^c = \event_2^c \cap \event_3^c \cap \event_5^c$,
which holds with probability at least $1-{c_4}/{(n \vee m)^4}$,
$\fnorm{A^{\diamond}(R_j, e_i)} \le u_f(i,j)$ for all $i \not=j$, where
\ben
\label{eq::R2Fnormlocal}
u_f(i,j) := ({3}/{\sqrt{2}})
\big(\twonorm{A_0} a_{\infty} \sum_{j=1}^m 
  p_j^2 \big)^{1/2} \sqrt{b_{ii} b_{jj} } \twonorm{\Omega^{(j)}}/{\omega_{jj}}.
\een
\end{lemma}
The proof of Lemma~\ref{lemma::AFbounds} follows identical steps to the proof of 
Theorem~\ref{thm::uninorm-intro}, though simpler, and is thus omitted.
These two bounds lead to a uniform bound on $ S_I(R_j, e_i)$,
\ben
\nonumber
\lefteqn{\forall i \not=j, \quad
  S_I(R_j, e_i):=  \abs{Z^T A^{\diamond}({R}_j, e_i) Z - \E(Z^T 
  A^{\diamond}({R}_j, e_i) Z | U) } \le }\\
&  &
\label{eq::partAI}
 C_1   \sqrt{b_{ii} b_{jj}} (\shnorm{\Omega^{(j)}}_2/{\omega_{jj}})
\big(\twonorm{A_0} a_{\infty} \log(m \vee n)  \sum_{j=1}^m  p_j^2 \big)^{1/2}
\een
on event $\F_{14}^c \cap \F_{15}^c$, for some universal constant $C_1$.
We then allow $U$ to be random, and
obtain on event $\F_{9}^c$ as in Lemma~\ref{lemma::eventF9}, a uniform
bound on $S_{\star}(R_j, e_i)$ for all $i \not=j$ that is essentially
at the same order as~\eqref{eq::partAI}; cf.~\eqref{eq::partAII}.

\begin{proofof}{Lemma~\ref{lemma::ABsum}}
Recall $\omega_{jj} =b_{jj}  \theta_{jj} \ge 1$ for all $j$ and
$$\abs{\beta_k^{j*} } = \abs{\theta_{jk}}/{\theta_{jj}} =
\abs{\omega_{jk}}/\big(\theta_{jj}\sqrt{b_{jj} b_{kk}} \big)$$
by Proposition~\ref{prop::projection}.
Hence we have
\ben
\nonumber
\lefteqn{\sum_{k \not=i, j} \abs{\beta_k^{j*} }\abs{b_{ik}}
= 
\inv{\theta_{jj} \sqrt{b_{jj}} }\sum_{k \not=i, j}
\frac{\abs{b_{ik}  \omega_{jk}}}{\sqrt{b_{kk}}}
= \frac{\sqrt{b_{ii}}}{\theta_{jj} \sqrt{b_{jj}} }
\sum_{k \not=i, j}  
\frac{\abs{\omega_{jk} b_{ik}}}{\sqrt{b_{ii} b_{kk}}} =}\\
&&
\label{eq::ABsum}
\quad \quad ({\sqrt{b_{ii} b_{jj}}}/{\omega_{jj}})
\sum_{k \not=i, j} \abs{\omega_{jk}\rho_{ik}(B)} \le 
({\sqrt{b_{ii} b_{jj}}}/{\omega_{jj}})
\twonorm{\rho^{(i)}_{\minus \{i,j\}}(B)} \twonorm{\Omega^{(j)}_{\minus \{i,j\}}} 
\een
where we use the following relation:   for all $i \not=j$,
\bens
\label{eq::morph}
{\sqrt{b_{ii}}}/(\theta_{jj}  \sqrt{b_{jj}})
& = &
{\sqrt{b_{ii}}}/{\sqrt{\theta_{jj} \omega_{jj}}} ={\sqrt{b_{ii} b_{jj}}}/{\omega_{jj}}
\eens
Next we bound for $\twonorm{c_i c_k^T} =  \fnorm{c_i c_k^T} = \sqrt{\tr(c_i c_k^T c_k
  c_i)} = \sqrt{b_{ii} b_{kk}}$,
\ben
\label{eq::wsum}
&& \sum_{k \not= i, j} \abs{\beta_k^{j*}} \twonorm{c_i c_k^T} =\sum_{k
  \not=i, j}^n \frac{\abs{\omega_{jk}} }{\theta_{jj}}
\sqrt{\frac{b_{ii}}{b_{jj}}} \le\sqrt{d_0} \twonorm{\Omega^{(j)}}
\sqrt{b_{ii} b_{jj}}/\omega_{jj} \\
\nonumber
\forall i, k, &&
\twonorm{A_0^{1/2} \diag(v^i \otimes v^k) A_0^{1/2}} \le
\twonorm{A_0} \twonorm{\diag(v^i \otimes v^k)} \le \twonorm{A_0}  \\
\nonumber
\text{ and hence}
&& \twonorm{A^{\diamond}(R_j, e_i)} \le 
    \sum_{k \not= i, j} \abs{\beta_k^{j*}} \twonorm{ (c_i c_k^T)}
    \twonorm{A_0^{1/2} \diag(v^i \otimes v^k) A_0^{1/2}} \\
    & \le &
    \nonumber 
      \twonorm{A_0}  \sqrt{d_0} {\sqrt{b_{ii} b_{jj}}}
  \norm{\Omega^{(j)} }_2 /\omega_{jj}
  \een 
  where in \eqref{eq::wsum}, we use the fact that row vectors of 
$\Omega$ are $d_0$-sparse.
\end{proofof}

\begin{proofof}{Lemma~\ref{lemma::S3main}}
We now show on event $\F_{10}^c$, where $\prob{\F_{10}^c} \ge 1-{c_{10}}/{(n \vee
  m)^4}$, for all $i \not=j$,
\ben
\label{eq::S3}
\abs{S_3(e_i, \beta^{j*} )} \le 
C_3 \kappa_\rho \ul{r_{\offd}} {\sqrt{b_{ii} b_{jj}}}/{\omega_{jj} }
\text{ where} \; \omega_{jj} \ge 1.
\een
It remains to show  the proof on Part I and Part II before we put
things together. Let $c, C, C_1, \ldots$ be some absolute constants.\\
 \noindent{\bf  Part I.}
We show that~\eqref{eq::partAI} holds.
By Theorem~\ref{thm::HW},~\eqref{eq::ARopnorm},
Lemma~\ref{lemma::AFbounds},
we have on event $\F_{15}^c$ for $\tau =\Omega(u_f(i, j) \log^{1/2} (m
\vee n))$ and 
\bens
\label{eq::eventF10}
\lefteqn{\quad \quad \quad \prob{\exists i\not= j, \; \abs{S_I({R}_j, e_i) }     > C_1 u_f(i, j) \log^{1/2}
      (m \vee n) } =:    \prob{\F_{14}}} \\
  \nonumber
& &  \le \prob{\left\{\exists i\not=j,\; \abs{S_I({R}_j, e_i)} > C_1
    u_f(i, j) \log^{1/2} (m \vee n)  \right\}\cap \F_{15}^c} +  \prob{\F_{15}}  \\
  \nonumber
   & & \le 2n^2 \exp(- c_7 \log (m \vee n)) + \prob{\F_{15}} \le {c_9}/{(m \vee n)^4} \; \text{  by \eqref{eq::forte3}
     and~\eqref{eq::ARopnorm}, } \\
 &&\text{where  on } \; \F_{15}^c, \; 
 u_f(i, j) = \Omega \big(\twonorm{A^{\diamond}(R_j, e_i)} \log^{1/2} (m \vee n) 
 \vee \fnorm{A^{\diamond}(R_j, e_i)} \big) \\
  \nonumber
\lefteqn{\text{ and hence} \; \; \prob{\left\{\abs{Z^T 
      A^{\diamond}({R}_j, e_i) Z - \E(Z^T A^{\diamond}(R_j, e_i) 
      Z | U) } > \tau \right\} \cap \F_{15}^c}} \\
& \leq &
  \nonumber
2 \exp \big(-c \min\big(\frac{\tau^2}{u_f(i,j) },   \frac{\tau}{
  \twonorm{A^{\diamond}(R_j, e_i)} }\big)\big) \le 2\exp(- c_5\log (m \vee n))
 \eens
\text{ since} $\big(\twonorm{A_0} a_{\infty} \sum_{j=1}^m p_j^2  \big)^{1/2} \ge 
 12 \twonorm{A_0}  \sqrt{d_0} \log^{1/2} (m \vee n)$.
\noindent{\bf  Part II.}
Now on event $\F_9^c$ and for $T_{ik}$ as defined in
Lemma~\ref{lemma::eventF9}, by~\eqref{eq::ABsum},
\ben
\label{eq::partAII}
\lefteqn{\quad \quad \quad \quad \quad \quad \quad S_{\star}(R_j, e_i) :=\abs{\E( Z^T  A^{\diamond}(R_j, e_i) Z | U) - \E( Z^T 
    A^{\diamond}(R_j, e_i) Z )}}\\
\nonumber
& := & \abs{\tr\big(\sum_{k \not= i, j} \beta_k^{j*} c_i c_k^T 
\otimes A_0^{1/2}  \diag(v^i \otimes v^k) A_0^{1/2}\big) - \big(
b_{ij} +  b_{ii} {\theta_{ji}}/{\theta_{jj}}\big)\sum_{j=1}^m 
a_{jj} p_j^2} \\
\nonumber
& \le & 
\sum_{k \not=i, j} \abs{\beta_k^{j*} }\abs{b_{ik}} T_{ik} \le C_2
\twonorm{\rho(B)^{(i)}_{\minus \{i,j\}}}
\twonorm{\Omega} \frac{\sqrt{b_{ii}  b_{jj}}}{\omega_{jj}}  \big(\log (m \vee n) a^2_{\infty} {\sum_{j=1}^m  p_j^2}\big)^{1/2} 
\een
Finally, we have by~\eqref{eq::partAI} and~\eqref{eq::partAII}, on event
 $\F_{14}^c \cap \F_{9}^c \cap \F_{15}^c =: \F_{10}^c$,
\ben
\nonumber
\lefteqn{\forall i \not=j, \quad
  \abs{S_3(e_i, \beta^{j*} )} \le   \big(S_I(R_j, e_i) +S_{\star}(R_j, e_i) \big)/{\norm{\M}_{\offd}}} \\
& \le &
\nonumber
(C_1 \vee C_2) \big({\sqrt{b_{jj}b_{ii}}}/{\omega_{jj}} \big) \twonorm{\Omega}  \ul{r_{\offd}}
\big(1  +   \twonorm{\rho(B)^{(i)}_{\minus \{i,j\}}}\big)\\
\label{eq::S3norm}
&  & \quad \quad \quad \le  C_3 \twonorm{\Omega} \twonorm{\rho(B)} \ul{r_{\offd}}  {\sqrt{b_{jj}b_{ii}}}/{\omega_{jj}} 
\een
Thus~\eqref{eq::S3} holds on event $\F_{10}^c$ and $\prob{\F_{10}^c}
\ge 1 - {c_{10}}/{(m \vee n)^4}$ by the union bound. 
The rest of the proof for \eqref{eq::S3hat} follows that 
of Theorem~\ref{thm::main-coro}, in view of~\eqref{eq::S3norm}, and 
hence omitted.
\end{proofof}

\bibliography{subgaussian}

\begin{thebibliography}{65}
\expandafter\ifx\csname natexlab\endcsname\relax\def\natexlab#1{#1}\fi
\expandafter\ifx\csname url\endcsname\relax
  \def\url#1{\texttt{#1}}\fi
\expandafter\ifx\csname urlprefix\endcsname\relax\def\urlprefix{URL }\fi

\bibitem[{Agarwal et~al.(2012)Agarwal, Negahban and Wainwright}]{ANW12}
\textsc{Agarwal, A.}, \textsc{Negahban, S.} and \textsc{Wainwright, M.} (2012).
\newblock Fast global convergence of gradient methods for high-dimensional
  statistical recovery.
\newblock \textit{Annals of Statistics} \textbf{40}.

\bibitem[{Allen and Tibshirani(2010)}]{AT10}
\textsc{Allen, G.} and \textsc{Tibshirani, R.} (2010).
\newblock Transposable regularized covariance models with an application to
  missing data imputation.
\newblock \textit{Annals of Applied Statistics} \textbf{4} 764--790.

\bibitem[{Banerjee et~al.(2008)Banerjee, Ghaoui and d'Aspremont}]{BGA08}
\textsc{Banerjee, O.}, \textsc{Ghaoui, L.~E.} and \textsc{d'Aspremont, A.}
  (2008).
\newblock Model selection through sparse maximum likelihood estimation for
  multivariate {G}aussian or binary data.
\newblock \textit{Journal of Machine Learning Research} \textbf{9} 485--516.

\bibitem[{Belloni et~al.(2017)Belloni, Rosenbaum and Tsybakov}]{BRT14}
\textsc{Belloni, A.}, \textsc{Rosenbaum, M.} and \textsc{Tsybakov, A.} (2017).
\newblock Linear and conic programming estimators in high-dimensional
  errors-in-variables models.
\newblock \textit{Journal of the Royal Statistical Society. Series B
  (Statistical Methodology)} \textbf{79} 939--956.

\bibitem[{Bickel et~al.(2009)Bickel, Ritov and Tsybakov}]{BRT09}
\textsc{Bickel, P.}, \textsc{Ritov, Y.} and \textsc{Tsybakov, A.} (2009).
\newblock Simultaneous analysis of {L}asso and {D}antzig selector.
\newblock \textit{The Annals of Statistics} \textbf{37} 1705--1732.

\bibitem[{Brockwell and Davis(2006)}]{BD06}
\textsc{Brockwell, P.~J.} and \textsc{Davis, R.~A.} (2006).
\newblock \textit{Time Series: Theory and Methods (Second Edition)}.
\newblock Springer: NY.

\bibitem[{Bryson et~al.(2019)Bryson, Vershynin and Zhao}]{BVZ19}
\textsc{Bryson, J.}, \textsc{Vershynin, R.} and \textsc{Zhao, H.} (2019).
\newblock Marchenko–pastur law with relaxed independence conditions.
\newblock \textit{http://arxiv.org/pdf/1912.12724v1.pdf} .

\bibitem[{Cai et~al.(2015)Cai, Cai and Zhang}]{CCZ15}
\textsc{Cai, T.}, \textsc{Cai, T.~T.} and \textsc{Zhang, A.} (2015).
\newblock Structured matrix completion with applications in genomic data
  integration.
\newblock \textit{Journal of American Statistical Association} .

\bibitem[{Candes and Recht(2009)}]{CR09}
\textsc{Candes, E.} and \textsc{Recht, B.} (2009).
\newblock Exact matrix completion via convex optimization.
\newblock \textit{Foundations of Computational Mathematics} \textbf{9}
  717–772.

\bibitem[{Candes and Tao(2010)}]{CT10}
\textsc{Candes, E.} and \textsc{Tao, T.} (2010).
\newblock The power of convex relaxation: near-optimal matrix completion.
\newblock \textit{IEEE Trans. Inform. Theory} \textbf{56} 2053–2080.

\bibitem[{Carroll et~al.(2006)Carroll, Ruppert, Stefanski and
  Crainiceanu}]{carr:rupp:2006}
\textsc{Carroll, R.}, \textsc{Ruppert, D.}, \textsc{Stefanski, L.} and
  \textsc{Crainiceanu, C.~M.} (2006).
\newblock \textit{Measurement Error in Nonlinear Models (Second Edition)}.
\newblock Chapman \& Hall.

\bibitem[{Carroll et~al.(1993)Carroll, Gail and Lubin}]{CGL93}
\textsc{Carroll, R.~J.}, \textsc{Gail, M.~H.} and \textsc{Lubin, J.~H.} (1993).
\newblock Case-control studies with errors in predictors.
\newblock \textit{Journal of American Statistical Association} \textbf{88} 177
  -- 191.

\bibitem[{Chen et~al.(1998)Chen, Donoho and Saunders}]{Chen:Dono:Saun:1998}
\textsc{Chen, S.}, \textsc{Donoho, D.} and \textsc{Saunders, M.} (1998).
\newblock Atomic decomposition by basis pursuit.
\newblock \textit{SIAM Journal on Scientific and Statistical Computing}
  \textbf{20} 33--61.

\bibitem[{Cressie and Wikle(2011)}]{CW11}
\textsc{Cressie, N.} and \textsc{Wikle, C.} (2011).
\newblock \textit{Statistics for Spatio-Temporal Data}.
\newblock Wiley.

\bibitem[{Dawid(1981)}]{Dawid81}
\textsc{Dawid, A.~P.} (1981).
\newblock Some matrix-variate distribution theory: Notational considerations
  and a bayesian application.
\newblock \textit{Biometrika} \textbf{68} 265--274.

\bibitem[{Dempster et~al.(1977)Dempster, Laird and Rubin}]{DLR77}
\textsc{Dempster, A.}, \textsc{Laird, N.} and \textsc{Rubin, D.} (1977).
\newblock Maximum likelihood from incomplete data via the em algorithm.
\newblock \textit{Journal of the Royal Statistical Society, Series B}
  \textbf{39} 1--38.

\bibitem[{Duncan and Pearson(1991)}]{duncan:91}
\textsc{Duncan, G.} and \textsc{Pearson, R.} (1991).
\newblock Enhancing access to microdata while protecting confidentiality:
  Prospects for the future.
\newblock \textit{Statistical Science} \textbf{6} 219--232.

\bibitem[{Friedman et~al.(2008)Friedman, Hastie and Tibshirani}]{FHT07}
\textsc{Friedman, J.}, \textsc{Hastie, T.} and \textsc{Tibshirani, R.} (2008).
\newblock Sparse inverse covariance estimation with the graphical {L}asso.
\newblock \textit{Biostatistics} \textbf{9} 432--441.

\bibitem[{Fuller(1987)}]{Full:1987}
\textsc{Fuller, W.~A.} (1987).
\newblock \textit{Measurement error models}.
\newblock John Wiley and Sons.

\bibitem[{Greenewald et~al.(2019)Greenewald, Zhou and Hero}]{GZH19}
\textsc{Greenewald, K.}, \textsc{Zhou, S.} and \textsc{Hero, A.} (2019).
\newblock The {T}ensor graphical {L}asso ({T}era{L}asso).
\newblock \textit{Journal of the Royal Statistical Society: Series B
  (Statistical Methodology)} \textbf{81} 901--931.

\bibitem[{Gupta and Varga(1992)}]{GV92}
\textsc{Gupta, A.} and \textsc{Varga, T.} (1992).
\newblock Characterization of matrix variate normal distributions.
\newblock \textit{Journal of Multivariate Analysis} \textbf{41} 80--88.

\bibitem[{Hanson and Wright(1971)}]{HW71}
\textsc{Hanson, D.~L.} and \textsc{Wright, E.~T.} (1971).
\newblock A bound on tail probabilities for quadratic forms in independent
  random variables.
\newblock \textit{Ann. Math. Statist.} \textbf{42} 1079--1083.

\bibitem[{Honaker and King(2010)}]{HK10}
\textsc{Honaker, J.} and \textsc{King, G.} (2010).
\newblock What to do about missing values in time-series cross-section data.
\newblock \textit{American Journal of Political Science} \textbf{54} 561--581.

\bibitem[{Hornstein et~al.(2019)Hornstein, Fan, Shedden and Zhou}]{Horns19}
\textsc{Hornstein, M.}, \textsc{Fan, R.}, \textsc{Shedden, K.} and
  \textsc{Zhou, S.} (2019).
\newblock Joint mean and covariance estimation for unreplicated matrix-variate
  data.
\newblock \textit{Journal of the American Statistical Association (Theory and
  Methods)} \textbf{114} 682--696.

\bibitem[{Hwang(1986)}]{HWang86}
\textsc{Hwang, J.~T.} (1986).
\newblock Multiplicative errors-in-variables models with applications to recent
  data released by the u.s. department of energy.
\newblock \textit{Journal of American Statistical Association} \textbf{81}
  680--688.

\bibitem[{Iturria et~al.(1999)Iturria, Carroll and Firth}]{ICF99}
\textsc{Iturria, S.~J.}, \textsc{Carroll, R.~J.} and \textsc{Firth, D.} (1999).
\newblock Polynomial regression and estimating functions in the presence of
  multiplicative measurement error.
\newblock \textit{Journal of the Royal Statistical Society, Series B,
  Methodological} \textbf{61} 547--561.

\bibitem[{Kalaitzis et~al.(2013)Kalaitzis, Lafferty, Lawrence and
  Zhou}]{KLLZ13}
\textsc{Kalaitzis, A.}, \textsc{Lafferty, J.}, \textsc{Lawrence, N.} and
  \textsc{Zhou, S.} (2013).
\newblock The bigraphical lasso.
\newblock In \textit{Proceedings of The 30th International Conference on
  Machine Learning {ICML}-13}.

\bibitem[{Lauritzen(1996)}]{laur96}
\textsc{Lauritzen, S.~L.} (1996).
\newblock \textit{Graphical Models}.
\newblock Oxford University Press.

\bibitem[{Leng and Tang(2012)}]{LT12}
\textsc{Leng, C.} and \textsc{Tang, C.} (2012).
\newblock Sparse matrix graphical models.
\newblock \textit{Journal of American Statistical Association} \textbf{107}
  1187--1200.

\bibitem[{Little and Rubin(2002)}]{LR02}
\textsc{Little, R. J.~A.} and \textsc{Rubin, D.~B.} (2002).
\newblock \textit{Statistical Analysis with Missing Data (Second Edition)}.
\newblock John Wiley and Sons.

\bibitem[{Loh and Wainwright(2012{\natexlab{a}})}]{LW12b}
\textsc{Loh, P.} and \textsc{Wainwright, M.} (2012{\natexlab{a}}).
\newblock Corrupted and missing predictors: Minimax bounds for high-dimensional
  linear regression.
\newblock In \textit{Proceedings of the IEEE International Symposium on
  Information Theory (ISIT)}.

\bibitem[{Loh and Wainwright(2012{\natexlab{b}})}]{LW12}
\textsc{Loh, P.} and \textsc{Wainwright, M.} (2012{\natexlab{b}}).
\newblock High-dimensional regression with noisy and missing data: Provable
  guarantees with nonconvexity.
\newblock \textit{The Annals of Statistics} \textbf{40} 1637--1664.

\bibitem[{Loh and Wainwright(2012{\natexlab{c}})}]{LW12supp}
\textsc{Loh, P.} and \textsc{Wainwright, M.} (2012{\natexlab{c}}).
\newblock Supplementary material for: High-dimensional regression with noisy
  and missing data: Provable guarantees with nonconvexity .

\bibitem[{Meinshausen and B\"{u}hlmann(2006)}]{MB06}
\textsc{Meinshausen, N.} and \textsc{B\"{u}hlmann, P.} (2006).
\newblock High dimensional graphs and variable selection with the {L}asso.
\newblock \textit{Annals of Statistics} \textbf{34} 1436--1462.

\bibitem[{Milman and Schechtman(1986)}]{MS86}
\textsc{Milman, V.~D.} and \textsc{Schechtman, G.} (1986).
\newblock \textit{Asymptotic Theory of Finite Dimensional Normed Spaces.
  Lecture Notes in Mathematics 1200}.
\newblock Springer.

\bibitem[{Pigott(2001)}]{Pig01}
\textsc{Pigott, T.~D.} (2001).
\newblock A review of methods for missing data.
\newblock \textit{Educational Research and Evaluation} \textbf{7} 353--383.

\bibitem[{Plan et~al.(2016)Plan, Vershynin and Yudovina}]{PVY16}
\textsc{Plan, Y.}, \textsc{Vershynin, R.} and \textsc{Yudovina, E.} (2016).
\newblock High-dimensional estimation with geometric constraints.
\newblock \textit{Information and Inference} \textbf{0} 1–40.

\bibitem[{Raskutti et~al.(2010)Raskutti, Wainwright and Yu}]{RWY10}
\textsc{Raskutti, G.}, \textsc{Wainwright, M.} and \textsc{Yu, B.} (2010).
\newblock Restricted nullspace and eigenvalue properties for correlated
  gaussian designs.
\newblock \textit{Journal of Machine Learning Research} \textbf{11} 2241--2259.

\bibitem[{Ravikumar et~al.(2011)Ravikumar, Wainwright, Raskutti and
  Yu}]{RWRY08}
\textsc{Ravikumar, P.}, \textsc{Wainwright, M.}, \textsc{Raskutti, G.} and
  \textsc{Yu, B.} (2011).
\newblock High-dimensional covariance estimation by minimizing
  $\ell_1$-penalized log-determinant divergence.
\newblock \textit{Electronic Journal of Statistics} \textbf{4} 935--980.

\bibitem[{Ren et~al.(2015)Ren, Sun, Zhang and Zhou}]{RSZZ15}
\textsc{Ren, Z.}, \textsc{Sun, T.}, \textsc{Zhang, C.-H.} and \textsc{Zhou,
  H.~H.} (2015).
\newblock Asymptotic normality and optimalities in estimation of large gaussian
  graphical model.
\newblock \textit{Annals of Statistics} \textbf{43} 991--1026.

\bibitem[{Rosenbaum and Tsybakov(2010)}]{RT10}
\textsc{Rosenbaum, M.} and \textsc{Tsybakov, A.} (2010).
\newblock Sparse recovery under matrix uncertainty.
\newblock \textit{The Annals of Statistics} \textbf{38} 2620--2651.

\bibitem[{Rosenbaum and Tsybakov(2013)}]{RT13}
\textsc{Rosenbaum, M.} and \textsc{Tsybakov, A.} (2013).
\newblock Improved matrix uncertainty selector.
\newblock \textit{IMS Collections} \textbf{9} 276--290.

\bibitem[{Rothman et~al.(2008)Rothman, Bickel, Levina and Zhu}]{RBLZ08}
\textsc{Rothman, A.}, \textsc{Bickel, P.}, \textsc{Levina, E.} and \textsc{Zhu,
  J.} (2008).
\newblock Sparse permutation invariant covariance estimation.
\newblock \textit{Electronic Journal of Statistics} \textbf{2} 494--515.

\bibitem[{Rudelson(2013)}]{Rudelson13}
\textsc{Rudelson, M.} (2013).
\newblock Private communication.

\bibitem[{Rudelson and Vershynin(2013)}]{RV13}
\textsc{Rudelson, M.} and \textsc{Vershynin, R.} (2013).
\newblock {H}anson-{W}right inequality and sub-gaussian concentration.
\newblock \textit{Electronic Communications in Probability} \textbf{18} 1--9.

\bibitem[{Rudelson and Zhou(2013)}]{RZ13}
\textsc{Rudelson, M.} and \textsc{Zhou, S.} (2013).
\newblock Reconstruction from anisotropic random measurements.
\newblock \textit{IEEE Transactions on Information Theory} \textbf{59}
  3434--3447.

\bibitem[{Rudelson and Zhou(2017)}]{RZ17}
\textsc{Rudelson, M.} and \textsc{Zhou, S.} (2017).
\newblock Errors-in-variables models with dependent measurements.
\newblock \textit{Electron. J. Statist.} \textbf{11} 1699--1797.

\bibitem[{Sigrist et~al.(2015)Sigrist, K\"{u}nsch and Werner}]{SKS15}
\textsc{Sigrist, F.}, \textsc{K\"{u}nsch, K.} and \textsc{Werner, S.} (2015).
\newblock Stochastic partial differential equation based modelling of large
  space-time data sets.
\newblock \textit{Journal of the Royal Statistical Society: Series B
  (Statistical Methodology)} \textbf{77} 3--33.

\bibitem[{Smith et~al.(2003)Smith, Kolenikov and Cox}]{Smith03}
\textsc{Smith, R.}, \textsc{Kolenikov, S.} and \textsc{Cox, L.} (2003).
\newblock Spatiotemporal modeling of {PM}$_{2.5}$ data with missing values.
\newblock \textit{Journal of Geophysical Research} \textbf{108}.

\bibitem[{St\"{a}dler and B\"{u}hlmann(2012)}]{SB12}
\textsc{St\"{a}dler, N.} and \textsc{B\"{u}hlmann, P.} (2012).
\newblock Missing values: sparse inverse covariance estimation and an extension
  to sparse regression.
\newblock \textit{Statistics and Computing} \textbf{22} 219–235.

\bibitem[{St\"{a}dler et~al.(2014)St\"{a}dler, Stekhoven and
  B\"{u}hlmann}]{SSB14}
\textsc{St\"{a}dler, N.}, \textsc{Stekhoven, D.~J.} and \textsc{B\"{u}hlmann,
  P.} (2014).
\newblock Pattern alternating maximization algorithm for missing data in
  high-dimensional problems.
\newblock \textit{Journal of Machine Learning Research} \textbf{15} 1903--1928.

\bibitem[{Talagrand(1995)}]{Tal95}
\textsc{Talagrand, M.} (1995).
\newblock Sections of smooth convex bodies via majorizing measures.
\newblock \textit{Acta. Math.} \textbf{175} 273--300.

\bibitem[{Tibshirani(1996)}]{Tib96}
\textsc{Tibshirani, R.} (1996).
\newblock Regression shrinkage and selection via the {L}asso.
\newblock \textit{J. Roy. Statist. Soc. Ser. B} \textbf{58} 267--288.

\bibitem[{Tsiligkaridis et~al.(2013)Tsiligkaridis, Hero and Zhou}]{THZ13}
\textsc{Tsiligkaridis, T.}, \textsc{Hero, A.} and \textsc{Zhou, S.} (2013).
\newblock On convergence of kronecker graphical lasso algorithms.
\newblock \textit{IEEE Transactions on Signal Processing} \textbf{61} 1743 --
  1755.

\bibitem[{van~de Geer and Buhlmann(2009)}]{GB09}
\textsc{van~de Geer, S.} and \textsc{Buhlmann, P.} (2009).
\newblock On the conditions used to prove oracle results for the lasso.
\newblock \textit{Electronic Journal of Statistics} \textbf{3} 1360--1392.

\bibitem[{Vershynin(2018)}]{Vers18}
\textsc{Vershynin, R.} (2018).
\newblock \textit{High-Dimensional Probability: An Introduction with
  Applications in Data Science}.
\newblock Cambridge University Press.

\bibitem[{Wainwright(2019)}]{Wain19}
\textsc{Wainwright, M.} (2019).
\newblock \textit{High-Dimensional Statistics: A Non-Asymptotic Viewpoint}.
\newblock Cambridge: Cambridge University Press.

\bibitem[{Wright(1973)}]{HW73}
\textsc{Wright, E.~T.} (1973).
\newblock A bound on tail probabilities for quadratic forms in independent
  random variables whose distributions are not necessarily symmetric.
\newblock \textit{Ann. Probab.} \textbf{1} 1068--1070.

\bibitem[{Yuan(2010)}]{Yuan10}
\textsc{Yuan, M.} (2010).
\newblock High dimensional inverse covariance matrix estimation via linear
  programming.
\newblock \textit{Journal of Machine Learning Research} \textbf{11} 2261--2286.

\bibitem[{Yuan and Lin(2007)}]{YL07}
\textsc{Yuan, M.} and \textsc{Lin, Y.} (2007).
\newblock Model selection and estimation in the gaussian graphical model.
\newblock \textit{Biometrika} \textbf{94} 19--35.

\bibitem[{Zhou(2014)}]{Zhou14a}
\textsc{Zhou, S.} (2014).
\newblock Gemini: Graph estimation with matrix variate normal instances.
\newblock \textit{Annals of Statistics} \textbf{42} 532--562.

\bibitem[{Zhou(2019)}]{Zhou19}
\textsc{Zhou, S.} (2019).
\newblock Sparse {H}anson-{W}right inequalities for subgaussian quadratic
  forms.
\newblock \textit{Bernoulli} \textbf{25} 1603--1639.

\bibitem[{Zhou et~al.(2009)Zhou, Lafferty and Wasserman}]{ZLW09}
\textsc{Zhou, S.}, \textsc{Lafferty, J.} and \textsc{Wasserman, L.} (2009).
\newblock Compressed and privacy sensitive sparse regression.
\newblock \textit{IEEE Transactions on Information Theory} \textbf{55}
  846--866.

\bibitem[{Zhou et~al.(2010)Zhou, Lafferty and Wasserman}]{ZLW08}
\textsc{Zhou, S.}, \textsc{Lafferty, J.} and \textsc{Wasserman, L.} (2010).
\newblock Time varying undirected graphs.
\newblock \textit{Machine Learning} \textbf{80} 295--319.

\bibitem[{Zhou et~al.(2011)Zhou, R\"{u}timann, Xu and B\"{u}hlmann}]{ZRXB11}
\textsc{Zhou, S.}, \textsc{R\"{u}timann, P.}, \textsc{Xu, M.} and
  \textsc{B\"{u}hlmann, P.} (2011).
\newblock High-dimensional covariance estimation based on {G}aussian graphical
  models.
\newblock \textit{Journal of Machine Learning Research} \textbf{12} 2975--3026.

\end{thebibliography}
\end{document}